\documentclass[reqno]{memo-l}
\usepackage{amssymb,pifont}
\usepackage{pgf,pgfautomata,tikz}
\usepackage{euscript} 
\usepackage{extarrows}
\usepackage{mathrsfs} 
\usepackage{units}   
\usepackage{color}
\usepackage{mathbbol}
\usepackage[all]{xy}
\usepackage{graphicx}
\usepackage{subfigure}

\makeindex

\newcommand*{\idx}[2]{\index{{\hspace{-2ex}\color{white}#2}{#1}}{#1}}

\makeatletter
\@namedef{subjclassname@2010}{%
\textup{2010} Mathematics Subject Classification} \makeatother

\newtheorem{thm}{Theorem}
\newtheorem{cor}[thm]{Corollary}
\newtheorem{lem}[thm]{Lemma}
\newtheorem{pro}[thm]{Proposition}

\newtheorem*{thm*}{Theorem}
\newtheorem{opq}[thm]{Problem}
\newtheorem*{opq*}{Problem}
\newtheorem*{cau}{Caution}

\theoremstyle{remark}
\newtheorem{rem}[thm]{Remark}

\theoremstyle{definition}
\newtheorem{exa}[thm]{Example}
\newtheorem{preexa}[thm]{Procedure}

\DeclareMathOperator{\lin}{lin}
\DeclareMathOperator{\D}{d\hspace{-0.1ex}}

\DeclareMathOperator{\dzii}{{\mathsf{Chi}}}

\DeclareMathOperator{\koo}{{\mathsf{root}}}

\DeclareMathOperator{\paa}{{\mathsf{par}}}

\DeclareMathOperator{\bigU}{{\bigsqcup}} 

\newcommand*{\ascr}{\mathscr A}
\newcommand*{\at}[1]{\mathsf{At}(#1)}
\newcommand*{\atb}{{\ascr \otimes \borel{\rbb_+}}}
\newcommand*{\ats}{\mathrm{At}}
\newcommand*{\borel}[1]{{\mathfrak B}(#1)}
\newcommand*{\bscr}{\mathscr B}
\newcommand*{\card}[1]{\mathrm{card}(#1)}
\newcommand*{\cbb}{\mathbb C}
\newcommand*{\ceb}[1]{\mathsf{E}(#1;\bscr,\nu)}
\newcommand*{\cen}[1]{\mathsf{E}_{\phi^{#1}}}
\newcommand*{\cfw}{C_{\phi,w}}
\newcommand*{\cfwm}{C_{\phi,|w|}}
\newcommand*{\cfww}{C_{\phi,\widetilde w}}
\newcommand*{\cpv}{C_{\psi,v}}
\newcommand*{\ctfw}{C_{\tilde\phi,w}}
\newcommand*{\CFW}{C_{\varPhi,W}}
\newcommand*{\cscr}{\mathscr{C}}

\newcommand*{\dscr}{\mathscr{D}}
\newcommand*{\dz}[1]{{\mathscr D}(#1)}
\newcommand*{\dzi}[1]{\dzii(#1)}
\newcommand*{\dziplus}[1]{\dzii_\lambdab^+(#1)}
\newcommand*{\dzn}[1]{{\mathscr D}^\infty(#1)}
\newcommand*{\efw}{\mathsf{E}_{\phi,w}}

\newcommand*{\efww}{\mathsf{E}_{\phi,\widetilde w}}
\newcommand*{\etfw}{\mathsf{E}_{\tilde\phi,w}}
\newcommand*{\escr}{\mathscr E}
\newcommand*{\fscr}{\mathscr{F}}

\newcommand*{\Ge}{\geqslant}

\newcommand*{\hh}{\mathcal H}
\newcommand*{\hscr}{\mathscr H}
\newcommand*{\hfw}{\mathsf{h}_{\phi,w}}
\newcommand*{\hfwm}{\mathsf{h}_{\phi,|w|}}
\newcommand*{\htfw}{\mathsf{h}_{\tilde\phi,w}}
\newcommand*{\hfwi}[1]{\mathsf{h}_{\phi_{#1},w_{#1}}}
\newcommand*{\hfww}{\mathsf{h}_{\phi,\widetilde w}}
\newcommand*{\hfwn}[1]{\mathsf{h}_{\phi^{#1},\widehat w_{#1}}}
\newcommand*{\HFW}{\mathsf{h}_{\varPhi,W}}
\newcommand*{\HtFW}{\mathsf{h}_{\tilde\varPhi,W}}

\newcommand*{\id}{\mathrm{id}}
\newcommand*{\is}[2]{\langle#1,#2\rangle}
\newcommand*{\jd}[1]{\mathscr N(#1)}
\newcommand*{\kk}{\mathcal K}

\newcommand*{\lambdab}{{\boldsymbol\lambda}}

\newcommand*{\Le}{\leqslant}

\newcommand*{\nbb}{\mathbb N}

\newcommand*{\ogr}[1]{\boldsymbol B(#1)}
\newcommand*{\ob}[1]{{\mathscr R}(#1)}
\newcommand*{\pa}[1]{\paa(#1)}

\newcommand*{\pscr}{{\mathscr P}}

\newcommand*{\qscr}{\mathscr{Q}}
\newcommand*{\rbb}{\mathbb R}
\newcommand*{\rbop}{{\overline{\rbb}_+}}
\newcommand*{\slam}{S_{\boldsymbol \lambda}}
\newcommand*{\smalloplus}{\raise0pt
\hbox{$\scriptscriptstyle \oplus$}}

\newcommand*{\supp}[1]{\mathrm{supp}\,#1}
\newcommand*{\tcal}{{\mathscr T}}

\newcommand*{\zbb}{\mathbb Z}

\begin{document}
   \title[Unbounded Weighted Composition
Operators in $L^2$-Spaces]{Unbounded Weighted
Composition Operators \\ in $L^2$-Spaces}
   \author[P.\ Budzy\'{n}ski]{Piotr Budzy\'{n}ski}
   \address{Katedra Zastosowa\'{n} Matematyki,
Uniwersytet Rolniczy w Krako\-wie, ul.\ Balicka 253c,
PL-30198 Krak\'ow}
\email{piotr.budzynski@ur.krakow.pl}
   \author[Z.\ J.\ Jab{\l}o\'nski]{Zenon Jan
Jab{\l}o\'nski}
   \address{Instytut Matematyki,
Uniwersytet Jagiello\'nski, ul.\ \L ojasiewicza 6, PL-30348 Kra\-k\'ow,
Poland} \email{Zenon.Jablonski@im.uj.edu.pl}
   \author[I.\ B.\ Jung]{Il Bong Jung}
   \address{Department of Mathematics, Kyungpook National University,
Da\-egu 41566, Korea} \email{ibjung@knu.ac.kr}
   \author[J.\ Stochel]{Jan Stochel}
\address{Instytut Matematyki, Uniwersytet
Jagiello\'nski, ul.\ \L ojasiewicza 6, PL-30348 Kra\-k\'ow, Poland}
\email{Jan.Stochel@im.uj.edu.pl}
   \thanks{The research of the
second and fourth authors was supported by the NCN
(National Science Center), decision No.
DEC-2013/11/B/ST1/03613. The research of the third
author was supported by Basic Science Research Program
through the National Research Foundation of Korea
(NRF) funded by the Ministry of Science, ICT and
future Planning (2015R1A2A2A01006072).}
   \begin{abstract}
Bounded and unbounded weighted composition operators
on $L^2$ spaces over $\sigma$-finite measure spaces
are investigated. A variety of questions related to
seminormality of such operators are discussed.
   \end{abstract}
   \maketitle

   \setcounter{tocdepth}{2} \tableofcontents

   \chapter{Preliminaries}
   \section{Introduction}
The classical Banach-Stone theorem (see
\cite[Th\'{e}or\`{e}me XI.3]{Ban} and \cite{Ston}, see
also \cite[Theorem 2.1.1]{Fl-Jam}) states that if $X$
and $Y$ are compact Hausdorff topological spaces and
$A\colon C(X) \to C(Y)$ is a surjective linear
isometry, then there exist a continuous function
$w\colon Y\to \rbb$ and a homeomorphism $\phi \colon
Y\to X$ such that $|w|\equiv 1$ and
   \begin{align} \label{mariu}
Af = w \cdot (f \circ \phi)
   \end{align}
for every $f \in C(X)$; here $C(Z)$ stands for the
Banach space of real valued continuous functions on a
compact Hausdorff topological space $Z$ equipped with
the supremum norm. This result provided a strong
motivation to study isometric operators between
function spaces (see e.g., the monographs
\cite{Fl-Jam,Fl-Jam2} and the references therein). It
also brought attention to the investigation of
operators of the form \eqref{mariu}, without assuming
that $|w|\equiv 1$, acting in spaces of analytic
functions (cf.\ \cite{Sha,Cow-MC95}) or in
$L^p$-spaces (cf.\ \cite{nor,sin-man}). Another
important source of motivation comes from the ergodic
theory where Koopman operators (which are isometric
operators of the form \eqref{mariu} with $w\equiv 1$)
play an important role (cf.\ \cite{Kren,EFH}).

Let $(X,\ascr,\mu)$ be a $\sigma$-finite measure space
and $\phi\colon X\to X$ and $w\colon X \to \cbb$ be
measurable transformations. A linear operator in
$L^2(\mu)$ with the domain
   \begin{align*}
\big\{f \in L^2(\mu)\colon w \cdot (f\circ \phi) \in
L^2(\mu)\big\}
   \end{align*}
acting in accordance with the formula
   \begin{align*}
f \longmapsto w \cdot (f \circ \phi)
   \end{align*}
is called a weighted composition operator with a
symbol $\phi$ and a weight $w$; it will be denoted by
$\cfw$. If $w\equiv 1$, then we call it a composition
operator and abbreviate $\cfw$ to $C_{\phi}$. The
reader is referred to Section \ref{Sec2.2} for the
discussion of when the so-defined operator is
well-defined. Note that to each weighted composition
operator $\cfw$ there corresponds the composition
operator $C_{\phi}$. The relationship between
well-definiteness of $\cfw$ and $C_{\phi}$ is
discussed in Section \ref{Sec7.1}.

Weighted composition operators (in $L^2$-spaces) turn
out to be interesting objects of Operator Theory. The
class of these operators includes multiplication
operators, partial (weighted and unweighted)
composition operators, weighted shifts on directed
trees, unilateral and bilateral weighted shifts and
their adjoints (see Section \ref{pco}). Much effort
was put into the investigation of bounded weighted
composition operators. In particular, under certain
restrictive assumptions (see below for more details),
characterizations of their selfadjointness, normality,
quasinormality, hyponormality and cohyponormality were
given (see
\cite{ca-jam-corr,ca-em-fl-nar,lam3,ca-hor}; see also
\cite{car1} for the discrete case and
\cite{wh,sin-kum,sin,ha-wh} for the case of
composition operators). Moreover, criteria for
subnormality and cosubnormality of bounded composition
operators were invented (see \cite{lam1,emb-lam3}).
Questions related to compactness, commutants and
spectra of weighted and unweighted composition
operators were studied as well (see
\cite{Rid,Pet,car2,Tak,Cha-Par-Poz,Par-Poz,Ma-Av,Sch};
see also \cite{Kam}, which was an inspiration for
\cite{Tak}). For more references and an overview of
results on bounded weighted composition operators the
reader is referred to \cite{sin-man}.

Until now, little was known about the properties of
unbounded weighted composition operators. The
questions of their hyponormality and cohyponormality
were investigated by Campbell and Hornor in
\cite{ca-hor}, still under restrictive assumptions
(see also \cite{jj} for the case of unbounded
cohyponormal composition operators with bimeasurable
symbols). The problem of finding reasonable criteria
for subnormality of unbounded weighted composition
operators is much more challenging. The celebrated
Lambert's characterizations of subnormality solve the
problem for bounded composition operators (see
\cite{lam1}; see also \cite{lam2}). None of them is
true for unbounded ones (see \cite{b-j-j-sC,j-j-s0}).
Successful attempts to solve this problem in the
unbounded case were undertaken by the present authors
for weighted shifts on directed trees (see
\cite{b-j-j-sA,b-j-j-sB}) and, very recently, for
composition operators (see \cite{b-j-j-sS}). The
solutions were given in terms of families of
probability measures satisfying the so-called
consistency condition.

In most papers concerning weighted composition
operators (including
\cite{ca-jam-corr,ca-em-fl-nar,lam3,ca-hor}), the
authors assume that the measure spaces under
consideration are complete and that the corresponding
composition operators are densely defined. This
enables them to use the conditional expectation
$\mathsf{E}(\,\cdot\,; \phi^{-1}(\ascr), \mu)$ with
respect to the $\sigma$-algebra $\phi^{-1}(\ascr)$ and
to regard a weighted composition operator $\cfw$ as
the product $M_w C_{\phi}$ of the operator $M_w$ of
multiplication by $w$ and the composition operator
$C_{\phi}$. In particular, such $\cfw$ is well-defined
(but not necessarily densely defined, see Example
\ref{ciach4}). Surprisingly, however, if the
aforementioned assumptions are dropped, it may even
happen that $\cfw$ is an isometry while $C_{\phi}$ is
not well-defined (see Example \ref{cohyp-count2}), or
that $\cfw$ is bounded whereas the conditional
expectation $\mathsf{E}(\,\cdot\,; \phi^{-1}(\ascr),
\mu)$ does not exist (see Example \ref{ciach2}). This
means that the approach proposed by these authors
excludes, a priori, a variety of weighted composition
operators.

In 1950 Halmos introduced the notion of a bounded
subnormal operator and gave its first characterization
(cf.\ \cite{hal1}), which was successively simplified
by Bram \cite{bra}, Embry \cite{emb} and Lambert
\cite{lam}. Neither of them is true for unbounded
operators (see \cite{Con} and \cite{StSz1,StSz2,StSz3}
for foundations of the theory of bounded and unbounded
subnormal operators). The only known general
characterizations of subnormality of unbounded
operators refer to semispectral measures or elementary
spectral measures (cf.\ \cite{bis,foi,FHSz}). They
seem to be useless in the context of particular
classes of operators. The other known criteria for
subnormality (with the exception of \cite{Sz4})
require the operator in question to have an invariant
domain (cf.\ \cite{StSz2,StSz,c-s-sz,Al-V}). In this
paper, we give a criterion for subnormality of densely
defined weighted composition operators in $L^2$-spaces
with no additional restrictions.

The question of subnormality of bounded weighted
composition operators was studied by Azimi in a recent
paper \cite{Azi}, still under the restrictive
assumptions mentioned above. Unfortunately, the author
uses an invalid formula for $\|\cfw^n f\|^2$ (see the
proof of \cite[Theorem 3.3]{Azi}). As a consequence,
most of the results in his paper are incorrect. The
reader is referred to Theorem \ref{gsms+} and
\eqref{poz1-1} for the correct ones.
   \begin{cau}
In this paper, weighted composition operators are
considered only in $L^2$-spaces. For brevity, we will
omit the expression ``in $L^2$-space''.
   \end{cau}
   \section{Notations and prerequisites}
We write \idx{$\zbb$}{01}, \idx{$\rbb$}{02} and
\idx{$\cbb$}{03} for the sets of integers, real
numbers and complex numbers, respectively. We denote
by \idx{$\nbb$}{04}, \idx{$\zbb_+$}{05} and
\idx{$\rbb_+$}{06} the sets of positive integers,
nonnegative integers and nonnegative real numbers,
respectively. Set \idx{$\rbop$}{07}$ = \rbb_+ \cup
\{\infty\}$. In what follows, we adhere to the
convention that
   \begin{align} \label{conv-j}
0 \cdot \infty = \infty \cdot 0 = 0, \quad \frac{1}{0}
= \infty, \quad \frac{0}{0} = 0 \quad \text{and} \quad
\frac{1}{\infty} = 0.
   \end{align}
We write \idx{$\delta_{i,j}$}{08} for the Kronecker's
delta. The expression ``a countable set'' means a
finite set or a countably infinite set. We denote by
\idx{$\card{X}$}{09} the cardinal number of a set $X$.
We put \idx{$\varDelta\vartriangle
\varDelta^\prime$}{10}$= (\varDelta \setminus
\varDelta^\prime) \cup (\varDelta^\prime \setminus
\varDelta)$ for subsets $\varDelta$ and
$\varDelta^\prime$ of $X$. Given subsets $\varDelta,
\varDelta_n$ of $X$, $n\in \nbb$, we write
\idx{$\varDelta_n \nearrow \varDelta$}{11} as $n\to
\infty$ if $\varDelta_n \subseteq \varDelta_{n+1}$ for
every $n\in \nbb$ and $\varDelta =
\bigcup_{n=1}^\infty \varDelta_n$. The notation
``\idx{$\bigU$}{12}'' is reserved to indicate pairwise
disjointness of union of sets. For a family $\cscr
\subseteq 2^X$, we denote by
\idx{$\sigma_X(\cscr)$}{13} the $\sigma$-algebra
generated by $\cscr$; we will abbreviate
$\sigma_X(\cscr)$ to \idx{$\sigma(\cscr)$}{14} if this
does not lead to ambiguity. We also put \idx{$\cscr
\cap A$}{15}$=\{C \cap A\colon C \in \cscr\}$ for
$A\in 2^X$. The characteristic function of a subset
$\varDelta$ of $X$ is denoted by
\idx{$\chi_\varDelta$}{16}. If no confusion can arise,
we write \idx{$\boldsymbol{1}$}{17} for $\chi_X$. If
$f$ is a $\cbb$-valued or an $\rbop$-valued function
on a set $X$, then we put $\{f=0\}=\{x\in X\colon
f(x)=0\}$ and $\{f\neq 0\}=\{x\in X\colon f(x)\neq
0\}$; if $f$ is an $\rbop$-valued function on $X$,
then we set $\{f> 0\}= \{x\in X\colon f(x) > 0\}$,
$\{f < \infty\}=\{x\in X\colon f(x) < \infty\}$,
$\{f=\infty\}=\{x\in X\colon f(x)=\infty\}$ and $\{0 <
f < \infty\} = \{f>0\} \cap \{f < \infty\}$. Given
functions $f,f_n\colon X \to \rbop$, $n\in \nbb$, we
write $f_n \nearrow f$ as $n\to \infty$ if the
sequence $\{f_n(x)\}_{n=1}^\infty$ is monotonically
increasing and converging to $f(x)$ for every $x \in
X$. All measures considered in this paper are assumed
to be nonnegative. To simplify the notation, we write
   \begin{align*}
\nu(x)=\nu(\{x\}), \quad x \in X,
   \end{align*}
whenever $\nu$ is a measure on a $\sigma$-algebra
$\ascr\subseteq 2^X$ such that $\{x\}\in \ascr$ for
all $x\in X$. If $(X,\ascr,\nu)$ is a measure space
and $\varDelta, \tilde \varDelta \in \ascr$ are such
that $\nu(\varDelta \setminus \tilde \varDelta)=0$,
then we write $\varDelta \subseteq \tilde \varDelta$
a.e.\ $[\nu]$; if $\varDelta \subseteq \tilde
\varDelta$ a.e.\ $[\nu]$ and $\tilde\varDelta
\subseteq \varDelta$ a.e.\ $[\nu]$ (equivalently
$\nu(\varDelta\vartriangle \tilde \varDelta)=0$), then
we write $\varDelta = \tilde \varDelta$ a.e.\ $[\nu]$.
Clearly, $\varDelta \subseteq \tilde \varDelta$ a.e.\
$[\nu]$ (resp., $\varDelta = \tilde \varDelta$ a.e.\
$[\nu]$) if and only if $\chi_{\varDelta} \Le
\chi_{\tilde \varDelta}$ a.e.\ $[\nu]$ (resp.,
$\chi_{\varDelta} = \chi_{\tilde \varDelta}$ a.e.\
$[\nu]$). Given two measures $\mu$ and $\nu$ on the
same $\sigma$-algebra, we write $\mu \ll \nu$ if $\mu$
is absolutely continuous with respect to $\nu$; if
this is the case, then $\frac{\D \mu}{\D\nu}$ stands
for the Radon-Nikodym derivative of $\mu$ with respect
to $\nu$ (provided it exists). The $\sigma$-algebra of
all Borel sets of a topological space $Z$ is denoted
by $\borel{Z}$. We write $\supp \nu$ for the closed
support of a Borel measure $\nu$ on $Z$ (provided it
exists). Given $t\in \rbb_+$, we denote by $\delta_t$
the Borel probability measure on $\rbb_+$ concentrated
at $t$. Let $(X,\ascr, \nu)$ be a measure space. If $p
\in [1,\infty)$, then $L^p(\nu)=L^p(X,\ascr, \nu)$
stands for the Banach space of all $p$-integrable
(with respect to $\nu$) complex functions on $X$
equipped with the standard $L^p$-norm. In turn,
$L^{\infty}(\nu)=L^{\infty}(X,\ascr, \nu)$ denotes the
Banach space of all $\nu$-essentially bounded complex
functions on $X$ equipped with the standard
$L^{\infty}$-norm. As usual, an element of $L^p(\nu)$
may be regarded either as an equivalence class of
functions or simply as a function, depending on
circumstances. For $p\in [1,\infty]$, we denote by
$L_+^p(\nu)$ the convex cone $\{f \in L^p(\nu)\colon f
\Ge 0 \text{ a.e.\ $[\nu]$}\}$. The space $L^2(\nu)$
is regarded as a Hilbert space equipped with the
standard inner product and the corresponding norm
denoted by \mbox{$\|\cdot\|=\|\cdot\|_{\nu}$}. If $X$
is an nonempty set, then $\ell^2(X)$ will be
identified with the Hilbert space $L^2(X,2^X,\mu)$,
where $\mu$ is the counting measure on $X$. In what
follows, $\cbb[t]$ stands for the ring of all complex
polynomials in one real variable $t$.

The following lemma is a direct consequence of \cite[Proposition
I-6-1]{Nev} and \cite[Theorem 1.3.10]{Ash}.
   \begin{lem} \label{2miary}
Let $\pscr$ be a semi-algebra of subsets of a set $X$
and $\nu_1, \nu_2$ be measures on $\sigma(\pscr)$ such
that $\nu_1(\varDelta) = \nu_2(\varDelta)$ for all
$\varDelta \in \pscr$. Suppose there exists a sequence
$\{\varDelta_n\}_{n=1}^\infty \subseteq \pscr$ such
that $\varDelta_n \nearrow X$ as $n\to \infty$ and
$\nu_1(\varDelta_k) < \infty$ for every $k \in \nbb$.
Then $\nu_1 = \nu_2$.
   \end{lem}
The proof of the next lemma is left to the reader.
   \begin{lem} \label{f=g}
If $(X,\ascr,\nu)$ is a $\sigma$-finite measure space
and $f,g$ are $\ascr$-measurable complex functions on
$X$ such that $\int_{\varDelta} |f| \D \nu < \infty$,
$\int_{\varDelta} |g| \D \nu < \infty$ and
$\int_{\varDelta} f \D \nu = \int_{\varDelta} g \D
\nu$ for every $\varDelta \in \ascr$ such that
$\nu(\varDelta) < \infty$, then $f=g$ a.e.\ $[\nu]$.
   \end{lem}
A sequence $\{\gamma_n\}_{n=0}^\infty \subseteq \rbb_+$ is said to be a
{\em Stieltjes moment sequence} if there exists a Borel measure $\nu$ on
$\rbb_+$ such that (we write $\int_0^\infty$ in place of $\int_{\rbb_+}$)
   \begin{align*}
\gamma_n = \int_0^\infty s^n \D \nu(s), \quad n \in \zbb_+;
   \end{align*}
such a $\nu$ is called a representing measure of
$\{\gamma_n\}_{n=0}^\infty$. We say that a Stieltjes
moment sequence is {\em determinate} if it has a
unique representing measure; otherwise, we call it
{\em indeterminate}. A Stieltjes moment sequence
$\{\gamma_n\}_{n=0}^\infty$ is called {\em
non-degenerate} if $\gamma_n \neq 0$ for all $n\in
\zbb_+$. Using \cite[Theorem 1.39]{Rud}, we obtain the
following.
   \begin{lem} \label{lemS3} If $\nu$ is a Borel
probability measure on $\rbb_+$, then $\int_0^\infty
t^k \D \nu(t) = 0$ for some $k \in \nbb$
$($equivalently, for every $k \in \nbb$$)$ if and only
if $\nu=\delta_0$.
   \end{lem}
A sequence $\{\gamma_n\}_{n=0}^\infty \subseteq
\rbb_+$ is said to be {\em positive definite} if
   \begin{align*}
\sum_{i,j=0}^n \gamma_{i+j} \alpha_i \bar\alpha_j \Ge 0, \quad
\alpha_0,\ldots, \alpha_n \in \cbb, \ n\in \zbb_+.
   \end{align*}
The Stieltjes theorem (see \cite[Theorem 6.2.5]{b-c-r}) states that
   \begin{align} \label{Sti}
   \begin{minipage}{70ex}
{\em a sequence $\{\gamma_n\}_{n=0}^\infty \subseteq
\rbb_+$ is a Stieltjes moment sequence if and only if
the sequences $\{\gamma_n\}_{n=0}^\infty$ and
$\{\gamma_{n+1}\}_{n=0}^\infty$ are positive
definite.}
   \end{minipage}
   \end{align}
Applying \eqref{Sti} and \cite[Exercise 4(e), Chapter 3]{Rud} (see also
\cite[p.\ 50]{fug} for the determinacy issue), we obtain the following
characterization of Stieltjes moment sequences with polynomial growth.
   \begin{align} \label{Stiogr}
   \begin{minipage}{70ex}
{\em If $\{\gamma_n\}_{n=0}^\infty \subseteq \rbb_+$
and $r \in \rbb_+$, then $\{\gamma_n\}_{n=0}^\infty$
is a Stieltjes moment sequence with a representing
measure whose closed support is contained in $[0,r]$
if and only if $\{\gamma_n\}_{n=0}^\infty$ and
$\{\gamma_{n+1}\}_{n=0}^\infty$ are positive definite
and $\gamma_n \Le c r^n$ for all $n \in \zbb_+$ and
for some $c \in \rbb_+$. Moreover, if one of these
equivalent conditions holds, then
$\{\gamma_n\}_{n=0}^\infty$ is determinate.}
   \end{minipage}
   \end{align}
The particular case of $r=1$ is of special interest.
Namely, a sequence $\{\gamma_n\}_{n=0}^\infty
\subseteq \rbb_+$ is called a {\em Hausdorff moment
sequence} if there exists a Borel measure $\nu$ on
$[0,1]$ such that
   \begin{align*}
\gamma_n = \int_{[0,1]} s^n \D \nu(s), \quad n \in
\zbb_+.
   \end{align*}
In view of \eqref{Stiogr}, each Hausdorff moment
sequence is determinate as a Stieltjes moment
sequence.

Let $A$ be an operator in a complex Hilbert space
$\hh$ (all operators considered in this paper are
linear). Denote by $\dz{A}$, $\jd{A}$, $\ob{A}$, $\bar
A$ and $A^*$ the domain, the kernel, the range, the
closure and the adjoint of $A$ (in case they exist)
respectively. Set $\dzn{A} = \bigcap_{n=0}^\infty
\dz{A^n}$ with $A^0=I$, where $I=I_{\hh}$ is the
identity operator on $\hh$. Members of $\dzn{A}$ are
called {\em $C^\infty$-vectors} of $A$. We write
$\is{\cdot}{\mbox{-}}_A$ and $\|\cdot\|_A$ for the
graph inner product and the graph norm of $A$, i.e.,
$\is{f}{g}_A = \is{f}{g} + \is{Af}{Ag}$ and $\|f\|_A^2
= \is{f}{f}_A$ for $f,g\in \dz{A}$. We also use the
norm $\|\cdot\|_{A;n}$ on $\dz{A^n}$ defined by
$\|f\|_{A;n}^2 = \sum_{j=0}^n \|A^j f\|^2$ for $f\in
\dz{A^n}$ and $n\in\zbb_+$. We say that a vector
subspace $\escr$ of $\dz{A}$ is a {\em core} for $A$
if $\escr$ is dense in $\dz{A}$ with respect to the
graph norm of $A$. It is well-known that if $A$ is
closable, then $\escr$ is a core for $A$ if and only
if $\bar A = \overline{A|_{\escr}}$. Given two
operators $A$ and $B$ in $\hh$, we write $A \subseteq
B$ if $\dz{A} \subseteq \dz{B}$ and $Af=Bf$ for all
$f\in \dz{A}$. An operator $A$ in $\hh$ is called {\em
positive} if $\is{Af}{f} \Ge 0$ for all $f\in \dz{A}$.
A densely defined operator $N$ in $\hh$ is said to be
{\em normal} if $N$ is closed and $N^*N=NN^*$ (or
equivalently if and only if $\dz{N}=\dz{N^*}$ and
$\|Nf\|=\|N^*f\|$ for all $f \in \dz{N}$, see
\cite{b-s}). We say that a densely defined operator
$A$ in $\hh$ is {\em formally normal} if $\dz{A}
\subseteq \dz{A^*}$ and $\|Af\|=\|A^*f\|$ for all $f
\in \dz{A}$. It is well-known that normal operators
are formally normal, but not conversely (cf.\
\cite{Cod1,Cod2}). A densely defined operator $S$ in
$\hh$ is said to be {\em selfadjoint} or {\em
symmetric} if $S=S^*$ or $S \subseteq S^*$,
respectively. Recall that symmetric operators may not
be selfadjoint and that each symmetric operator has a
selfadjoint extension possibly in a larger Hilbert
space (see \cite[Theorem 1 in Appendix I.2]{a-g}; see
also \cite{b-s,Weid} for more information on symmetric
operators). A closed densely defined operator $Q$ in
$\hh$ is said to be {\em quasinormal} if $U |Q|
\subseteq |Q|U$, where $|Q|$ is the modulus of $Q$ and
$Q=U|Q|$ is the polar decomposition of $Q$ (see
\cite{bro,StSz2}). It was shown in \cite[Theorem
3.1]{j-j-s1} that
   \begin{align} \label{QQQ}
   \begin{minipage}{70ex}
{\em a closed densley defined operator $Q$ in $\hh$ is
quasinormal if and only if $Q|Q|^2 = |Q|^2Q.$}
   \end{minipage}
   \end{align}
We say that a densely defined operator $S$ in $\hh$ is
{\em subnormal} if there exist a complex Hilbert space
$\kk$ and a normal operator $N$ in $\kk$ such that
$\hh \subseteq \kk$ (isometric embedding), $\dz{S}
\subseteq \dz{N}$ and $Sf = Nf$ for all $f \in
\dz{S}$. A densely defined operator $A$ in $\hh$ is
called {\em hyponormal} if $\dz{A} \subseteq \dz{A^*}$
and $\|A^* f\| \Le \|Af\|$ for all $f \in \dz{A}$. We
say that a densely defined operator $A$ in $\hh$ {\em
cohyponormal} if $\dz{A^*} \subseteq \dz{A}$ and $\|A
f\| \Le \|A^*f\|$ for all $f \in \dz{A^*}$. If
additionally $A$ is closed, then, by the von Neumann
theorem, $A$ is cohyponormal if and only if $A^*$ is
hyponormal. As a consequence, a closed operator is
normal if and only if it is hyponormal and
cohyponormal. Operators which are either hyponormal or
cohyponormal are called {\em seminormal}. An operator
$A$ in $\hh$ is said to be {\em paranormal} if
$\|Af\|^2 \Le \|f\| \|A^2f\|$ for all $f \in
\dz{A^2}$. It is well-known that quasinormal operators
are subnormal (see \cite[Theorem 1]{bro} and
\cite[Theorem 2]{StSz2}), subnormal operators are
hyponormal (see \cite[Lemma 2.8]{ot-sch}) and
hyponormal operators are paranormal (see \cite[Lemma
3.1]{ot-sch}), but none of these implications can be
reversed in general (this can be seen by considering
weighted shifts on directed trees, see e.g.,
\cite{j-j-s}). As shown by Daniluk in \cite{dan}, a
paranormal operator may not be closable and the
closure of a closable paranormal operator may not be
paranormal; in both cases the operators in question
may have invariant domains. In what follows
$\ogr{\hh}$ stands for the $C^*$-algebra of all
bounded operators in $\hh$ whose domains are equal to
$\hh$. The linear span of a set $\fscr$ of vectors in
$\hh$ will be denoted by $\lin\fscr$.

The following lemma gives a necessary and sufficient condition for two
positive selfadjoint operators to be equal.
   \begin{lem} \label{useful}
Let $A$ and $B$ be positive selfadjoint operators in
$\hh$ such that $\dz{A} = \dz{B}$ and $\|Af\|=\|Bf\|$
for every $f \in \dz{A}$. Then $A=B$.
   \end{lem}
   \begin{proof}
Note that $\overline{\ob{A}} = \jd{A}^\perp = \jd{B}^\perp
=\overline{\ob{B}}$. It follows from our assumptions that there exists a
unique unitary operator $\tilde U \in \ogr{\overline{\ob{A}}}$ such that
$\tilde U A = B$. Then $U:=\tilde U \oplus I_{\jd{A}}$ is unitary and
$UA=B$. Hence $B^2 = B^*B = A^* U^* UA = A^2$, which, by uniqueness of
square roots, implies that $A=B$.
   \end{proof}
      \chapter{Preparatory Concepts}
This chapter introduces some concepts of measure
theory that will be useful for studying weighted
composition operators (including the Radon-Nikodym
derivative $\hfw$ and the conditional expectation
$\mathsf{E}(\,\cdot\,;\phi^{-1}(\ascr),\mu_w)$; see
Sections \ref{sec-2.1} and \ref{Sec2.4}). Weighted
composition operators are introduced and initially
investigated in Section \ref{Sec2.2}. Assorted classes
of weighted composition operators including classical
(unilateral and bilateral) weighted shifts and their
adjoints are discussed in Section \ref{pco}. The
adjoint and the polar decomposition of a weighted
composition operator are explicitly described in
Section \ref{Sec2.5}. The chapter is concluded with a
characterization of quasinormality of weighted
composition operators (see Theorem \ref{quain}). This
result is used in the proof of Theorem \ref{MAIN1},
which is the main criterion for subnormality of
weighted composition operators.
   \section{\label{sec-2.1}Measure-theory background}
   Let $(X, \ascr, \nu)$ be a measure space and let
$\bscr\subseteq \ascr$ be a $\sigma$-algebra. We say
that $\bscr$ is {\em relatively $\nu$-complete} if
$\ascr_0 \subseteq \bscr$, where $\ascr_0= \{\varDelta
\in \ascr\colon \nu(\varDelta)=0\}$ (see \cite[Chapter
8]{Rud}). The smallest relatively $\nu$-complete
$\sigma$-algebra containing $\bscr$, denoted by
$\bscr^\nu$ and called the {\em relative
$\nu$-completion} of $\bscr$, is equal to the
$\sigma$-algebra generated by $\bscr \cup \ascr_0$.
Moreover, we have
   \begin{align} \label{complascr}
\bscr^\nu = \{\varDelta \in \ascr \, | \, \exists \varDelta^\prime \in
\bscr \colon \nu(\varDelta \vartriangle \varDelta^\prime)=0\}.
   \end{align}
The $\bscr^\nu$-measurable functions are described in
\cite[Lemma 1, p.\ 169]{Rud}.

Let $(X,\ascr, \mu)$ be a $\sigma$-finite measure
space. We shall abbreviate the expressions ``almost
everywhere with respect to $\mu$'' and ``for
$\mu$-almost every $x$'' to ``a.e.\ $[\mu]$'' and
``for $\mu$-a.e.\ $x$'', respectively. We call a
mapping $\phi\colon X\to X$ a {\em transformation} of
$X$ and write $\phi^{-1}(\ascr) =
\{\phi^{-1}(\varDelta)\colon \varDelta \in \ascr\}$.
For $n \in \nbb$, we denote by $\phi^n$ the $n$-fold
composition of $\phi$ with itself. We write $\phi^0$
for the identity transformation $\id_X$ of $X$. Set
$\phi^{-n}(\varDelta) = (\phi^n)^{-1}(\varDelta)$ for
$\varDelta \in \ascr$ and $n \in \zbb_+$. A
transformation $\phi$ of $X$ is said to be {\em
$\ascr$-measurable} if $\phi^{-1}(\ascr) \subseteq
\ascr$. Clearly, if $\phi$ is $\ascr$-measurable, then
so is $\phi^n$ for every $n\in \zbb_+$. Let $\phi$ be
an $\ascr$-measurable transformation of $X$. If
$\nu\colon \ascr\to \rbop$ is a measure, then $\nu
\circ \phi^{-1}$ stands for the measure on $\ascr$
given by $\nu \circ \phi^{-1} (\varDelta) =
\nu(\phi^{-1} (\varDelta))$ for $\varDelta \in \ascr$.
Let $w$ be a complex $\ascr$-measurable function on
$X$. Define the measures $\mu^w, \mu_w\colon \ascr \to
\rbop$ by
   \begin{align*}
\mu^w(\varDelta) = \mu(\varDelta \cap \{w\neq 0\}) \text{ and }
\mu_w(\varDelta) = \int_{\varDelta} |w|^2 \D\mu \text{ for } \varDelta \in
\ascr.
   \end{align*}
Clearly, the measures $\mu^w$ and $\mu_w$ are
$\sigma$-finite and mutually absolutely continuous.
Moreover, if $u\colon X \to \cbb$ is
$\ascr$-measurable and $u=w$ a.e.\ $[\mu]$, then
$\mu(\{u\neq 0\} \vartriangle \{w\neq 0\}) = 0$,
$\mu^u=\mu^w$ and $\mu_u=\mu_w$. Note also that if
$\mathcal P$ is a property which a point $x \in X$ may
or may not have, then $\mathcal P$ holds a.e.\
$[\mu_w]$ if and only if $\mathcal P$ holds a.e.\
$[\mu]$ on $\{w\neq 0\}$.

If $\mu_w \circ \phi^{-1} \ll \mu$, then by the
Radon-Nikodym theorem (see \cite[Theorem 2.2.1]{Ash})
there exists a unique (up to a set of $\mu$-measure
zero) $\ascr$-measurable function $\hfw\colon X \to
\rbop$ such that
   \begin{align} \label{l1}
\mu_w \circ \phi^{-1}(\varDelta) = \int_{\varDelta}
\hfw \D \mu, \quad \varDelta \in \ascr.
   \end{align}
It follows from \cite[Theorem 1.6.12]{Ash} and
\cite[Theorem 1.29]{Rud} that for every
$\ascr$-measurable function $f \colon X \to \rbop$, or
for every $\ascr$-measurable function $f\colon X \to
\cbb$ such that $f\circ \phi \in L^1(\mu_w)$
(equivalently $f \, \hfw \in L^1(\mu)$),
   \begin{align} \label{l2}
\int_X f \circ \phi \D\mu_w = \int_X f \, \hfw \D \mu.
   \end{align}
In view of \eqref{l2}, $f\circ \phi \in L^1(\mu_w)$ if
and only if $f \hfw \in L^1(\mu)$. Note that the set
$\{\hfw > 0\}$ is determined up to a set of
$\mu$-measure zero.

To avoid the repetition, we gather the following two
assumptions which will be used frequently throughout
this paper.
   \begin{align} \label{stand1} \tag{$\mathrm{AS1}$}
   \begin{minipage}{60ex}
The triplet $(X,\ascr,\mu)$ is a $\sigma$-finite
measure space, $w$ is an $\ascr$-meas\-urable complex
function on $X$ and $\phi$ is an $\ascr$-measurable
transformation of $X$.
   \end{minipage}
   \end{align}
   \begin{align} \label{stand2} \tag{$\mathrm{AS2}$}
   \begin{minipage}{60ex}
The triplet $(X,\ascr,\mu)$ is a $\sigma$-finite
measure space, $w$ is an $\ascr$-meas\-urable complex
function on $X$ and $\phi$ is an $\ascr$-measurable
transformation of $X$ such that $\mu_w \circ \phi^{-1}
\ll \mu$.
   \end{minipage}
   \end{align}
Now we formulate a ``cancellation'' rule for the
operation of composing functions.
   \begin{lem} \label{nuklear}
Suppose \eqref{stand2} holds and $f$ and $g$ are
$\ascr$-measurable $\rbop$-valued or $\cbb$-valued
functions on $X$. Then
   \begin{align*}
f\circ \phi =g\circ \phi \text{ a.e.\ $[\mu_w]$} &
\iff \chi_{\{\hfw > 0\}} \cdot f = \chi_{\{\hfw
> 0\}}\cdot g \text{ a.e.\ $[\mu]$}
   \\
&\iff f=g \text{ a.e.\ $[\hfw\D\mu]$}
   \\
&\iff f=g \text{ a.e.\ $[\mu_w\circ \phi^{-1}]$.}
   \end{align*}
   \end{lem}
   \begin{proof}
This follows from the equality $\mu_w(\phi^{-1}(Y))
\overset{\eqref{l1}}= \int_Y \hfw \D \mu$, where $Y =
\{x \in X\colon f(x) \neq g(x)\}$.
   \end{proof}
The next lemma is written in the spirit of
\cite{ha-wh,Em-Lam-center,b-j-j-sC,b-j-j-sS}.
   \begin{lem} \label{lemS6}
Suppose \eqref{stand2} holds. Then the following
assertions are valid{\em :}
   \begin{enumerate}
   \item[(i)] $\hfw \circ \phi > 0$ a.e.\ $[\mu_w]$,
   \item[(ii)] if  $\hfw < \infty$ a.e.\ $[\mu]$,
then $\hfw \circ \phi < \infty$ a.e.\ $[\mu_w]$ and
   \begin{align}  \label{hle2}
\int_X \frac{f\circ \phi}{\hfw\circ \phi} \D\mu_w =
\int_{\{\hfw > 0\}} f \D\mu
   \end{align}
for every $\ascr$-measurable $f \colon X \to \rbop$.
   \end{enumerate}
   \end{lem}
   \begin{proof}
(i) This follows from the equalities
   \begin{align*}
\mu_w(\{\hfw\circ \phi = 0\}) = \int_X \chi_{\{\hfw =
0\}} \circ \phi \D \mu_w \overset{\eqref{l1}}=
\int_{\{\hfw = 0\}} \hfw \D\mu = 0.
   \end{align*}

(ii) Since $\mu(\{\hfw=\infty\})=0$, we get
   \begin{align*}
\mu_w(\{\hfw\circ \phi = \infty\}) = \int_X
\chi_{\{\hfw = \infty\}} \circ \phi \D \mu_w
\overset{\eqref{l1}}= \int_{\{\hfw = \infty\}} \hfw
\D\mu = 0.
   \end{align*}
This and (i) yield
   \begin{align*}
\int_X \frac{f\circ \phi}{\hfw\circ \phi} \D\mu_w & =
\int_{\{\hfw \circ \phi > 0\}} \frac{f\circ
\phi}{\hfw\circ \phi} \D\mu_w
   \\
&= \int_X \chi_{\{\hfw > 0\}} \circ \phi \cdot
\frac{f\circ \phi}{\hfw\circ \phi} \D\mu_w
   \\
&\overset{\eqref{l2}}= \int_{\{\hfw > 0\}} f \D\mu,
   \end{align*}
which completes the proof.
   \end{proof}
   \section{\label{Sec2.2}Invitation to weighted composition operators}
Let $(X,\ascr,\mu)$ be a $\sigma$-finite measure
space, $w$ be an $\ascr$-measurable complex function
on $X$ and $\phi$ be an $\ascr$-measur\-able
transformation of $X$. Denote by $L^2(\mu)$ the
complex Hilbert space of all square summable (with
respect to $\mu$) $\ascr$-measurable complex functions
on $X$ (with the standard inner product). By a {\em
weighted composition operator} in $L^2(\mu)$ we mean a
mapping $\cfw \colon L^2(\mu) \supseteq \dz{\cfw} \to
L^2(\mu)$ formally defined by
   \begin{align*}
\dz{\cfw} & = \{f \in L^2(\mu) \colon w \cdot (f\circ \phi) \in
L^2(\mu)\},
   \\
\cfw f & = w \cdot (f\circ \phi), \quad f \in \dz{\cfw}.
   \end{align*}
We call $\phi$ and $w$ the {\em symbol} and the {\em
weight} of $\cfw$ respectively. In general, such
operator may not be well-defined. To be more precise,
   \begin{align*}
   \begin{minipage}{70ex}
$\cfw$ is said to be {\em well-defined} if $w \cdot
(f\circ \phi)= w \cdot (f\circ \phi)$ a.e.\ $[\mu]$
whenever $f,g\colon X \to \cbb$ are $\ascr$-measurable
functions such that $f=g$ a.e.\ $[\mu]$ and
$f,g,w\cdot (f\circ \phi), w\cdot (g\circ \phi)\in
L^2(\mu)$.
   \end{minipage}
   \end{align*}
Below, we describe circumstances under which $\cfw$ is
well-defined, and show that the operator $\cfw$
remains unchanged if $w$ and $\phi$ are modified on
sets of measure zero.
   \begin{pro}\label{wco1}
Let $(X,\ascr,\mu)$ be a $\sigma$-finite measure space, $w$ be an
$\ascr$-measur\-able complex function on $X$ and $\phi$ be an
$\ascr$-measurable transformation of $X$. Then the following conditions
are equivalent{\em :}
   \begin{enumerate}
   \item[(i)] $\cfw$ is well-defined,
   \item[(ii)] $\mu^w \circ \phi^{-1} \ll \mu$,
   \item[(iii)] $\mu_w \circ \phi^{-1} \ll \mu$.
   \end{enumerate}
Moreover, if $\cfw$ is well-defined, $u\colon X \to
\cbb$ and $\psi\colon X \to X$ are $\ascr$-measurable,
$u=w$ a.e.\ $[\mu]$ and $\psi=\phi$ a.e.\ $[\mu_w]$,
then $C_{\psi,u}$ is well-defined, $C_{\psi,u}=\cfw$,
$\mu_u=\mu_w$, $\mathsf{h}_{\psi,u}=\hfw$ a.e.\
$[\mu]$ and
$(\psi^{-1}(\ascr))^{\mu_u}=(\phi^{-1}(\ascr))^{\mu_w}$.
   \end{pro}
   \begin{proof} First note that for any two functions
$f,g\colon X \to \cbb$,
   \begin{align} \label{aj1}
\{w\cdot (f\circ \phi) \neq w\cdot (g\circ \phi)\} =
\phi^{-1}(\{f\neq g\}) \cap \{w\neq 0\}.
   \end{align}

(i)$\Rightarrow$(ii) What is important here is that
$g\equiv 0$ has the property that $w \cdot f\circ \phi
\in L^2(\mu)$. Take $\varDelta \in \ascr$ such that
$\mu(\varDelta) = 0$. Then $\chi_{\varDelta} = 0$
a.e.\ $[\mu]$, and thus, by (i), $w \cdot
(\chi_{\varDelta} \circ \phi) = 0$ a.e.\ $[\mu]$.
Using \eqref{aj1} with $f=\chi_{\varDelta}$, we get
$\mu^w(\phi^{-1}(\varDelta))=0$.

(ii)$\Rightarrow$(i) If $f,g \colon X \to \cbb$ are
$\ascr$-measurable functions and $f=g$ a.e.\ $[\mu]$,
then $\mu(\{f\neq g\})=0$, which, by \eqref{aj1} and
(ii), gives $w\cdot (f\circ \phi) = w\cdot (g\circ
\phi)$ a.e.\ $[\mu]$. Hence $\cfw$ is well-defined.

(i)$\Leftrightarrow$(iii) Note that the measures
$\mu^w$ and $\mu_w$ are mutually absolutely continuous
and apply the equivalence (i)$\Leftrightarrow$(ii).

To justify the ``moreover'' part, note that $\mu_w=\mu_u$, and $\mu_w
\circ \phi^{-1} = \mu_u \circ \psi^{-1}$ because
   \begin{align*}
\mu_w(\psi^{-1}(\varDelta) \vartriangle \phi^{-1}(\varDelta)) = \int_X
|\chi_{\varDelta} \circ \psi - \chi_{\varDelta} \circ \phi|^2 \D\mu_w = 0,
\quad \varDelta \in \ascr.
   \end{align*}
This, together with \eqref{complascr}, completes the proof.
   \end{proof}
An inspection of the proof of Proposition \ref{wco1}
shows that for every $p \in (0,\infty]$, any of the
equivalent conditions (ii) and (iii) is necessary and
sufficient for $\cfw$ to be well-defined in
$L^p(\mu)$.

   Now we show that each weighted composition operator
$\cfw$ is closed. We also give a necessary and
sufficient condition for $\cfw$ to be bounded.
   \begin{pro} \label{lemS1}
Suppose \eqref{stand2} holds. Then the following
assertions are valid{\em :}
   \begin{enumerate}
   \item[(i)] $\dz{\cfw}  = L^2((1+\hfw)\D \mu)$,
   \item[(ii)] $\|f\|_{\cfw}^2 = \int_X |f|^2(1+\hfw)\D \mu$
for $f \in \dz{\cfw}$,
   \item[(iii)] $\overline{\dz{\cfw}} = \chi_{\{\hfw < \infty\}}
\cdot L^2(\mu)$,
   \item[(iv)] $\cfw$ is closed,
   \item[(v)] $\cfw \in \ogr{L^2(\mu)}$ if and
only if $\hfw \in L^\infty(\mu)$; if this is the case,
then $\|\cfw\|^2 = \|\hfw\|_{L^{\infty}(\mu)}$,
   \item[(vi)] $\cfw$ is the zero operator on
$L^2(\mu)$ $\iff$ $\hfw = 0$ a.e.\ $[\mu]$ $\iff$
$\mu_w=0$ $\iff$ $w=0$ a.e.\ $[\mu]$.
   \end{enumerate}
   \end{pro}
   \begin{proof}
In view of Proposition \ref{wco1}, $\cfw$ is
well-defined. If $f\colon X \to \cbb$ is
$\ascr$-measurable, then
   \begin{align*}
\int_X |f\circ \phi|^2 |w|^2 \D\mu = \int_X |f\circ
\phi|^2 \D\mu_w \overset{\eqref{l2}}= \int_X |f|^2
\hfw \D\mu,
   \end{align*}
which implies (i) and (ii), and thus (iv). The proofs
of (iii) and (v), which are straightforward
adaptations of those for composition operators (see
\cite[Eq.\ (3.8)]{b-j-j-sC} and \cite[Theorem 1]{nor}
respectively), are omitted. The assertion (vi) follows
from \eqref{l1} and (v).
   \end{proof}
For later use we single out the following fact whose
proof is left to the reader.
   \begin{lem} \label{aproks}
If \eqref{stand2} holds and $\hfw < \infty$ a.e.\
$[\mu]$, then there exists a sequence
$\{X_n\}_{n=1}^\infty \subseteq \ascr$ such that
$\mu(X_n) < \infty$ and $\hfw \Le n$ a.e.\ $[\mu]$ on
$X_n$ for every $n\in \nbb$, and $X_n \nearrow X$ as
$n\to \infty$.
   \end{lem}
The question of when $\cfw$ is densely defined is
answered below.
   \begin{pro} \label{lemS2}
If \eqref{stand2} holds, then the following conditions
are equivalent{\em :}
   \begin{enumerate}
   \item[(i)] $\cfw$ is densely defined,
   \item[(ii)] $\hfw < \infty$ a.e.\ $[\mu]$,
   \item[(iii)] $\mu_w \circ \phi^{-1}$ is
$\sigma$-finite,
   \item[(iv)] $\mu_w|_{\phi^{-1}(\ascr)}$  is
$\sigma$-finite.
   \end{enumerate}
   \end{pro}
   \begin{proof}
(i)$\Rightarrow$(ii) Apply Lemma \ref{lemS1}(iii) and
the assumption that $\mu$ is $\sigma$-finite.

(ii)$\Rightarrow$(iii) Let $\{X_n\}_{n=1}^\infty$ be
as in Lemma \ref{aproks}. Then
   \begin{align*}
\mu_w \circ \phi^{-1}(X_n) \overset{\eqref{l1}}=
\int_{X_n} \hfw \D \mu \Le n \mu(X_n) <\infty, \quad
n\in \nbb.
   \end{align*}
This yields (iii).

(iii)$\Rightarrow$(iv) Evident.

(iv)$\Rightarrow$(i) Let $\{X_n\}_{n=1}^\infty
\subseteq \ascr$ be a sequence such that
$\phi^{-1}(X_n) \nearrow X$ as $n\to \infty$ and
$\mu_w(\phi^{-1}(X_k)) < \infty$ for every $k\in
\nbb$. Without loss of generality we can assume that
$X_n \nearrow X_\infty := \bigcup_{k=1}^\infty X_k$ as
$n\to \infty$. It follows from \eqref{l1} that $\hfw <
\infty$ a.e.\ $[\mu]$ on $X_k$ for every $k\in \nbb$.
This implies that $\hfw < \infty$ a.e.\ $[\mu]$ on
$X_\infty$. Since $\phi^{-1}(X_n) \nearrow
\phi^{-1}(X_\infty) = X$ as $n\to \infty$, we get
$\phi^{-1}(X\setminus X_\infty) = \emptyset$. By
\eqref{l1}, we have $\int_{X\setminus X_\infty} \hfw
\D \mu = 0$, which yields $\hfw = 0$ a.e.\ $[\mu]$ on
$X\setminus X_\infty$. Hence, $\hfw < \infty$ a.e.\
$[\mu]$. Applying Proposition \ref{lemS1}(i) and
\cite[Lemma 12.1]{b-j-j-sC}, we obtain (i).
   \end{proof}
   \begin{cau}
To simplify terminology throughout the rest of the
paper, in saying that ``$\cfw$ is densely defined'',
we tacitly assume that $\cfw$ is well-defined.
   \end{cau}
Using \eqref{l2} we can describe the kernel of $\cfw$.
   \begin{lem} \label{jadro}
Suppose \eqref{stand2} holds. Then $\jd{\cfw}=
\chi_{\{\hfw=0\}} L^2(\mu)$. Moreover, $\jd{\cfw} =
\{0\}$ if and only if $\hfw > 0$ a.e.\ $[\mu]$.
   \end{lem}
Our next aim is to characterize weighted composition
operators satisfying the condition ``$\hfw > 0$ a.e.\
$[\mu_w]$'' that plays an essential role in our study.
In general, even composition operators do not satisfy
this condition (see Example \ref{pico}). Proposition
\ref{hsfd} below is an adaptation of \cite[Proposition
6.2]{b-j-j-sC} to our setting.
   \begin{pro} \label{hsfd}
Suppose \eqref{stand2} holds. Then the following
conditions are equivalent{\em :}
   \begin{enumerate}
   \item[(i)] $\chi_{\{w \neq 0\}} \jd{\cfw} = \{0\}$,
   \item[(ii)] $\mu(\{\hfw=0\}\cap \{w \neq
0\})=0$,
   \item[(iii)] $\hfw > 0$ a.e.\ $[\mu_w]$,
   \item[(iv)] $\chi_{\{\hfw=0\}} = \chi_{\{\hfw=0\}}
\circ \phi$ a.e.\ $[\mu_w]$.
   \end{enumerate}
Moreover, if $\cfw$ is densely defined, then any of the above conditions
is equivalent to the following one{\em :}
   \begin{enumerate}
   \item[(v)] $\chi_{\{w \neq 0\}} \jd{\cfw}
\subseteq \jd{\cfw^*}$.
   \end{enumerate}
   \end{pro}
   \begin{proof}
(i)$\Leftrightarrow$(ii) This follows from the $\sigma$-finiteness of
$\mu$ and Lemma \ref{jadro}.

(ii)$\Leftrightarrow$(iii) Clear.

(ii)$\Leftrightarrow$(iv) Since, by Lemma
\ref{lemS6}(i), $\chi_{\{\hfw=0\}} \circ \phi=0$ a.e.\
$[\mu_w]$, we are done.

Now we assume that $C_{\phi}$ is densely defined.

(i)$\Rightarrow$(v) Evident.

(v)$\Rightarrow$(iii) By Proposition \ref{lemS2},
$\hfw < \infty$ a.e.\ $[\mu]$. Hence, there exists a
sequence $\{X_n\}_{n=1}^\infty \subseteq \ascr$ such
that $X_n \nearrow X$ as $n\to \infty$ and $\mu(X_k) <
\infty$, $\hfw \Le k$ a.e.\ $[\mu]$ on $X_k$ and $|w|
\Le k$ a.e.\ $[\mu]$ on $X_k$ for every $k\in \nbb$.
Set $Y_n = X_n \cap \{\hfw=0\}$ for $n \in \nbb$. It
follows from Proposition \ref{lemS1}(i) that $\{w
\cdot \chi_{Y_n}\}_{n=1}^\infty \subseteq \chi_{\{w
\neq 0\}} \dz{\cfw}$ and $\{\chi_{X_n}\}_{n=1}^\infty
\subseteq \dz{\cfw}$. This implies that
   \begin{align*}
\|\cfw(w \cdot \chi_{Y_n})\|^2 & = \int_X |w \circ \phi|^2 \cdot
\chi_{Y_n} \circ \phi \D \mu_w
   \\
&\hspace{-.3ex}\overset{\eqref{l2}}= \int_X |w|^2
\chi_{Y_n} \hfw \D \mu = 0, \quad n \in \nbb,
   \end{align*}
and thus by our assumptions $\{w \cdot \chi_{Y_n}\}_{n=1}^\infty \subseteq
\jd{\cfw^*}$. As a consequence, we have
   \begin{align*}
0 = \is{w \cdot \chi_{Y_n}}{\cfw \chi_{X_n}} & = \int_X \chi_{X_n} \circ
\phi \cdot \chi_{Y_n} \D \mu_w
   \\
&= \mu_w(Y_n \cap \phi^{-1}(X_n)), \quad n \in \nbb.
   \end{align*}
By continuity of measures, we conclude that
$\mu_w(\{\hfw = 0\})=0$, which gives (iii). This
completes the proof.
   \end{proof}
Since $\jd{A} \subseteq \jd{A^*}$ for every hyponormal operator $A$, we
get the following.
   \begin{cor} \label{hipinj}
If \eqref{stand2} holds and $\cfw$ is hyponormal, then
$\hfw > 0$ a.e.\ $[\mu_w]$ $($or equivalently{\em :}
$\mu(\{\hfw = 0\} \cap \{w \neq 0\})=0$$)$.
   \end{cor}
It is known that hyponormal composition operators are
automatically injective (see \cite[Corollary
6.3]{b-j-j-sC}). However, there are hyponormal
weighted composition operators which are not
injective. The simplest possible example seems to be a
multiplication operator $M_{w}$ for which
$\mu(\{w=0\}) > 0$ (see Remark \ref{22.X.2013Daegu}).
The case of a quasinormal non-injective weighted
composition operator with nontrivial symbol is
discussed in Example \ref{quasif}.
   \section{\label{pco}Assorted classes of weighted
composition operators} In this section, we single out
some classes of weighted composition operators. Below,
if not stated otherwise, $(X,\ascr,\mu)$ stands for a
$\sigma$-finite measure space.

\vspace{1ex}

(a) [{\sc multiplication operators}] If $w$ is an
$\ascr$-measurable complex function on $X$, then the
operator $M_w := C_{\id_X, w}$ is well-defined; it is
called the {\em operator of multiplication} by $w$ in
$L^2(\mu)$. Recall that the operator $M_w$ is normal
(see \cite[Sect.\ 7.2]{b-s}; see also Remark
\ref{22.X.2013Daegu}). If $u\colon X \to \rbop$ is an
$\ascr$-measurable function which is finite a.e.\
$[\mu]$, then $M_u$ will be understood as the
multiplication operator $M_{\tilde u}$ by any
$\ascr$-measurable function $\tilde u\colon X \to
\rbb_+$ such that $\tilde u=u$ a.e.\ $[\mu]$; clearly,
this definition is correct (cf.\ Proposition
\ref{wco1}). For more information on multiplication
operators, the reader is referred to
\cite{b-s,Con-m,Schb,Weid}.

\vspace{1ex}

(b) [{\sc composition operators}] If $\phi$ is an
$\ascr$-measurable transformation of $X$, then the
operator $C_{\phi}:=C_{\phi,\boldsymbol{1}} $ is
called the {\em composition operator} in $L^2(\mu)$
with {\em symbol} $\phi$. By Proposition \ref{wco1},
$C_{\phi}$ is well-defined if and only if $\mu \circ
\phi^{-1} \ll \mu$; if this is the case, then the
transformation $\phi$ is called {\em nonsingular}. To
simplify the notation, we write $\mathsf h_{\phi}$ in
place of $\mathsf h_{\phi,\boldsymbol{1}}$. The
subject of composition operators has been studied by
many authors over the past 60 years, see, e.g.,
\cite{dun-sch,nor,lam1,lam2,ml,sto,da-st} and
\cite{sin-man}.

\vspace{1ex}

(c) [{\sc partial composition operators}] Let $Y$ be a
nonempty subset of $X$ and $\psi\colon Y \to X$ be an
$\ascr$-measurable mapping, i.e.,
$\psi^{-1}(\varDelta) \in \ascr$ for every $\varDelta
\in \ascr$. The operator $C_{\psi}\colon L^2(\mu)
\supseteq \dz{C_{\psi}} \to L^2(\mu)$ given by
   \allowdisplaybreaks
   \begin{align*}
\dz{C_{\psi}} & = \bigg\{f\in L^2(\mu)\colon \int_Y |f
\circ \psi|^2 \D \mu < \infty\bigg\},
   \\
(C_{\psi} f)(x) & =
   \begin{cases}
f(\psi(x)) & \text{ if } x \in Y,
   \\
0 & \text{ if } x \in X \setminus Y,
   \end{cases}
\quad f \in \dz{C_{\psi}},
   \end{align*}
will be called the {\em partial composition operator}
in $L^2(\mu)$ with the {\em symbol} $\psi$ (note that
$Y=\psi^{-1}(X)\in \ascr$). Arguing as in \cite[p.\
38]{nor}, one can show that $C_{\psi}$ is well-defined
if and only if $\mu\circ \psi^{-1} \ll \mu$, where
$\mu \circ \psi^{-1}(\varDelta) =
\mu(\psi^{-1}(\varDelta))$ for $\varDelta \in \ascr$.
Set $w=\chi_{Y}$ and take any $\ascr$-measurable
transformation $\phi$ of $X$ which extends $\psi$.
Since $\mu(\psi^{-1}(\varDelta)) =
\mu^w(\phi^{-1}(\varDelta))$ for $\varDelta \in
\ascr$, we deduce from Proposition \ref{wco1} that
$C_{\psi}$ is well-defined if and only if $\cfw$ is
well-defined. If this is the case, then $C_{\psi} =
\cfw$. The definition of what we call here a partial
composition operator has been given by Nordgren in
\cite{nor}. Particular classes of partial composition
operators have been studied in \cite{sto}.

\vspace{1ex}

(d) [{\sc weighted partial composition operators}] Let
$Y$ be a nonempty subset of $X$, $v$ be an
$\ascr$-measurable complex function on $X$ and
$\psi\colon Y \to X$ be $\ascr$-measurable mapping
(cf.\ (c)). The operator $\cpv \colon L^2(\mu)
\supseteq \dz{\cpv} \to L^2(\mu)$ given by
   \allowdisplaybreaks
   \begin{align*}
\dz{\cpv} & = \bigg\{f\in L^2(\mu)\colon \int_Y |v
\cdot (f \circ \psi)|^2 \D \mu < \infty\bigg\},
   \\
(\cpv f)(x) & =
   \begin{cases}
v(x)f(\psi(x)) & \text{ if } x \in Y,
   \\
0 & \text{ if } x \in X \setminus Y,
   \end{cases}
\quad f \in \dz{\cpv},
   \end{align*}
will be called the {\em weighted partial composition
operator} in $L^2(\mu)$ with the {\em symbol} $\psi$
and the {\em weight} $v$. As in the proof of
Proposition \ref{wco1}, one can show that $\cpv$ is
well-defined if and only if $\mu^{v}\circ \psi^{-1}
\ll \mu$, where $\mu^{v} (\varDelta) = \mu(\varDelta
\cap \{v \neq 0\})$ for $\varDelta \in \ascr$. Set
$w=v\cdot \chi_{Y}$ and take any $\ascr$-measurable
transformation $\phi$ of $X$ which extends $\psi$.
Since $\mu^v(\psi^{-1}(\varDelta)) =
\mu^w(\phi^{-1}(\varDelta))$ for all $\varDelta \in
\ascr$, we deduce from Proposition \ref{wco1} that
$\cpv$ is well-defined if and only if $\cfw$ is
well-defined. If this is the case, then $\cpv = \cfw$.
Weighted partial composition operators over countably
infinite discrete measure spaces have been
investigated by Carlson in \cite{car1,car2}.

\vspace{1ex}

(e) [{\sc weighted shifts on directed trees}] Suppose
$\tcal=(V,E)$ is a directed tree, where $V$ is the set
of vertices of $\tcal$ and $E$ is the set of edges of
$\tcal$. If $u\in V$, then a (unique) vertex $v \in V$
such that $(v,u)\in E$ is called a {\em parent} of
$u$; it is denoted by $\pa{u}$. A vertex which has no
parent is called a {\em root} of $\tcal$. A root is
unique (provided it exists) and is denoted here by
$\koo$. A directed tree without root is called {\em
rootless}. Set $V^\circ=V \setminus \{\koo\}$ if
$\tcal$ has a root and $V^\circ=V$ otherwise. For
$u\in V$, we write $\dzi{u}=\{v\in V\colon (u,v) \in
E\}$ and call a member of $\dzi{u}$ a {\em child} of
$u$. By a {\em weighted shift on} $\tcal$ with {\em
weights} $\lambdab=\{\lambda_v\}_{v \in V^{\circ}}
\subseteq \cbb$ we mean the operator $\slam$ in
$\ell^2(V)$ defined by
   \begin{align*}
   \begin{aligned}
\dz {\slam} & = \{f \in \ell^2(V) \colon
\varLambda_\tcal f \in \ell^2(V)\},
   \\
\slam f & = \varLambda_\tcal f, \quad f \in
\dz{\slam},
   \end{aligned}
   \end{align*}
where $\varLambda_\tcal$ is a complex mapping on
$\cbb^V$ given by
   \begin{align*}
(\varLambda_\tcal f) (v) =
   \begin{cases}
\lambda_v \cdot f\big(\pa v\big) & \text{ if } v\in
V^\circ,
   \\
0 & \text{ if } v=\koo,
   \end{cases}
\quad f\in \cbb^V.
   \end{align*}
Note that if $\card{V} \Le \aleph_0$, then we can view
a weighted shift on a directed tree as a weighted
partial composition operator in $L^2(V,2^V,\mu)$,
where $\mu$ is the counting measure (cf.\ (d)). The
foundations of the theory of weighted shifts on
directed trees have been established in \cite{j-j-s}.
Very preliminary investigations of a particular class
of weighted adjacency operators of directed graphs,
called also adjacency operators of directed fuzzy
graphs, have been undertaken by Fujii et al.\ in
\cite{f-f-s-w}.

\vspace{1ex}

(f) [{\sc almost nowhere-vanishing weights}] Suppose
$w$ is an $\ascr$-measurable complex function on $X$
such that $w \neq 0$ a.e.\ $[\mu]$ and $\phi$ is an
$\ascr$-measurable transformation of $X$. Then, by
Proposition \ref{wco1}, $\cfw$ is well-defined if and
only if $\mu\circ \phi^{-1} \ll \mu$. This resembles
the case of composition operators (see also (c)).

\vspace{1ex}

(g) [{\sc unilateral weighted shifts}] Let
$\{\lambda_n\}_{n=0}^{\infty}$ be a sequence of
complex numbers. Set $X=\zbb_+$ and $\ascr=2^X$. Let
$\mu$ be the counting measure on $X$. Define the
functions $\phi\colon X \to X$ and $w\colon X \to
\cbb$ by
   \begin{align*}
\phi(n)=
   \begin{cases}
n-1 & \text{for } n\in\nbb,
   \\
0 & \text{for } n=0,
   \end{cases}
\quad \text{and} \quad w(n)=
   \begin{cases}
\lambda_{n-1} & \text{for } n\in \nbb,
   \\
0 & \text{for } n=0.
   \end{cases}
   \end{align*}
Clearly, the weighted composition operator $\cfw$ in
$\ell^2(\zbb_+)$ is well-defined. Set
$e_n=\chi_{\{n\}}$ for $n\in \zbb_+$. Then
$\{e_n\}_{n=0}^{\infty}$ is an orthonormal basis of
$\ell^2(\zbb_+)$. Following \cite{ml2}, we define the
{\em unilateral weighted shift} $W$ in
$\ell^2(\zbb_+)$ with weights
$\{\lambda_n\}_{n=0}^{\infty}$ to be equal to the
product $SD$, where $S$ is the isometric unilateral
shift in $\ell^2(\zbb_+)$ (i.e., $S\in
\ogr{\ell^2(\zbb_+)}$ and $Se_n = e_{n+1}$ for all
$n\in \zbb_+$) and $D$ is the diagonal operator in
$\ell^2(\zbb_+)$ with the diagonal elements
$\{\lambda_n\}_{n=0}^{\infty}$ (i.e., $D$ is a normal
operator, $\{e_n\}_{n=0}^{\infty} \subseteq \dz{D}$
and $De_n=\lambda_n e_n$ for all $n \in \zbb_+$). We
show that $\cfw = W$. It is easily seen that
$\{e_n\}_{n=0}^{\infty} \subseteq \dz{\cfw}$ and
   \begin{align} \label{jajo}
\cfw e_n = \lambda_n e_{n+1}, \quad n\in \zbb_+.
   \end{align}
If $f\in \ell^2(\zbb_+)$, then
   \begin{align*}
\sum_{n=0}^{\infty} |w(n)f(\phi(n))|^2 =
\sum_{n=1}^{\infty} |\lambda_{n-1}f(n-1)|^2 =
\sum_{n=0}^{\infty} |\lambda_{n}f(n)|^2,
   \end{align*}
which implies that
   \begin{gather*}
\dz{\cfw} = \dz{W} = \dz{D} = \bigg\{f\in \cbb^X\colon
\sum_{n=0}^{\infty} |f(n)|^2(1+|\lambda_{n}|^2)
<\infty\bigg\},
   \\
\|f\|^2_{\cfw} = \|f\|^2_{W} = \sum_{n=0}^{\infty}
|f(n)|^2(1+|\lambda_{n}|^2), \quad f\in \dz{\cfw}.
   \end{gather*}
Hence the operators $\cfw$ and $W$ are closed and
$\lin \{e_n\colon n \in \zbb_+\}$ is a core for $\cfw$
and $W$ (see \cite[Eq.\ (1.7)]{ml2} for the case of
$W$). Since $We_n = \lambda_n e_{n+1}$ for all $n\in
\zbb_+$, we infer from \eqref{jajo} that $\cfw = W$.
The topic of weighted shift operators is an immanent
part of operator theory (see, e.g.,
\cite{g-w,hal2,shi,nik}).

\vspace{1ex}

(h) [{\sc bilateral weighted shifts}] Let
$\{\lambda_n\}_{n\in \zbb}$ be a two-sided sequence of
complex numbers. Set $X=\zbb$ and $\ascr=2^X$. Let
$\mu$ be the counting measure on $X$ and
$\{e_n\}_{n\in \zbb}$ be the standard orthonormal
basis of $\ell^2(\zbb)$, i.e., $e_n=\chi_{\{n\}}$ for
all $n\in \zbb$. Define the functions $\phi\colon X
\to X$ and $w\colon X \to \cbb$ by $\phi(n)= n-1$ and
$w(n)= \lambda_{n-1}$ for $n\in\zbb$. Let $V$ be the
unitary bilateral shift in $\ell^2(\zbb)$ (i.e., $V\in
\ogr{\ell^2(\zbb)}$ and $Ve_n = e_{n+1}$ for all $n\in
\zbb$) and let $D$ be the diagonal operator in
$\ell^2(\zbb)$ with the diagonal elements
$\{\lambda_n\}_{n\in \zbb}$ (i.e., $D$ is a normal
operator, $\{e_n\}_{n\in \zbb} \subseteq \dz{D}$ and
$De_n=\lambda_n e_n$ for all $n \in \zbb$). Set
$W=VD$. The operator $W$ is called a {\em bilateral
weighted shift} in $\ell^2(\zbb)$ with weights
$\{\lambda_n\}_{n\in \zbb}$. Arguing similarly as in
(g), one can show that $\lin \{e_n\colon n \in \zbb\}$
is a core for $\cfw$ and $W$, and $\cfw=W$. Let us
recall that a bilateral weighted shift with nonzero
weights is unitarily equivalent to a composition
operator in an $L^2$-space over a $\sigma$-finite
measure space (see, e.g., \cite[Lemma 4.3.1]{j-j-s0}).
For more information on general as well as particular
properties of bilateral weighted shifts see
\cite{her,shi,cu-fi2}.

\vspace{1ex}

(i) [{\sc adjoints of unilateral weighted shifts}] Let
$W=SD$ be a unilateral weighted shift in
$\ell^2(\zbb_+)$ with weights
$\{\lambda_n\}_{n=0}^{\infty}$ (cf.\ (g)). Set
$X=\zbb_+$ and $\ascr=2^X$. Let $\mu$ be the counting
measure on $X$. Define the functions $\phi\colon X \to
X$ and $w\colon X \to \cbb$ by $\phi(n)= n+1$ and
$w(n)= \bar \lambda_{n}$ for $n\in\zbb_+$. We show
that $\cfw=W^*$. For this, first note that $W^* = D^*
S^*$. Then it is clear that $\{e_n\}_{n=0}^{\infty}
\subseteq \dz{\cfw} \cap \dz{W^*}$ and
   \begin{align} \label{nocc}
\cfw e_n = \bar \lambda_{n-1} e_{n-1} = W^* e_n, \quad
n\in \zbb_+,
   \end{align}
with the convention that $e_{-1}=0$ and
$\lambda_{-1}=0$. Arguing similarly as in (g), we
verify that $\dz{\cfw}=\dz{W^*}$. Since $\lin
\{e_n\colon n \in \zbb_+\}$ is a core for $W^*$ (see,
e.g., \cite[Eq.\ (1.11)]{ml2}) and $\cfw$ is closed,
we deduce from \eqref{nocc} that $\cfw=W^*$.

\vspace{1ex}

(j) [{\sc backward unilateral weighted shifts}] A
backward unilateral weight\-ed shift in
$\ell^2(\zbb_+)$ with weights
$\{\lambda_n\}_{n=0}^{\infty}$ can be defined as the
product $W=DS^*$, where $S$ and $D$ are as in (g).
Clearly, such an operator is equal to the adjoint of
the unilateral weighted shift $SD^{*}$ whose weights
have the form $\{\bar \lambda_n\}_{n=0}^{\infty}$.
Hence, by (i), it is a weighted composition operator.

\vspace{1ex}

(k) [{\sc adjoints of bilateral weighted shifts}] We
begin by observing that if $W$ is a bilateral weighted
shift in $\ell^2(\zbb)$ with weights
$\{\lambda_n\}_{n\in \zbb}$ (cf.\ (h)), then $U\tilde
W = W^* U$, where $\tilde W$ is the bilateral weighted
shift in $\ell^2(\zbb)$ with weights
$\{\bar\lambda_{-(n+1)}\}_{n\in \zbb}$ and $U\in
\ogr{\ell^2(\zbb)}$ is the unitary operator such that
$Ue_n=e_{-n}$ for all $n\in\zbb$ (this fact is
well-known in the case of bounded operators). To see
this, verify that $U\tilde W|_{\mathcal E}=
W^*U|_{\mathcal E}$, where $\mathcal E=\lin
\{e_n\colon n \in \zbb\}$, and use the fact that
$\mathcal E$ is a core for $\tilde W$ and $W^*$. In
view of (h), this means that $W^*$ is unitarily
equivalent to a weighted composition operator. On the
other hand, the argument provided in (i) enables one
to show that the adjoint of a bilateral weighted shift
is in fact a weighted composition operator. As in (j)
we can view the adjoint of a bilateral weighted shift
as a backward bilateral weighted shift.
   \section{\label{Sec2.4}Conditional expectation} In this
section, we discuss some properties of the conditional
expectation $\efw$ which plays a crucial role in our
considerations. We refer the reader to Appendix
\ref{AppB} for the necessary background (including
notation) on conditional expectation in a general
non-probabilistic setting.

   Suppose \eqref{stand2} holds and the measure
$\mu_w|_{\phi^{-1}(\ascr)}$ is $\sigma$-finite (or
equivalently, by Proposition \ref{lemS2}, $\hfw <
\infty$ a.e.\ $[\mu]$). Thus we can consider the
conditional expectation
$\mathsf{E}(f;\phi^{-1}(\ascr),\mu_w)$ of a function
$f$ with respect to the $\sigma$-algebra
$\phi^{-1}(\ascr)$ and the measure $\mu_w$. To
simplify the notation, we write $\efw(f)$ in place of
$\mathsf{E}(f;\phi^{-1}(\ascr),\mu_w)$. If
$w=\boldsymbol{1}$, we abbreviate
$\mathsf{E}_{\phi,\boldsymbol{1}}(f)$ to
$\mathsf{E}_{\phi}(f)$. Recall that for a given $p\in
[1,\infty]$, the conditional expectation $\efw(\cdot)$
can be regarded as a linear contraction on
$L^p(\mu_w)$ which leaves invariant the convex cone
$L_+^p(\mu_w)$ (see Theorem \ref{Ath}). In view of
\eqref{B3}, \eqref{B9.5} and the well-known fact that
(cf.\ \cite[Problem 13.3]{bill})
   \begin{align} \label{Il-Zen2}
   \begin{minipage}{72ex}
{\em a function $\tilde g\colon X \to \rbop$ $($resp.,
$\tilde g\colon X \to \rbb_+$, $\tilde g\colon X \to
\cbb$$)$ is $\phi^{-1}(\ascr)$-measurable if and only
if there exists $\ascr$-measurable function $g\colon X
\to \rbop$ $($resp., $g\colon X \to \rbb_+$, $g\colon
X \to \cbb$$)$ such that $\tilde g=g \circ \phi$,}
   \end{minipage}
   \end{align}
we see that
   \begin{align}  \label{wazny}
   \begin{aligned}
   \begin{minipage}{72ex}
{\em if $f,g \colon X \to \rbop$ are
$\ascr$-measurable functions $($resp., $f,g \colon X
\to \cbb$ are $\ascr$-measurable functions such that
$f \in L^p(\mu_w)$ and $g\circ \phi\in L^q(\mu_w)$,
where $p,q \in [1,\infty]$ satisfy $\frac{1}{p} +
\frac {1}{q} =1$$)$, then}
   $$ \int_X g \circ \phi \cdot f \D \mu_w =
\int_X g \circ \phi \cdot \efw(f) \D \mu_w.
   $$
  \end{minipage}
   \end{aligned}
   \end{align}

The following proposition is patterned on \cite[page
325]{ca-hor}.
   \begin{pro} \label{lemS4}
Suppose \eqref{stand2} holds. Assume that $f\colon X
\to \rbop$ $($resp., $f\colon X \to \cbb$\/$)$ is an
$\ascr$-measurable function. Then there exists an
$\ascr$-measurable function $g\colon X \to \rbop$
$($resp., $g\colon X \to \cbb$\/$)$ such that $f\circ
\phi = g \circ \phi$ a.e.\ $[\mu_w]$ and $g=0$ a.e.\
$[\mu]$ on $\{\hfw = 0\}$. Moreover, such $g$ is
unique up to a set of $\mu$-measure zero and is given
by $g=f\cdot \chi_{\{\hfw>0\}}$.
   \end{pro}
   \begin{proof}
By Lemma \ref{lemS6}(i), $\chi_{\{\hfw
> 0\}} \circ \phi = 1$ a.e.\ $[\mu_w]$. This implies
that
   \begin{align*}
(f\cdot \chi_{\{\hfw>0\}}) \circ \phi = (f \circ \phi)
\cdot \chi_{\{\hfw
> 0\}} \circ \phi =f \circ \phi \text{ a.e.\ $[\mu_w]$}.
   \end{align*}
The uniqueness statement follows from Lemma \ref{nuklear}.
   \end{proof}
Assume that \eqref{stand2} holds and $\hfw < \infty$
a.e. $[\mu]$. Suppose that $f\colon X \to \rbop$ is an
$\ascr$-measurable function (resp., $f\in
L^2(\mu_w)$). Then, by the well-known description of
$\phi^{-1}(\ascr)$-measurable functions and
Proposition \ref{lemS4}, $\efw(f) = g\circ \phi$ a.e.\
$[\mu_w]$ with some $\ascr$-measurable $\rbop$-valued
(resp., $\cbb$-valued) function $g$ on $X$ such that
$g=g\cdot \chi_{\{\hfw>0\}}$ a.e.\ $[\mu]$. Set
$\efw(f) \circ \phi^{-1} = g$ a.e.\ $[\mu]$. By
Proposition \ref{lemS4}, this definition is correct,
and by Lemma \ref{nuklear} the following equality
holds
   \begin{align} \label{fifi}
(\efw(f) \circ \phi^{-1})\circ \phi = \efw(f) \quad
\text{a.e.\ $[\mu_w|_{\phi^{-1}(\ascr)}]$.}
   \end{align}
(Of course, the expression ``a.e.\
$[\mu_w|_{\phi^{-1}(\ascr)}]$'' in \eqref{fifi} can be
replaced by ``a.e.\ $[\mu_w]$''.) If $\tilde f\colon X
\to \rbop$ is an $\ascr$-measurable function (resp.,
$\tilde f\in L^2(\mu_w)$) is such that $f=\tilde f$
a.e.\ $[\mu_w]$, then $\efw(f) = \efw(\tilde f)$ a.e.\
$[\mu_w]$ and consequently, in view of Lemma
\ref{nuklear}, $\efw(f) \circ \phi^{-1} = \efw(\tilde
f) \circ \phi^{-1}$ a.e.\ $[\mu]$. It is clear that
$\efw(\boldsymbol{1}) = \boldsymbol{1}$ a.e.\
$[\mu_w]$ and
   \begin{align} \label{jedynka}
\efw(\boldsymbol{1}) \circ \phi^{-1} = \chi_{\{\hfw >
0\}} \text{ a.e.\ $[\mu]$.}
   \end{align}
   \begin{pro}  \label{6 XII 2013}
Suppose \eqref{stand2} holds, $\hfw < \infty$ a.e.
$[\mu]$ and $f,g\colon X \to \rbop$ are
$\ascr$-measurable functions. Then the following
assertions are valid{\em :}
   \begin{enumerate}
   \item[(i)] if $f \Le g$ a.e.\
$[\mu_w]$, then $\efw(f) \circ \phi^{-1} \Le \efw(g)
\circ \phi^{-1}$ a.e.\ $[\mu]$,
   \item[(ii)] if $f \Le c$ a.e.\
$[\mu_w]$ for some $c\in \rbb_+$, then $\efw(f) \circ
\phi^{-1} \Le c$ a.e.\ $[\mu]$.
   \end{enumerate}
   \end{pro}
   \begin{proof}
(i) If $f \Le g$ a.e.\ $[\mu_w]$, then, by \eqref{B4},
$\efw(f) \Le \efw(g)$ a.e.\ $[\mu_w]$, and hence by
\eqref{l2} and \eqref{fifi}, we have
   \begin{multline*}
\int_{\varDelta} \efw(f) \circ \phi^{-1} \cdot \hfw \D
\mu = \int_{\phi^{-1}(\varDelta)} \efw(f) \D \mu_w
   \\
\Le \int_{\phi^{-1}(\varDelta)} \efw(g) \D \mu_w =
\int_{\varDelta} \efw(g) \circ \phi^{-1} \cdot \hfw \D
\mu, \quad \varDelta \in \ascr.
   \end{multline*}
Since $\mu$ is $\sigma$-finite, $\hfw < \infty$ a.e.
$[\mu]$ and $\efw(q) \circ \phi^{-1} = 0$ a.e.\
$[\mu]$ on $\{\hfw=0\}$ for any $\ascr$-measurable
$q\colon X \to \rbop$, we get (i).

(ii) Apply (i) and \eqref{jedynka}.
   \end{proof}
Regarding Proposition \ref{6 XII 2013}, note that part
(i) remains true if ``$\Le$'' is replaced by ``$\Ge$''
(or by ``$=$''). However, both of these replacements
make part (ii) false (see \eqref{jedynka} and Example
\ref{quasif}; see also Remark \ref{22.X.2013Daegu}).

The Radon-Nikodym derivative $\frac{\D \mu^w \circ
\phi^{-1}}{\D \mu}$ can be expressed in terms of
$\hfw$.
   \begin{pro}\label{lemS7}
If \eqref{stand2} holds and $\hfw < \infty$ a.e.\
$[\mu]$, then $\mu^w \circ \phi^{-1} \ll \mu$ and
   \begin{align*}
\frac{\D \mu^w \circ \phi^{-1}}{\D \mu} = \hfw \cdot
\efw\Big(\chi_{\{w \neq 0\}} \cdot
\frac{1}{|w|^2}\Big)\circ \phi^{-1} \text{ a.e.\
$[\mu]$.}
   \end{align*}
   \end{pro}
   \begin{proof}   Since
   \begin{align*}
\mu^w(\phi^{-1}(\varDelta)) & =
\int_{\phi^{-1}(\varDelta)} \chi_{\{w \neq 0\}} \cdot
\frac{1}{|w|^2} \D \mu_w \overset{\eqref{wazny}}=
\int_{\phi^{-1}(\varDelta)} \efw\Big(\chi_{\{w \neq
0\}} \cdot \frac{1}{|w|^2}\Big) \D \mu_w
   \\
& \overset{\eqref{l2} \& \eqref{fifi}}=
\int_{\varDelta} \hfw \cdot \efw\Big(\chi_{\{w \neq
0\}} \cdot \frac{1}{|w|^2}\Big)\circ \phi^{-1} \D \mu,
\quad \varDelta \in \ascr,
   \end{align*}
the proof is complete.
   \end{proof}
   \section{\label{Sec2.5}Adjoint and polar decomposition}
An unexplicit description of the adjoint of a weighted
composition operator has been given in \cite[Lemma
6.4]{ca-hor}. Below, we provide another one which is
complete and written in terms of the conditional
expectation $\efw$.

First, we single out the following fact.
   \begin{align} \label{fw}
   \begin{minipage}{70ex}
{\em If $(X,\ascr,\mu)$ is a measure space and
$w\colon X \to \cbb$ is $\ascr$-measurable, then the
mapping $\varPsi_w \colon L^2(\mu)\ni f \longmapsto
f_w \in L^2(\mu_w)$, where $f_w := \chi_{\{w \neq 0\}}
\cdot \frac{f}{w}$ for $f \in L^2(\mu)$, is a partial
isometry with the initial space $\chi_{\{w \neq 0\}}
\cdot L^2(\mu)$ and the final space $L^2(\mu)$;
therefore, $\varPsi_w $ is a coisometry.}
   \end{minipage}
   \end{align}
Hence the mapping $L^2(\mu) \ni f \longmapsto
\efw(f_w) \in L^2(\mu_w)$, which appears frequently in
this section, coincides with the product $\efw
\varPsi_w$, where $\efw$ is understood as a bounded
operator on $L^2(\mu_w)$ (in fact, $\efw$ is an
orthogonal projection in $L^2(\mu_w)$; see Theorem
\ref{Ath}). Clearly, this product is a bounded
operator.
   \begin{pro}\label{adj}
Suppose \eqref{stand2} holds and $\cfw$ is densely
defined. Then the following assertions are valid{\em
:}
   \begin{enumerate}
   \item[(i)]
$\dz{\cfw^*} = \big\{f \in L^2(\mu)\colon \hfw \cdot
\efw(f_w) \circ \phi^{-1} \in L^2(\mu)\big\}$,
   \item[(ii)]
$\cfw^*(f) = \hfw \cdot \efw(f_w) \circ \phi^{-1}$ for
all $f \in \dz{\cfw^*}$,
   \item[(iii)]
$\jd{\cfw^*} = \big\{f \in L^2(\mu)\colon \efw(f_w)=0
\text{ a.e.\ $[\mu_w]$}\big\}$,
   \item[(iv)]
$\chi_{\{w=0\}} \cdot L^2(\mu) \subseteq \jd{\cfw^*}$,
   \item[(v)]
$\dz{\cfw^*} = \chi_{\{w\neq 0\}} \cdot \dz{\cfw^*}
\oplus \chi_{\{w=0\}} \cdot L^2(\mu)$ and $\cfw^* f =
\cfw^* (\chi_{\{w\neq 0\}} \cdot f)$ for all $f\in
\dz{\cfw^*}$,
   \item[(vi)] if $\cfw$ has dense range, then
$w \neq 0$ a.e.\ $[\mu]$.
   \end{enumerate}
   \end{pro}
   \begin{proof}
It follows from \eqref{fw} that $\efw(f_w) \in
L^2(\mu_w)$ for every $f \in L^2(\mu)$. In turn, if $g
\in \dz{\cfw}$, then by \eqref{l2} and Proposition~
\ref{lemS1}(i), we get $g\circ \phi \in L^2(\mu_w)$.
This, \eqref{wazny}, \eqref{B9}, \eqref{fifi} and
\eqref{l2} yield
   \begin{align} \notag
\is{\cfw g}{f} & = \int_X g\circ \phi \cdot
\overline{f_w} \D \mu_w
   \\   \notag
& = \int_X g\circ \phi \cdot \overline{\efw(f_w)} \D
\mu_w
   \\  \label{dae1}
& = \int_X g \cdot \hfw \cdot \overline{\efw(f_w)\circ
\phi^{-1}} \D \mu, \quad g \in \dz{\cfw}, \quad f \in
L^2(\mu).
   \end{align}
Denote by $\escr$ the right-hand side of (i). Clearly,
if $f \in \escr$, then, by \eqref{dae1}, $f \in
\dz{\cfw^*}$ and (ii) holds. To complete the proof of
(i) and (ii), it suffices to show that if $f \in
\dz{\cfw^*}$, then $\xi:=\hfw \cdot
\overline{\efw(f_w)\circ \phi^{-1}} \in L^2(\mu)$. By
\eqref{dae1}, $g \cdot \xi \in L^1(\mu)$ and $\int_X
g\cdot \xi \D \mu = \int_X g \cdot \tilde \xi \D \mu$
for every $g \in \dz{\cfw}$, where $\tilde \xi :=
\cfw^* f \in L^2(\mu)$. Let $\{X_n\}_{n=1}^\infty$ be
as in Lemma \ref{aproks}. Considering
$g=\chi_{\varDelta\cap X_n}$, $\varDelta\in \ascr$,
and applying Lemma \ref{f=g} we get $\xi=\tilde\xi$
a.e.\ $[\mu]$ on $X_n$ for every $n\in \nbb$. Hence
$\xi = \tilde\xi$ a.e.\ $[\mu]$, which completes the
proof of (i) and (ii). Now we prove (iii). If $f \in
\jd{\cfw^*}$, then, by (ii), $\hfw \cdot \efw(f_w)
\circ \phi^{-1} = 0$ a.e.\ $[\mu]$. Since, by
definition, $\efw(f_w) \circ \phi^{-1}=0$ a.e.\
$[\mu]$ on $\{\hfw = 0\}$, we deduce that $\efw(f_w)
\circ \phi^{-1}=0$ a.e.\ $[\mu]$. This together with
Lemma \ref{nuklear} and the equality \eqref{fifi}
yield $\efw(f_w)=0$ a.e.\ $[\mu_w]$. Conversely, if
$f\in L^2(\mu)$ is such that $\efw(f_w)=0$ a.e.\
$[\mu_w]$, then, by definition, $\efw(f_w) \circ
\phi^{-1}=0$ a.e.\ $[\mu]$, which according to (i) and
(ii) gives $f \in \jd{\cfw^*}$. This completes the
proof of (iii).

The assertion (iv) is a direct consequence of (iii),
while the assertion (v) follows from (i), (ii) and
(iv). Finally, the assertion (vi) is to be deduced
from (iv).
   \end{proof}
The assertion (iii) of Proposition \ref{adj} says that
the kernel of $\cfw^*$ coincides with the kernel of
the product $\efw \varPsi_w$. By the assertion (iv) of
Proposition \ref{adj}, the range of the multiplication
operator $M_{\chi_{\{w=0\}}}$ is always contained in
the kernel of $\cfw^*$ and it is a (closed) invariant
vector space for $\cfw^*$. Observe also that in
general the implication (vi) of Proposition \ref{adj}
can not be reversed even for bounded composition
operators (see Example \ref{byby}).

Now we describe the polar decompositions of $\cfw$ and
$\cfw^*$.
   \begin{pro}\label{polar}
Suppose \eqref{stand2} holds and $\cfw$ is densely
defined. Let $\cfw = U|\cfw|$ be the polar
decomposition of $\cfw$. Then the following assertions
are valid{\em :}
   \begin{enumerate}
   \item[(i)] $|\cfw| = M_{\hfw^{1/2}}$ $($see
Sect.\ {\em \ref{pco}(a)} for notation$)$,
   \item[(ii)] $U = \cfww$, where $\widetilde w\colon X \to \cbb$
is an $\ascr$-measurable function such
that\footnote{\;Because of Lemma \ref{lemS6} and
Proposition \ref{lemS2}, the rational function
appearing on the right-hand side of the equality in
\eqref{dofood} takes complex values a.e.\ $[\mu]$.
What is important here is that $\widetilde w$
satisfies the equality $\{\widetilde w= 0\} = \{w=0\}$
a.e. $[\mu]$.}
   \begin{align} \label{dofood}
\widetilde w = \frac{w}{(\hfw \circ \phi)^{1/2}}
\text{ a.e.\ $[\mu]$,}
   \end{align}
   \item[(iii)] $U^* f = \hfw^{1/2} \cdot \efw(f_w)
\circ \phi^{-1}$ for $f \in L^2(\mu)$, where $f_w$ is as in \eqref{fw},
   \item[(iv)] the modulus $|\cfw^*|$ is given by
   \begin{align}\label{dzmc*} &
   \left.\begin{aligned} \dz{|\cfw^*|} & = \{f\in
   L^2(\mu)\colon w \cdot (\hfw \circ \phi)^{1/2}
   \cdot \efw(f_w) \in L^2(\mu)\},
   \\ & = \{f\in L^2(\mu) \colon \hfw \cdot \efw(f_w)  \circ
\phi^{-1} \in L^2(\mu)\},
   \end{aligned} \right\}
   \\  \label{pmc*}
&\hspace{3.4ex} |\cfw^*| f = w \cdot (\hfw \circ
\phi)^{1/2} \cdot \efw(f_w), \quad f \in
\dz{|\cfw^*|}.
   \end{align}
   \end{enumerate}
   \end{pro}
   \begin{proof}
First note that, by Proposition \ref{wco1}, $\hfw <
\infty$ a.e.\ $[\mu]$.

(i) By the well-known properties of multiplication
operators, $M_{\hfw^{1/2}}$ is positive and
selfadjoint. So is the operator $|\cfw|$. By
Proposition \ref{lemS1}(i), $\dz{|\cfw|} =
\dz{M_{\hfw^{1/2}}}$ and thus
   \begin{align*}
\||\cfw|f\|^2 = \|\cfw f\|^2 = \int_X |f|^2 \circ \phi
\D \mu_w \overset{\eqref{l2}} = \int_X |f|^2 \hfw \,
\D \mu = \|M_{\hfw^{1/2}} f\|^2
   \end{align*}
for every $f \in \dz{|\cfw|}$. It follows from Lemma
\ref{useful} that $|\cfw| = M_{\hfw^{1/2}}$.

(ii) By Lemma \ref{lemS6}(ii), we have
   \begin{align} \label{isaknox}
\int_X \Big|w \cdot \frac{f\circ \phi}{(\hfw \circ
\phi)^{1/2}}\Big|^2 \D \mu = \int_{\{\hfw > 0\}} |f|^2
\D \mu, \quad f\in L^2(\mu),
   \end{align}
which implies that the operator $\cfww$ is
well-defined and $\cfww \in \ogr{L^2(\mu)}$. According
to \eqref{isaknox} and Lemma \ref{jadro}, we see that
$\jd{\cfww} = \jd{\cfw} = \chi_{\{\hfw > 0\}}
L^2(\mu)$ and $\cfww|_{L^2(\mu) \ominus \jd{\cfww}}$
is an isometry. This means that $\cfww$ is a partial
isometry. It follows from (i) that $\cfw = \cfww
|\cfw|$. By the uniqueness statement in the polar
decomposition theorem, $U=\cfww$, which yields (ii).

(iii) Clearly, $\D \mu_{\widetilde w} = \frac{1}{\hfw
\circ \phi} \D \mu_w$, which means that the measures
$\mu_{\widetilde w}$ and $\mu_w$ are mutually
absolutely continuous and thus $\mu_{\widetilde
w}\circ \phi^{-1} \ll \mu$. By Lemma \ref{lemS6}(ii),
we have
   \begin{align} \label{mtil}
(\mu_{\widetilde w}\circ \phi^{-1})(\varDelta) =
\int_X \frac{\chi_{\varDelta}\circ \phi}{\hfw \circ
\phi} \D\mu_w = \int_{\varDelta} \chi_{\{\hfw
> 0\}} \D \mu, \quad \varDelta \in \ascr,
   \end{align}
which implies that $\hfww = \chi_{\{\hfw > 0\}}$ a.e.\
$[\mu]$.

Now we show that
   \begin{align} \label{nakola}
\efww(f_{\widetilde w}) = \hfw^{1/2} \circ \phi \cdot
\efw(f_w) \text{ a.e.\ $[\mu_{\widetilde w}]$,} \quad
f \in L^2(\mu).
   \end{align}
For this, define $q_{\varDelta} = \chi_{\{\hfw > 0\}}
\cdot \frac{\chi_{\varDelta}}{\hfw^{1/2}}$ a.e.\
$[\mu]$ for $\varDelta \in \ascr$. Then, by Lemma
\ref{lemS6}(i) and Lemma \ref{nuklear}, we have
   \begin{align} \label{cg}
q_{\varDelta} \circ \phi = \frac{\chi_{\varDelta}
\circ \phi}{(\hfw\circ \phi)^{1/2}} \text{ a.e.\
$[\mu_w]$,} \quad \varDelta \in \ascr.
   \end{align}
Take $f \in L^2(\mu)$. Let $\varDelta \in \ascr$ be such that
$\mu_{\widetilde w}(\phi^{-1}(\varDelta)) < \infty$. Then
   \begin{align*}
\int_X |q_{\varDelta} \circ \phi|^2 \D \mu_w
\overset{\eqref{hle2}}= \int_{\{\hfw > 0\}}
\chi_{\varDelta} \D \mu \overset{\eqref{mtil}}=
\mu_{\widetilde w}(\phi^{-1}(\varDelta)) < \infty.
   \end{align*}
This combined with \eqref{wazny}, \eqref{fw} and \eqref{cg} yields
   \begin{align*}
\int_{\phi^{-1}(\varDelta)} \efww(f_{\widetilde w}) \D
\mu_{\widetilde w} &= \int_{X} q_{\varDelta} \circ
\phi \cdot f_w \D \mu_w
   \\
&= \int_{X} q_{\varDelta} \circ \phi \cdot \efw(f_w)
\D \mu_w
   \\
&= \int_{\phi^{-1}(\varDelta)} \hfw^{1/2}\circ \phi
\cdot \efw(f_w) \D \mu_{\widetilde w}.
   \end{align*}
Applying Lemma \ref{f=g} to the measure $\mu_{\widetilde
w}|_{\phi^{-1}(\ascr)}$ gives \eqref{nakola}.

Since the measures $\mu_{\widetilde w}$ and $\mu_w$ are mutually
absolutely continuous, we infer from \eqref{nakola} and Proposition
\ref{lemS4} that
   \begin{align*}
\efww(f_{\widetilde w}) \circ \phi^{-1} = \hfw^{1/2}
\cdot \efw(f_w)\circ \phi^{-1} \text{ a.e.\ $[\mu]$,}
\quad f \in L^2(\mu).
   \end{align*}
This together with Proposition \ref{adj}, applied to $\cfww$, yields
(iii).

(iv) It follows from \cite[Exercise 7.26(b)]{Weid})
that $|\cfw^*| = \cfw U^*$. In view of Proposition
\ref{lemS1}(i), $f\in L^2(\mu)$ belongs to
$\dz{|\cfw^*|}$ if and only if $U^*f \in L^2(\hfw \D
\mu)$. Since, by (iii), \eqref{l2} and \eqref{fifi},
the following equalities hold
   \begin{align*}
\int_X |U^*f|^2 \cdot \hfw \D \mu & = \int_X \hfw^2
\cdot |\efw(f_w) \circ \phi^{-1}|^2 \D \mu
   \\
& = \int_X \hfw \circ \phi \cdot |\efw(f_w)|^2 \D
\mu_w, \quad f\in L^2(\mu),
   \end{align*}
we get \eqref{dzmc*}. The formula \eqref{pmc*} follows
from the equality $|\cfw^*| = \cfw U^*$, the condition
(iii) and the equality \eqref{fifi}. This completes
the proof.
   \end{proof}
   \begin{rem} \label{adj-uw}
Regarding Proposition \ref{polar}, note that
$\efww(f)=\efw(f)$ a.e.\ $[\mu_{w}]$ for every
$\ascr$-measurable function $f\colon X \to \rbop$.
Indeed, this is because
   \begin{align*}
\int_{\phi^{-1}(\varDelta)} f \D \mu_{\widetilde w} &
\overset{\eqref{cg}}= \int_X (q_{\varDelta}\circ \phi)^2
\cdot f \D \mu_w
\\
&\hspace{.7ex}= \int_X (q_{\varDelta}\circ \phi)^2
\cdot \efw(f) \D \mu_w \overset{\eqref{cg}}=
\int_{\phi^{-1}(\varDelta)} \efw(f) \D \mu_{\widetilde
w}, \quad \varDelta \in \ascr.
   \end{align*}
   \end{rem}
Using Proposition \ref{polar}(i) and the well-known
fact that a Hilbert space operator $V\in \ogr{\hh}$ is
an isometry if and only if $|V|=I$, we obtain the
following characterization of isometric weighted
composition operators.
   \begin{align} \label{chisom}
   \begin{minipage}{70ex}
{\em If \eqref{stand1} holds, then $\cfw$ is an
isometry on $L^2(\mu)$ if and only if $\hfw=1$ a.e.\
$[\mu]$.}
   \end{minipage}
   \end{align}
   \section{A characterization of quasinormality}
In this section, we characterize quasinormal weighted
composition operators. This basic characterization
will be used several times in subsequent chapters.
   \begin{thm} \label{quain}
If \eqref{stand2} holds and $\cfw$ is densely defined,
then $\cfw$ is quasinormal if and only if $\hfw \circ
\phi = \hfw$ a.e.\ $[\mu_w]$.
   \end{thm}
   \begin{proof}
It follows from Proposition \ref{polar}(i) that
$|\cfw|^2 = M_{\hfw}$. We claim that
   \begin{align} \label{cl1}
\dz{\cfw |\cfw|^2} = L^2((1+ \hfw^3) \D \mu).
   \end{align}
Indeed, if $f \in \dz{\cfw |\cfw|^2}$, then $f \in
\dz{|\cfw|^2} = L^2((1+\hfw^2) \D \mu)$ and, by
Proposition \ref{lemS1}(i),
   \begin{align*}
\int_X |f|^2 \hfw^3 \D \mu = \int_X |M_{\hfw} f|^2
\hfw \D\mu < \infty,
   \end{align*}
which yields $f \in L^2((1+\hfw^2 + \hfw^3) \D \mu) =
L^2((1+\hfw^3) \D \mu)$. Reversing the above reasoning
proves \eqref{cl1}.

Suppose $\cfw$ is quasinormal. By Proposition
\ref{lemS2}, $\hfw < \infty$ a.e.\ $[\mu]$. Let
$\{X_n\}_{n=1}^\infty$ be as in Lemma \ref{aproks}. In
view of \eqref{cl1}, $\{\chi_{X_n}\}_{n=1}^\infty
\subseteq \dz{\cfw |\cfw|^2}$. Hence
   \begin{align} \notag
w \cdot (\hfw \circ \phi) \cdot (\chi_{X_n} \circ
\phi) & = \cfw |\cfw|^2 \chi_{X_n}
\overset{\eqref{QQQ}}= |\cfw|^2\cfw \chi_{X_n}
   \\  \label{quaZ}
&= w \cdot \hfw \cdot (\chi_{X_n} \circ \phi) \text{
a.e.\ $[\mu]$}, \quad n\in \nbb.
   \end{align}
Since $\phi^{-1}(X_n) \nearrow X$ as $n\to \infty$, we
see that $w\cdot(\hfw\circ \phi) = w \cdot \hfw$ a.e.\
$[\mu]$, or equivalently that $\hfw \circ \phi = \hfw$
a.e.\ $[\mu_w]$.

Assume now that $\hfw \circ \phi = \hfw$ a.e.\
$[\mu_w]$. We claim that
   \begin{align} \label{cl2}
\dz{|\cfw|^2 \cfw} = L^2((1+ \hfw^3) \D \mu).
   \end{align}
Indeed, if $f \in \dz{|\cfw|^2 \cfw}$, then it follows
from Proposition \ref{polar}(i) that $\cfw f \in
L^2((1+ \hfw^2) \D \mu)$ and
   \begin{align*}
\int_X \hfw^3 |f|^2 \D \mu &\overset{\eqref{l2}}=
\int_X (\hfw^2 \circ \phi) \cdot (|f|^2 \circ \phi) \D
\mu_w
   \\&= \int_X \hfw^2 \cdot (|f|^2 \circ \phi) \D
\mu_w = \int_X \hfw^2 |\cfw f|^2 \D \mu < \infty,
   \end{align*}
which implies that $f \in L^2((1+ \hfw^3)\D \mu)$.
Reversing the above reasoning proves \eqref{cl2}.
Combining \eqref{cl1} and \eqref{cl2} shows that
$\dz{\cfw |\cfw|^2} = \dz{|\cfw|^2\cfw}$. An
appropriate modification of \eqref{quaZ} gives $\cfw
|\cfw|^2 = |\cfw|^2\cfw $. Applying \eqref{QQQ}
completes the proof.
   \end{proof}
   \chapter{Subnormality - General Criteria}
The main goal of this chapter is to provide criteria
for subnormality of (not necessarily bounded) weighted
composition operators. The first criterion, which is
given in Section \ref{Sec3.1}, requires that $\hfw >
0$ a.e.\ $[\mu_w]$ and that there exists a measurable
family of Borel probability measures on $\rbb_+$ which
satisfies the consistency condition \eqref{cc} (see
Theorem \ref{MAIN1}). Section \ref{sectcc-1} provides
the second criterion which involves another, stronger
than \eqref{cc}, condition \eqref{cc-1mu} (see Theorem
\ref{main2}). In Section \ref{sec2.4}, we discuss the
interplay between the conditions \eqref{cc} and
\eqref{cc-1mu} (see Theorem \ref{MAIN2}). Section
\ref{Sec3.2} shows that the consistency condition
\eqref{cc} itself is not sufficient for subnormality.
By Theorem \ref{main2}, this means that \eqref{cc}
does not imply \eqref{cc-1mu}.
   \section{\label{Sec3.1}General scheme}
Let $(X, \ascr)$ be a measurable space. A mapping
$P\colon X \times \borel{\rbb_+} \to [0,1]$ is called
an {\em $\ascr$-measurable family of probability
measures} if the set-function $P(x,\cdot)$ is a
probability measure for every $x \in X$ and the
function $P(\cdot,\sigma)$ is $\ascr$-measurable for
every $\sigma \in \borel{\rbb_+}$. Denote by $\atb$
the $\sigma$-algebra generated by the family
   \begin{align*}
\text{$\ascr \boxtimes \borel{\rbb_+}:=\{\varDelta \times \sigma\colon
\varDelta \in \ascr, \, \sigma \in \borel{\rbb_+}\}$.}
   \end{align*}
If $\mu\colon \ascr \to \rbop$ is a $\sigma$-finite
measure, then, by \cite[Theorem 2.6.2]{Ash}, there
exists a unique $\sigma$-finite measure $\rho$ on
$\atb$ such that
   \begin{align} \label{rhoabs}
\rho(\varDelta \times \sigma) = \int_{\varDelta} P(x,\sigma) \mu(\D x),
\quad \varDelta\in \ascr, \sigma \in \borel{\rbb_+}.
   \end{align}
Moreover, for every $\atb$-measurable function
$f\colon X \times \rbb_+ \to \rbop$,
   \begin{gather} \label{pme}
\text{the function $X \ni x \to \int_0^\infty f (x,t) P(x,\D t) \in \rbop$
is $\ascr$-measurable}
   \end{gather}
and
   \begin{gather}\label{intrho}
\int_{X\times \rbb_+} f \D\rho= \int_X\int_0^\infty f (x,t) P(x,\D t)
\mu(\D x).
   \end{gather}
Clearly the measure $\rho$ depends on $P$, but since
we do not exploit this fact, we will not make this
dependence explicit. Let $w\colon X \to \cbb$ be an
$\ascr$-measurable function and $\phi$ be an
$\ascr$-measurable transformation of $X$. Define the
function $W\colon X \times \rbb_+ \to \cbb$ and the
transformation $\varPhi$ of $X\times \rbb_+$ by
   \begin{align}  \label{W}
 W(x,t) & = w(x), \quad \hspace{3.3ex}x \in X, \, t \in \rbb_+,
   \\ \label{vaPhi}
\varPhi(x,t) & = (\phi(x),t), \quad x \in X, t \in \rbb_+.
   \end{align}
It is easily seen that $W$ and $\varPhi$ are $\atb$-measurable. According
to our convention, the measure $\rho_W$ is defined as follows
   \begin{align*}
\rho_{W}(E) = \int_E |W|^2 \D\rho \overset{
\eqref{intrho}}= \int_X \int_0^\infty \chi_E(x,t)
P(x,\D t) \D \mu_w(x), \quad E \in \atb.
   \end{align*}
In what follows, we regard $\CFW$ as a weighted composition operator in
$L^2(\rho)$. There is a natural way of looking at $L^2(\mu)$ as a subspace
of $L^2(\rho)$. Namely, by \eqref{intrho}, the mapping $U\colon L^2(\mu)
\to L^2(\rho)$ given by
   \begin{align*}
(Uf)(x,t)=f(x), \quad x \in X, \, t \in \rbb_+, \, f \in L^2(\mu),
   \end{align*}
is well-defined, linear and isometric. Moreover, if $\CFW$ is
well-defined, then, combining Proposition \ref{wco1}, Lemma \ref{lemS8}
and \eqref{intrho}, we deduce that $\cfw$ is well-defined~ and
   \begin{align} \label{ucu}
U\cfw = \CFW U.
   \end{align}

In order to make the paper more readable, we gather
the following assumptions.
   \begin{align} \label{stand3}  \tag{$\mathrm{AS3}$}
   \begin{minipage}{69ex} The triplet $(X,\ascr,\mu)$
is a $\sigma$-finite measure space, $w$ is an
$\ascr$-meas\-urable complex function on $X$, $\phi$
is an $\ascr$-measurable transformation of $X$ and
$P\colon X \times \borel{\rbb_+} \to [0,1]$ is an
$\ascr$-measurable family of probability measures. The
measure $\rho$, the function $W$ and the
transformation $\varPhi$ are determined by
\eqref{rhoabs}, \eqref{W} and \eqref{vaPhi},
respectively.
   \end{minipage}
   \end{align}
We begin by proving a formula that connects $\hfw$
with $\HFW$ via $\efw$, where $\HFW \colon X\times
\rbb_+\to \rbop$ is a unique (up to a set of
$\rho$-measure zero) $\atb$-measurable function such
that
   \begin{align*}
\rho_W \circ \varPhi^{-1}(\varDelta) =
\int_{\varDelta} \HFW \D \rho, \quad E \in \atb.
   \end{align*}
Since $\rho$ is $\sigma$-finite, such an $\HFW$ exists
due to the Radon-Nikodym theorem.
   \begin{cau}
The reader should be aware of the fact that both the
measure $\rho$ and the Radon-Nikodym derivative $\HFW$
depend on $P$. For notational convenience this
dependence will not be expressed explicitly.
   \end{cau}
   \begin{lem} \label{lemS8}
Suppose \eqref{stand3} holds and $\rho_W \circ
\varPhi^{-1} \ll \rho$. Then $\mu_w \circ \phi^{-1}
\ll \mu$. Moreover, if $\hfw < \infty$ a.e.\ $[\mu]$,
then $\HFW < \infty$ a.e.\ $[\rho]$ and
   \begin{multline}
\Big(\efw(P(\cdot,\sigma)) \circ \phi^{-1}\Big)(x)
\cdot \hfw(x)
   \\ \label{for1}
= \int_{\sigma} \HFW(x,t) P(x,\D t) \text{ for
$\mu$-a.e.\ } x \in X, \quad \sigma \in
\borel{\rbb_+}.
   \end{multline}
   \end{lem}
   \begin{proof}
To prove that $\mu_w \circ \phi^{-1} \ll \mu$, take $\varDelta \in \ascr$
such that $\mu(\varDelta)=0$. Then, by \eqref{rhoabs}, $\rho(\varDelta
\times \rbb_+) = 0$. Hence, in view of \eqref{intrho}, we have
   \begin{align*}
\mu_w \circ \phi^{-1}(\varDelta) = \int_{\phi^{-1}(\varDelta)} |w(x)|^2
\int_0^\infty P(x,\D t) \mu(\D x) = \int_{\varPhi^{-1}(\varDelta \times
\rbb_+)} |W|^2 \D \rho = 0.
   \end{align*}

Assume additionally that $\hfw < \infty$ a.e.\
$[\mu]$. If $\varDelta \in \ascr$ and $\sigma \in
\borel{\rbb_+}$, then \allowdisplaybreaks
   \begin{align} \notag
\int_{\varDelta} \int_{\sigma} \HFW(x,t) P(x,\D t)
\mu(\D x) & \overset{\eqref{intrho}}= \int_{\varDelta
\times \sigma} \HFW \D \rho \overset{\eqref{l1}}=
\int_{\varPhi^{-1}(\varDelta \times \sigma)} |W|^2 \D
\rho
   \\  \notag
& \overset{\eqref{intrho}}= \int_{\phi^{-1}(\varDelta)} |w(x)|^2
\int_{\sigma} P(x,\D t) \mu(\D x)
   \\  \notag
& \hspace{.7ex} = \int_{\phi^{-1}(\varDelta)} P(x,\sigma) \mu_w(\D x)
   \\  \notag
& \hspace{0ex} \overset{\eqref{wazny}}=
\int_{\phi^{-1}(\varDelta)} \efw(P(\cdot,\sigma)) \D
\mu_w
\\   \label{hle}
& \hspace{.5ex}\overset{(\dag)}= \int_{\varDelta}
\efw(P(\cdot,\sigma))\circ \phi^{-1} \cdot \hfw \D
\mu,
   \end{align}
where $(\dag)$ follows from \eqref{l2} and \eqref{fifi}. Since $\mu$ is
$\sigma$-finite, \eqref{for1} holds.

It suffices to show that $\HFW < \infty$ a.e.\
$[\rho]$. Let $\{X_n\}_{n=1}^\infty$ be as in Lemma
\ref{aproks}. Then \allowdisplaybreaks
   \begin{align*}
\int_{X_n \times \rbb_+} \HFW \D \rho
&\overset{\eqref{intrho}}= \int_{X_n} \int_{\rbb_+}
\HFW(x,t) P(x,\D t) \mu(\D x)
   \\
&\overset{\eqref{for1}}= \int_{X_n}
\Big(\efw(P(\cdot,\rbb_+)) \circ \phi^{-1}\Big)(x)
\cdot \hfw(x) \mu(\D x)
   \\
&\overset{\eqref{jedynka}}= \int_{X_n} \hfw \D \mu \Le
n \cdot \mu(X_n) < \infty, \quad n \in \nbb,
   \end{align*}
which implies that $\HFW < \infty$ a.e.\ $[\rho]$ on
$X_n \times \rbb_+$ for every $n\in \nbb$, and thus
$\HFW < \infty$ a.e.\ $[\rho]$. This completes the
proof.
   \end{proof}
Below we introduce the {\em consistency condition}
\eqref{cc} which plays the crucial role in this paper.
In the case of composition operators, it coincides
with the consistency condition that appeared in
\cite{b-j-j-sS} (consult Sect.\ \ref{pco}).
   \begin{lem} \label{lemS9}
Suppose \eqref{stand3} holds, $\mu_w \circ \phi^{-1}
\ll \mu$, $\hfw < \infty$ a.e.\ $[\mu]$ and the
following condition is satisfied\/\footnote{\;By Lemma
\ref{lemS6} and \eqref{pme}, the right-hand side of
the equality in \eqref{cc} is an $\ascr$-measurable
function defined a.e.\ $[\mu]$.}{\em :}
   \begin{align} \tag{CC} \label{cc}
\efw(P(\cdot, \sigma)) (x) = \frac{\int_{\sigma} t
P(\phi(x),\D t)}{\hfw(\phi(x))} \; \text{ for
$\mu_w$-a.e.\ $x \in X$}, \quad \sigma \in
\borel{\rbb_+}.
   \end{align}
Then $\rho_W \circ \varPhi^{-1} \ll \rho$ and
   \begin{align*}
\HFW(x,t) = \chi_{\{\hfw > 0\}} (x) \cdot t \; \text{
for $\rho$-a.e.\ $(x,t) \in X \times \rbb_+$.}
   \end{align*}
   \end{lem}
   \begin{proof}
Arguing as in \eqref{hle}, we get \allowdisplaybreaks
   \begin{align} \notag
\rho_W (\varPhi^{-1}(\varDelta \times \sigma)) &=
\int_{\varPhi^{-1}(\varDelta \times \sigma)} |W|^2
\D\rho = \int_{\phi^{-1}(\varDelta)} \efw(P(\cdot,
\sigma)) \D \mu_w
   \\      \notag
&\hspace{-1.1ex}\overset{\eqref{cc}}= \int_X
\chi_{\varDelta}(\phi(x)) \cdot \frac{\int_{\sigma} t
P(\phi(x),\D t)}{\hfw(\phi(x))} \mu_w (\D x)
   \\            \notag
&\hspace{-.2ex}\overset{\eqref{hle2}}= \int_{\{\hfw >
0\}} \chi_{\varDelta}(x) \int_{\sigma} t P(x,\D t)
\mu(\D x)
      \\   \label{HSF}
&\hspace{-.5ex}\overset{\eqref{intrho}}=
\int_{\varDelta \times \sigma} \chi_{\{\hfw > 0\}}(x)
\cdot t \D \rho(x,t), \quad \varDelta \in \ascr, \,
\sigma \in \borel{\rbb_+}.
   \end{align}
It is clear that $\pscr:=\ascr \boxtimes
\borel{\rbb_+}$ is a semi-algebra and
$\atb=\sigma(\pscr)$. Let $\{X_n\}_{n=1}^\infty$ be as
in Lemma \ref{aproks}. Then $\{X_n \times
\rbb_+\}_{n=1}^\infty \subseteq \pscr$ and
   \begin{align} \notag
\rho_W(\varPhi^{-1}(X_n \times \rbb_+)) & \overset{\eqref{rhoabs}}=
\int_{\phi^{-1}(X_n)} |w|^2 \D \mu =\mu_w \circ \phi^{-1}(X_n)
   \\   \label{HSF2}
& \hspace{.7ex}= \int_{X_n} \hfw \D \mu \Le n \cdot
\mu(X_n)< \infty, \quad n \in \nbb.
   \end{align}
Combining \eqref{HSF} and \eqref{HSF2} with Lemma \ref{2miary}, we get
   \begin{align*}
\rho_W \circ \varPhi^{-1}(E)=\int_{E} \chi_{\{\hfw >
0\}}(x) \cdot t \D \rho(x,t), \quad E \in \atb.
   \end{align*}
This completes the proof.
   \end{proof}
Now we provide some characterizations of the consistency condition
\eqref{cc} that will be used later in this paper.
   \begin{lem} \label{lemS10}
Suppose \eqref{stand3} holds, $\mu_w \circ \phi^{-1}
\ll \mu$ and $\hfw < \infty$ a.e.\ $[\mu]$. Then the
following conditions are equivalent{\em :}
   \begin{enumerate}
   \item[(i)] $P$ satisfies \eqref{cc},
   \item[(ii)] $\big(\efw(P(\cdot, \sigma))\circ
\phi^{-1}\big) (x) \cdot \hfw(x) = \chi_{\{\hfw >
0\}}(x) \int_{\sigma} t P(x,\D t)$ for $\mu$-a.e.\ $x
\in X$ and for every $\sigma \in \borel{\rbb_+}$,
   \item[(iii)] $\rho_W \circ \varPhi^{-1} \ll \rho$
and $\HFW(x,t) = \chi_{\{\hfw > 0\}} (x) \cdot t$ for
$\rho$-a.e.\ $(x,t) \in X \times \rbb_+$,
   \item[(iv)] $\rho_W \circ \varPhi^{-1} \ll \rho$
and $\int_{\sigma} \HFW(\phi(x),t) P(\phi(x), \D t) =
\int_{\sigma} t P(\phi(x), \D t)$ for $\mu_w$-a.e.\
$x\in X$ and for every $\sigma \in \borel{\rbb_+}$.
   \end{enumerate}
   \end{lem}
   \begin{proof}
(i)$\Leftrightarrow$(ii) This can be proved by using \eqref{fifi}, Lemma
\ref{nuklear} and Lemma \ref{lemS6}(i).

(i)$\Rightarrow$(iii) Apply Lemma \ref{lemS9}.

(iii)$\Rightarrow$(iv) Note that \allowdisplaybreaks
   \begin{align*}
\int_{\phi^{-1}(\varDelta)} \int_{\sigma} &
\HFW(\phi(x),t) P(\phi(x), \D t) \mu_w(\D x)
   \\
& = \int_{\varDelta} \hfw(x) \int_{\sigma} \HFW(x,t)
P(x, \D t) \mu(\D x)
   \\
& = \int_{\varDelta \times \sigma} \hfw(x) \HFW(x,t)
\D \rho(x,t)
   \\
& \hspace{-.6ex}\overset{\mathrm{(iii)}}=
\int_{\varDelta \times \sigma} \hfw(x) \cdot
\chi_{\{\hfw > 0\}}(x) \cdot t \D \rho(x,t)
   \\
& = \int_{\varDelta} \hfw(x) \chi_{\{\hfw
> 0\}}(x) \int_{\sigma} t P(x,\D t) \mu(\D x)
   \\
& = \int_{\phi^{-1}(\varDelta)} \chi_{\{\hfw
> 0\}}(\phi(x)) \int_{\sigma} t P(\phi(x),\D t) \mu_w(\D x)
   \\
& \hspace{-1.8ex}\overset{\mathrm{Lem.} \, \ref{lemS6}}=
\int_{\phi^{-1}(\varDelta)} \int_{\sigma} t P(\phi(x),\D t) \mu_w(\D x),
\quad \varDelta \in \ascr, \, \sigma \in \borel{\rbb_+}.
   \end{align*}
Since, by Proposition \ref{lemS2}, the measure $\mu_w|_{\phi^{-1}(\ascr)}$
is $\sigma$-finite, we get (iv).

(iv)$\Rightarrow$(i) By Lemma \ref{lemS8},
\eqref{for1} holds. Composing both sides of
\eqref{for1} with $\phi$ and using Lemma
\ref{nuklear}, the equality \eqref{fifi} and (iv), we
obtain (i). This completes the proof.
   \end{proof}
The next lemma deals with a variation of the consistency condition
\eqref{cc}.
   \begin{lem} \label{lemS10.1}
Suppose \eqref{stand3} holds, $\mu_w \circ \phi^{-1}
\ll \mu$, $\hfw < \infty$ a.e.\ $[\mu]$ and
   \begin{multline} \label{cc-1}
\big(\efw(P(\cdot, \sigma))\circ \phi^{-1}\big) (x)
\cdot \hfw(x) = \int_{\sigma} t P(x,\D t) \quad
\text{for $\mu_w$-a.e.\ $x \in X$}
   \\
\text{and for every $\sigma \in \borel{\rbb_+}$.}
   \end{multline}
Then for every Borel function $f\colon \rbb_+ \to \rbop$,
   \begin{multline} \label{calka}
\bigg(\efw\bigg(\int_0^\infty f(t) P(\cdot, \D
t)\bigg)\circ \phi^{-1}\bigg) (x) \cdot \hfw(x) =
\int_0^\infty t \cdot f(t) P(x,\D t)
   \\
\text{ for $\mu_w$-a.e.\ $x \in X$}.
   \end{multline}
   \end{lem}
   \begin{proof}
By Proposition \ref{lemS4}, \eqref{calka} holds for every simple Borel
function $f\colon \rbb_+ \to \rbb_+$. Let $f\colon \rbb_+ \to \rbop$ be a
Borel function. Take a sequence $\{s_n\}_{n=1}^\infty$ of simple Borel
functions $s_n\colon \rbb_+ \to \rbb_+$ which is monotonically increasing
and pointwise convergent to $f$. Then, by Lebesgue's monotone convergence
theorem, we have \allowdisplaybreaks
   \begin{align*}
\int_{\varDelta} \efw\bigg(&\int_0^\infty s_n(t)
P(\cdot, \D t)\bigg)\circ \phi^{-1} \cdot \hfw \D
\mu_w
   \\
&\overset{(\dag)}= \int_{\phi^{-1}(\varDelta)}
\efw\bigg(\int_0^\infty s_n(t) P(\cdot, \D t)\bigg)
\cdot |w\circ \phi|^2 \D \mu_w
   \\
&= \int_{\phi^{-1}(\varDelta)} \int_0^\infty s_n(t) P(x, \D t)
|w(\phi(x))|^2 \mu_w(\D x)
   \\
& \hspace{22ex} \bigg\downarrow \mbox{\small{$n\to\infty$}}
   \\
&\hspace{2.7ex}\int_{\phi^{-1}(\varDelta)} \int_0^\infty f(t) P(x, \D t)
|w(\phi(x))|^2 \mu_w(\D x)
   \\
&\hspace{6ex}\overset{(\dag)}= \int_{\varDelta}
\efw\bigg(\int_0^\infty f(t) P(\cdot, \D t)\bigg)\circ
\phi^{-1} \cdot \hfw \D \mu_w, \quad \varDelta \in
\ascr,
  \end{align*}
where $(\dag)$ follows from \eqref{l2} and \eqref{fifi}. Since
   \begin{align*}
\int_{\varDelta} \int_0^\infty t \cdot s_n(t) P(x,\D t) \mu_w(\D x)
\xlongrightarrow{n \to \infty} \int_{\varDelta} \int_0^\infty t \cdot f(t)
P(x,\D t) \mu_w(\D x), \quad \varDelta \in \ascr,
   \end{align*}
and $\mu_w$ is $\sigma$-finite, the proof is complete.
   \end{proof}
Applying the equivalence (i)$\Leftrightarrow$(ii) of Lemma \ref{lemS10},
we obtain the following.
   \begin{pro}\label{nicto}
Suppose \eqref{stand3} holds, $\mu_w \circ \phi^{-1}
\ll \mu$ and $\hfw < \infty$ a.e.\ $[\mu]$. Then the
following two assertions are valid.
   \begin{enumerate}
   \item[(i)] If \eqref{cc-1} holds and
$\big(\efw(P(\cdot, \sigma))\circ \phi^{-1}\big) (x)
\cdot \hfw(x) = \int_{\sigma} t P(x,\D t)$ for
$\mu$-a.e.\ $x \in \{\hfw > 0\} \cap \{w=0\}$ and for
every $\sigma \in \borel{\rbb_+}$, then \eqref{cc}~
holds.
   \item[(ii)] If \eqref{cc} holds and
$\int_0^\infty t P(x,\D t) = 0$ for $\mu$-a.e.\ $x \in
\{\hfw = 0\} \cap \mbox{$\{w \neq 0\}$}$, then
\eqref{cc-1} holds.
   \end{enumerate}
   \end{pro}
Regarding Proposition \ref{nicto}(i), we note that if
$\big(\efw(P(\cdot, \sigma)) \circ \phi^{-1}\big) (x)
\cdot \hfw(x) = \int_{\sigma} t P(x,\D t)$ for
$\mu$-a.e.\ $x \in \varTheta_{+0} := \{\hfw > 0\} \cap
\{w=0\}$ and for every $\sigma \in \borel{\rbb_+}$,
then, by \eqref{jedynka}, $\int_0^\infty t P(x,\D t) =
0$ for $\mu$-a.e.\ $x \in \varTheta_{+0}$ if and only
if $\mu(\varTheta_{+0})=0$, or equivalently, if and
only if $\big(\efw(P(\cdot, \sigma)) \circ
\phi^{-1}\big) (x)= 0$ for $\mu$-a.e.\ $x \in
\varTheta_{+0}$.

The $n$-th power of a weighted composition operator
$\cfw$ is related in a natural way to the weighted
composition operator $C_{\phi^n,\widehat w_n}$ with
explicitly given weight $\widehat w_n$. Below, we
provide recurrence formulas for the Radon-Nikodym
derivatives $\hfwn{n}$ attached to $C_{\phi^n,\widehat
w_n}$.
   \begin{lem} \label{lemS11}
Suppose \eqref{stand2} holds. Then the following
assertions are valid{\em :}
   \begin{enumerate}
   \item[(i)] $C_{\phi^n,\widehat w_n}$ is well-defined
and $\cfw^n \subseteq C_{\phi^n,\widehat w_n}$ for
every $n\in \zbb_+$, where
   \begin{align} \label{Zak}
\text{$\widehat w_0 = \boldsymbol{1}$ and $\widehat
w_{n+1} = \prod_{j=0}^n w\circ \phi^j$ for $n\in
\zbb_+$,}
   \end{align}
   \item[(ii)] if $\cfw$ is densely defined, then
   \allowdisplaybreaks
   \begin{align}  \label{rec1}
\hfwn{n+1} & = \efw\big(\hfwn{n}\big)\circ \phi^{-1}
\cdot \hfw \text{ a.e.\ $[\mu]$}, \quad n\in\zbb_+,
      \\  \label{rec2}
\hfwn{n+1} \circ \phi & = \efw\big(\hfwn{n}\big) \cdot
(\hfw \circ \phi) \text{ a.e.\ $[\mu_w]$}, \quad
n\in\zbb_+,
   \end{align}
   \item[(iii)] if $\cfw$ is quasinormal, then
$\hfwn{n}=\hfw^n$ a.e.\ $[\mu]$ for all $n\in \zbb_+$.
   \end{enumerate}
   \end{lem}
   \begin{proof}
(i) Take $\varDelta \in \ascr$ such that
$\mu(\varDelta)=0$. Then by Proposition \ref{wco1}, we
see that $\widehat w_n \cdot
\chi_{(\phi^n)^{-1}(\varDelta)} = \widehat w_n \cdot
(\chi_{\varDelta}\circ \phi^n) = 0$ a.e.\ $[\mu]$ for
every $n\in \zbb_+$, which means that $(\mu_{\widehat
w_n} \circ (\phi^n)^{-1}) (\varDelta)=0$. Applying
Proposition \ref{wco1} again, we conclude that
$C_{\phi^n,\widehat w_n}$ is well-defined. The
inclusion $\cfw^n \subseteq C_{\phi^n,\widehat w_n}$
is easily seen to be true.

(ii) By Lemma \ref{nuklear} and \eqref{fifi}, the
equality \eqref{rec2} follows from \eqref{rec1}. To
prove \eqref{rec1}, note that $\widehat w_{n+1} = w
\circ \phi^n \cdot \widehat w_n$ for $n\in \zbb_+$.
Hence, by \eqref{l2} and \eqref{fifi}, we have
   \allowdisplaybreaks
   \begin{align*}
\mu_{\widehat w_{n+1}}((\phi^{n+1})^{-1}(\varDelta)) &
= \int_X \chi_{\phi^{-1}(\varDelta)} \circ \phi^n
\cdot |w\circ \phi^n|^2 \D \mu_{\widehat w_n}
   \\
&= \int_{\phi^{-1}(\varDelta)} |w|^2 \hfwn{n} \D \mu =
\int_{\phi^{-1}(\varDelta)} \efw(\hfwn{n}) \D \mu_w
   \\
&= \int_{\varDelta} \efw(\hfwn{n})\circ \phi^{-1}
\cdot \hfw \D \mu, \quad \varDelta \in \ascr, \, n \in
\zbb_+,
   \end{align*}
which implies \eqref{rec1}.

(iii) We prove it by induction on $n$. The case of
$n=0$ is obviously true. If $\hfwn{n}=\hfw^n$ a.e.\
$[\mu]$ for a fixed $n\in \zbb_+$, then by Theorem
\ref{quain} we have
   \begin{align*}
\hfwn{n+1} & \overset{\eqref{rec1}}=
\efw\big(\hfwn{n}\big)\circ \phi^{-1} \cdot \hfw
   \\
&\hspace{.75ex} = \efw\big(\hfw^{n}\big)\circ
\phi^{-1} \cdot \hfw
   \\
&\hspace{.75ex} = \efw\big(\hfw^{n}\circ
\phi\big)\circ \phi^{-1} \cdot \hfw
   \\
& \hspace{.45ex}\overset{(\dag)}= \hfw^{n} \cdot \hfw
\\
& \hspace{.75ex} = \hfw^{n+1} \text{ a.e.\ $[\mu]$,}
   \end{align*}
where $(\dag)$ follows from the fact that $\hfw^n=0$
on $\{\hfw=0\}$ a.e.\ $[\mu]$.
   \end{proof}
The result that follows will be used in the proof of
Theorem \ref{MAIN1}. It clarifies the role played by
the assumption ``$\hfw > 0$ a.e.\ $[\mu_w]$'' in this
theorem.
   \begin{thm}  \label{main1}
Suppose \eqref{stand3} holds, $\mu_w \circ \phi^{-1}
\ll \mu$ and $\hfw < \infty$ a.e.\ $[\mu]$. If $P$
satisfies \eqref{cc}, then the following conditions
are equivalent{\em :}
   \begin{enumerate}
   \item[(i)] $P$ satisfies \eqref{cc-1},
   \item[(ii)] $\hfwn{n}(x) =
\int_0^\infty t^n P(x,\D t)$ for every $n \in \zbb_+$ and for
$\mu_w$-a.e.\ $x\in X$,
   \item[(iii)] $\int_0^\infty t P(x,\D t) = 0$ for
$\mu_w$-a.e.\ $x \in \{\hfw = 0\}$,
   \item[(iv)] $P(x,\cdot) = \delta_0(\cdot)$ for
$\mu_w$-a.e.\ $x \in \{\hfw = 0\}$,
   \item[(v)] $\rho_W \circ \varPhi^{-1} \ll \rho$ and
$\HFW(x,t) = t$ for $\rho_W$-a.e.\ $(x,t) \in X \times
\rbb_+$,
   \item[(vi)] $\rho_W \circ \varPhi^{-1} \ll \rho$ and
$\HFW \circ \varPhi = \HFW$ a.e.\ $[\rho_W]$,
   \item[(vii)] $\hfw > 0$ a.e.\ $[\mu_w]$.
   \end{enumerate}
   \end{thm}
   \begin{proof}
(i)$\Rightarrow$(ii) We use induction to prove that
$\hfwn{n}=H_{n}$ a.e.\ $[\mu_w]$ for all $n\in
\zbb_+$, where $H_n(x) = \int_0^\infty t^n P(x,\D t)$
for $x\in X$. The case of $n=0$ is obvious. Suppose
$\hfwn{n}=H_{n}$ a.e.\ $[\mu_w]$ for some unspecified
$n \in \zbb_+$. Then, by Lemma \ref{lemS10.1},
\eqref{calka} holds with $f(t) = t^n$. Hence, we have
   \begin{align}  \notag
H_{n+1}(x) & \overset{\eqref{calka}}= (\efw(H_n)\circ
\phi^{-1}) (x) \cdot \hfw(x)
   \\  \notag
&\hspace{0.7ex}= (\efw(\hfwn{n})\circ \phi^{-1}) (x)
\cdot \hfw(x)
   \\ \label{Arab}
& \overset{\eqref{rec1}}= \hfwn{n+1} \text{ for
$\mu_w$-a.e.\ $x\in X$,}
   \end{align}
where the second and third equalities in \eqref{Arab}
hold a.e.\ $[\mu]$. This completes the induction
argument and gives (ii).

(ii)$\Rightarrow$(iii) Consider the equality in (ii) with $n=1$.

(iii)$\Leftrightarrow$(iv) Use Lemma \ref{lemS3}.

(iii)$\Rightarrow$(v) It follows that
   \begin{align} \label{dinner}
\int_0^\infty t \cdot \chi_E(x,t) P(x,\D t) = 0 \text{
for $\mu_w$-a.e.\ $x \in \{\hfw = 0\}$}, \quad E \in
\atb.
   \end{align}
Hence, by the implication (i)$\Rightarrow$(iii) of Lemma \ref{lemS10},
$\rho_W \circ \varPhi^{-1} \ll \rho$ and \allowdisplaybreaks
   \begin{align*}
\int_E \HFW \D \rho_W & = \int_E |w(x)|^2 \HFW(x,t) \D
\rho(x,t)
   \\
& = \int_E \chi_{\{\hfw > 0\}} (x) \cdot t \cdot
|w(x)|^2\D \rho(x,t)
   \\
& = \int_X \chi_{\{\hfw > 0\}} (x) \int_0^\infty t
\cdot \chi_E(x,t) P(x,\D t) \mu_w(\D x)
   \\
& \hspace{-.8ex}\overset{\eqref{dinner}}= \int_X \int_0^\infty t \cdot
\chi_E(x,t) P(x,\D t) \mu_w(\D x)
   \\
& = \int_E t \D \rho_W (x,t), \quad E \in \atb.
   \end{align*}
Since $\rho_W$ is $\sigma$-finite, (v) holds.

(v)$\Rightarrow$(vi) By the implication (i)$\Rightarrow$(iii) of Lemma
\ref{lemS10} and Lemma \ref{nuklear}, we have
   \begin{align}  \label{zima}
(\HFW \circ \varPhi)(x,t) = \chi_{\{\hfw >
0\}}(\phi(x)) \cdot t \; \text{ for $\rho_W$-a.e.\
$(x,t) \in X \times \rbb_+$.}
   \end{align}
Now we note that if $f, g\colon X \to \rbop$ are $\ascr$-measurable
functions such that $f=g$ a.e.\ $[\mu_w]$, then $f (x) \cdot t = g (x)
\cdot t$ for $\rho_W$-a.e.\ $(x,t)\in X\times \rbb_+$. Indeed, this is
because
   \begin{align*}
\int_{E} f(x) \cdot t \D \rho_W(x,t) & = \int_X f(x) \int_0^\infty
\chi_E(x,t) \cdot t \, P(x,\D t) \mu_w(\D x)
   \\
& = \int_{E} g(x) \cdot t \D \rho_W(x,t), \quad E \in \atb.
   \end{align*}
The above property combined with \eqref{zima} and the
fact that $\chi_{\{\hfw > 0\}} \circ \phi = 1$ a.e.\
$[\mu_w]$ (see Proposition \ref{lemS6}(i)) yields
(vi).

(vi)$\Rightarrow$(vii) By Lemma \ref{lemS6}(i) and Proposition \ref{hsfd},
$\chi_{\{W\neq 0\}}\jd{\CFW} = \{0\}$. Take $f\in \chi_{\{w \neq 0\}}
\jd{\cfw}$. By Lemma \ref{jadro}, $f\in \jd{\cfw}$ and thus, by
\eqref{ucu}, $Uf \in \jd{\CFW}$. Since $\{W=0\} = \{w = 0\} \times \rbb_+$
and $f=0$ a.e.\ $[\mu]$ on $\{w=0\}$, we deduce that
   \begin{align*}
\int_{\{W=0\}} |Uf|^2 \D \rho = \int_{\{w=0\}} |f(x)|^2 \int_0^\infty
P(x,\D t) \mu(\D x) = \int_{\{w=0\}} |f(x)|^2 \mu(\D x) = 0.
   \end{align*}
As a consequence, $Uf \in \chi_{\{W\neq 0\}}\jd{\CFW}
= \{0\}$. Since $U$ is injective, we get $f=0$ a.e.\
$[\mu]$. This means that $\chi_{\{w \neq 0\}}
\jd{\cfw} = \{0\}$. It follows from Proposition
\ref{hsfd} that $\hfw
> 0$ a.e.\ $[\mu_w]$.

(vii)$\Rightarrow$(i) Apply the implication (i)$\Rightarrow$(ii) of Lemma
\ref{lemS10}.
   \end{proof}
   \begin{rem} \label{pb}
The implication (vi)$\Rightarrow$(vii) of Theorem
\ref{main1} can be proved in a shorter (but more
advanced) way by applying Theorem \ref{quain}. Indeed,
by Lemmas \ref{lemS2} and \ref{lemS8}, and Theorem
\ref{quain} the operator $\CFW$ is quasinormal. In
view of \eqref{ucu} and the fact that quasinormal
operators are subnormal, $\cfw$ is subnormal. Since
subnormal operators are hyponormal, an application of
Corollary \ref{hipinj} yields $\hfw > 0$ a.e.\
$[\mu_w]$.
   \end{rem}
Theorem \ref{main1} enables us to formulate a
criterion for subnormality of unbounded weighted
composition operators (see Theorem \ref{MAIN1} below).
In Section \ref{sec2.4}, we shall supply an extension
of this criterion based on the condition
\eqref{cc-1mu}, which is a stronger version of
\eqref{cc-1} (see Theorem \ref{MAIN2}). Note that the
assumption ``$\hfw
> 0$ a.e.\ $[\mu_w]$'' that appears in Theorem
\ref{MAIN1} is not restrictive because it is always
satisfied whenever $\cfw$ is subnormal (see Corollary
\ref{hipinj}).
   \begin{thm} \label{MAIN1}
Let $(X,\ascr,\mu)$ be a $\sigma$-finite measure
space, $w$ be an $\ascr$-measurable complex function
on $X$ and $\phi$ be an $\ascr$-measurable
transformation of $X$ such that $\cfw$ is densely
defined and $\hfw > 0$ a.e.\ $[\mu_w]$. Suppose there
exists an $\ascr$-measurable family of probability
measures $P\colon X \times \borel{\rbb_+} \to [0,1]$
that satisfies \eqref{cc}. Then $\cfw$ is subnormal,
$\CFW$ is its quasinormal extension $($see
\eqref{stand3}$)$~ and
   \begin{align} \label{momom}
\hfwn{n}(x) = \int_0^\infty t^n P(x,\D t) \text{ for
every $n \in \zbb_+$ and for $\mu_w$-a.e.\ $x\in X$.}
   \end{align}
   \end{thm}
   \begin{proof}
By Propositions \ref{wco1} and \ref{lemS2}, the assumptions of Theorem
\ref{main1} are satisfied. Hence \eqref{momom} holds and, by Lemma
\ref{lemS8} and Theorem \ref{quain}, $\CFW$ is quasinormal. Employing
\eqref{ucu} completes the proof.
   \end{proof}
   \begin{rem} \label{22.X.2013Daegu}
It is worth pointing out that the above criterion can
be applied to prove the normality of the
multiplication operator $M_w$. This is possible
because in this particular case we can describe
explicitly an $\ascr$-measurable family $P$ of
probability measures that satisfies \eqref{cc}.
Indeed, since $M_w=\cfw$ with $\phi=\id_X$ (see Sect.\
\ref{pco}(a)), we see that the conditional expectation
$\efw$ acts as the identity mapping,
$\mu_w\circ\phi^{-1}\ll\mu$ and $\hfw=|w|^2$ a.e.\
$[\mu]$. Hence $M_w$ is densely defined (see
Proposition \ref{lemS2}). Set
$P(x,\sigma)=\chi_{\sigma}(|w(x)|^2)$ for $x\in X$ and
$\sigma\in\borel{\rbb_+}$. It is easily seen that $P$
is an $\ascr$-measurable family of probability
measures which satisfies \eqref{cc}. By Theorem
\ref{MAIN1}, $M_w$ is subnormal. It follows from
Proposition \ref{adj} that $M_w^*=M_{\bar w}$. As a
consequence, $M_w$ and $M_w^*$ are subnormal. Since
subnormal operators are hyponormal, we conclude that
$M_w$ is normal.
   \end{rem}
   \section{\label{Sec3.2}Injectivity versus $(\mathrm{CC})$}
In this section, we discuss the question of whether
the consistency condition \eqref{cc} implies the
injectivity of a weighted composition operator. A
parallel question, which is related to \cite[Theorem
9]{b-j-j-sS}, can be stated for composition operators.
We answer the latter (and so the former) question
negatively, showing that there exists a non-injective
composition operators $C_{\phi}$ in an $L^2$-space
which admits an $\ascr$-measurable family of
probability measures satisfying \eqref{cc} for
$w=\boldsymbol{1}$. This is possible in both
infinite-dimensional and finite-dimensional
$L^2$-spaces (including the two-dimensional case).
Hence, by \cite[Corollary 6.3]{b-j-j-sC}, the
consistency condition \eqref{cc} alone is not
sufficient for subnormality of $C_{\phi}$. As shown
below, the consistency condition \eqref{cc} may not
even be sufficient for paranormality of $C_{\phi}$.
The lack of paranormality in turn implies that
$C_{\phi}$ does not generate Stieltjes moment
sequences (see the proof of \cite[Theorem
4.1.1(iii)]{j-j-s0}; see also Sect.\ \ref{GSMS}).
   \begin{exa} \label{pico}
Fix $N \in \nbb \cup \{\infty\}$. Set $X = \{j \in
\zbb_+\colon j \Le N\}$ and $\ascr = 2^X$. Let $\mu$
be an arbitrary finite measure on $\ascr$ such that
$\mu(j) > 0$ for every $j\in X$. Define a mapping
$\phi\colon X \to X$ by $\phi(j)=0$ for every $j\in
X$. Let $C_{\phi}$ be the composition operator in
$L^2(\mu)$ (see Sect.\ \ref{pco}(b)). It is clear that
$C_{\phi}$ is a well-defined rank-one operator and
$\mathsf h_{\phi} = \frac{\mu(X)}{\mu(0)} \,
\chi_{\{0\}}$, which implies that $\|C_{\phi}\| =
\sqrt{\frac{\mu(X)}{\mu(0)}}$ (see Proposition
\ref{lemS1}(v)). Since $\mu(\{\mathsf h_{\phi} = 0\})
> 0$, we infer from Lemma \ref{jadro} that $C_{\phi}$
is not injective. What is more, $C_{\phi}$ is not
paranormal because
   \begin{align*}
\|\chi_{\{0\}}\| \|C_{\phi}^2 \chi_{\{0\}}\| =
\sqrt{\mu(0) \mu(X)} < \mu(X) = \|C_{\phi}
\chi_{\{0\}}\|^2.
   \end{align*}
Continuing the example, we now describe
$\ascr$-measurable families of probability measures
that satisfies \eqref{cc} for $w=\boldsymbol{1}$
(though in the discrete case $\ascr$-measurabili\-ty
is automatic, we preserve the original terminology).
   \begin{lem} \label{cc-opis}
Suppose $X$, $\ascr$, $\mu$ and $\phi$ are as above
and $P\colon X \times \borel{\rbb_+} \to [0,1]$ is an
$\ascr$-measurable family of probability measures.
Then the following assertions are valid{\em :}
   \begin{enumerate}
   \item[(i)] $P$ satisfies \eqref{cc} for $w=\boldsymbol{1}$
if and only if the following conditions hold{\em :}
   \begin{gather} \label{ccA}
P(j,[0,1]) = 0, \quad j\in \nbb \cap [1,N],
   \\ \label{ccB}
\theta := \sum_{j=1}^N \frac{\mu(j)}{\mu(0)} \int_0^\infty \frac{1}{t-1}
P(j,\D t) \Le 1,
   \\ \label{ccC}
P(0, \sigma) = \sum_{j=1}^N \frac{\mu(j)}{\mu(0)} \int_{\sigma}
\frac{1}{t-1} P(j,\D t) + (1-\theta) \delta_1(\sigma), \quad \sigma \in
\borel{\rbb_+},
   \end{gather}
   \item[(ii)]  the following condition is necessary
and sufficient for $P$ to be
$\phi^{-1}(\ascr)$-measurable and to satisfy
\eqref{cc} for $w=\boldsymbol{1}${\em :}
   \begin{align} \label{ccD}
P(j,\cdot) = \delta_{\frac{\mu(X)}{\mu(0)}}(\cdot), \quad j\in X.
   \end{align}
   \end{enumerate}
   \end{lem}
   \begin{proof}
(i) Since $\phi^{-1}(\{0\})=X$, we see that
   \begin{align*}
\mathsf{E}_{\phi}(f)(y) = \frac{1}{\mu(X)}\int_X f \D
\mu
   \end{align*}
for all $y \in X$ and for each function $f\colon X \to
\rbb_+$ (we refer the reader to Section \ref{Sec5.5}
for more information on conditional expectation in the
case of discrete measure spaces). Hence, we have
   \begin{align*}
\mathsf h_{\phi}(\phi(y)) \mathsf{E}_{\phi}(P(\cdot,
\sigma))(y) = \sum_{j=0}^N \frac{\mu(j)}{\mu(0)} P(j,
\sigma), \quad y \in X, \, \sigma \in \borel{\rbb_+}.
   \end{align*}
This implies that $P$ satisfies \eqref{cc} for
$w=\boldsymbol{1}$ if and only if
   \begin{align} \label{cc2}
\int_{\sigma} t P(0, \D t) = \sum_{j=0}^N \frac{\mu(j)}{\mu(0)} P(j,
\sigma), \quad \sigma \in \borel{\rbb_+}.
   \end{align}

Suppose $P$ satisfies \eqref{cc} for
$w=\boldsymbol{1}$. Then, by \eqref{cc2}, we have
   \begin{align} \label{cc3}
\int_{\sigma} (t-1) P(0, \D t) = \sum_{j=1}^N \frac{\mu(j)}{\mu(0)} P(j,
\sigma), \quad \sigma \in \borel{\rbb_+}.
   \end{align}
Substituting $\sigma=[0,1]$ into \eqref{cc3}, we deduce that \eqref{ccA}
holds and $P(0,[0,1))=0$. This and \eqref{cc3} lead to
   \begin{align} \notag
P(0, \sigma) & = \int_{\sigma \cap (1,\infty)} \frac{1}{t-1} (t-1) P(0, \D
t) + P(0, \sigma \cap \{1\})
      \\  \label{cc4}
& = \sum_{j=1}^N \frac{\mu(j)}{\mu(0)} \int_{\sigma} \frac{1}{t-1} P(j,\D
t) + P(0, \{1\}) \delta_{1}(\sigma), \quad \sigma \in \borel{\rbb_+}.
   \end{align}
Since $P(0,\cdot)$ is probabilistic, \eqref{cc4} implies \eqref{ccB} and
\eqref{ccC}.

Conversely, if \eqref{ccA}, \eqref{ccB} and
\eqref{ccC} hold, then \eqref{cc3} is easily seen to
be satisfied. Hence, by \eqref{cc2}, $P$ satisfies
\eqref{cc} for $w=\boldsymbol{1}$.

(ii) Clearly $\phi^{-1}(\ascr) = \{\emptyset, X\}$,
and so $P$ is $\phi^{-1}(\ascr)$-measurable if and
only if for every $\sigma \in \borel{\rbb_+}$, $P(j,
\sigma) = P(0,\sigma)$ for all $j\in X$. If $P$ is
$\phi^{-1}(\ascr)$-measurable and satisfies \eqref{cc}
for $w=\boldsymbol{1}$, then, by \eqref{ccA} and
\eqref{ccC}, $P(0, \{1\}) = 0$ and
   \begin{align*}
P(0, \sigma) = \int_{\sigma} \frac{\sum_{j=1}^N
\frac{\mu(j)}{\mu(0)}}{t-1} P(0,\D t), \quad \sigma
\in \borel{\rbb_+}.
   \end{align*}
This implies that $t=\frac{\mu(X)}{\mu(0)}$ for
$P(0,\cdot)$-a.e.\ $t \in \rbb_+$, which yields
\eqref{ccD}. The converse is obvious. This completes
the proof.
   \end{proof}
Going back to our example, we note that if
$\{P(j,\cdot)\}_{j=1}^N$ are Borel probability
measures on $\rbb_+$ which satisfy \eqref{ccA} and
\eqref{ccB}, then the formula \eqref{ccC} defines a
Borel probability measure $P(0, \cdot)$ on $\rbb_+$.
In this way, in view of Lemma \ref{cc-opis}, we get an
$\ascr$-measurable family of probability measures
$P\colon X \times \borel{\rbb_+} \to [0,1]$ that
satisfies \eqref{cc} for $w=\boldsymbol{1}$. In
particular, any sequence $\{P(j,\cdot)\}_{j=1}^N$ of
Borel probability measures on $\rbb_+$ such that
$\supp P(j,\cdot) \subseteq
\big[\frac{\mu(X)}{\mu(0)}, \infty\big)$ for all $j
\in \nbb \cap [1,N]$ fulfils \eqref{ccA} and
\eqref{ccB}, which gives rise to a family $P$
satisfying \eqref{cc} for $w=\boldsymbol{1}$. Hence,
there are plenty of $\ascr$-measurable families of
probability measures $P$ satisfying \eqref{cc} for
$w=\boldsymbol{1}$. On the other hand, if
$\{P(j,\cdot)\}_{j=1}^N$ are Borel probability
measures on $\rbb_+$ that satisfy \eqref{ccA} and the
inequality $\int_0^\infty \frac{1}{t-1} P(j,\D t) <
\infty$ for every $j \in \nbb \cap [1,N]$, then we can
always find a sequence $\{\mu(j)\}_{j=0}^N \subseteq
(0,\infty)$ (and consequently a measure $\mu$) such
that \eqref{ccB} holds. This again gives rise to a
family $P$ satisfying \eqref{cc} for
$w=\boldsymbol{1}$.
   \end{exa}
   \section{\label{sectcc-1}The condition
$(\mathrm{CC}^{-1})$} If \eqref{stand3} holds and
$\cfw$ is densely defined, then one can consider the
following version of the condition \eqref{cc-1}:
   \begin{multline}  \tag{CC$^{-1}$}  \label{cc-1mu}
\big(\efw(P(\cdot, \sigma))\circ \phi^{-1}\big) (x)
\cdot \hfw(x) = \int_{\sigma} t P(x,\D t) \quad
\text{for $\mu$-a.e.\ $x \in X$}
   \\
\text{and for every $\sigma \in \borel{\rbb_+}$.}
   \end{multline}
It follows from Lemma \ref{nuklear} that
   \begin{align} \label{cc-1to-cc}
\text{\eqref{cc-1mu} implies \eqref{cc} with the same $P$.}
   \end{align}
Our current goal is to prove Theorem \ref{main2}, an analogue of Theorem
\ref{main1}, in which we will give some characterizations of the condition
\eqref{cc-1mu}. The role of \eqref{cc-1mu} and its relationship to
\eqref{cc} will be explained in the subsequent section.

We begin by proving the following lemma.
   \begin{lem} \label{lemS12}
Suppose \eqref{stand3} holds, $\mu_w \circ \phi^{-1}
\ll \mu$, $\hfw < \infty$ a.e.\ $[\mu]$ and $P$
satisfies \eqref{cc}. Then for every Borel function
$f\colon \rbb_+ \to \rbop$,
   \begin{multline} \label{calka2}
\bigg(\efw\bigg(\int_0^\infty f(t) P(\cdot, \D
t)\bigg)\circ \phi^{-1}\bigg) (x) \cdot \hfw(x) =
\chi_{\{\hfw > 0\}}(x) \int_0^\infty t \cdot f(t)
P(x,\D t) \;\;\;
   \\
\text{ for $\mu$-a.e.\ $x \in X$}.
   \end{multline}
   \end{lem}
   \begin{proof}
It follows from Proposition \ref{lemS4} and Lemma \ref{lemS10} that
\eqref{calka2} holds for every simple Borel function $f\colon \rbb_+ \to
\rbb_+$. Let $f\colon \rbb_+ \to \rbop$ be an arbitrary Borel function.
Take a sequence $\{s_n\}_{n=1}^\infty$ of simple Borel functions
$s_n\colon \rbb_+ \to \rbb_+$ which is monotonically increasing and
pointwise convergent to $f$. Arguing as in the proof of Lemma
\ref{lemS10.1}, we see that
   \begin{multline*}
\int_{\varDelta} \bigg(\efw\bigg(\int_0^\infty f(t)
P(\cdot, \D t)\bigg)\circ \phi^{-1}\bigg) (x) \cdot
\hfw(x) \D \mu(x)
   \\
= \lim_{n \to \infty} \int_{\varDelta}
\bigg(\efw\bigg(\int_0^\infty s_n(t) P(\cdot, \D
t)\bigg)\circ \phi^{-1}\bigg) (x) \cdot \hfw(x) \D
\mu(x), \quad \varDelta \in \ascr.
   \end{multline*}
Since \allowdisplaybreaks
   \begin{multline*}
\int_{\varDelta} \chi_{\{\hfw > 0\}}(x) \int_0^\infty
t \cdot f(t) P(x,\D t) \mu(\D x)
   \\
= \lim_{n \to \infty} \int_{\varDelta} \chi_{\{\hfw >
0\}}(x) \int_0^\infty t \cdot s_n(t) P(x,\D t) \mu(\D
x), \quad \varDelta \in \ascr,
   \end{multline*}
and $\mu$ is $\sigma$-finite, we get \eqref{calka2}. This completes the
proof.
   \end{proof}
We are now in a position to formulate and prove the
aforementioned analogue of Theorem \ref{main1}. Note
that the relations ``a.e.'' appearing in the
conditions (i)-(vii) of this theorem are related to
the measures $\mu_w$ and $\rho_W$, while their
counterparts in Theorem \ref{main2} below are related
to the measures $\mu$ and $\rho$.
   \begin{thm} \label{main2}
Suppose \eqref{stand3} holds, $\mu_w \circ \phi^{-1}
\ll \mu$, $\hfw < \infty$ a.e.\ $[\mu]$ and $P$
satisfies \eqref{cc}. Consider the following
conditions{\em :}
   \begin{enumerate}
   \item[(i$^\star$)] $P$ satisfies  \eqref{cc-1mu},
   \item[(ii$^\star$)] $\hfwn{n}(x) =
\int_0^\infty t^n P(x,\D t)$ for every $n \in \zbb_+$
and for $\mu$-a.e.\ $x\in X$,
   \item[(iii$^\star$)] $\int_0^\infty t P(x,\D t) = 0$ for
$\mu$-a.e.\ $x \in \{\hfw = 0\}$,
   \item[(iv$^\star$)] $P(x,\cdot) = \delta_0(\cdot)$ for
$\mu$-a.e.\ $x \in \{\hfw = 0\}$,
   \item[(v$^\star$)] $\rho_W \circ \varPhi^{-1} \ll \rho$ and
$\HFW(x,t) = t$ for $\rho$-a.e.\ $(x,t) \in X \times
\rbb_+$,
   \item[(vi$^\star$)] $\rho_W \circ \varPhi^{-1} \ll \rho$ and
   there exist an $\ascr$-measurable function $g
\colon X\times \rbb_+ \to \rbb_+$ such that $g=\HFW$
a.e.\ $[\rho]$ and $g \circ \varPhi = g$ a.e.\
$[\rho]$,
   \item[(vii$^\star$)] $\hfw > 0$ a.e.\ $[\mu]$.
   \end{enumerate}
Then the conditions {\em
\mbox{(i$^\star$)-(v$^\star$)}} are equivalent and the
implications
   \begin{align} \label{futro}
\text{\em
\mbox{(v$^\star$)}$\Rightarrow$\mbox{(vi$^\star$)},
\mbox{(vii$^\star$)}$\Rightarrow$\mbox{(i$^\star$)}
and
\mbox{(vii$^\star$)}$\Rightarrow$\mbox{(vi$^\star$)}}
   \end{align}
hold. Moreover, if {\em \mbox{(i$^\star$)}} holds,
then $\hfw > 0$ a.e.\ $[\mu_w]$, $\cfw$ is subnormal
and $\CFW$ is a quasinormal extension of $\cfw$.
   \end{thm}
   \begin{proof}
(i$^\star$)$\Rightarrow$(iii$^\star$) Substituting
$\sigma=\rbb_+$ into \eqref{cc-1mu}, we obtain
   \begin{align*}
\int_0^\infty t P(x, \D t) &= \big(\efw(P(\cdot,
\rbb_+)) \circ \phi^{-1}\big)(x) \cdot \hfw(x)
   \\
&\hspace{-.8ex}\overset{\eqref{jedynka}}= \hfw(x)
\quad \text{for $\mu$-a.e.\ $x \in X$.}
   \end{align*}
This implies \mbox{(iii$^\star$)}.

(iii$^\star$)$\Rightarrow$(v$^\star$) By the
definitions of $\rho$ and $\rho_W$, we have
\allowdisplaybreaks
   \begin{align}  \notag
\int_{\varPhi^{-1}(\varDelta \times \sigma)} \D \rho_W & =
\int_{\phi^{-1}(\varDelta)} |w(x)|^2 P(x,\sigma) \D \mu(x)
   \\ \notag
& = \int_{\phi^{-1}(\varDelta)} P(x,\sigma) \D \mu_w(x)
   \\  \notag
&= \int_{\phi^{-1}(\varDelta)} \efw(P(\cdot,\sigma))
\D \mu_w
   \\ \notag
&\hspace{-.5ex} \overset{\mbox{\small $(\dag)$}}=
\int_{\varDelta} \efw(P(\cdot,\sigma))\circ \phi^{-1}
\cdot \hfw \D \mu
   \\ \notag
&\hspace{-.5ex}\overset{\mbox{\small $(\ddag)$}}=
\int_{\varDelta} \int_{\sigma} t P(x,\D t) \D \mu(x)
   \\ \label{ptaszek}
& = \int_{\varDelta \times \sigma} t \D \rho(x,t),
\quad \varDelta \in \ascr, \, \sigma \in
\borel{\rbb_+},
   \end{align}
where $(\dag)$ follows from \eqref{l2} and
\eqref{fifi}, while $(\ddag)$ can be deduced from
Lemma \ref{lemS10} and \mbox{(iii$^\star$)}. Take a
sequence $\{X_n\}_{n=1}^\infty \subseteq \ascr$ such
that $X_n \nearrow X$ as $n \to \infty$, and $\mu(X_k)
< \infty$ for every $k\in \nbb$. Set $\sigma_n =
[0,n]$ for $n \in \nbb$. Then clearly $\{X_n \times
\sigma_n\}_{n=1}^\infty \subseteq \ascr \otimes
\borel{\rbb_+}$, $X_n \times \sigma_n \nearrow X
\times \rbb_+$ as $n \to \infty$ and
   \begin{align*}
\int_{X_n \times \sigma_n} t \D \rho(x,t) = \int_{X_n} \int_{\sigma_n} t
P(x, \D t) \D \mu(x) \Le n \cdot \mu(X_n) < \infty, \quad n\in \nbb.
   \end{align*}
Hence, by \eqref{ptaszek} and Lemma \ref{2miary}, the
measures $E \mapsto \int_{\varPhi^{-1}(E)} \D \rho_W$
and $E \mapsto \int_{E} t \D \rho(x,t)$ coincide on
$\ascr \otimes \borel{\rbb_+}$. As a consequence,
$\rho_W \circ \varPhi^{-1} \ll \rho$ and $\HFW(x,t)=t$
for $\rho$-a.e.\ $(x,t) \in X \times \rbb_+$. This
yields \mbox{(v$^\star$)}.

(v$^\star$)$\Rightarrow$(i$^\star$) By the definitions
of $\rho$ and $\rho_W$, we have \allowdisplaybreaks
   \begin{align*}
\int_{\varDelta} \int_{\sigma} t P(x,\D t) \D \mu(x) &
= \int_{\varDelta \times \sigma} t \D \rho(x,t)
   \\
&\hspace{-1ex}\overset{\mbox{\small (v$^\star$)}}=
\int_{\varDelta \times \sigma} \HFW \D \rho
   \\
& = \int_{\varPhi^{-1}(\varDelta \times \sigma)} 1 \D
\rho_W =
   \\
&= \int_{\phi^{-1}(\varDelta)} P(x,\sigma) \D\mu_w (x)
   \\
&= \int_{\phi^{-1}(\varDelta)} \efw(P(\cdot,\sigma))
\D\mu_w
   \\
&\hspace{-.4ex}\overset{\mbox{\small$(\dag)$}}=
\int_{\varDelta} \efw(P(\cdot,\sigma)) \circ \phi^{-1}
\cdot \hfw \D\mu, \quad \varDelta \in \ascr, \sigma
\in \borel{\rbb_+},
   \end{align*}
where $(\dag)$ can be inferred from \eqref{l2} and
\eqref{fifi}. Since $\mu$ is $\sigma$-finite, we get
\mbox{(i$^\star$)}.

(iv$^\star$)$\Rightarrow$(ii$^\star$) Set $H_n(x) =
\int_0^\infty t^n P(x,\D t)$ for $x\in X$ and $n\in
\zbb_+$. Clearly, by \eqref{pme}, each $H_n\colon X
\to \rbop$ is $\ascr$-measurable. It follows from
\mbox{(iv$^\star$)} that
   \begin{align} \label{hnx0}
H_n(x) = 0 \text{ for $\mu$-a.e.\ $x \in \{\hfw=0\}$
and for all $n\in \nbb$.}
   \end{align}
Using induction, we will show that for every $n\in \zbb_+$,
   \begin{align} \label{hnhni}
H_n(x) = \hfwn{n}(x) \text{ for $\mu$-a.e.\ $x \in
X$.}
   \end{align}
The case of $n=0$ is obvious. Now assume that \eqref{hnhni} holds for a
fixed $n \in \zbb_+$. Then
   \begin{align*}
H_{n+1}(x) &\overset{\eqref{hnx0}} = \chi_{\{\hfw
> 0\}} (x) \cdot \int_0^\infty t \cdot t^{n} P(x,
\D t)
   \\
&\overset{\eqref{calka2}}=
\bigg(\efw\bigg(\int_0^\infty t^n P(\cdot, \D
t)\bigg)\circ \phi^{-1}\bigg) (x) \cdot \hfw(x)
   \\
&\overset{\eqref{hnhni}}= \big(\efw(\hfwn{n}) \circ
\phi^{-1}\big) (x) \cdot \hfw(x)
   \\
&\overset{\eqref{rec1}}= \hfwn{n+1}(x) \qquad
\text{for $\mu$-a.e.\ $x\in X$,}
   \end{align*}
which completes the induction argument.

(ii$^\star$)$\Rightarrow$(iii$^\star$) Clearly
$\int_0^\infty t P(x,\D t) = \hfw(x) = 0$ for
$\mu$-a.e.\ $x\in \{\hfw = 0\}$.

(iii$^\star$)$\Rightarrow$(iv$^\star$) Use Lemma
\ref{lemS3}.

Summarizing, we have shown that the conditions
\mbox{(i$^\star$)}-\mbox{(v$^\star$)} are equivalent.

(v$^\star$)$\Rightarrow$(vi$^\star$) Observe that the
function $g\colon X\times \rbb_+ \to \rbb_+$ defined
by $g(x,t)=t$ for $(x,t) \in X\times \rbb_+$ has the
required properties.

\mbox{(vii$^\star$)}$\Rightarrow$\mbox{(i$^\star$)}
This follows from the chain of implications
\mbox{(vii$^\star$)}$\Rightarrow$\mbox{(iii$^\star$)}$\Rightarrow$\mbox{(i$^\star$)}.

(vii$^\star$)$\Rightarrow$(vi$^\star$) For this, note
that
\mbox{(vii$^\star$)}$\Rightarrow$\mbox{(i$^\star$)}$\Rightarrow$\mbox{(v$^\star$)}$\Rightarrow$\mbox{(vi$^\star$)}.

To prove the ``moreover'' part, assume that
\mbox{(i$^\star$)} holds. By \mbox{(v$^\star$)} and
Lemma \ref{nuklear} (applied to $\rho$ and $\rho_W$ in
place of $\mu$ and $\mu_w$), we see that $\HFW \circ
\varPhi = \HFW$ a.e.\ $[\rho_W]$. According to
Proposition \ref{lemS2}, Theorem \ref{quain} and
\eqref{ucu}, $\CFW$ is a quasinormal extension of
$\cfw$ and consequently $\cfw$ is subnormal. This and
Corollary \ref{hipinj} imply that $\hfw > 0$ a.e.\
$[\mu_w]$, which completes the proof.
   \end{proof}
   \begin{rem} \label{rem1}
Under the assumptions of Theorem \ref{main2}, if
$\cfw$ is subnormal, then the condition
\mbox{(iii$^\star$)} of this theorem is equivalent to
   \begin{enumerate}
   \item[(iii$^\prime$)] $\int_0^\infty t P(x,\D t) = 0$ for
$\mu$-a.e.\ $x \in \{\hfw = 0\} \cap \{w=0\}$.
   \end{enumerate}
Indeed, if (iii$^\prime$) holds, then by Corollary
\ref{hipinj}, $\hfw > 0$ a.e.\ $[\mu_w]$, and
consequently $\mu(\{\hfw = 0\} \cap \{w \neq 0\}) =
0$. Hence (iii$^\star$) holds. The reverse implication
is obvious.
   \end{rem}
Comparing Theorems \ref{main1} and \ref{main2}, one
can ask whether any of the conditions
\mbox{(vi$^\star$)} and \mbox{(vii$^\star$)} is
equivalent to \mbox{(iii$^\star$)}. The answer is
negative. What is more, the conditions (vi$^\star$)
and (vii$^\star$) are not equivalent. All this will be
shown in Example~ \ref{rem0+}.
   \section{\label{sec2.4}Subnormality via
$(\mathrm{CC}^{-1})$} In this section, we will show
that a densely defined weighted composition operator
that admits an $\ascr$-measurable family of
probability measures $P$ satisfying \eqref{cc-1mu} is
subnormal (see Theorem \ref{MAIN2}). Moreover, we will
prove that the condition \eqref{cc-1mu} is equivalent
to the conjunction of the conditions \eqref{cc} and
``$\hfw > 0$ a.e.\ $[\mu_w]$'', not necessarily with
the same $P$ (see Theorem \ref{MAIN2} again and Remark
\ref{zzz}). In other words, if the assumptions of
Theorem \ref{MAIN1} are satisfied, then the family $P$
appearing therein can always be modified so as to
satisfy \eqref{cc-1mu}. In particular, by Theorem
\ref{main2}, the so-modified $P$ satisfies the
condition \mbox{(vi$^\star$)} of this theorem.

First, we propose a method of modifying the family $P$
under which the consistency condition \eqref{cc} is
preserved.
   \begin{lem} \label{lemS13}
Suppose \eqref{stand3} holds, $\mu_w \circ \phi^{-1}
\ll \mu$, $\hfw < \infty$ a.e.\ $[\mu]$ and $P$
satisfies \eqref{cc}. Let $Q \colon X \times
\borel{\rbb_+} \to [0,1]$ be an $\ascr$-measurable
family of probability measures such that
   \begin{align} \label{odw}
\text{$Q(x,\cdot) = P(x,\cdot)$ for $\mu$-a.e.\ $x\in
\{\hfw > 0\}$.}
   \end{align}
Then the following conditions are equivalent{\em :}
   \begin{enumerate}
   \item[(i)] $Q$ satisfies  \eqref{cc} with
$Q$ in place of $P$,
   \item[(ii)] $\efw(Q(\cdot,\sigma)) = \efw(P(\cdot,\sigma))$
a.e.\ $[\mu_w]$ for every $\sigma \in \borel{\rbb_+}$.
   \end{enumerate}
Moreover, if \eqref{odw} holds with $\mu_w$ in place of $\mu$ and
   \begin{align} \label{odw2}
\int_{\phi^{-1}(\varDelta) \cap \{\hfw=0\}}
Q(x,\sigma) \D \mu_w(x) = \int_{\phi^{-1}(\varDelta)
\cap \{\hfw=0\}} P(x,\sigma) \D \mu_w(x)
   \end{align}
for all $\varDelta \in \ascr$ and $\sigma \in \borel{\rbb_+}$, then {\em
(i)} is valid.
   \end{lem}
   \begin{proof}
First note that
   \begin{align*}
\int_{\phi^{-1}(\varDelta)} \int_{\sigma} t Q(\phi(x),
& \D t) \D \mu_w(x) \overset{\eqref{l1}}=
\int_{\varDelta} \hfw(x) \int_{\sigma} t Q(x, \D t) \D
\mu(x)
   \\
& = \int_{\varDelta \cap \{\hfw > 0\}} \hfw(x)
\int_{\sigma} t Q(x, \D t) \D \mu(x)
   \\
& \hspace{-.7ex}\overset{\eqref{odw}} =
\int_{\varDelta \cap \{\hfw
> 0\}} \hfw(x) \int_{\sigma} t P(x, \D t) \D \mu(x)
   \\
& = \int_{\phi^{-1}(\varDelta)} \int_{\sigma} t P(\phi(x), \D t) \D
\mu_w(x), \quad \varDelta \in \ascr, \, \sigma \in \borel{\rbb_+}.
   \end{align*}
Since the measure $\mu_w|_{\phi^{-1}(\ascr)}$ is
$\sigma$-finite (see Proposition \ref{lemS2}), we have
   \begin{align*}
\int_{\sigma} t Q(\phi(x), \D t) = \int_{\sigma} t P(\phi(x), \D t) \quad
\text{for $\mu_w$-a.e.\ $x\in X, \, \sigma \in \borel{\rbb_+}$.}
   \end{align*}
Using the above equality, we verify that (i) is equivalent to (ii).

Now we prove the ``moreover'' part. First, observe
that $\phi^{-1}(\varDelta)$ is a disjoint union of the
sets $\phi^{-1}(\varDelta) \cap \{\hfw=0\}$ and
$\phi^{-1}(\varDelta) \cap \{\hfw>0\}$, then apply
\eqref{odw} with $\mu_w$ in place of $\mu$ and finally
use \eqref{odw2}. What we get is
   \begin{align*}
\int_{\phi^{-1}(\varDelta)} Q(x, \sigma) \D \mu_w(x) =
\int_{\phi^{-1}(\varDelta)} P(x, \sigma) \D \mu_w(x), \quad \varDelta \in
\ascr.
   \end{align*}
This implies (ii) and completes the proof.
   \end{proof}
The following realization of the above general scheme is particularly
useful.
   \begin{pro} \label{lemS14}
Suppose \eqref{stand3} holds, $\mu_w \circ \phi^{-1}
\ll \mu$, $\hfw < \infty$ a.e.\ $[\mu]$ and $P$
satisfies \eqref{cc}. Set $\varTheta_{00} = \{\hfw =
0\}\cap \{w=0\}$. Let $\{\tau_x\}_{x \in
\varTheta_{00}}$ be a family of Borel probability
measures on $\rbb_+$ such that the mapping
$\varTheta_{00} \ni x \longmapsto \tau_x(\sigma) \in
[0,1]$ is $\ascr$-measurable for every $\sigma \in
\borel{\rbb_+}$. Define the mapping $Q\colon X \times
\borel{\rbb_+} \to [0,1]$~ by
   \begin{align*}
Q(x,\sigma) =
   \begin{cases}
   P(x,\sigma) & \text{if} \quad x \in X \setminus \varTheta_{00},
   \\
   \tau_x(\sigma) & \text{if} \quad x \in \varTheta_{00},
   \end{cases}
\quad \sigma \in \borel{\rbb_+}.
   \end{align*}
Then $Q$ is an $\ascr$-measurable family of probability measures which
satisfies \eqref{cc}. In particular, the family $\{\tau_x\}_{x \in
\varTheta_{00}}$ given by $\tau_x = \delta_0$ for $x \in \varTheta_{00}$
meets our requirements and $\int_0^\infty t Q(x,\D t) = 0$ for all $x \in
\varTheta_{00}$.
   \end{pro}
   \begin{proof}
By definition, $Q$ is an $\ascr$-measurable family of probability measures
which satisfies \eqref{odw} and has the property that for all $\sigma \in
\borel{\rbb_+}$ and $\varDelta \in \ascr$,
   \begin{align*}
\int\limits_{\phi^{-1}(\varDelta)} Q(x,\sigma) \D \mu_w(x) =
\int\limits_{\phi^{-1}(\varDelta)\cap \{w\neq 0\}} Q(x,\sigma) \D \mu_w(x)
= \int\limits_{\phi^{-1}(\varDelta)} P(x,\sigma) \D \mu_w(x).
   \end{align*}
Hence, by Lemma \ref{lemS13}, $Q$ satisfies \eqref{cc}. The ``moreover''
part is obvious.
   \end{proof}
   \begin{cor} \label{lemS15}
Suppose \eqref{stand3} holds, $\mu_w \circ \phi^{-1}
\ll \mu$, $\hfw < \infty$ a.e.\ $[\mu]$ and $P$
satisfies \eqref{cc}. Then the following assertions
are valid{\em :}
   \begin{enumerate}
   \item[(i)]  if  $\hfw > 0$ a.e.\ $[\mu_w]$, then
there exists an $\ascr$-measurable family of
probability measures $Q\colon X \times \borel{\rbb_+}
\to [0,1]$ which satisfies \eqref{cc} with $Q$ in
place of $P$ and such that
   \begin{align} \label{fish1}
\int_0^\infty t Q(x,\D t) = 0 \quad \text{for
$\mu$-a.e.\ $x \in \{\hfw=0\}$,}
   \end{align}
   \item[(ii)] if $\mu(\varTheta_{00}) > 0$ with
$\varTheta_{00} :=\{\hfw = 0\} \cap \{w=0\}$, then
there exists an $\ascr$-measurable family of
probability measures $Q\colon X \times \borel{\rbb_+}
\to [0,1]$ which satisfies \eqref{cc} with $Q$ in
place of $P$ and such that
   \begin{align} \label{fish2}
\int_0^\infty t Q(x,\D t) = \infty \quad \text{for
$\mu$-a.e.\ $x \in \varTheta_{00}$.}
   \end{align}
   \end{enumerate}
   \end{cor}
   \begin{proof}
(i) By our assumption, $\hfw > 0$ a.e.\ $[\mu]$ on
$\{w \neq 0\}$ and consequently $\mu(\{\hfw=0\} \cap
\{w\neq 0\}) =0$. Applying Proposition \ref{lemS14}
with $\tau_x=\delta_0$, we get \eqref{fish1}.

(ii) Take any Borel probability measure $\tau$ on
$\rbb_+$ such that $\int_0^\infty t \D \tau(t) =
\infty$ (e.g., $\tau = \frac{6}{\pi^2}
\sum_{n=1}^\infty \frac{1}{n^2} \delta_n$). Applying
Proposition \ref{lemS14} with $\tau_x = \tau$ for
every $x \in \varTheta_{00}$, we get \eqref{fish2}.
This completes the proof.
   \end{proof}
 \begin{rem} \label{4sets}
In view of Remark \ref{rem1} and Corollary
\ref{lemS15}, it is of some interest to know which of
the following sets allow for modifications (modulo
$\mu$) of $P$ that preserve \eqref{cc}:
   \allowdisplaybreaks
   \begin{align*}
\varTheta_{0+} & =\{\hfw = 0\} \cap \{w \neq 0\}, &
\varTheta_{++} = \{\hfw > 0\} \cap \{w \neq 0\},
   \\
\varTheta_{00} & = \{\hfw = 0\} \cap \{w=0\}, &
\varTheta_{+0} = \{\hfw > 0\} \cap \{w=0\}.
   \end{align*}
First, we consider the case when $\cfw$ is bounded and subnormal. Then, by
Corollary \ref{hipinj}, $\mu(\varTheta_{0+})=0$. According to Theorem
\ref{bsubn}, $\cfw$ has an $\ascr$-measurable family of probability
measures $P\colon X \times \borel{\rbb_+} \to [0,1]$ satisfying
\eqref{cc}, and each such family is uniquely determined (modulo $\mu$) on
$\varTheta_{++}$. Also, in general, such $P$ may happen to be uniquely
determined (modulo $\mu$) on $\varTheta_{+0}$ with $\mu(\varTheta_{+0})
> 0$ (see Example \ref{quasif}). This means that $\varTheta_{00}$ is
the largest subset of $X$ (modulo $\mu$) that allows
for modifications of $P$ that preserve \eqref{cc} (see
Proposition \ref{lemS14}). However, if $\cfw$ is not
subnormal, then the set $\varTheta_{0+}$ may have
positive $\mu$-measure. This may happen even in the
case of a composition operator $C_{\phi}$ which admits
an $\ascr$-measurable family of probability measures
$P$ satisfying \eqref{cc}; in fact, $C_{\phi}$ may
have plenty of such families (see Example \ref{pico}).

The following instances completely illustrate the
interplay between $\mu$-positivity and $\mu$-nullity
of the sets defined above in the case of subnormal
weighted composition operators that admit
$\ascr$-measurable families of probability measures
$P$ satisfying \eqref{cc} (in fact, most of these
operators are quasinormal, which, even in the
unbounded case, admit $\ascr$-measurable families of
probability measures $P$ satisfying \eqref{cc}, see
Proposition \ref{munich2}). Recall that in this case
$\mu(\varTheta_{0+}) = 0$.

1) $\mu(\varTheta_{++})=0$. Then $\mu_w=0$, which, by
Proposition \ref{lemS1}(vi), implies that
$\mu(\varTheta_{+0})=0$ and $\cfw=0$ (clearly,
$\cfw=0$ implies $\mu(\varTheta_{++})=0$). If $\cfw$
is a composition operator, then evidently
$\mu(\varTheta_{00})=0$. However, in the weighted
case, it may happen that $\mu(\varTheta_{00}) > 0$
(e.g., for the zero multiplication operator).

2) $\mu(\varTheta_{++}) > 0$. Then each of the following cases may occur:
   \begin{enumerate}
   \item[$\bullet$] $\mu(\varTheta_{00})=0$ and
$\mu(\varTheta_{+0})=0$ (e.g., for a composition operator),
   \item[$\bullet$] $\mu(\varTheta_{00})=0$
and $\mu(\varTheta_{+0})> 0$ (e.g., for $C_{\zeta,v}$ as in Remark
\ref{Zenchce}),
   \item[$\bullet$] $\mu(\varTheta_{00})> 0$ and
$\mu(\varTheta_{+0})=0$ (e.g., for a non-injective multiplication
operator),
   \item[$\bullet$] $\mu(\varTheta_{00}) > 0$ and
$\mu(\varTheta_{+0}) > 0$ (e.g., for $\cfw$ as in Example \ref{quasif}).
   \end{enumerate}
   \end{rem}
Theorem \ref{MAIN2} below, which is the main result of
this paper, extends the criterion for subnormality of
unbounded weighted composition operators given in
Theorem \ref{MAIN1}. The quasi-moment formula
\eqref{momomu} below generalizes that for composition
operators established in \cite[Theorem 17]{b-j-j-sS}.
   \begin{thm} \label{MAIN2}
Let $(X,\ascr,\mu)$ be a $\sigma$-finite measure space, $w$ be an
$\ascr$-measurable complex function on $X$ and $\phi$ be an
$\ascr$-measurable transformation of $X$ such that $\cfw$ is densely
defined. Then the following three conditions are equivalent{\em :}
   \begin{enumerate}
   \item[(i)] $\hfw > 0$ a.e.\ $[\mu_w]$ and there
exists $P\colon X \times \borel{\rbb_+} \to [0,1]$, an
$\ascr$-measur\-able family of probability measures, which satisfies
\eqref{cc},
   \item[(ii)] there exists $P\colon X \times
\borel{\rbb_+} \to [0,1]$, an $\ascr$-measurable family of probability
measures, which satisfies \eqref{cc} and the condition below
   \begin{align} \label{momomu}
\hfwn{n}(x) = \int_0^\infty t^n P(x,\D t) \quad
\text{for every $n \in \zbb_+$ and for $\mu$-a.e.\
$x\in X$,}
   \end{align}
   \item[(iii)] there exists $P\colon X \times
\borel{\rbb_+} \to [0,1]$, an $\ascr$-measurable family of probability
measures, which satisfies \eqref{cc-1mu}.
   \end{enumerate}
Moreover, if {\em (iii)} holds, then $\cfw$ is
subnormal and $\CFW$ is its quasinormal extension
$($see \eqref{stand3}$)$.
   \end{thm}
   \begin{proof}
(i)$\Rightarrow$(ii) Apply Corollary \ref{lemS15}(i)
and Theorem \ref{main2}.

(ii)$\Rightarrow$(iii) Apply Theorem \ref{main2} again.

(iii)$\Rightarrow$(i) Multiplying both sides of the
equality in \eqref{cc-1mu} by $\chi_{\{\hfw > 0\}}$,
we get
   \begin{align*}
\big(\efw(P(\cdot, \sigma))\circ \phi^{-1}\big) (x)
\cdot \hfw(x) = \chi_{\{\hfw > 0\}}(x) \int_{\sigma} t
P(x,\D t) \text{ for $\mu$-a.e.\ $x \in X$},
   \end{align*}
for every $\sigma \in \borel{\rbb_+}$. This combined
with Lemma \ref{lemS10} implies that $P$ satisfies
\eqref{cc}. Substituting $\sigma = \rbb_+$ into
\eqref{cc-1mu}, we deduce that $\int_0^\infty t P(x,\D
t) = 0$ for $\mu$-a.e.\ $x \in \{\hfw = 0\}$. Hence,
by Theorem \ref{main1}, $\hfw > 0$ a.e.\ $[\mu_w]$,
which gives (i).

The ``moreover'' part follows from Theorem \ref{MAIN1}. This completes the
proof.
   \end{proof}
   \begin{rem} \label{zzz}
A close inspection of the proof of Theorem \ref{MAIN2}
reveals that the conditions (ii) and (iii) are
equivalent with the same $P$. Similarly, (ii) implies
(i) with the same $P$. However, in general, (i) does
not imply (ii) with the same $P$ (see Example
\ref{rem0+}).
   \end{rem}
The following corollary is related to \cite[Theorem
7]{b-j-j-sS} (see the equivalence
(i)$\Leftrightarrow$(iv) therein).
   \begin{cor}\label{nicto2} Suppose
\eqref{stand3} holds and $\cfw$ is densely defined.
Consider the following two assertions{\em :}
   \begin{enumerate}
   \item[(i)] $P$ satisfies \eqref{cc} and $\hfw
> 0$ a.e.\ $[\mu_w]$,
   \item[(ii)] $P$ satisfies \eqref{cc-1}.
   \end{enumerate}
Then {\em (i)} implies {\em (ii)}. Moreover, if $w \neq 0$ a.e.\ $[\mu]$,
then {\em (ii)} implies {\em (i)}.
   \end{cor}
   \begin{proof}
That (i) implies (ii) follows from Theorem
\ref{main1}. The moreover ``part'' can be deduced from
the implication (iii)$\Rightarrow$(i) of Theorem
\ref{MAIN2} (see Remark~ \ref{zzz}).
   \end{proof}
   \chapter{$C^\infty$-vectors}
In this chapter, we turn our interest to weighted
composition operators that have sufficiently many
$C^\infty$-vectors. In Section \ref{sec4.1}, we give
necessary and sufficient conditions for the $n$th
power of a weighted composition operator $\cfw$ to be
densely defined or to be closed (see Lemmas
\ref{lemS11p} and \ref{13-12-13}, and Proposition
\ref{potegi-p}). The question of when
$C^\infty$-vectors are dense in $L^2(\mu)$ is answered
in Theorem \ref{Mittag}. Section \ref{GSMS} is devoted
to characterizing weighted composition operators
generating Stieltjes moment sequences (see Theorem
\ref{gsms}). Finally, in Section \ref{Sec4.3}, we
provide a new characterization of subnormality of
bounded weighted composition operators (see Theorem
\ref{bsubn}).
   \section{\label{sec4.1}Powers of $\cfw$}
Some basic properties of powers of unbounded
composition operators have been established in
\cite[Sect.\ 4]{b-j-j-sC}. As shown in this section,
most of them, but not all, remain true for powers of
unbounded weighted composition operators.

Our first aim is to give necessary and sufficient
conditions for the $n$th power of a weighted
composition operator $\cfw$ to be densely defined (see
\eqref{Zak} for the definition of $\widehat w_n$).
   \begin{lem} \label{lemS11p}
Suppose \eqref{stand2} holds. Then the following
assertions are valid{\em :}
   \begin{enumerate}
   \item[(i)] $\int_X |f g| \, \hfwn{n} \D \mu< \infty$
and $\is{\cfw^n f}{\cfw^n g} = \int_X f\bar g \,
\hfwn{n} \D \mu$ for all $f,g \in \dz{\cfw^n}$ and
$n\in \zbb_+,$
   \item[(ii)] $\|f\|_{\cfw^n}^2 = \int_X |f|^2
(1 + \hfwn{n}\big) \D \mu$ for all $f \in \dz{\cfw^n}$
and $n\in \nbb$,
   \item[(iii)] $\dz{\cfw^n}  =
L^2\big(\big(\sum_{j=0}^n \hfwn{j}\big)\D \mu\big)$
for every $n\in \zbb_+$,
   \item[(iv)] for every $n \in \nbb$,
$\dz{\cfw^n} = \bigcap_{j=1}^n \dz{C_{\phi^j,\widehat
w_j}}$,
   \item[(v)] for every $n \in \nbb$, $\cfw^n$
is densely defined if and only if $\sum_{j=1}^n
\hfwn{j} < \infty$ a.e.\ $[\mu]$,
   \item[(vi)] for every $n \in \nbb$, $\cfw^n$
is densely defined if and only if $C_{\phi^j,\widehat
w_j}$ is densely defined for every $j \in \{1, \dots,
n\}$.
   \end{enumerate}
   \end{lem}
   \begin{proof}
(i) Apply Lemma \ref{lemS11}(i) and \eqref{l2} (with
$\phi^n$, $\widehat w_n$, $\hfwn{n}$ in place of
$\phi$, $w$, $\hfw$, respectively).

(ii) This follows from (i).

(iii) Employ Lemma \ref{lemS11}(i), \eqref{l2} and induction on $n$.

(iv) This can be deduced from (iii) and Proposition \ref{lemS1}(i).

(v) The ``if'' part follows from (iii) and \cite[Lemma
12.1]{b-j-j-sC}. To prove the ``only if'' part,
suppose that $\mu(\varDelta)> 0$, where $\varDelta :=
\{x \in X\colon \sum_{j=1}^n \hfwn{j}(x)=\infty\}$.
One can deduce from (iii) that $f=0$ a.e.\ $[\mu]$ on
$\varDelta$ for all $f \in \dz{\cfw^n}$, and
consequently for all $f \in L^2(\mu)$, which
contradicts the $\sigma$-finiteness of $\mu$.

(vi) This follows from (v) and Proposition \ref{lemS1}(i).
   \end{proof}
Now we give sufficient conditions for the $n$th power of a weighted
composition operator $\cfw$ to be closed.
   \begin{lem} \label{13-12-13}
Suppose \eqref{stand2} holds. Fix an integer $n \Ge
2$. Then the following assertions are valid{\em :}
   \begin{enumerate}
   \item[(i)] $\cfw^n$ is closable,
   \item[(ii)] if $\cfw^{n-1}$ is densely defined, then
$\overline{\cfw^n} = C_{\phi^n,\widehat w_n}$,
   \item[(iii)] if $\cfw^{n-1}$ is densely defined,
then $\cfw^n$ is closed if and only $\cfw^n =
C_{\phi^n,\widehat w_n}$,
   \item[(iv)] $\cfw^n = C_{\phi^n,\widehat w_n}$
if and only if there exists $c \in \rbb_+$ such that
   \begin{align*}
\sum_{j=1}^{n-1} \hfwn{j} \Le c (1 + \hfwn{n}) \text{
a.e.\ $[\mu]$.}
   \end{align*}
   \item[(v)] if $\cfw$ is quasinormal, then
$C_{\phi^n,\widehat w_n}$ is quasinormal and $\cfw^n =
C_{\phi^n,\widehat w_n}$.
   \end{enumerate}
   \end{lem}
   \begin{proof}
(i) This is a direct consequence of Lemma
\ref{lemS11}(i) and Proposition \ref{lemS1}(iv).

(ii) Adapt the proof of \cite[Proposition 4.1(iii)]{b-j-j-sC} to the
present situation and use Lemma \ref{lemS11p}.

(iii) This follows from (ii) and Proposition
\ref{lemS1}(iv).

(iv) It follows from Lemma \ref{lemS11}(i) that
$\cfw^n = C_{\phi^n,\widehat w_n}$ if and only if
$\dz{C_{\phi^n,\widehat w_n}} \subseteq \dz{\cfw^n}$.
In turn, by Lemma \ref{lemS11p}(iii),
$\dz{C_{\phi^n,\widehat w_n}} = L^2\big(\big(1 +
\hfwn{n}\big)\D \mu\big)$ and $\dz{\cfw^n} =
L^2\big(\big(\sum_{j=0}^n \hfwn{j}\big)\D \mu\big)$.
Hence an application of \cite[Corollary
12.4]{b-j-j-sC} gives (iv).

(v) This follows from (iii) and the fact that powers
of quasinormal operators being quasinormal are closed
and densely defined (see, e.g., \cite[Corollary
3.8]{j-j-s1}).
   \end{proof}
Regarding Lemma \ref{13-12-13}, we note that the
assertion (ii) is no longer true if we do not assume
that $\cfw^{n-1}$ is densely defined. It may also
happen that $\cfw$ has a dense set of
$C^\infty$-vectors but $\cfw^n$ is not closed for
every integer $n\Ge 2$ (consequently, $\cfw^n \neq
C_{\phi^n,\widehat w_n}$ for every integer $n\Ge 2$).
Both cases can be illustrated with the aid of
composition operators (see \cite[Examples 5.1 and
5.4]{b-j-j-sC}).

It follows from the assertions (ii) and (iv) of
\cite[Proposition 4.1]{b-j-j-sC} that if a composition
operator $C_{\phi}$ is well-defined and $n$ is an
integer greater than or equal to~ $2$, then
$C_{\phi}^n$ is densely defined if and only if
$C_{\phi^n}$ is densely defined (or equivalently, by
\cite[Proposition 3.2]{b-j-j-sC}, if and only if
$\mathsf{h}_{\phi^n} < \infty$ a.e.\ $[\mu]$). The
corresponding statement for a weighted composition
operator $\cfw$ is false in general (see Example
\ref{powers}). However, it is true if $w(x) \neq 0$
for every $x\in X$ (see Proposition \ref{potegi-p});
this covers the case of composition operators.
   \begin{pro}\label{potegi-p}
Suppose \eqref{stand2} holds and $w \neq 0$ a.e.\
$[\mu]$. Fix an integer $n \Ge 2$. Then $\widehat w_n
\neq 0$ a.e.\ $[\mu]$ and the following conditions are
equivalent{\em :}
   \begin{enumerate}
   \item[(i)] $\cfw^n$ is densely defined,
   \item[(ii)] $C_{\phi^n,\widehat w_n}$ is densely
defined,
   \item[(iii)] $\hfwn{n} < \infty$ a.e.\ $[\mu]$.
   \end{enumerate}
   \end{pro}
   \begin{proof}
By Lemma \ref{lemS11}(i), the operator
$C_{\phi^j,\widehat w_j}$ is well-defined for every $j
\in \zbb_+$. Since $w \neq 0$ a.e.\ $[\mu]$, there
exists $\ascr$-measurable function $u\colon X \to
\cbb\setminus \{0\}$ such that $w=u$ a.e.\ $[\mu]$. We
infer from Lemma \ref{nuklear}, using induction on
$k$, that $\widehat w_k = \widehat u_k$ a.e.\ $[\mu]$
for every $k\in\zbb_+$. Applying the ``moreover'' part
of Lemma \ref{wco1}, we see that there is no loss of
generality in assuming that $w(x) \neq 0$ for all
$x\in X$.

(i)$\Rightarrow$(ii) Apply Lemma \ref{lemS11}(i).

(ii)$\Rightarrow$(i) It follows from Lemma
\ref{aproks} that there exists a sequence
$\{X_k\}_{k=1}^\infty \subseteq \ascr$ such that
$\mu_{\widehat w_n} \circ \phi^{-n}(X_k) < \infty$ for
every $k \in \nbb$ and $X_k \nearrow X$ as $k\to
\infty$. Set
   \begin{align*}
\text{$\varDelta_k = \Big\{x \in X\colon |w(x)| \Ge \frac 1 k \Big\}$ and
$\nabla_k = \phi^{-n}(X_k) \cap \phi^{-(n-1)}(\varDelta_k)$ for $k\in
\nbb$.}
   \end{align*}
Then
   \begin{multline*}
\mu_{\widehat w_{n-1}} \circ
\phi^{-(n-1)}\Big(\phi^{-1}(X_k) \cap \varDelta_k\Big)
= \mu_{\widehat w_{n-1}}(\nabla_k) = \int_{\nabla_k}
\frac{|\widehat w_{n}|^2}{|w\circ \phi^{n-1}|^2} \D
\mu
   \\
\Le k^2 \int_{\nabla_k} |\widehat w_{n}|^2 \D \mu \Le
k^2 \mu_{\widehat w_{n}} \circ \phi^{-n}(X_k) <
\infty, \quad k\in \nbb.
   \end{multline*}
Since $\phi^{-1}(X_k) \cap \varDelta_k \nearrow X$ as
$k\to \infty$, we infer from Proposition \ref{lemS2}
that the operator $C_{\phi^{n-1},\widehat w_{n-1}}$ is
densely defined. By reverse induction, we deduce that
$C_{\phi^{j},\widehat w_{j}}$ is densely defined for
every $j \in \{1, \ldots, n\}$. This and Lemma
\ref{lemS11p}(vi) imply that $\cfw^n$ is densely
defined.

(ii)$\Leftrightarrow$(iii) Apply Proposition \ref{lemS2}.
   \end{proof}
Now we show that Proposition \ref{potegi-p} is no
longer true if the assumption ``$w \neq 0$ a.e.\
$[\mu]$'' is removed.
   \begin{exa} \label{powers}
Fix an integer $n$ greater than or equal to $2$. Set $X =
\bigsqcup_{k=0}^{n-1} \nbb^k$ with $\nbb^0 = \{0\}$ and $\ascr = 2^{X}$.
Let $\mu$ be the counting measure on $X$. Define the transformation
$\phi\colon X \to X$ by
   \begin{align*}
\phi(x) =
   \begin{cases}
0 & \text{ if } x \in \nbb^0 \cup \nbb^1,
   \\
(i_1, \ldots, i_{k-1}) & \text{ if } x=(i_1, \ldots, i_{k}) \in \nbb^{k},
\, 2 \Le k \Le n-1 \text{ and } n \Ge 3.
   \end{cases}
   \end{align*}
Set $w=\chi_{Y}$ with $Y=X \setminus \nbb^0$. Clearly,
$\cfw$ is a well-defined partial composition operator
(see Sect.\ \ref{pco}(c)). Since $\phi^{n-1}(x) = 0$
for all $x \in X$, we see that $\hfwn{n} = 0$. Hence,
$C_{\phi^{n},\widehat w_{n}}$ is the zero operator on
$L^2(\mu)$. Now we calculate the function $\hfwn{j}$
for $j \in \{1, \ldots, n-1\}$. Fix $j \in \{1,
\ldots, n-1\}$ and $k \in \{0,\ldots,n-1\}$. If $k \Le
n-1-j$, then for every $y \in \nbb^k$,
$\phi^{-j}(\{y\})$ is an infinite subset of
$\nbb^{k+j}$ and $\widehat w_j(x)=|w(x)|^2 \cdot
\ldots \cdot |w(\phi^{j-1}(x))|^2 = 1$ for every $x
\in \phi^{-j}(\{y\})$. This implies that
   \begin{align}  \label{21.12.13}
\hfwn{j}(y) = \int_{\phi^{-j}(\{y\})} |\widehat w_j|^2
\D \mu = \infty, \quad y \in \nbb^k, \, 0 \Le k \Le
n-1-j.
   \end{align}
In turn, if $k > n-1-j$, then for every $y \in
\nbb^k$, $\phi^{-j}(\{y\}) = \emptyset$. Hence, we
have
   \begin{align*}
\hfwn{j}(y) = 0, \quad y \in \nbb^k, \, n-1-j < k \Le
n-1.
   \end{align*}
It follows from \eqref{21.12.13} that $\hfwn{j}(0) =
\infty$ for every $j \in \{1, \ldots, n-1\}$.
Therefore, by Proposition \ref{lemS2}, the operator
$C_{\phi^j,\widehat w_j}$ is not densely defined for
each $j\in\{1, \dots, n-1\}$. As a consequence,
$\cfw^{n}$ is not densely defined. Since
$C_{\phi^{n},\widehat w_{n}}$ is densely defined, the
implications (ii)$\Rightarrow$(i) and
(iii)$\Rightarrow$(i) of Proposition \ref{potegi-p} do
not hold.
   \end{exa}
It was proved in \cite{b-j-j-sC} that a composition operator $C_{\phi}$
whose powers are all densely defined has a dense set of
$C^\infty$-vectors. As shown below, the same is true for weighted
composition operators (see also Appendix \ref{App} for abstract versions
of this result). For completeness we include the proof which is an
adaptation of that of \cite[Theorem 4.7]{b-j-j-sC}.
   \begin{thm} \label{Mittag}
Suppose \eqref{stand2} holds. Then the following
conditions are equivalent\/{\em :}
   \begin{enumerate}
   \item[(i)] $\dz{\cfw^n}$ is dense in $L^2(\mu)$ for every $n
\in \nbb$,
   \item[(ii)] $\dzn{\cfw}$ is dense in $L^2(\mu)$,
   \item[(iii)] $\dzn{\cfw}$ is a core for $\cfw^n$
for every $n\in \zbb_+$ $($in particular, $\dzn{\cfw}$ is dense in
$L^2(\mu)$$)$,
   \item[(iv)] $\dzn{\cfw}$
is dense in $(\dz{\cfw^n}, \|\cdot\|_{\cfw;n})$ for every $n\in \zbb_+$.
   \end{enumerate}
   \end{thm}
   \begin{proof}
Only the implication (i)$\Rightarrow$(iv) requires a
proof. Assume (i) holds. By Lemma \ref{lemS11p}(v), $0
\Le \hfwn{n} < \infty$ a.e.\ $[\mu]$ for every $n\in
\zbb_+$. It follows from assertions (i) and (iii) of
Lemma \ref{lemS11p} that for every $n\in \zbb_+$, the
vector space $\dz{\cfw^n}$ equipped with the norm
$\|\cdot\|_{\cfw;n}$ coincides with the Hilbert space
$L^2((\sum_{j=0}^n \hfwn{j})\D\mu)$, and thus by
\cite[Lemma 12.1]{b-j-j-sC}, $\dz{\cfw^{n+1}}$ is
dense in $\dz{\cfw^n}$ with respect to the norm
$\|\cdot\|_{\cfw;n}$. Applying Lemma \ref{closedpow}
with $k_n = n$ completes the proof.
   \end{proof}
   \section{\label{GSMS}Generating Stieltjes moment sequences}
   Following \cite{j-j-s0}, we say that an operator
$A$ in a complex Hilbert space $\hh$ {\em generates
Stieltjes moment sequences} if $\dzn{A}$ is dense in
$\hh$ and $\{\|A^n f\|^2\}_{n=0}^\infty$ is a
Stieltjes moment sequence for every $f \in \dzn{A}$.
It is known that if $A$ is a subnormal operator in
$\hh$ such that $\overline{\dzn{A}} = \hh$, then $A$
generates Stieltjes moment sequences; however, the
converse implication does not hold in general (see,
e.g., \cite[Sect.\ 3.2]{b-j-j-sA}).

The following theorem generalizes characterizations of
composition operators generating Stieltjes moment
sequences given in \cite[Theorem 10.4]{b-j-j-sC} to
the case of weighted composition operators. They
resemble the celebrated Lambert's characterizations of
subnormality of bounded composition operators given in
\cite{lam1}. As shown in \cite[Theorem 4.3.3]{j-j-s0}
and \cite[Sect.\ 11]{b-j-j-sC} Lambert's
characterizations are no longer true for unbounded
composition operators. Below, $\cbb[t]_+$ stands for
the set of all polynomials $p \in \cbb[t]$ that are
nonnegative on $\rbb_+$ and ${\EuScript M}_X$ for the
set of all $\ascr$-measurable complex functions on
$X$.
   \begin{thm} \label{gsms}
Suppose \eqref{stand2} holds. Then the following
conditions are equivalent $($see \eqref{Zak} for the
definition of $\widehat w_n$$)$\/{\em :}
   \begin{enumerate}
   \item[(i)] $\cfw$ generates  Stieltjes moment
sequences,
   \item[(ii)] $\{\|\cfw^{n}
f\|^2\}_{n=0}^\infty$ is a Stieltjes moment sequence for all $f \in
\dzn{\cfw}$, and $\overline{\dz{\cfw^k}} = L^2(\mu)$ for all $k\in \nbb$,
   \item[(iii)] $\{\hfwn{n}(x)\}_{n=0}^\infty$ is a
Stieltjes moment sequence for $\mu$-a.e.\ $x \in X$,
   \item[(iv)] $\{\mu_{\widehat w_n}(\phi^{-n}
(\varDelta))\}_{n=0}^\infty$ is a Stieltjes mo\-ment
sequence for all $\varDelta \in \ascr$ such that
$\mu_{\widehat w_k}(\phi^{-k}(\varDelta)) < \infty$
for all $k \in \zbb_+$, and $\overline{\dz{\cfw^k}} =
L^2(\mu)$ for all~ $k\in \nbb$,
   \item[(v)] there exists
a linear mapping $L\colon \cbb[t] \to {\EuScript M}_X$
such that\/\footnote{\;Note that the condition
``$\hfwn{n} < \infty$ a.e.\ $[\mu]$ for all $n\in
\nbb$'' is necessary and sufficient for the existence
of a linear mapping $L\colon \cbb[t] \to {\EuScript
M}_X$ satisfying \eqref{ltn}. Of course, if $\tilde L$
is another such mapping, then $L(p) = \tilde L(p)$
a.e.\ $[\mu]$ for every $p \in \cbb[t]$.}
   \begin{align} \label{ltn}
L(t^n) & = \hfwn{n} \text{ a.e.\ $[\mu]$,} \quad n \in
\zbb_+,
   \\  \notag
L(p) & \Ge 0 \text{ a.e.\ $[\mu]$,} \quad p\in \cbb[t]_+.
   \end{align}
   \end{enumerate}
Moreover, if {\em (i)} holds, then
   \begin{enumerate}
   \item[(a)] $\dzn{\cfw}$ is a core for $\cfw^n$ for every
$n\in \zbb_+$,
   \item[(b)] $\cfw^n$ is closed and
$\cfw^n = C_{\phi^n,\widehat w_n}$ for every $n \in
\zbb_+$.
   \end{enumerate}
   \end{thm}
   \begin{proof}
By Lemma \ref{lemS11p}(v) and Theorem \ref{Mittag},
any of the conditions (i) to (v) implies that (a)
holds and $0\Le \hfwn{n} (x) < \infty$ for $\mu$-a.e.\
$x \in X$ and every $n\in \zbb_+$; hence, there is no
loss of generality in assuming that $0\Le \hfwn{n} (x)
< \infty$ for all $x \in X$ and $n\in \zbb_+$. Take
$\boldsymbol \alpha = \{\alpha_i\}_{i=0}^n \subseteq
\cbb$ with $n\in \zbb_+$. Set $g_{\boldsymbol
\alpha}=\sum_{i,j=0}^n \alpha_i\bar\alpha_j
\hfwn{i+j}$. It follows from Lemma \ref{lemS11p}(i)
that for every $f \in \dz{\cfw^{2n}}$,
   \begin{align} \label{pdef}
\int_X |g_{\boldsymbol \alpha}| |f|^2 \D \mu < \infty
\text{ and } \int_X g_{\boldsymbol \alpha} |f|^2 \D\mu
= \sum_{i,j=0}^n \|C_{\phi}^{i+j} f\|^2 \alpha_i
\bar\alpha_j.
   \end{align}

(i)$\Leftrightarrow$(ii) This is evident due to
Theorem \ref{Mittag}.

(i)$\Rightarrow$(iii) In view of \eqref{Sti} and
\eqref{pdef}, $\int_X g_{\boldsymbol \alpha} |f|^2
\D\mu \Ge 0$ for every $f \in \dzn{\cfw}$. According
to Lemma \ref{lemS11p}(iii), $\chi_{\varDelta} \cdot f
\in \dzn{\cfw}$ for all $f \in \dzn{\cfw}$ and
$\varDelta \in \ascr$. Applying \eqref{pdef} and
\cite[Corollary 12.6]{b-j-j-sC} with
$\escr=\dzn{\cfw}$, we get
   \begin{gather*}
\text{$\sum_{i,j=0}^n \alpha_i\bar\alpha_j \hfwn{i+j}
\Ge 0$ a.e.\ $[\mu]$ for all $\{\alpha_i\}_{i=0}^n
\subseteq \cbb$ and $n\in\zbb_+$.}
   \end{gather*}
Arguing as in the proof of \cite[Lemma
10.1(a)]{b-j-j-sC}, we deduce that the sequence
$\{\hfwn{n}(x)\}_{n=0}^\infty$ is positive definite
for $\mu$-a.e.\ $x\in X$. Replacing $f$ by $\cfw f$ in
the above reasoning, we conclude that the sequence
$\{\hfwn{n+1}(x)\}_{n=0}^\infty$ is positive definite
for $\mu$-a.e.\ $x\in X$. Employing \eqref{Sti} gives
(iii).

(iii)$\Rightarrow$(i) By \eqref{Sti} and \eqref{pdef}, the sequence
$\{\|\cfw^{n} f\|^2\}_{n=0}^\infty$ is positive definite for every $f \in
\dzn{\cfw}$. Replacing $f$ by $\cfw f$ and using \eqref{Sti} again, we
obtain (i).

(i)$\Rightarrow$(iv) Clearly, if $\varDelta \in \ascr$
is such that $\mu_{\widehat w_k}(\phi^{-k}(\varDelta))
< \infty$ for all $k \in \zbb_+$, then
$\chi_{\varDelta} \in \dz{C_{\phi^n,\widehat w_n}}$
and $\|C_{\phi^n,\widehat w_n} \chi_{\varDelta}\|^2 =
\mu_{\widehat w_n}(\phi^{-n}(\varDelta))$ for every $n
\in \zbb_+$. Applying Lemma \ref{lemS11}(i) and Lemma
\ref{lemS11p}(iv) yields (iv).

(iv)$\Rightarrow$(i) Let $s$ be a simple $\ascr$-measurable function on
$X$ of the form $s = \sum_{i=1}^k \alpha_i \chi_{\varDelta_i}$, where
$\{\varDelta_i\}_{i=1}^k \subseteq \ascr$ are pairwise disjoint and
$\{\alpha_i\}_{i=1}^k \subseteq (0,\infty)$. Suppose $s \in \dzn{\cfw}$.
Then, by assertions (i) and (iii) of Lemma \ref{lemS11p},
$\chi_{\varDelta_i} \in \dzn{\cfw}$ for every $i\in \{1, \ldots,k\}$ and
   \begin{align*}
\|\cfw^n s\|^2 & = \sum_{i,j=1}^k \alpha_i \alpha_j
\int_{\varDelta_i \cap \varDelta_j} \hfwn{n} \D\mu
   \\
& = \sum_{i=1}^k \alpha_i^2 \int_{\varDelta_i}
\hfwn{n} \D\mu = \sum_{i=1}^k \alpha_i^2 \mu_{\widehat
w_n}(\phi^{-n}(\varDelta_i)), \quad n\in \zbb_+.
   \end{align*}
Hence, by (iv), $\{\|C_{\phi}^n s\|^2\}_{n=0}^\infty$
is a Stieltjes moment sequence (clearly, this is the
case for $s = 0$ as well). If $f \in \dzn{C_{\phi}}$,
then there exists a sequence $\{s_n\}_{n=1}^\infty$ of
simple $\ascr$-measurable functions $s_n\colon X \to
\rbb_+$ such that $s_n(x) \Le s_{n+1}(x) \Le |f(x)|$
and $\lim_{k \to \infty} s_k(x) = |f(x)|$ for all $x
\in X$ and $n\in \nbb$. By Lemma \ref{lemS11p}(iii),
this implies that $\{s_n\}_{n=1}^\infty \subseteq
\dzn{C_{\phi}}$. Applying Lemma \ref{lemS11p}(i) and
Lebesgue's monotone convergence theorem yields
   \begin{align*}
\|\cfw^n f\|^2 & = \int_X |f|^2 \hfwn{n} \D\mu
   \\
& = \lim_{k\to \infty} \int_X s_k^2 \hfwn{n} \D\mu =
\lim_{k\to \infty} \|\cfw^n s_k\|^2, \quad n \in
\zbb_+.
   \end{align*}
Since the class of Stieltjes moment sequences is
closed under the operation of taking pointwise limits
(use \eqref{Sti}), we see that $\{\|\cfw^n
f\|^2\}_{n=0}^\infty$ is a Stieltjes moment sequence.

(iii)$\Rightarrow$(v) If $p\in \cbb[t]_+$, then there exist
$q_1,q_2\in\cbb[t]$ such that $p(t)=t|q_1(t)|^2 + |q_2(t)|^2$ for all $t
\in \rbb$ (see \cite[Problem 45, p.\ 78]{P-S}). This, together with
\eqref{Sti}, implies that $L(p) \Ge 0$ a.e.\ $[\mu]$.

(v)$\Rightarrow$(iii) Let $Q$ be a countable dense subset of $\cbb$. Since
$L(|q|^2) \Ge 0$ a.e.\ $[\mu]$ and $L(t|q|^2) \Ge 0$ a.e.\ $[\mu]$ for
every $q \in \cbb[t]$, there exists $\varDelta \in \ascr$ such that $\mu(X
\setminus \varDelta) = 0$,
   \begin{align} \label{pd2}
\text{$\sum_{i,j=0}^n \hfwn{i+j}(x) \alpha_i \bar
\alpha_j \Ge 0$ and $\sum_{i,j=0}^n \hfwn{i+j+1}(x)
\alpha_i \bar \alpha_j \Ge 0$}
   \end{align}
for all $x \in \varDelta$, $\alpha_0, \ldots, \alpha_n \in Q$ and $n\in
\zbb_+$. By the density of $Q$ in $\cbb$, we conclude that \eqref{pd2}
holds for all $x \in \varDelta$, $\alpha_0, \ldots, \alpha_n \in \cbb$ and
$n\in \zbb_+$. Applying \eqref{Sti} gives (iii).

(i)$\Rightarrow$(b) Arguing as in the proof of \cite[Lemma
10.1(b)]{b-j-j-sC}, we see that
   \begin{align*}
\sum_{j=0}^{n} \hfwn{j} \Le (n+1) (1 + \hfwn{n})
\text{ a.e.\ $[\mu]$,} \quad n \in \zbb_+.
   \end{align*}
Hence, applying Lemma \ref{13-12-13}(iv) and
Proposition \ref{lemS1}(iv), we get (b).
   \end{proof}
   \section{\label{Sec4.3}Subnormality in the bounded case}
We begin by stating the weighted composition operator
counterpart of the celebrated Lambert's
characterizations of bounded subnormal composition
operators (see \cite{lam1}). It is a direct
consequence of Theorem \ref{gsms} and \cite[Theorem
3.1]{lam} (see also \cite[Theorem 7]{StSz2} for the
general, not necessarily injective, case). Recall that
subnormal weighted composition operators may not be
injective which is not the case for composition
operators (see \cite[Corollary 6.3]{b-j-j-sC}). For
the definition of $\widehat w_n$ we refer the reader
to Lemma \ref{lemS11}.
   \begin{thm} \label{gsms+}
Suppose \eqref{stand2} holds and $\cfw \in
\ogr{L^2(\mu)}$. Then the following conditions are
equivalent{\em :}
   \begin{enumerate}
   \item[(i)] $\cfw$ is subnormal,
   \item[(ii)] $\{\|\cfw^{n}
f\|^2\}_{n=0}^\infty$ is a Stieltjes moment sequence for all $f \in
L^2(\mu)$,
   \item[(iii)] $\{\hfwn{n}(x)\}_{n=0}^\infty$ is a
Stieltjes moment sequence for $\mu$-a.e.\ $x \in X$,
   \item[(iv)] $\{\mu_{\widehat w_n}(\phi^{-n}
(\varDelta))\}_{n=0}^\infty$ is a Stieltjes mo\-ment
sequence for all $\varDelta \in \ascr$ such that
$\mu(\varDelta) < \infty$,
   \item[(v)] there exists
a linear mapping $L\colon \cbb[t] \to {\EuScript M}_X$ such that
   \begin{align*}
L(t^n) & = \hfwn{n} \text{ a.e.\ $[\mu]$,} \quad n \in
\zbb_+,
   \\
L(p) & \Ge 0 \text{ a.e.\ $[\mu]$,} \quad p\in \cbb[t]_+.
   \end{align*}
   \end{enumerate}
   \end{thm}
Before proving the main result of this section, we recall a ``moment
measurability" lemma which is a variant of \cite[Lemma 1.3]{lam2}.
   \begin{lem}[\mbox{\cite[Lemma 11]{b-j-j-sS}}]
\label{mommea} Let $(X,\ascr)$ be a measurable space
and $K$ be a compact subset of $\rbb_+$. Suppose that
$\{P(x,\cdot)\}_{x\in X}$ is a family of Borel
probability measures on $\rbb_+$ such that $\supp
P(x,\cdot) \subseteq K$ for every $x\in X$, and the
mapping $X\ni x \longmapsto \int_{K} t^n P(x, \D t)
\in \rbb_+$ is $\ascr$-measurable for every $n\in
\zbb_+$. Then the mapping $P\colon X \times
\borel{\rbb_+} \ni (x,\sigma) \longmapsto P(x,\sigma)
\in [0,1]$ is an $\ascr$-measurable family of
probability measures.
   \end{lem}
The following theorem generalizes a characterization of subnormality of
bound\-ed composition operators given in \cite[Theorem 13]{b-j-j-sS} to
the case of bounded weighted composition operators. The present proof is
completely different and fits nicely into our framework.
   \begin{thm} \label{bsubn}
Suppose $(X,\ascr,\mu)$ is a $\sigma$-finite measure space, $w \colon X
\to \cbb$ is an $\ascr$-measurable function and $\phi$ is an
$\ascr$-measurable transformation of $X$ such that $\cfw \in
\ogr{L^2(\mu)}$. Then the following three conditions are equivalent{\em :}
   \begin{enumerate}
   \item[(i)] $\cfw$ is subnormal,
   \item[(ii)] $\hfw > 0$ a.e.\ $[\mu_w]$ and
there exists $P\colon X \times \borel{\rbb_+} \to [0,1]$, an
$\ascr$-measur\-able family of probability measures, which satisfies
\eqref{cc},
   \item[(iii)] there exists $P\colon X \times
\borel{\rbb_+} \to [0,1]$, an $\ascr$-measurable family of probability
measures, which satisfies \eqref{cc-1mu}.
   \end{enumerate}
Moreover, the following three assertions are valid{\em :}
   \begin{enumerate}
   \item[(a)] if {\em (i)} holds, then there exists $P\colon X
\times \borel{\rbb_+} \to [0,1]$, an
$\ascr$-measur\-able family of probability measures,
which satisfies {\em \eqref{cc-1mu}} $($and thus {\em
\eqref{cc}}$)$, and which has the property that $\supp
P(x,\cdot) \subseteq [0,\|\cfw\|^2]$ for each $x\in
X$,
   \item[(b)] if $P_1,P_2\colon X \times
\borel{\rbb_+} \to [0,1]$ are $\ascr$-measurable
families of probability measures satisfying \eqref{cc}
and $\hfw > 0$ a.e.\ $[\mu_w]$, then $P_1(x, \cdot) =
P_2(x, \cdot)$ for $\mu_w$-a.e.\ $x\in X$,
   \item[(c)] if $P_1,P_2\colon X \times \borel{\rbb_+} \to
[0,1]$ are $\ascr$-measurable families of probability measures satisfying
\eqref{cc-1mu}, then $P_1(x, \cdot) = P_2(x, \cdot)$ for $\mu$-a.e.\ $x\in
X$.
   \end{enumerate}
   \end{thm}
   \begin{proof}
The implication (ii)$\Rightarrow$(i) follows from Theorem \ref{MAIN1},
while the equivalence (ii)$\Leftrightarrow$(iii) is a consequence of
Theorem \ref{MAIN2}.

To prove the implication (i)$\Rightarrow$(ii), we
assume that $\cfw$ is subnormal. If $\cfw=0$, then by
Proposition \ref{lemS1}(vi) the family $P(x, \cdot) =
\delta_0(\cdot)$, $x\in X$, meets our requirements.
Hence, we can assume that $r := \|\cfw\|^2 > 0$. It
follows from Corollary \ref{hipinj} that $\hfw > 0$
a.e.\ $[\mu_w]$. Set $J=[0,r]$. According to
\cite[Proposition 3.2.1]{b-j-j-sA}, $\cfw$ generates
Stieltjes moment sequences. By Lemma \ref{lemS11p}(i),
we have
   \begin{align*}
\int_{\varDelta} \hfwn{n} \D \mu =
\|\cfw^n(\chi_{\varDelta})\|^2 \Le r^n \mu(\varDelta),
\quad \varDelta \in \ascr, \, \mu(\varDelta) < \infty.
   \end{align*}
Since $\mu$ is $\sigma$-finite, we infer from
\cite[Theorem 1.40]{Rud} that $\hfwn{n} \Le r^n$ a.e.\
$[\mu]$ for every $n \in \zbb_+$. This, \eqref{Sti}
and Theorem \ref{gsms}(iii) imply that there exists
$X_0 \in \ascr$ with $\mu(X \setminus X_0)=0$ such
that for every $x \in X_0$, \allowdisplaybreaks
   \begin{align}  \label{17-0}
\hfwn{0}(x) & = 1,
   \\ \label{17-1}
0 \Le \hfwn{n}(x) & \Le r^n, \quad n \in \zbb_+,
   \\    \label{17-2}
\sum_{i,j=0}^n \alpha_i\bar\alpha_j \hfwn{i+j}(x) &
\Ge 0, \quad \alpha_0, \ldots, \alpha_n \in \cbb, \, n
\in \zbb_+,
   \\ \label{17-3}
\sum_{i,j=0}^n \alpha_i\bar\alpha_j \hfwn{i+j+1}(x) &
\Ge 0, \quad \alpha_0, \ldots, \alpha_n \in \cbb, \, n
\in \zbb_+.
   \end{align}
Setting $\hfwn{n}(x) = \int_J t^n \D \delta_r(t)$ for
all $x \in X \setminus X_0$ and $n \in \zbb_+$, we may
assume without loss of generality that the conditions
\eqref{17-0}--\eqref{17-3} hold for all $x \in X$. By
\eqref{Stiogr}, for every $x \in X$ there exists a
unique Borel probability measure $P(x, \cdot)$ on
$\rbb_+$ such that $\supp P(x, \cdot) \subseteq J$ and
   \begin{align} \label{smom0}
\hfwn{n}(x) = \int_{J} t^n P(x, \D t), \quad x \in X,
\, n\in \zbb_+.
   \end{align}
Applying Lemma \ref{mommea}, we see that the mapping
$P\colon X \times \borel{\rbb_+} \ni (x, \sigma)
\longmapsto P(x, \sigma) \in [0,1]$ is an
$\ascr$-measurable family of probability measures.
Using Theorem \ref{Ath} (with $p=\infty$) and arguing
as in \cite[Lemma 10.1(a)]{b-j-j-sC}, we deduce from
\eqref{17-1}, \eqref{17-2} and \eqref{17-3} that there
exists $X_1 \in \phi^{-1}(\ascr)$ with $\mu_w(X
\setminus X_1) = 0$ such that for every $x \in X_1$,
\allowdisplaybreaks
   \begin{align*}
\efw(\hfwn{0})(x) & = 1,
   \\
0 \Le \efw(\hfwn{n})(x) & \Le r^n, \quad n \in \zbb_+,
   \\
\sum_{i,j=0}^n \alpha_i\bar\alpha_j
\efw(\hfwn{i+j})(x) & \Ge 0, \quad \alpha_0, \ldots,
\alpha_n \in \cbb, \, n \in \zbb_+,
   \\
\sum_{i,j=0}^n \alpha_i\bar\alpha_j
\efw(\hfwn{i+j+1})(x) & \Ge 0, \quad \alpha_0, \ldots,
\alpha_n \in \cbb, \, n \in \zbb_+.
   \end{align*}
Hence, applying \eqref{Stiogr} again, we see that for every $x \in X_1$
there exists a unique Borel probability measure $Q(x, \cdot)$ on $\rbb_+$
such that $\supp Q(x, \cdot) \subseteq J$ and
   \begin{align} \label{smom1}
\efw(\hfwn{n})(x) = \int_{J} t^n \D Q(x, \D t), \quad
n\in \zbb_+.
   \end{align}
Setting $Q(x, \cdot) = \delta_r(\cdot)$ and
$\efw(\hfwn{n})(x) = r^n$ for all $x \in X \setminus
X_1$ and $n \in \zbb_+$, we may assume without loss of
generality that \eqref{smom1} holds for all $x \in X$.
Hence, by Lemma \ref{mommea}, $Q\colon X \times
\borel{\rbb_+} \ni (x, \sigma) \longmapsto Q(x,
\sigma) \in [0,1]$ is a $\phi^{-1}(\ascr)$-measurable
family of probability measures such that $\supp
Q(x,\cdot) \subseteq J$ for every $x \in X$. Combining
\eqref{smom1} (which holds for all $x\in X$),
\eqref{rec2} and \eqref{smom0}, we see that for
$\mu_w$-a.e.\ ~ $x \in X$,
   \begin{align} \label{etacc}
\hfw(\phi(x)) \int_{J} t^n Q(x,\D t) = \int_{J} t^n t
P(\phi(x), \D t), \quad n \in \zbb_+.
   \end{align}
However, since the measures $Q(x, \cdot)$, $x\in X$, are compactly
supported, \eqref{Stiogr} yields
   \begin{align} \label{cceta}
\hfw(\phi(x)) \cdot Q(x, \sigma) = \int_{\sigma} t
P(\phi(x), \D t) \text{ for $\mu_w$-a.e.\ $x \in X$},
\quad \sigma \in \borel{\rbb_+}.
   \end{align}
Hence, by \eqref{smom1} (which is valid for all $x\in X$), we have
\allowdisplaybreaks
   \begin{align*}
\int_{\phi^{-1}(\varDelta)} \int_{J} t^n Q(x, \D t) \D
\mu_w(x) & = \int_{\phi^{-1}(\varDelta)}
\efw(\hfwn{n})(x) \D \mu_w(x)
   \\
& \hspace{-.7ex} \overset{\eqref{wazny}}=
\int_{\phi^{-1}(\varDelta)} \hfwn{n}(x) \D \mu_w(x)
   \\
& \hspace{-.7ex} \overset{\eqref{smom0}}= \int_{\phi^{-1}(\varDelta)}
\int_{J} t^n P(x, \D t) \D \mu_w(x), \quad \varDelta \in \ascr, \, n \in
\zbb_+.
   \end{align*}
As a consequence, we obtain
   \begin{multline} \label{zas1}
\int_{\phi^{-1}(\varDelta)} \int_{J} p(t) Q(x, \D t) \D \mu_w(x) =
\int_{\phi^{-1}(\varDelta)} \int_{J} p(t) P(x, \D t) \D \mu_w(x),
   \\
\varDelta \in \ascr, \, p\in \cbb[t].
   \end{multline}
Since $\mu$ is $\sigma$-finite, there exists a
sequence $\{Y_k\}_{k=1}^{\infty} \subseteq \ascr$ such
that $Y_k \nearrow X$ as $k \to \infty$, and $\mu(Y_k)
< \infty$ for every $k\in \nbb$. Fix $k \in \nbb$ and
$\varDelta \in \ascr$ such that $\varDelta \subseteq
Y_k$. Let $f\in C(J)$ (as usual, $C(J)$ stands for the
set of all continuous complex functions on $J$). Then,
by the Weierstrass theorem, there exists a sequence
$\{p_n\}_{n=1}^{\infty} \subseteq \cbb[t]$ such that
$\sup_J |f-p_n| \to 0$ as $n \to \infty$. Since
$P(x,\cdot)$ and $Q(x,\cdot)$ are probability measures
for all $x \in X$, we infer from \eqref{l1} and
Proposition \ref{lemS1}(v) that
   \begin{multline}  \label{zas0}
\bigg|\int_{\phi^{-1}(\varDelta)} \int_{J} (f(t) - p_n(t)) S(x, \D t) \D
\mu_w(x) \bigg| \Le \sup_J |f-p_n| \, \mu_w(\phi^{-1}(\varDelta))
   \\
\Le r \sup_J |f-p_n| \, \mu(Y_k), \quad n \in \nbb,
   \end{multline}
where $S \in \{P, Q\}$. This and \eqref{zas1} imply
   \begin{align} \label{zas2}
\int_{\phi^{-1}(\varDelta)} \int_{J} f(t) Q(x, \D t)
\D \mu_w(x) = \int_{\phi^{-1}(\varDelta)} \int_{J}
f(t) P(x, \D t) \D \mu_w(x), \quad f \in C(J).
   \end{align}
Let $\sigma = [0,a)$ with $a \in (0,r]$. Then there
exists a sequence $\{f_n\}_{n=1}^{\infty} \subseteq
C(J)$ such that $0 \Le f_n \Le 1$ for all $n \in
\nbb$, and $f_n(t) \to \chi_{\sigma}(t)$ as $n\to
\infty$ for every $t\in J$. Since
$\mu_w(\phi^{-1}(\varDelta)) < \infty$ (see
\eqref{zas0}), we may apply \eqref{zas2} and
Lebesgue's dominated convergence theorem to get the
following equality
   \begin{align} \label{zas4}
\int_{\phi^{-1}(\varDelta)} Q(x, \sigma) \D \mu_w(x) =
\int_{\phi^{-1}(\varDelta)} P(x, \sigma) \D \mu_w(x).
   \end{align}
Set $\pscr=\big\{[a,b) \cap J \colon a, b \in
\rbb\big\}$. Noticing that $[a,b)\cap J = ([0,b)\cap
J) \setminus ([0,a) \cap J)$ whenever $a \Le b$, we
deduce that \eqref{zas4} holds for every $\sigma \in
\pscr$. Since $\pscr$ is a semi-algebra,
$\borel{J}=\sigma_J(\pscr)$ and
$\mu_w(\phi^{-1}(\varDelta)) < \infty$, we infer from
Lemma \ref{2miary} that \eqref{zas4} holds for every
$\sigma \in \borel{J}$. If $\varDelta \in \ascr$, then
$\phi^{-1}(\varDelta \cap Y_k) \nearrow
\phi^{-1}(\varDelta)$ as $k \to \infty$. Applying
\eqref{zas4} to $\varDelta \cap Y_k$ in place of
$\varDelta$ and using Lebesgue's monotone convergence
theorem, we see that \eqref{zas4} holds for all
$\varDelta \in \ascr$ and $\sigma \in \borel{J}$.
Since $\supp P(x,\cdot) \subseteq J$ and $\supp
Q(x,\cdot) \subseteq J$ for all $x \in X$, we deduce
that \eqref{zas4} holds for all $\varDelta \in \ascr$
and $\sigma \in \borel{\rbb_+}$. In view of the fact
that $Q(\cdot, \sigma)$ is
$\phi^{-1}(\ascr)$-measurable for all $\sigma \in
\borel{\rbb_+}$, we conclude that
   \begin{align*}
Q(x, \sigma) = \efw(P(\cdot, \sigma))(x) \text{ for
$\mu_w$-a.e.\ $x\in X$, } \quad \sigma \in
\borel{\rbb_+}.
   \end{align*}
This combined with \eqref{cceta} shows that $P$
satisfies \eqref{cc} (hence, (ii) holds). Since $P$
also satisfies \eqref{smom0}, we infer from Theorem
\ref{main2} that $P$ satisfies \eqref{cc-1mu} as well.
Recall that $\supp P(x, \cdot) \subseteq
[0,\|\cfw\|^2]$ for every $x\in X$ (hence, (a) holds).
This, together with \eqref{Stiogr} and the implication
(vii)$\Rightarrow$(ii) of Theorem \ref{main1}, proves
(b). In turn, using \eqref{Stiogr} and the implication
\mbox{(i$^\star$)}$\Rightarrow$\mbox{(ii$^\star$)} of
Theorem \ref{main2} (recall that \eqref{cc-1mu}
implies \eqref{cc} with the same $P$) gives (c). This
completes the proof.
   \end{proof}
   \begin{rem} \label{eta}
Regarding the proof of the implication
(i)$\Rightarrow$(ii) of Theorem \ref{bsubn}, we note
that in the case of composition operators the second
usage of Lemma \ref{mommea} is to be omitted. Indeed,
in this case we can assume that $\mathsf h_{\phi}(x) >
0$ for all $x\in X$ (see \cite[Corollary
6.3]{b-j-j-sC}). Arguing as in the original proof, we
find a family $\{\eta_x \colon x \in X_1\}$ of Borel
probability measures on $\rbb_+$ supported in $J$,
which satisfies \eqref{smom1}, \eqref{etacc} and
\eqref{cceta} with $\eta_x$ in place of $Q(x, \cdot)$
(the first two conditions hold for $\mu$-a.e.\ $x\in
X$). Next we define a $\phi^{-1}(\ascr)$-measurable
family of probability measures $Q\colon X \times
\borel{\rbb_+} \to [0,1]$~ by
   \begin{align*}
Q(x, \sigma) = \frac{\int_{\sigma} t P(\phi(x), \D t)}{\mathsf
h_{\phi}(\phi(x))}, \quad x \in X, \, \sigma \in \borel{\rbb_+}.
   \end{align*}
Clearly, this $Q$ satisfies \eqref{cceta}. Beginning
with \eqref{cceta} we may repeat the rest of the
original proof. It is worth pointing out that this
procedure breaks down in the case of weighted
composition operators.
   \end{rem}
   \chapter{\label{Chap5}Seminormality}
In this chapter, we give characterizations of
seminormal, formally normal, symmetric, selfadjoint
and positive selfadjoint weighted composition
operators. Hyponormality and cohyponormality are
characterized in Sections \ref{Sec5.1} and
\ref{Sec5.2}, respectively (see Theorems \ref{hypon}
and \ref{cohypon-m}). The introductory part of Section
\ref{Sec5.2} is devoted to the study of the range of
the conditional expectation regarded either as a
mapping on the set of $\rbop$-valued
$\ascr$-measurable functions or as an operator in
$L^2$-space. In Section \ref{Sec5.3}, we characterize
normal weighted composition operators (see Theorem
\ref{normal-m}). We also show that formally normal (in
particular, symmetric) weighted composition operators
are automatically normal (see Theorem \ref{fn2n}). In
Section \ref{Sec5.4}, we characterize selfadjoint and
positive selfadjoint weighted composition operators
(see Theorems \ref{self} and \ref{psa}).

The characterizations of hyponormal and cohyponormal
(not necessarily bound\-ed) weighted composition
operators were given by Campbell and Hornor in
\cite{ca-hor} under quite restrictive assumptions.
Namely, in contrast to our paper, they made the
following assumptions (in the notation of Section
\ref{pco}(b)):
   \begin{enumerate}
   \item[\ding{202}] the underlying measure space $(X,\ascr, \mu)$ is
complete,
   \item[\ding{203}] the measure space
$(X,\phi^{-1}(\ascr),\mu|_{\phi^{-1}(\ascr)})$ is
complete,
   \item[\ding{204}] $\mu\circ \phi^{-1} \ll
\mu$ (equivalently, $C_{\phi}$ is well-defined),
   \item[\ding{205}] $\mathsf{h}_{\phi} < \infty$ a.e.\
$[\mu]$ (equivalently, $C_{\phi}$ is densely defined),
   \item[\ding{206}] $w \Ge 0$ a.e.\ $[\mu]$.
   \end{enumerate}
They wrote that the assumption \ding{205} (of course,
combined with \ding{204}) ``plays an important role in
obtaining our results; it is implicit in most of the
definitions and used explicitly in many of the
calculations''. It guarantees the existence of the
conditional expectation $\mathsf{E}(\,\cdot\,;
\phi^{-1}(\ascr),\mu)$, which, together with
\ding{206}, enables them to deal with $\mathsf{E}(w;
\phi^{-1}(\ascr),\mu)$. It is worth pointing out that,
under the assumptions \ding{204} and \ding{205}, a
weighted composition operator $\cfw$ may not coincide
with the product $M_wC_{\phi}$ even if the operators
$C_{\phi}$ and $\cfw$ are subnormal (see Example
\ref{mwc8}; see also Section \ref{Sec7.1}). What is
worse, it may happen that a weighted composition
operator $\cfw$ is an isometry while the corresponding
composition operator $C_{\phi}$ is even not
well-defined (see Example \ref{cohyp-count2}). The
assumptions \ding{202} and \ding{203} may not be
satisfied even in the case of discrete measure spaces
(see Remark \ref{figurar}).

The reader should be aware that in this chapter the
conditional expectation is regarded either as a
mapping on the set of $\rbop$-valued
$\ascr$-measurable functions (modulo $\mu_w$) or as an
operator in the Hilbert space $L^2(\mu_w)$. In
particular, writing $\jd{\efw}$ and $\ob{\efw}$, we
regard $\efw$ as an operator in $L^2(\mu_w)$.
   \section{\label{Sec5.1}Hyponormality}
   Below we give a few useful characterizations of
hyponormality of densely defined weighted composition
operators.
   \begin{thm} \label{hypon} Suppose \eqref{stand2}
holds and $\cfw$ is densely defined. Then the
following conditions are equivalent{\em :}
   \begin{enumerate}
   \item[(i)] $\cfw$ is hyponormal,
   \item[(ii)] $\hfw > 0$ a.e.\ $[\mu_w]$ and
$\efw\Big(\sqrt{\frac{\hfw \circ \phi}{\hfw}}\cdot
f\Big)^2 \Le \efw(f^2)$ a.e.\ $[\mu_w]$ for every
$\ascr$-measurable function $f\colon X \to \rbb_+$,
   \item[(iii)] $\hfw > 0$ a.e.\ $[\mu_w]$ and
$\efw\Big(\frac{\hfw\circ \phi}{\hfw}\Big) \Le 1$
a.e.\ $[\mu_w]$,
   \item[(iv)] $\hfw > 0$ a.e.\ $[\mu_w]$ and
$\efw\Big(\frac{1}{\hfw}\Big) \Le \frac{1}{\hfw \circ
\phi}$ a.e.\ $[\mu_w]$.
   \end{enumerate}
   \end{thm}
   \begin{proof}  In view of
Lemma \ref{lemS6}, Proposition \ref{lemS2} and
Corollary \ref{hipinj}, we can assume, without loss of
generality, that
   \begin{align} \label{dyszcz}
0 < \hfw < \infty \text{ a.e.\ $[\mu_w]$ and } 0 <
\hfw \circ \phi < \infty \text{ a.e.\ $[\mu_w]$.}
   \end{align}

(i)$\Leftrightarrow$(ii) It follows from \eqref{l2},
Proposition \ref{lemS1}(i) and Proposition
\ref{adj}(i)\&(ii) that $\cfw$ is hyponormal if and
only if the following two conditions hold:
   \begin{gather} \notag
L^2((1+\hfw)\D \mu) \subseteq \big\{f \in
L^2(\mu)\colon \hfw \cdot \efw(f_w) \circ \phi^{-1}
\in L^2(\mu)\big\},
   \\ \label{pozni2}
\int_X \hfw^2 \cdot |\efw(f_w) \circ \phi^{-1}|^2 \D
\mu \Le \int_X \hfw |f|^2 \D \mu, \quad f \in
L^2((1+\hfw)\D \mu).
   \end{gather}
Clearly, this implies that $\cfw$ is hyponormal if and
only if \eqref{pozni2} holds. Making the substitution
$f \leftrightsquigarrow f\cdot w$ and noting that $f =
f\cdot \chi_{\{w\neq 0\}}$ a.e.\ $[\mu_w]$, we deduce
that \eqref{pozni2} holds if and only if the following
inequality is satisfied
   \begin{align} \label{pozni3}
\int_X \hfw^2 \cdot |\efw(f) \circ \phi^{-1}|^2 \D \mu
\Le \int_X \hfw |f|^2 \D \mu_w, \quad f \in
L^2((1+\hfw)\D \mu_w).
   \end{align}
Since for every $f \in L^2((1+\hfw)\D \mu_w)$,
   \begin{align*}
\int_X \hfw^2 \cdot |\efw(f) \circ \phi^{-1}|^2 \D \mu
& \overset{\eqref{l2}} = \int_X \hfw \circ \phi \cdot
|\efw(f) \circ \phi^{-1} \circ \phi|^2 \D \mu_w
   \\
&\hspace{-.5ex} \overset{\eqref{fifi}} = \int_X \hfw
\circ \phi \cdot |\efw(f)|^2 \D \mu_w,
   \end{align*}
we see that \eqref{pozni3} is equivalent to
   \begin{align} \label{pozni4}
\int_X \hfw \circ \phi \cdot |\efw(f)|^2 \D \mu_w \Le
\int_X \hfw |f|^2 \D \mu_w, \quad f \in L^2((1+\hfw)\D
\mu_w).
   \end{align}

Our next goal is to show that \eqref{pozni4} is
equivalent to
   \begin{multline} \label{pozni5}
\int_X \hfw \circ \phi \cdot \efw(f)^2 \D \mu_w \Le
\int_X \hfw f^2 \D \mu_w, \quad
\text{$\ascr$-measurable $f\colon X \to \rbb_+$.}
   \end{multline}
For this, suppose that \eqref{pozni4} holds. Then, by
Theorem \ref{Ath}(ii), we have
   \begin{align} \label{pozni6}
\int_X \hfw \circ \phi \cdot \efw(f)^2 \D \mu_w \Le
\int_X \hfw f^2 \D \mu_w, \quad f \in L_+^2((1+\hfw)\D
\mu_w).
   \end{align}
Define the measure $\nu\colon \ascr \to \rbop$ by
$\nu(\varDelta) = \int_{\varDelta} (1+\hfw) \D \mu_w$
for $\varDelta \in \ascr$. Since $\mu_w$ is
$\sigma$-finite, we infer from Proposition \ref{lemS2}
that the measure $\nu$ is $\sigma$-finite. Let
$\{\varOmega_k\}_{k=1}^{\infty} \subseteq \ascr$ be
such that $\nu(\varOmega_k) < \infty$ for all $k\in
\nbb$ and $\varOmega_k \nearrow X$ as $k\to \infty$.
Take an $\ascr$-measurable function $f \colon X \to
\rbb_+$. Fix $k\in \nbb$. Then, by \cite[Theorem
1.17]{Rud}, there exists a sequence
$\{s_n\}_{n=1}^{\infty}$ of $\ascr$-measurable simple
functions such that $0 \Le s_n \nearrow
\chi_{\varOmega_k} f$ as $n\to \infty$. It is easily
seen that $\{s_n\}_{n=1}^{\infty} \subseteq
L_+^2(\nu)$. In view of \eqref{B5}, we have $0 \Le
\efw(s_n) \nearrow \efw(\chi_{\varOmega_k}f)$ a.e.\
$[\mu_w]$ as $n\to \infty$. This together with
Lebesgue's monotone convergence theorem and
\eqref{pozni6} yields
   \allowdisplaybreaks
   \begin{align*}
\int_X \hfw \circ \phi \cdot
\efw(\chi_{\varOmega_k}f)^2 \D \mu_w & =
\lim_{n\to\infty} \int_X \hfw \circ \phi \cdot
\efw(s_n)^2 \D \mu_w
   \\
&\Le \lim_{n\to\infty} \int_X \hfw s_n^2 \D \mu_w
   \\
&= \int_X \hfw \chi_{\varOmega_k} f^2 \D \mu_w.
   \end{align*}
The same reasoning applied to $\{\chi_{\varOmega_k}
f\}_{k=1}^{\infty}$ in place of
$\{s_n\}_{n=1}^{\infty}$ leads to \eqref{pozni5}.
Conversely, assume that \eqref{pozni5} holds. Then, by
Proposition \ref{BL3}(i), we have
   \begin{align*}
\int_X \hfw \circ \phi \cdot |\efw(f)|^2 \D \mu_w &
\Le \int_X \hfw \circ \phi \cdot \efw(|f|)^2 \D\mu_w
   \\
& \overset{\eqref{pozni5}} \Le \int_X \hfw |f|^2 \D
\mu_w, \quad f\in L^2((1+\hfw)\D \mu_w),
   \end{align*}
which gives \eqref{pozni4} and completes the proof of
the equivalence
\eqref{pozni4}$\Leftrightarrow$\eqref{pozni5}.

Using \eqref{B6}, we deduce that \eqref{pozni5} is
equivalent to
   \begin{align*}
\int_X \efw(\sqrt{\hfw \circ \phi}\cdot f)^2 \D \mu_w
\Le \int_X \hfw f^2 \D \mu_w, \quad
\text{$\ascr$-measurable } f\colon X \to \rbb_+.
   \end{align*}
Hence, by making the substitution $f
\leftrightsquigarrow \sqrt{\hfw}\cdot f$ and using
\eqref{dyszcz}, we see that \eqref{pozni5} is
equivalent to
   \begin{align}  \notag
\int_X \efw(\vartheta f)^2 \D \mu_w & \Le \int_X f^2
\D \mu_w
   \\ \label{pozni7}
& = \int_X \efw(f^2) \D \mu_w, \quad
\text{$\ascr$-measurable } f\colon X \to \rbb_+,
   \end{align}
where $\vartheta\colon X \to \rbb_+$ is an
$\ascr$-measurable function such that $\vartheta=
\sqrt{\frac{\hfw \circ \phi}{\hfw}}$ a.e.\ $[\mu_w]$
(in view of \eqref{dyszcz} such $\vartheta$ exists).
Substituting $\chi_{\phi^{-1}(\varDelta)}f$ in place
of $f$ and using \cite[Theorem 1.6.11]{Ash} together
with \eqref{B6} and the fact that $\mu_w$ is
$\sigma$-finite, we deduce that \eqref{pozni7} is
equivalent to
   \begin{align} \label{pozni8}
\efw(\vartheta f)^2 \Le \efw(f^2) \text{ a.e.\
$[\mu_w]$,} \quad \text{$\ascr$-measurable } f\colon X
\to \rbb_+.
   \end{align}
Thus (i) and (ii) are equivalent.

(ii)$\Leftrightarrow$(iii) If \eqref{pozni8} holds,
then by substituting $f=\chi_{\{\vartheta \Le n\}}
\vartheta$ into \eqref{pozni8}, where $\{\vartheta \Le
n\}:=\{x \in X\colon \vartheta(x) \Le n\}$, we get
   \begin{align*}
\efw(\chi_{\{\vartheta \Le n\}}\vartheta^2)^2 \Le
\efw(\chi_{\{\vartheta \Le n\}}\vartheta^2) \overset{
\eqref{B4}}\Le n^2 \efw(\boldsymbol{1})=n^2
\boldsymbol{1} \text{ a.e.\ $[\mu_w]$}, \quad n\in
\nbb.
   \end{align*}
This implies that $\efw(\chi_{\{\vartheta \Le
n\}}\vartheta^2)\Le 1$ a.e.\ $[\mu_w]$ for all $n\in
\nbb$. Passing to the limit with $n\to\infty$ and
using \eqref{B5}, we get $\efw(\vartheta^2) \Le 1$
a.e.\ $[\mu_w]$. Conversely, if $\efw(\vartheta^2) \Le
1$ a.e.\ $[\mu_w]$, then by Lemma \ref{BL1} we have
   \begin{align*}
\efw(\vartheta f)^2 &\Le \efw(\vartheta^2)\efw(f^2)
\\
& \Le \efw(f^2) \text{ a.e.\ $[\mu_w]$}, \quad
\text{$\ascr$-measurable } f\colon X \to \rbb_+.
   \end{align*}
This completes the proof of the equivalence
(ii)$\Leftrightarrow$(iii).

The equivalence (iii)$\Leftrightarrow$(iv) is a direct
consequence of \eqref{dyszcz} and \eqref{B6}.
   \end{proof}
   \begin{cor} \label{hypon2}
Suppose \eqref{stand2} holds and $\cfw$ is hyponormal.
Then the following assertions are valid{\em :}
   \begin{enumerate}
   \item[(i)] $\efw(\vartheta^{n+1})^2 \Le
\efw(\vartheta^{2n})$ a.e.\ $[\mu_w]$ for all
$n\in\zbb$,
   \item[(ii)] $\efw(\vartheta) \Le 1$ a.e.\ $[\mu_w]$,
   \item[(iii)] $\efw(\frac{1}{\vartheta^2}) \Ge 1$ a.e.\
$[\mu_w]$,
   \end{enumerate}
where $\vartheta\colon X \to \rbb_+$ is an
$\ascr$-measurable function such that $\vartheta=
\sqrt{\frac{\hfw \circ \phi}{\hfw}}$ a.e.\ $[\mu_w]$.
   \end{cor}
   \begin{proof}
Substituting $f=\vartheta^n$ into Theorem
\ref{hypon}(ii), we obtain (i). In turn, substituting
$n=0$ and $n=-1$ into (i) gives (ii) and (iii),
respectively.
   \end{proof}
   \section{\label{Sec5.2}Cohyponormality}
   Before characterizing cohyponormality of weighted
composition operators (see Theorem \ref{cohypon-m}),
we prove a few auxiliary lemmas. The first two provide
various descriptions of the range of the conditional
expectation $\efw$, where $\efw$ is regarded either as
a mapping on the set of $\rbop$-valued
$\ascr$-measurable functions (see Lemma \ref{peter})
or as an operator in the Hilbert space $L^2(\mu_w)$
(see Lemma \ref{roce}).
   \begin{lem} \label{peter}
Suppose \eqref{stand2} holds and $\cfw$ is densely
defined. Let $f\colon X \to \rbop$ be an
$\ascr$-measurable function such that $f<\infty$ a.e.\
$[\mu_w]$ and $\efw(f)=f$ a.e.\ $[\mu_w]$. Then there
exists an $\ascr$-measurable function $g\colon X \to
\rbb_+$ such that $f=g \circ \phi$ a.e.\ $[\mu_w]$.
   \end{lem}
   \begin{proof}
Since $\efw(f)=f$ a.e.\ $[\mu_w]$ and $0 \Le f<\infty$
a.e.\ $[\mu_w]$, there exists a
$\phi^{-1}(\ascr)$-measurable function $\tilde f
\colon X \to \rbb_+$ such that $\efw(f)=\tilde f$
a.e.\ $[\mu_w]$. By \eqref{Il-Zen2}, there exists an
$\ascr$-measurable function $g\colon X \to \rbb_+$
such that $\tilde f = g \circ \phi$, which completes
the proof.
   \end{proof}
   \begin{lem} \label{roce}
Suppose \eqref{stand2} holds and $\cfw$ is densely
defined. Then the following assertions are valid{\em
:}
   \begin{enumerate}
   \item[(i)] $\hfw\D \mu=\D \mu_w\circ \phi^{-1}$,
   \item[(ii)] if $f \in L^2(\mu_w)$, then
$f\in \ob{\efw}$ if and only if there exists an
$\ascr$-measurable function $g\colon X \to \cbb$ such
that $f=g\circ \phi$ a.e.\ $[\mu_w]$,
   \item[(iii)] the mapping
$V\colon L^2(\hfw\D \mu) \ni g\mapsto g\circ \phi \in
L^2(\mu_w)$ is well-defined and it is a linear
isometry such that
   \begin{align} \label{2xob}
\ob{\efw} = \ob{V} =
L^2(\mu_w|_{\phi^{-1}(\ascr)^{\mu_w}}),
   \end{align}
where $\phi^{-1}(\ascr)^{\mu_w}$ is the relative
$\mu_w$-completion of $\phi^{-1}(\ascr)$ $($see Sect.\
{\em \ref{sec-2.1}}$)$,
   \item[(iv)] if $f \in \ob{\efw}$ and $f\Ge 0$
a.e.\ $[\mu_w]$, then there exists $g\in L^2(\hfw\D\mu)$
such that $f=g\circ \phi$ a.e.\ $[\mu_w]$ and $g(x)\Ge
0$ for all $x\in X$,
   \item[(v)]  if $g \in L^2(\hfw\D\mu)$ and $g\Ge 0$
a.e.\ $[\hfw\D\mu]$, then $g\circ \phi\in \ob{\efw}$
and $g\circ \phi \Ge 0$ a.e.\ $[\mu_w]$.
   \end{enumerate}
   \end{lem}
   \begin{proof}
(i) Apply \eqref{l1}.

(ii) Apply \eqref{Il-Zen2} and \eqref{apme}.

(iii) Since
   \begin{align} \label{Il-Zen}
\int_X |g \circ \phi|^2 \D \mu_w \overset{\eqref{l2}}=
\int_X |g|^2 \hfw \D \mu, \quad
\text{$\ascr$-measurable $g\colon X \to \cbb$,}
   \end{align}
we deduce that the mapping $V$ is well-defined and it
is a linear isometry. Applying (ii) and
\eqref{Il-Zen}, we deduce that $\ob{\efw} = \ob{V}$.
In turn, employing (ii) and \cite[Lemma
13.3]{b-j-j-sC} with $\mu_w$ in place of $\mu$, we
conclude that $\ob{\efw} =
L^2(\mu_w|_{\phi^{-1}(\ascr)^{\mu_w}})$.

(iv) By (ii), $f=V\tilde g$ for some $\tilde g\in
L^2(\hfw\D\mu)$. Since, by Proposition \ref{lemS2},
the measure $\mu_w\circ \phi^{-1}$ is $\sigma$-finite,
there exists a sequence $\{X_n\}_{n=1}^{\infty}
\subseteq \ascr$ such that $X_n\nearrow X$ as $n\to
\infty$ and $\mu_w\circ \phi^{-1}(X_n)< \infty$ for
all $n\in \nbb$. Fix $n\in \nbb$ and take $\varDelta
\in \ascr$ such that $\varDelta \subseteq X_n$. Then
$\chi_{\phi^{-1}(\varDelta)}\in L^2(\mu_w)$ and
   \begin{align} \label{ptica1}
0 \Le \int_X \chi_{\phi^{-1}(\varDelta)} f\D \mu_w =
\int_X \chi_{\varDelta}\circ \phi \cdot \tilde g \circ
\phi \D \mu_w \overset{\eqref{l2}}= \int_{\varDelta}
\tilde g \, \hfw \D \mu.
   \end{align}
Since $\chi_{\varDelta}\circ \phi \cdot \tilde g \circ
\phi \in L^1(\mu_w)$ as the product of two functions
in $L^2(\mu_w)$, we deduce from \eqref{ptica1}, by
using \cite[Theorem 1.6.11]{Ash}, that $\tilde g\Ge 0$
a.e.\ $[\hfw\D\mu]$ on $X_n$. Knowing that
$X_n\nearrow X$ as $n\to \infty$, we deduce that
$\tilde g\Ge 0$ a.e.\ $[\hfw\D\mu]$. Let now $g\colon
X \to \rbb_+$ be such that $\tilde g=g$ a.e.\
$[\hfw\D\mu]$. Then, by Lemma \ref{nuklear}, $f= g
\circ \phi$ a.e.\ $[\mu_w]$. Clearly, $g\in
L^2(\hfw\D\mu)$.

(v) By assumptions, there exists $\tilde g\colon X \to
\rbb_+$ such that $g=\tilde g$ a.e.\ $[\hfw\D\mu]$. It
follows from Lemma \ref{nuklear} that $g\circ
\phi=\tilde g\circ \phi$ a.e.\ $[\mu_w]$. This implies
that $g\circ \phi\Ge 0$ a.e.\ $[\mu_w]$ and by (iii),
$g\circ \phi\in \ob{\efw}$.
   \end{proof}
   The next two lemmas provide several equivalent
variants of the conditions \mbox{(ii-a)} and
\mbox{(ii-b)} that appear in the characterization of
cohyponormality of weighted composition operators
given in Theorem \ref{cohypon-m}.
   \begin{lem} \label{cohypon1}
Suppose \eqref{stand2} holds, $\cfw$ is densely
defined and $\varOmega \in \ascr$. Then the following
conditions are equivalent{\em :}
   \begin{enumerate}
   \item[(i)] $\chi_{\varOmega} \cdot L^2(\mu_w) \subseteq
\ob{\efw}$,
   \item[(ii)] $\jd{\efw} \subseteq \chi_{X\setminus \varOmega}
\cdot L^2(\mu_w)$,
   \item[(iii)] $\big\{\varDelta \in \ascr\colon \varDelta
\subseteq \varOmega \text{ a.e.\ $[\mu_w]$}\big\}
\subseteq \phi^{-1}(\ascr)^{\mu_w}$,
   \item[(iv)] for every $\varDelta \in \ascr$ such
that $\varDelta \subseteq \varOmega$ a.e.\ $[\mu_w]$,
there exists $\tilde \varDelta \in \ascr$ such that
$\mu_w(\varDelta \vartriangle \phi^{-1}(\tilde
\varDelta))=0$,
   \item[(v)] for every $\ascr$-measurable function
$f\colon X \to \rbb_+$ such that $f=\chi_{\varOmega}
\cdot f$ a.e.\ $[\mu_w]$, there exists an
$\ascr$-measurable function $g\colon X \to \rbb_+$
such that $f=g\circ \phi$ a.e.\ $[\mu_w]$.
   \end{enumerate}
Moreover, if {\em (i)} holds, then there exists
$\tilde \varOmega \in \ascr$ such that
$\chi_{\varOmega} = \chi_{\phi^{-1}(\tilde
\varOmega)}$ a.e.\ $[\mu_w]$.
   \end{lem}
   \begin{proof}
(i)$\Leftrightarrow$(ii) Take adjoints and use the
fact that $\efw$ is selfadjoint as an operator in
$L^2(\mu_w)$ (cf.\ Theorem \ref{Ath}).

(i)$\Rightarrow$(iii) Let $\varDelta \in \ascr$ be
such that $\varDelta \subseteq \varOmega$ a.e.\
$[\mu_w]$. Without loss of generality we can assume
that $\varDelta \subseteq \varOmega$. Since $\mu_w$ is
$\sigma$-finite, there exits a sequence
$\{X_n\}_{n=1}^{\infty} \subseteq \ascr$ such that
$X_n \nearrow X$ as $n\to \infty$ and $\mu_w(X_n)<
\infty$ for all $n\in \nbb$. Then $\{\chi_{\varDelta
\cap X_n}\}_{n=1}^{\infty} \subseteq \chi_{\varOmega}
\cdot L^2(\mu_w)$. By (i) and Lemma \ref{roce}(iii),
$\{\chi_{\varDelta \cap X_n}\}_{n=1}^{\infty}\subseteq
L^2(\mu_w|_{\phi^{-1}(\ascr)^{\mu_w}})$ and thus
$\varDelta \cap X_n\in \phi^{-1}(\ascr)^{\mu_w}$ for
all $n\in\nbb$. Since $\varDelta \cap X_n\nearrow
\varDelta$ as $n\to\infty$, we get $\varDelta \in
\phi^{-1}(\ascr)^{\mu_w}$.

(iii)$\Leftrightarrow$(iv) This follows from
\eqref{complascr}.

(iv)$\Rightarrow$(v) Lat $s\colon X \to \rbb_+$ be an
$\ascr$-measurable simple function such that
$s=\chi_{\varOmega} \cdot s$ a.e.\ $[\mu_w]$. If $s$
is of the form $s=\sum_{j=1}^n \lambda_j \chi_{E_j}$,
where $n\in \nbb$, $\{\lambda_j\}_{j=1}^n \subseteq
\rbb_+$ and $\{E_j\}_{j=1}^n$ are pairwise disjoint
$\ascr$-measurable sets, then $\chi_{\varOmega} \cdot
s=\sum_{j=1}^n \lambda_j \chi_{E_j\cap \varOmega}$. By
(iv), there exists a sequence $\{F_j\}_{j=1}^n
\subseteq \ascr$ such that $\chi_{E_j\cap \varOmega} =
\chi_{\phi^{-1}(F_j)}$ a.e.\ $[\mu_w]$ for all $j \in
\{1, \ldots,n\}$. This implies that $s=
\chi_{\varOmega} \cdot s = (\sum_{j=1}^n \lambda_j
\chi_{F_j}) \circ \phi$ a.e.\ $[\mu_w]$, which shows
that (v) is valid for $\ascr$-measurable simple
functions.

Let now $f\colon X \to \rbb_+$ be an
$\ascr$-measurable function such that
$f=\chi_{\varOmega} \cdot f$ a.e.\ $[\mu_w]$. Then, by
\cite[Theorem 1.17]{Rud}, there exists a sequence
$\{s_n\}_{n=1}^{\infty}$ of $\ascr$-measurable simple
functions such that $0 \Le s_n \nearrow f$ as $n\to
\infty$. Clearly, $s_n = s_n \cdot \chi_{\varOmega}$
a.e.\ $[\mu_w]$ for all $n\in\nbb$. Hence, by the
previous paragraph, there exists a sequence $\{\tilde
s_n\}_{n=1}^{\infty}$ of $\ascr$-measurable simple
functions $\tilde s_n \colon X \to \rbb_+$ such that
$s_n = \tilde s_n \circ \phi$ a.e.\ $[\mu_w]$ for all
$n\in \nbb$. Thus there exists $Y\in \ascr$ such that
$\mu_w(X\setminus Y)=0$ and
   \begin{align} \label{pass1}
s_n = \tilde s_n \circ \phi \text{ on } Y, \quad n\in
\nbb.
   \end{align}
Set
   \begin{align*}
L=\{x\in X\colon \text{ the limit $\lim_{n\to \infty}
\tilde s_n (x)$ exists in $\rbb_+$}\}.
   \end{align*}
Let $g\colon X \to \rbb_+$ be a function defined by
$g(x)=\lim_{n\to \infty} \tilde s_n (x)$ for $x\in L$
and $g(x)=0$ for $x\in X \setminus L$. Note that $L\in
\ascr$ and $g$ is $\ascr$-measurable. It follows from
\eqref{pass1} that $\phi(Y) \subseteq L$. Hence, by
passing to the limit with $n\to\infty$ in
\eqref{pass1}, we see that $f=g\circ \phi$ on $Y$, so
$f=g\circ \phi$ a.e.\ $[\mu_w]$.

(v)$\Rightarrow$(i) Decomposing any function belonging
to $\chi_{\varOmega} \cdot L^2(\mu_w)$ into the linear
combination of four functions in $\chi_{\varOmega}
\cdot L_+^2(\mu_w)$ and using Lemma \ref{roce}(ii)
yields (i).

Now we prove the ``moreover'' part. Applying (iv) to
$\varDelta=\varOmega$, we get $\tilde \varOmega \in
\ascr$ such that $\mu_w(\varOmega \vartriangle
\phi^{-1}(\tilde \varOmega))=0$, or equivalently that
$\chi_{\varOmega} = \chi_{\phi^{-1}(\tilde
\varOmega)}$ a.e.\ $[\mu_w]$. This completes the
proof.
   \end{proof}
   \begin{lem} \label{cohypon0}
Suppose \eqref{stand2} holds. Then the following
conditions are equivalent{\em :}
   \begin{enumerate}
   \item[(i)] $\hfw = 0$ on $\{w=0\}$  a.e.\ $[\mu]$,
   \item[(ii)] $\mu_w(\phi^{-1}(\{w=0\}))=0$,
   \item[(iii)] $w\circ \phi\neq 0$  a.e.\ $[\mu_w]$.
   \end{enumerate}
   \end{lem}
   \begin{proof}
That (i) and (ii) are equivalent follows from the
following equality
   \begin{align*}
\mu_w(\phi^{-1}(\{w=0\})) \overset{\eqref{l1}} =
\int_X \chi_{\{w=0\}} \hfw \D \mu.
   \end{align*}
The conditions (ii) and (iii) are easily seen to be
equivalent.
   \end{proof}
   The following is the first step in the proof of
Theorem \ref{cohypon-m}.
   \begin{lem} \label{cohyponA}
Suppose \eqref{stand2} holds and $\cfw$ is densely
defined. Then the following conditions are
equivalent{\em :}
   \begin{enumerate}
   \item[(i)] $\cfw$ is cohyponormal,
   \item[(ii)] $\hfw = 0$ on $\{w=0\}$  a.e.\ $[\mu]$
and
   \begin{align} \label{cohypon4}
\int_X \theta^2 \, |g|^2 \D\mu_w \Le \int_X
|\efw(g)|^2 \D\mu_w, \quad g \in L^2(\mu_w),
   \end{align}
   \item[(iii)] $\hfw = 0$ on $\{w=0\}$  a.e.\ $[\mu]$
and
   \begin{align} \label{cohypon7}
\efw(\theta^2 \, |g|^2) \Le |\efw(g)|^2 \text{ a.e.\
$[\mu_w]$}, \quad g\in L^2(\mu_w),
   \end{align}
   \end{enumerate}
where $\theta\colon X \to \rbb_+$ is an
$\ascr$-measurable function such that $\theta=
\sqrt{\frac{\hfw}{\hfw \circ \phi}}$ a.e.\ $[\mu_w]$.
   \end{lem}
   \begin{proof}
It follows from Lemma \ref{lemS6} and Proposition
\ref{lemS2} that $\theta$ with the required properties
exists. By \eqref{fw} and Proposition \ref{adj}, the
operator $\cfw$ is cohyponormal if and only if for
every $f \in L^2(\mu)$ such that
$\hfw\cdot\efw(f_w)\circ \phi^{-1} \in L^2(\mu)$,
   \begin{align*}
\int_X |f\circ \phi|^2 |w|^2 \D\mu \Le \int_X \hfw^2
|\efw(f_w)\circ \phi^{-1}|^2 \D\mu.
   \end{align*}
Hence, by \eqref{fw}, $\cfw$ is cohyponormal if and
only if the following holds
   \begin{align} \label{cohypon1c}
\int_X |f\circ \phi|^2 |w|^2 \D\mu \Le \int_X \hfw^2
|\efw(f_w)\circ \phi^{-1}|^2 \D\mu, \quad f\in
L^2(\mu).
   \end{align}
Applying \eqref{l2} and \eqref{fifi}, we see that
\eqref{cohypon1c} is equivalent to
   \begin{align} \label{cohypon2}
\int_X |f|^2 \hfw \D\mu \Le \int_X \hfw\circ \phi
\cdot |\efw(f_w)|^2 \D\mu_w, \quad f\in L^2(\mu).
   \end{align}
Now we show that if \eqref{cohypon2} holds, then
\mbox{(ii-a)} is satisfied. Indeed, since $\mu$ is
$\sigma$-finite, there exists $\{X_n\}_{n=1}^{\infty}
\subseteq \ascr$ such that $X_n\nearrow X$ as
$n\to\infty$ and $\mu(X_n)<\infty$ for all $n\in
\nbb$. Substituting $f=\chi_{X_n\cap \{w=0\}}$ (which
is in $L^2(\mu)$) into \eqref{cohypon2}, we deduce
that $\int_{X_n\cap \{w=0\}} \hfw \D\mu = 0$ for all
$n\in \nbb$, which implies that $\int_{\{w=0\}} \hfw
\D\mu = 0$. As a consequence, \mbox{(ii-a)} is
satisfied. This means that we may assume that the
condition \mbox{(ii-a)} holds.

Making the substitution $f \leftrightsquigarrow f\cdot
w$ and using \mbox{(ii-a)} and the equality $f =
f\cdot \chi_{\{w\neq 0\}}$ a.e.\ $[\mu_w]$, we deduce
that \eqref{cohypon2} is equivalent to
   \begin{align} \label{cohypon3}
\int_X |f|^2 \hfw \D\mu_w \Le \int_X \hfw\circ \phi
\cdot |\efw(f)|^2 \D\mu_w, \quad f\in L^2(\mu_w).
   \end{align}
Summarizing, we have shown that (i) is equivalent to
\eqref{cohypon3} (still under the assumption that
$\hfw = 0$ on $\{w=0\}$ a.e.\ $[\mu]$).

\eqref{cohypon3}$\Rightarrow$\eqref{cohypon4}
Substituting $f=g/\sqrt{\hfw\circ \phi}$ into
\eqref{cohypon3} and using \eqref{B10}, we get
   \begin{align} \label{cohypon5}
\int_X \theta^2 \, |g|^2 \D\mu_w \Le \int_X
|\efw(g)|^2 \D\mu_w, \quad g\in
L^2\bigg(\Big(1+\frac{1}{\hfw\circ \phi}\Big)
\D\mu_w\bigg).
   \end{align}
Take $g\in L^2(\mu_w)$ and set $g_n=\chi_{\{\hfw\circ
\phi \Ge \frac{1}{n}\}}g$ for $n\in\nbb$. Then
   \begin{align*}
\{g_n\}_{n=1}^{\infty}\subseteq
L^2\bigg(\Big(1+\frac{1}{\hfw\circ \phi}\Big)
\D\mu_w\bigg).
   \end{align*}
Since $\hfw\circ \phi>0$ a.e.\ $[\mu_w]$, we deduce
that the sequence $\{\chi_{X\setminus \{\hfw\circ \phi
\Ge \frac{1}{n}\}}\}_{n=1}^{\infty}$ converges to $0$
a.e.\ $[\mu_w]$, so by Lebesgue's dominated
convergence theorem $g_n\to g$ in $L^2(\mu_w)$ as
$n\to\infty$. By the continuity of $\efw$ (see Theorem
\ref{Ath}), we see that $\efw(g_n)\to \efw(g)$ in
$L^2(\mu_w)$ as $n\to\infty$. Hence, using Fatou's
lemma, we obtain
   \allowdisplaybreaks
   \begin{align*}
\int_X \theta^2 \, |g|^2 \D\mu_w & \Le
\liminf_{n\to\infty}\int_X \theta^2 \, |g_n|^2 \D\mu_w
   \\
& \hspace{-.8ex}\overset{\eqref{cohypon5}}\Le
\lim_{n\to\infty} \int_X |\efw(g_n)|^2 \D\mu_w
   \\
&= \int_X |\efw(g)|^2 \D\mu_w,
   \end{align*}
which yields \eqref{cohypon4}.

\eqref{cohypon4}$\Rightarrow$\eqref{cohypon7} For
this, take $g\in L^2(\mu_w)$. Since
$\chi_{\phi^{-1}(\varDelta)} \cdot g \in L^2(\mu_w)$
for all $\varDelta \in \ascr$, an application of
\eqref{B10} yields
   \allowdisplaybreaks
   \begin{align*}
\int_{\phi^{-1}(\varDelta)} \efw(\theta^2 \, |g|^2)
\D\mu_w & = \int_{\phi^{-1}(\varDelta)} \theta^2 \,
|g|^2 \D\mu_w
   \\
& \hspace{-.8ex} \overset{\eqref{cohypon4}} \Le \int_X
|\efw(\chi_{\phi^{-1}(\varDelta)} g)|^2 \D\mu_w
   \\
&=\int_{\phi^{-1}(\varDelta)} |\efw(g)|^2 \D\mu_w,
\quad \varDelta \in \ascr.
   \end{align*}
Since $\mu_w$ is $\sigma$-finite, \eqref{cohypon7}
follows from the above inequalities and \cite[Theorem
1.6.11]{Ash}.

\eqref{cohypon7}$\Rightarrow$\eqref{cohypon3} Take
$f\in L^2(\mu_w)$. Set $X_n=\phi^{-1}(\{\hfw \Le
n\})\in\phi^{-1}(\ascr)$ for $n\in \nbb$. Substituting
$g=\chi_{X_n} \sqrt{\hfw\circ \phi}\, f\in L^2(\mu_w)$
into \eqref{cohypon7} and using \eqref{B6} and
\eqref{B10}, we get
   \allowdisplaybreaks
   \begin{align*}
\chi_{X_n} \efw(|f|^2\hfw) & = \efw(\chi_{X_n} |f|^2
\hfw)
   \\
&\Le |\efw(\chi_{X_n} \sqrt{\hfw\circ \phi} \,
f)|^2
   \\
&= \chi_{X_n} (\hfw\circ \phi) \, |\efw(f)|^2 \text{
a.e.\ $[\mu_w]$}, \quad n\in \nbb.
   \end{align*}
Since $\hfw\circ \phi<\infty$ a.e.\ $[\mu_w]$ and $X_n
\nearrow \{\hfw\circ \phi<\infty\}$ as $n\to\infty$,
we deduce that
   \begin{align*}
\efw(\hfw |f|^2) \Le \hfw\circ \phi \, |\efw(f)|^2
\text{ a.e.\ $[\mu_w]$.}
   \end{align*}
This in turn implies that
   \begin{align*}
\int_X |f|^2 \hfw \D\mu_w & = \int_X \efw(\hfw |f|^2)
\D\mu_w
   \\
& \Le \int_X \hfw\circ \phi \cdot |\efw(f)|^2 \D\mu_w,
\quad f\in L^2(\mu_w),
   \end{align*}
which gives \eqref{cohypon3}.
   \end{proof}
A careful inspection of the proof of Lemma
\ref{cohyponA} reveals that under the assumptions of
this lemma the inequalities \eqref{cohypon4},
\eqref{cohypon7} and \eqref{cohypon3} are equivalent
without assuming that $\hfw = 0$ on $\{w=0\}$ a.e.\
$[\mu]$.

   Now we are in a position to prove the
aforementioned characterization of cohyponormality of
weighted composition operators. Combining it with
Lemmas \ref{cohypon1} and \ref{cohypon0} we easily
obtain other characterizations.
   \begin{thm} \label{cohypon-m}
Suppose \eqref{stand2} holds and $\cfw$ is densely
defined. Then the following statements are
equivalent{\em :}
   \begin{enumerate}
   \item[(i)] $\cfw$ is cohyponormal,
   \item[(ii)] the following three conditions hold{\em :}
   \begin{enumerate}
   \item[(ii-a)] $\hfw = 0$ on $\{w=0\}$  a.e.\ $[\mu]$,
   \item[(ii-b)] $\chi_{\{\hfw > 0\}} \cdot L^2(\mu_w) \subseteq
\ob{\efw}$,
   \item[(ii-c)] $\hfw \Le \hfw \circ \phi$ a.e.\
$[\mu_w]$.
   \end{enumerate}
   \end{enumerate}
Moreover, if $\cfw$ is cohyponormal, then
   \begin{enumerate}
   \item[(iii)] $\efw(\hfw) = \hfw$ a.e.\ $[\mu_w]$,
   \item[(iv)] $M_{\theta} \in \ogr{L^2(\mu_w)}$,
$M_{\theta}$ is a contraction, $\ob{\efw}$ reduces
$M_{\theta}$ and
   \begin{align} \label{cohypon11}
M_{\theta} = M_{\theta}|_{\ob{\efw}} \oplus
0|_{\jd{\efw}}.
   \end{align}
   \end{enumerate}
   \end{thm}
   \begin{proof}
(i)$\Rightarrow$(ii) By Lemma \ref{cohyponA}(ii),
\mbox{(ii-a)} holds and
   \begin{align}  \label{cohypon8}
\|M_{\theta} \, g\|_{\mu_w} \Le \|\efw g\|_{\mu_w},
\quad g\in L^2(\mu_w),
   \end{align}
where $\|\cdot\|_{\mu_w}$ stands for the norm of the
Hilbert space $L^2(\mu_w)$. Since $\efw\in
\ogr{L^2(\mu_w)}$ is a contraction (see Theorem
\ref{Ath}(ii)), we have
   \begin{align*}
\|M_{\theta} \, g\|_{\mu_w} \Le \|g\|_{\mu_w}, \quad
g\in L^2(\mu_w),
   \end{align*}
which implies that $M_{\theta}\in \ogr{L^2(\mu_w)}$
and $\|M_{\theta}\|_{\mu_w} \Le 1$, or equivalently
that $\theta \Le 1$ a.e.\ $[\mu_w]$. This means that
\mbox{(ii-c)} holds.

Since the range of $\efw$ is closed (see Theorem
\ref{Ath}(iii)), it follows from \eqref{cohypon8} that
there exists a linear contraction $\tilde T \colon
\ob{\efw} \to \overline{\ob{M_{\theta}}}$ such that
   \begin{align} \label{cohypon6}
\tilde T \, \efw g = M_{\theta} \, g, \quad g\in
L^2(\mu_w).
   \end{align}
Define $T\in \ogr{L^2(\mu_w)}$ by $Tg=\tilde T \, \efw
g$ for $g\in L^2(\mu_w)$. Since $\efw$ is an
orthogonal projection in $L^2(\mu_w)$, we deduce from
\eqref{cohypon6} that $\|T\|_{\mu_w}\Le 1$ and
   \begin{align*}
T \, \efw = M_{\theta}.
   \end{align*}
Since both operators $\efw$ and $M_{\theta}$ are
selfadjoint, we obtain
   \begin{align} \label{cohypon9}
M_{\theta} = \efw T^*.
   \end{align}
Since
   \begin{align} \label{cohypon10}
\overline{\ob{M_{\theta}}}=\chi_{\{\theta > 0\}} \cdot
L^2(\mu_w)=\chi_{\{\hfw > 0\}} \cdot L^2(\mu_w),
   \end{align}
the equality \eqref{cohypon9} implies \mbox{(ii-b)}.

(ii)$\Rightarrow$(i) It follows from \mbox{(ii-c)}
that $M_{\theta} \in \ogr{L^2(\mu_w)}$ and
$\|M_{\theta}\|_{\mu_w} \Le 1$. By \mbox{(ii-b)} and
\eqref{cohypon10}, $\jd{\efw} \subseteq
\jd{M_{\theta}}$. Set $\efw^{\perp}:=I-\efw$. Since
$\efw$ is an orthogonal projection in $L^2(\mu_w)$, we
see that $\ob{\efw^{\perp}} \subseteq
\jd{M_{\theta}}$. Hence, we have
   \allowdisplaybreaks
   \begin{align*}
\|M_{\theta} \, g\|_{\mu_w} &= \|M_{\theta}\, (\efw g
\oplus \efw^{\perp} g)\|_{\mu_w}
   \\
& = \|M_{\theta} \, \efw g\|_{\mu_w}
   \\
&\Le \|\efw g\|_{\mu_w}, \quad g \in L^2(\mu_w),
   \end{align*}
which is equivalent to \eqref{cohypon4}. Applying
Lemma \ref{cohyponA} yields (i).

To prove the ``moreover'' part, assume that $\cfw$ is
cohyponormal. Since $\ob{\efw^{\perp}} \subseteq
\jd{M_{\theta}}$ and $\ob{M_{\theta}} \subseteq
\ob{\efw}$, we get
   \begin{align*}
M_{\theta} \, \efw g = M_{\theta} \, g = \efw
M_{\theta} \, g, \quad g \in L^2(\mu_w).
   \end{align*}
This implies that $\ob{\efw}$ reduces $M_{\theta}$ and
\eqref{cohypon11} is valid. Summarizing, we have shown
that (iv) holds.

By Proposition \ref{lemS2}, the measure $\hfw\D
\mu=\D\mu_w\circ \phi^{-1}$ is $\sigma$-finite, and
thus there exists $f\in L^2(\hfw\D\mu)$ such that
$f(x)\in (0,\infty)$ for all $x\in X$. By Lemma
\ref{roce}(iii), $g:=f\circ \phi \in \ob{\efw}$. Then,
since $\ob{\efw}$ reduces $M_{\theta}$, $\theta g \in
\ob{\efw}$. Applying Lemma \ref{roce}(iv), we find
$\eta\in L^2(\hfw\D\mu)$ such that $\eta(x)\Ge 0$ for
all $x\in X$ and $\theta g=\eta\circ \phi$ a.e.\
$[\mu_w]$. Therefore, we have $\sqrt{\hfw} \,
g=\sqrt{\hfw\circ \phi} \cdot \eta\circ \phi$ a.e.\
$[\mu_w]$. This implies that
   \begin{align*}
\hfw = \bigg(\frac{\hfw\cdot \eta^2}{f^2}\bigg) \circ
\phi \text{ a.e.\ $[\mu_w]$,}
   \end{align*}
which completes the proof of the ``moreover'' part.
   \end{proof}
   \begin{cor} \label{cohyponB}
Suppose \eqref{stand2} holds, $\cfw$ is densely
defined and $w\neq 0$ a.e.\ $[\mu]$. Then the
following conditions are equivalent{\em :}
   \begin{enumerate}
   \item[(i)] $\cfw$ is cohyponormal,
   \item[(ii)] $\chi_{\{\hfw > 0\}} \cdot L^2(\mu_w) \subseteq
\ob{\efw}$ and $\hfw \Le \hfw \circ \phi$ a.e.\
$[\mu_w]$.
   \end{enumerate}
   \end{cor}
   \section{\label{Sec5.3}Normality and formal normality}
The following lemma, which is a direct consequence of
Lemma \ref{cohypon1}, provides a few equivalent
variants of the condition \mbox{(ii-b)} that appears
in the characterization of normality of weighted
composition operators given in Theorem \ref{normal-m}.
   \begin{lem} \label{normal1}
Suppose \eqref{stand2} holds and $\cfw$ is densely
defined. Then the following conditions are
equivalent{\em :}
   \begin{enumerate}
   \item[(i)] $\ob{\efw}=L^2(\mu_w)$,
   \item[(ii)] $\jd{\efw} = \{0\}$,
   \item[(iii)] $\ascr = \phi^{-1}(\ascr)^{\mu_w}$,
   \item[(iv)] for every $\varDelta \in \ascr$,
there exists $\tilde \varDelta \in \ascr$ such that
$\mu_w(\varDelta \vartriangle \phi^{-1}(\tilde
\varDelta))=0$,
   \item[(v)] for every $\ascr$-measurable function
$f\colon X \to \rbb_+$, there exists an
$\ascr$-measura\-ble function $g\colon X \to \rbb_+$
such that $f=g\circ \phi$ a.e.\ $[\mu_w]$.
   \end{enumerate}
   \end{lem}
   \begin{thm} \label{normal-m}
Suppose \eqref{stand2} holds and $\cfw$ is densely
defined. Then the following statements are
equivalent{\em :}
   \begin{enumerate}
   \item[(i)] $\cfw$ is normal,
   \item[(ii)] the following three conditions are
satisfied{\em :}
   \begin{enumerate}
   \item[(ii-a)] $\hfw = 0$ on $\{w=0\}$  a.e.\ $[\mu]$,
   \item[(ii-b)] $\ob{\efw} = L^2(\mu_w)$,
   \item[(ii-c)] $\hfw = \hfw \circ \phi$ a.e.\ $[\mu_w]$.
   \end{enumerate}
   \end{enumerate}
Moreover, if $\cfw$ is normal, then $\{\hfw > 0\} =
\{w\neq 0\}$ a.e.\ $[\mu]$.
   \end{thm}
   \begin{proof}
(i)$\Rightarrow$(ii) First, we recall that normal
operators are simultaneously hyponormal and
cohyponormal. Thus the condition \mbox{(ii-a)} follows
from Theorem \ref{cohypon-m}. In view of Corollary
\ref{hipinj}, $\hfw>0$ a.e.\ $[\mu_w]$, so by Theorem
\ref{cohypon-m} the condition \mbox{(ii-b)} is
satisfied. Using Theorem \ref{cohypon-m} again, we see
that $\hfw \Le \hfw \circ \phi$ a.e.\ $[\mu_w]$ and
$\efw(\hfw) = \hfw$ a.e.\ $[\mu_w]$. On the other
hand, by Theorem \ref{hypon} we have
   \begin{align*}
\frac{\hfw\circ \phi}{\hfw} =
\efw\bigg(\frac{\hfw\circ \phi}{\hfw}\bigg) \Le 1
\text{ a.e.\ $[\mu_w]$.}
   \end{align*}
As a consequence, the condition \mbox{(ii-c)} holds.
(Note also that the condition \mbox{(ii-c)} is a
direct consequence of Theorem \ref{quain}.)

(ii)$\Rightarrow$(i) It follows from Theorem
\ref{cohypon-m} that $\cfw$ is cohyponormal. By
\mbox{(ii-c)} and Lemma \ref{lemS6}, $\hfw > 0$ a.e.\
$[\mu_w]$. Since $\efw\big(\frac{\hfw\circ
\phi}{\hfw}\big) = 1$ a.e.\ $[\mu_w]$, we infer from
Theorem \ref{hypon} that $\cfw$ is hyponormal (this
can also be derived from the quasinormality of $\cfw$
via \mbox{(ii-c)} and Theorem \ref{quain}). Knowing
that $\cfw$ is closed, we conclude that $\cfw$ is
normal.

To prove the ``moreover'' part, assume that $\cfw$ is
normal. Then, by (ii), $\hfw = 0$ on $\{w=0\}$ a.e.\
$[\mu]$. On the other hand, Corollary \ref{hipinj}
guarantees that $\hfw > 0$ on $\{w\neq 0\}$ a.e.\
$[\mu]$, which implies that $\{\hfw > 0\} = \{w\neq
0\}$ a.e.\ $[\mu]$.
   \end{proof}
   \begin{cor} \label{normal-B}
Suppose \eqref{stand2} holds and $\cfw$ is normal.
Then the following assertions are valid for every
$n\in \nbb$,
   \begin{enumerate}
   \item[(i)] $C_{\phi^n,\widehat w_n}$ is normal
and $\cfw^n=C_{\phi^n,\widehat w_n}$,
   \item[(ii)] $\hfwn{n}=\hfw^n$ a.e.\ $[\mu]$,
   \item[(iii)] $\hfw\circ \phi^n = \hfw$ a.e.\
$[\mu_w]$,
   \item[(iv)] $\{\widehat w_n \neq 0\} =
\{w\neq 0\} = \{\hfw > 0\}=\{\hfwn{n}>0\}$ a.e.\
$[\mu]$,
   \item[(v)] $\mu_w \ll \mu_{\widehat w_n}$ and
$\mu_{\widehat w_n} \ll \mu_w$.
   \end{enumerate}
   \end{cor}
   \begin{proof}
(i) Apply Lemma \ref{13-12-13}(v) and the fact that
powers of normal operators are normal (see
\cite[Proposition 4.22]{Schb}).

(ii) This is a consequence of Lemma \ref{lemS11}(iii).

(iv) Using the ``moreover'' part of Theorem
\ref{normal-m} and the assertions (i) and (ii), we get
the following equalities
   \begin{align*}
\{\widehat w_n \neq 0\} = \{\hfwn{n}>0\} = \{\hfw^n >
0\} = \{\hfw > 0\} = \{w\neq 0\} \text{ a.e.\
$[\mu]$.}
   \end{align*}

(v) This can be deduced from the assertion (iv).

(iii) It follows from the assertion (i) and Theorem
\ref{quain} that
   \begin{align*}
\hfwn{n}\circ \phi^n = \hfwn{n} \text{ a.e.\
$[\mu_{\widehat w_n}]$.}
   \end{align*}
By the assertion (ii) and Lemma \ref{nuklear}, this
implies that
   \begin{align*}
(\hfw \circ \phi^n)^n = \hfw^n \text{ a.e.\
$[\mu_{\widehat w_n}]$.}
   \end{align*}
Applying the assertion (v) and taking $n$th roots
completes the proof.
   \end{proof}
   \begin{cor} \label{normal-C}
Suppose \eqref{stand2} holds, $\cfw$ is densely
defined and $w\neq 0$ a.e.\ $[\mu]$. Then the
following conditions are equivalent{\em :}
   \begin{enumerate}
   \item[(i)] $\cfw$ is normal,
   \item[(ii)] $\ob{\efw} = L^2(\mu_w)$ and
$\hfw = \hfw \circ \phi$ a.e.\ $[\mu_w]$.
   \end{enumerate}
   \end{cor}
It was proved in \cite[Theorem 9.4]{b-j-j-sC} that
formally normal composition operators are
automatically normal. As shown below, the same is true
for weighted composition operators. The proof of this
fact is an adaptation of that given in \cite[Remark
9.5]{b-j-j-sC}.
   \begin{thm} \label{fn2n}
Suppose \eqref{stand2} holds and $\cfw$ is formally
normal. Then $\cfw$ is normal.
   \end{thm}
   \begin{proof}
Applying the polarization formula and Proposition
\ref{adj}, we obtain
   \allowdisplaybreaks
   \begin{align} \notag
\int_X f \bar g \, \hfw \D \mu &
\hspace{-.3ex}\overset{\eqref{l2}} = \is{\cfw f}{\cfw
g}
   \\ \notag
& = \is{\cfw^* f}{\cfw^* g}
   \\   \notag
& = \int_X (\efw(f_w) \circ \phi^{-1}) \,
\overline{\efw(g_w) \circ \phi^{-1}} \,\hfw^2 \D\mu
   \\ \notag
&\hspace{-.3ex}\overset{\eqref{l2}}= \int_X
(\efw(f_w)\circ \phi^{-1}) \circ \phi \cdot
\overline{(\efw(g_w)\circ \phi^{-1}) \circ \phi} \cdot
\hfw \circ \phi \D\mu_w,
   \\ \label{mordega}
&\hspace{-.75ex}\overset{\eqref{fifi}}= \int_X
\efw(f_w) \, \overline{\efw(g_w)} \; \hfw \circ \phi
\D\mu_w, \quad f,g\in \dz{\cfw}.
   \end{align}
Since $\mu$ is $\sigma$-finite, there exists a
sequence $\{X_n\}_{n=1}^{\infty} \subseteq \ascr$ such
that $X_n\nearrow X$ as $n\to \infty$ and $\mu(X_n) <
\infty$ for all $n\in \nbb$. Take $\varDelta \in
\ascr$. Knowing that $\hfw>0$ a.e.\ $[\mu_w]$, we can
define the sequences $\{Y_n\}_{n=1}^{\infty}\subseteq
\ascr$ and $\{f_n\}_{n=1}^{\infty},
\{g_n\}_{n=1}^{\infty} \subseteq L^2(\mu)$ by
   \begin{gather*}
   \begin{gathered}
Y_n = X_n \cap \{0 < |w|\Le n\} \cap \{\hfw \Ge 1/n\},
   \\
f_n= \chi_{\varDelta \cap Y_n} \frac{w}{\hfw} \text{
and } g_n= \chi_{\varDelta \cap Y_n} w,
   \end{gathered}
\quad n \in \nbb.
   \end{gather*}
Applying Proposition \ref{lemS1} we easily verify that
$\{f_n\}_{n=1}^{\infty}, \{g_n\}_{n=1}^{\infty}
\subseteq \dz{\cfw}$. Substituting $(f_n,g_n)$ into
\eqref{mordega} in place of $(f,g)$, we obtain
   \allowdisplaybreaks
   \begin{align*}
\mu_w(Y_n \cap \varDelta)&= \int_X
\efw(\hfw^{-1}\cdot\chi_{Y_n \cap \varDelta}) \,
\efw(\chi_{Y_n}) \, \hfw \circ \phi \D\mu_w
   \\
&\hspace{-2ex}\overset{\eqref{B3}}= \int_{Y_n \cap
\varDelta} \efw(\chi_{Y_n}) \frac{\hfw \circ
\phi}{\hfw} \D\mu_w, \quad n \in \nbb.
   \end{align*}
Using \eqref{B5} and Lebesgue's monotone convergence
theorem and the fact that $\chi_{Y_n \cap E} \nearrow
\chi_{E}$ a.e.\ $[\mu_w]$ as $n\to \infty$ for any
$E\in \ascr$, we deduce that
   \begin{align*}
\mu_w(\varDelta) = \int_{\varDelta} \frac{\hfw \circ
\phi}{\hfw} \D\mu_w, \quad \varDelta \in \ascr.
   \end{align*}
Since $\varDelta\in \ascr$ is arbitrary and $\mu_w$ is
$\sigma$-finite, we deduce from \cite[Theorem
1.6.11]{Ash} that $\hfw \circ \phi = \hfw$ a.e.\
$[\mu_w]$. By Proposition \ref{quain}, $\cfw$ is
quasinormal. Since quasinormal formally normal
operators are normal (cf.\ \cite[Corollary 4]{StSz2}),
the proof is complete.
   \end{proof}
   \section{\label{Sec5.4}Selfadjointness}
In this section, we concentrate on the study of
selfadjoint weighted composition operators. This
covers the cases of symmetric and positive weighted
composition operators due to the following
proposition.
   \begin{pro} \label{fn2nC}
Suppose \eqref{stand2} holds and $\cfw$ is densely
defined. If $\cfw$ is symmetric or positive, then
$\cfw$ is selfadjoint.
   \end{pro}
   \begin{proof}
Since densely defined positive complex Hilbert space
operators are symmetric, we can infer Proposition
\ref{fn2nC} from Theorem \ref{fn2n} by noting that
symmetric operators are formally normal and normal
symmetric operators are selfadjoint.
   \end{proof}
   The following lemma, which is of some independent
interest, will be one of the main tools in the proof
of the characterization of selfadjointness of weighted
composition operators given in Theorem \ref{self}.
   \begin{lem} \label{normsa}
Assume that $A$ is a normal operator in a complex
Hilbert space $\hh$. Then the the following two
conditions are equivalent{\em :}
   \begin{enumerate}
   \item[(i)] $A=A^*$,
   \item[(ii)] $A^2$ is selfadjoint and positive.
   \end{enumerate}
   \end{lem}
   \begin{proof}
(i)$\Rightarrow$(ii) This is clear due to the
well-known fact that powers of selfadjoint operators
are selfadjoint (see \cite[Proposition 4.22]{Schb}).

(ii)$\Rightarrow$(i) It follows from the spectral
mapping theorem (see \cite[Proposition 5.25]{Schb})
and the fact that the spectrum of a positive
selfadjoint operator is a subset of $\rbb_+$ (see
\cite[Theorem 5.7 and Proposition 5.10(i)]{Schb}) that
   \begin{align*}
\overline{\{z^2\colon z \in \mathrm{Sp}(A)\}} =
\mathrm{Sp}(A^2) \subseteq \rbb_+,
   \end{align*}
where $\mathrm{Sp}(T)$ denotes the spectrum of an
operator $T$. This implies that
$\mathrm{Sp}(A)\subseteq \rbb$. Hence, by
\cite[Propositions 4.17(ii) and 4.20(i)]{Schb}, the
operator $A$ is selfadjoint.
   \end{proof}
   To prove Theorem \ref{self}, we also need the
following technical lemma.
   \begin{lem} \label{sasq-1}
Suppose $(X,\ascr,\nu)$ is a $\sigma$-finite measure
space, $\tau\colon \ascr \to \rbop$ is a
$\sigma$-finite measure and $\phi$ is an
$\ascr$-measurable transformation of $X$. Then the
following conditions are equivalent\/\footnote{\;In
the condition (i) (see also Corollary \ref{sasq}(i)),
the expression ``for all $f \in L^2(\tau)$'' should be
understood as ``for all $\ascr$-measurable functions
$f\colon X \to \cbb$ such that $\int_X |f|^2
\D\tau<\infty$''.}{\em :}
   \begin{enumerate}
   \item[(i)] $f=f\circ\phi$ a.e.\ $[\nu]$ for all
$f \in L^2(\tau)$,
   \item[(ii)]
$f=f\circ\phi$ a.e.\ $[\nu]$ for all $f\in L^2(\nu)$,
   \item[(iii)] $f=f\circ\phi$ a.e.\ $[\nu]$
for all $\ascr$-measurable function $f\colon X \to
\cbb$.
   \end{enumerate}
   \end{lem}
   \begin{proof}
(i)$\Rightarrow$(iii) Take an $\ascr$-measurable
function $f \colon X \to \cbb$. Since $\tau$ is
$\sigma$-finite, there exists a sequence
$\{X_n\}_{n=1}^{\infty}$ such that $X_n\nearrow X$ as
$n\to \infty$ and $\tau(X_n)< \infty$ for all $n\in
\nbb$. Set $Y_n = \{|f|\Le n\} \cap X_n$ for $n\in
\nbb$. Then $\{\chi_{Y_n} f\}_{n=1}^{\infty} \subseteq
L^2(\tau)$, which yields
   \begin{align*}
\chi_{Y_n} f = \chi_{\phi^{-1}(Y_n)} \, f \circ \phi
\text{ a.e.\ $[\nu]$}, \quad n\in \nbb.
   \end{align*}
Passing to the limit with $n\to\infty$, we get $f = f
\circ \phi$ a.e.\ $[\nu]$.

(ii)$\Rightarrow$(iii) Apply the implication
(i)$\Rightarrow$(iii) to $\tau=\nu$.

The implications (iii)$\Rightarrow$(i) and
(iii)$\Rightarrow$(ii) are obvious.
   \end{proof}
An inspection of the proof of Lemma \ref{sasq-1} shows
that the conditions (i) and (iii) are equivalent
without assuming that the measure $\nu$ is
$\sigma$-finite.
   \begin{cor} \label{sasq}
Suppose \eqref{stand1} holds, $\tau\colon \ascr \to
\rbop$ is a $\sigma$-finite measure and $k\in \nbb$.
Then the following conditions are equivalent{\em :}
   \begin{enumerate}
   \item[(i)]
$f=f\circ\phi^k$ a.e.\ $[\mu_w]$ for every $f\in
L^2(\tau)$,
   \item[(ii)]
$C_{\phi^k}$ is well-defined as an operator in
$L^2(\mu_w)$ and $C_{\phi^k}=I_{L^2(\mu_w)}$,
   \item[(iii)] for every $\ascr$-measurable function
$f\colon X \to \cbb$, $f=f\circ\phi^k$ a.e.\
$[\mu_w]$.
   \end{enumerate}
Moreover, if {\em (iii)} holds and $\cfw$ is densely
defined, then $\ob{\efw}=L^2(\mu_w)$.
   \end{cor}
   \begin{proof}
The conditions (i)-(iii) are equivalent due to Lemma
\ref{sasq-1}. The ``moreover'' part is a direct
consequence of (iii) and Lemma \ref{roce}(ii).
   \end{proof}
   As shown below, every positive selfadjoint weighted
composition operator is a multiplication operator.
   \begin{pro}\label{psa-pre}
Suppose \eqref{stand2} holds and $\cfw$ is selfadjoint
and positive. Then $w=\sqrt{\hfw}$ a.e.\ $[\mu]$ and
$\cfw=M_w$.
   \end{pro}
   \begin{proof} Since, by
Proposition \ref{polar}(i),
$\cfw=|\cfw|=M_{\sqrt{\hfw}}$, we infer from
Proposition \ref{lemS1}(i) that
   \begin{align}  \label{xzmierz}
w \cdot f\circ \phi = \sqrt{\hfw} \cdot f, \quad f \in
L^2((1+\hfw)\D \mu).
   \end{align}
Let $\{X_n\}_{n=1}^{\infty}$ be as in Lemma
\ref{aproks}. Since $\{\chi_{X_n}\}_{n=1}^{\infty}
\subseteq L^2((1+\hfw)\D \mu)$, we infer from
\eqref{xzmierz} that $w\cdot \chi_{\phi^{-1}(X_n)} =
\sqrt{\hfw} \cdot \chi_{X_n}$ a.e.\ $[\mu]$ for all
$n\in \nbb$. Passing to the limit with $n\to\infty$
and using the fact that $X_n\nearrow X$ as
$n\to\infty$ completes the proof.
   \end{proof}
   Now we are ready to prove the theorem
characterizing the selfadjointness of weighted
composition operators.
   \begin{thm} \label{self}
Suppose \eqref{stand2} holds and $\cfw$ is densely
defined. Then the following statements are
equivalent{\em :}
   \begin{enumerate}
   \item[(i)] $\cfw$ is selfadjoint,
   \item[(ii)] the following two conditions hold{\em :}
   \begin{enumerate}
   \item[(ii-a)] $\widehat w_2 = \hfw$ a.e.\ $[\mu]$,
   \item[(ii-b)] $C_{\phi^2}$ is well-defined as an operator
in $L^2(\mu_w)$ and $C_{\phi^2}=I_{L^2(\mu_w)}$.
   \end{enumerate}
   \end{enumerate}
   \end{thm}
   \begin{proof}
(i)$\Rightarrow$(ii) Since powers of selfadjoint
operators are selfadjoint (see \cite[Theorem
5.9]{Schb}), we infer from Lemma \ref{13-12-13}(iii)
that
   \begin{align} \label{ldnyn}
\cfw^2 = C_{\phi^2,\widehat w_2}.
   \end{align}
Therefore, by Propositions \ref{lemS2} and
\ref{polar}(i), $\hfw < \infty$ a.e.\ $[\mu]$ and
   \begin{align} \label{Kamak}
M_{\hfw} = \Big(M_{\hfw^{1/2}}\Big)^2 = |\cfw|^2 =
\cfw^* \cfw = \cfw^2 \overset{\eqref{ldnyn}}=
C_{\phi^2,\widehat w_2}.
   \end{align}
This implies that the operator $C_{\phi^2,\widehat
w_2}=M_{\hfw}$ is selfadjoint and positive. By
\eqref{Kamak} and Proposition \ref{psa-pre}, we have
   \begin{align} \label{Kamak2}
M_{\hfw} = C_{\phi^2,\widehat w_2}=M_{\widehat w_2},
   \end{align}
which implies \mbox{(ii-a)}. It follows from
\eqref{Kamak} that
   \begin{align} \label{Kamak3}
\hfw \cdot f = \widehat w_2 \cdot f\circ \phi^2 \text{
a.e.\ $[\mu]$}, \quad f\in L^2((1+\hfw^2)\D\mu)
   \end{align}
However, by \mbox{(ii-a)} and the ``moreover'' part of
Theorem \ref{normal-m}, we have
   \begin{align*}
\text{$\{\widehat w_2\neq 0\}= \{\hfw > 0\} = \{w\neq
0\}$ a.e.\ $[\mu]$,}
   \end{align*}
which together with \eqref{Kamak3} yields
   \begin{align} \label{Kamak4}
f = f\circ \phi^2 \text{ a.e.\ $[\mu_w]$}, \quad f\in
L^2((1+\hfw^2)\D\mu).
   \end{align}
Applying Corollary \ref{sasq}(i) with
$\D\tau=(1+\hfw^2)\D\mu$ gives \mbox{(ii-b)}.

(ii)$\Rightarrow$(i) It follows from \mbox{(ii-a)}
that $\hfw=0$ on $\{w=0\}$ a.e.\ $[\mu]$. By Corollary
\ref{sasq} (with $k=2$), $\ob{\efw}=L^2(\mu_w)$ and
   \begin{align} \label{pow}
\text{$w=w\circ \phi^2$ a.e.\ $[\mu_w]$.}
   \end{align}
Applying Lemma \ref{nuklear}, we deduce from
\mbox{(ii-a)} that
   \begin{align*}
\text{$w\circ\phi \cdot w\circ \phi^2 = \hfw \circ
\phi$ a.e.\ $[\mu_w]$.}
   \end{align*}
This combined with \mbox{(ii-a)} and \eqref{pow}
yields $\hfw=\hfw\circ \phi$ a.e.\ $[\mu_w]$.
Summarizing, we have shown that the conditions
\mbox{(ii-a)}, \mbox{(ii-b)} and \mbox{(ii-c)} of
Theorem \ref{normal-m} hold. Hence, by this theorem,
$\cfw$ is normal. This implies that $\cfw^2$ is normal
(see \cite[Proposition 4.22]{Schb}). In particular
$\cfw^2$ is closed. In view of Lemma \ref{normsa}, it
suffices to show that $\cfw^2$ is selfadjoint and
positive. By Lemma \ref{13-12-13}(iii), the equality
\eqref{ldnyn} is valid. It follows from Proposition
\ref{lemS1}(i) and Lemma \ref{lemS11}(iii) that
   \begin{align}  \notag
\dz{\cfw^2} \overset{\eqref{ldnyn}}=
\dz{C_{\phi^2,\widehat w_2}} & = L^2((1 + \hfwn{2})\D
\mu)
   \\ \label{icib}
& = L^2((1 + \hfw^2)\D \mu) = \dz{M_{\hfw}}.
   \end{align}
According to \mbox{(ii-b)} and Corollary
\ref{sasq}(iii), for every $\ascr$-measurable function
$f\colon X \to \cbb$, $f=f\circ\phi^2$ a.e.\
$[\mu_w]$. This implies that
   \begin{align*}
\cfw^2 f \overset{\eqref{ldnyn}}= \widehat w_2 \cdot f
\circ \phi^2 & \overset{\mathrm{(ii\mbox{-}a)}}= \hfw
\cdot f \circ \phi^2
   \\
&\hspace{.9ex} \overset{(\dag)}= \hfw \cdot f =
M_{\hfw} f, \quad f\in \dz{\cfw^2},
   \end{align*}
where $(\dag)$ follows from the normality of $\cfw$
and the ``moreover'' part of Theorem \ref{normal-m}.
Hence, by \eqref{icib}, $\cfw^2=M_{\hfw}$. This
completes the proof.
   \end{proof}
The following result appeared in \cite[Proposition
B.1]{b-j-j-sS}.
   \begin{cor} \label{jfacor}
Let $(X, \ascr, \mu)$ be a $\sigma$-finite measure
space and $\phi$ be a nonsingular transformation of
$X$. Then the following assertions are valid{\em :}
   \begin{enumerate}
   \item[(i)] if $C_{\phi}$ is symmetric, then $C_{\phi}$ is
selfadjoint and unitary, and $C_{\phi}^2=I$,
   \item[(ii)] if $C_{\phi}$ is positive and symmetric, then
$C_{\phi}=I$.
   \end{enumerate}
   \end{cor}
   \begin{proof}
By Proposition \ref{fn2nC}, $C_{\phi}$ is selfadjoint.
Hence so is $C_{\phi}^2$.

(i) Since $C_{\phi}^2$ is closed, we infer from Lemma
\ref{13-12-13}(iii) and Theorem \ref{self} that
$C_{\phi}^2=C_{\phi^2}=I$.

(ii) Apply Proposition \ref{psa-pre} or (i) and the
square root lemma.
   \end{proof}
To prove Theorem \ref{psa} that characterizes positive
selfadjoint weighted composition operators, we need
the following lemma.
   \begin{lem} \label{psa-add}
Suppose \eqref{stand1} holds. Then the following
statements are equivalent{\em :}
   \begin{enumerate}
   \item[(i)] $\cfw$ is well-defined and $\cfw=M_w$,
   \item[(ii)] $C_{\phi}$ is well-defined as an operator
in $L^2(\mu_w)$ and $C_{\phi} = I_{L^2(\mu_w)}$.
   \end{enumerate}
Moreover, if {\em (ii)} holds, then $\hfw=|w|^2$ a.e.\
$[\mu]$.
   \end{lem}
   \begin{proof}
(i)$\Rightarrow$(ii) By assumption and Proposition
\ref{lemS1}(i) we see that $w\cdot f\circ \phi = w
\cdot f$ a.e.\ $[\mu]$ for all $f\in L^2((1+\hfw)\D
\mu)$. This implies that $f\circ \phi = f$ a.e.\
$[\mu_w]$ for all $f\in L^2((1+\hfw)\D \mu)$. Applying
Corollary \ref{sasq} completes the proof of (ii).

(ii)$\Rightarrow$(i) It follows from Corollary
\ref{sasq} that $w \cdot f\circ \phi = w \cdot f$
a.e.\ $[\mu]$ for all $\ascr$-measurable functions
$f\colon X \to \cbb$. This implies (i).

Now assume that (ii) holds. Applying Corollary
\ref{sasq}, we deduce that
   \begin{align*}
\mu_w\circ \phi^{-1}(\varDelta) = \int_X
\chi_{\varDelta} \circ \phi \D\mu_w = \int_{\varDelta}
|w|^2\D\mu, \quad \varDelta \in \ascr,
   \end{align*}
which proves the ``moreover'' part.
   \end{proof}
   \begin{cor} \label{psa-add-c}
Suppose \eqref{stand1} holds. If the condition {\em
(ii)} of Lemma {\em \ref{psa-add}} holds and $u\colon
X \to \cbb$ is an $\ascr$-measurable function such
that $\{w=0\} = \{u=0\}$ a.e.\ $[\mu]$, then
$C_{\phi,u}$ is well-defined and $C_{\phi,u}=M_u$.
   \end{cor}
   \begin{proof}
Combine Corollary \ref{sasq} and Lemma \ref{psa-add}.
   \end{proof}
   Positive selfadjoint weighted composition operators
can be characterized as follows.
   \begin{thm} \label{psa}
Suppose \eqref{stand1} holds. Then the following are
equivalent{\em :}
   \begin{enumerate}
   \item[(i)] $\cfw$ is well-defined, selfadjoint and positive,
   \item[(ii)] the following two conditions hold{\em :}
   \begin{enumerate}
   \item[(ii-a)] $w\Ge 0$ a.e.\ $[\mu]$,
   \item[(ii-b)] $C_{\phi}$ is well-defined as an
operator in $L^2(\mu_w)$ and
$C_{\phi}=I_{L^2(\mu_w)}$,
   \end{enumerate}
   \item[(iii)] $\cfw$ is well-defined, $\cfw=M_w$ and
$w\Ge 0$ a.e.\ $[\mu]$.
   \end{enumerate}
   \end{thm}
   \begin{proof}
(i)$\Rightarrow$(ii) In view of Proposition
\ref{psa-pre}, $w=\sqrt{\hfw} \Ge 0$ a.e.\ $[\mu]$ and
$\cfw=M_w$. Hence, by Lemma \ref{psa-add},
\mbox{(ii-b)} holds.

(ii)$\Rightarrow$(iii) Apply Lemma \ref{psa-add}.

(iii)$\Rightarrow$(i) Obvious.
   \end{proof}
   \begin{cor} \label{psa-cor}
Suppose \eqref{stand2} holds and $\cfw$ is selfadjoint
and positive. Let $u\colon X \to \cbb$ be an
$\ascr$-measurable function such that
   \begin{enumerate}
   \item[(i)] $u\Ge 0$ a.e.\ $[\mu]$,
   \item[(ii)] $\{w=0\} = \{u=0\}$ a.e.\ $[\mu]$.
   \end{enumerate}
Then $C_{\phi,u}$ is well-defined and it is
selfadjoint and positive.
   \end{cor}
   \begin{proof}
By Corollary \ref{sasq} and Theorem \ref{psa},
$f=f\circ\phi$ a.e.\ $[\mu_w]$ for every
$\ascr$-measurable function $f\colon X \to \cbb$.
Noting that the measures $\mu_w$ and $\mu_u$ are
mutually absolutely continuous and applying Corollary
\ref{sasq} and Theorem \ref{psa} to the weight $u$, we
complete the proof.
   \end{proof}
   We close this section by relating some of the
conditions that appeared in the characterizations of
normality, selfadjointness and positivity of weighted
composition operators.
   \begin{pro}\label{wkwlp}
Suppose \eqref{stand2} holds and $\cfw$ is densely
defined. Then the following assertions hold{\em :}
   \begin{enumerate}
   \item[(i)] if $w=\sqrt{\hfw}$ a.e.\ $[\mu]$ and
$\hfw=\hfw\circ \phi$ a.e.\ $[\mu_w]$, then $\widehat
w_2 = \hfw$ a.e.\ ~ $[\mu]$,
   \item[(ii)] the following two conditions are
equivalent
   \begin{enumerate}
   \item[(ii-a)] $w=\sqrt{\hfw}$ a.e.\ $[\mu]$,
$\hfw=\hfw\circ \phi$ a.e.\ $[\mu_w]$, $C_{\phi^2}$ is
well-defined as an operator in $L^2(\mu_w)$ and
$C_{\phi^2}=I_{L^2(\mu_w)}$,
   \item[(ii-b)] $w=\sqrt{\hfw}$ a.e.\ $[\mu]$ and
$\cfw=\cfw^*$.
   \end{enumerate}
   \end{enumerate}
   \end{pro}
   \begin{proof}
(i) Since $w=\sqrt{\hfw}$ a.e.\ $[\mu]$ and
$\hfw=\hfw\circ \phi$ a.e.\ $[\mu_w]$, we deduce from
Lemma \ref{nuklear} that
   \begin{align*}
w\circ \phi=\sqrt{\hfw\circ \phi} = \sqrt{\hfw} \text{
a.e.\ $[\mu_w]$.}
   \end{align*}
This gives the equality
   \begin{align} \label{w2zd}
\widehat w_2 = \hfw \text{ a.e.\ $[\mu_w]$.}
   \end{align}
In turn, the equality $w=\sqrt{\hfw}$ a.e.\ $[\mu]$
implies that $\{w=0\}=\{\hfw=0\}$ a.e.\ $[\mu]$. This
combined with \eqref{w2zd} yields $\widehat w_2 =
\hfw$ a.e.\ $[\mu]$.

\mbox{(ii-a)}$\Rightarrow$\mbox{(ii-b)} Apply the
assertion (i) and Theorem \ref{self}.

\mbox{(ii-b)}$\Rightarrow$\mbox{(ii-a)} Use Theorems
\ref{normal-m} and \ref{self}.
   \end{proof}
   \chapter{Discrete Measure Spaces}
In this chapter, we adapt our general results to the
context of discrete weighted composition operators,
i.e., weighted composition operators over discrete
measure spaces. Section \ref{Sec5.5} has an
introductory character. Section \ref{Sec5.6}
characterizes hyponormality, cohyponormality and
normality of discrete weighted composition operators
(see Theorems \ref{hypdisc}, \ref{cohypon-dis} and
\ref{normal-dis}). Section \ref{Sec6.3} provides two
criteria for subnormality of discrete weighted
composition operators, the second of which can be
thought of as a far reaching generalization of the
discrete version of Lambert's criterion (see Theorems
\ref{MAIN1-disc} and \ref{determcr}). The interplay
between the theory of moments, the geometry of graphs
induced by symbols and the injectivity problem is
discussed in Section \ref{Sec6.4} (see Theorem
\ref{determcr-c3} and Problems \ref{IP} and
\ref{IP2}). The chapter is concluded with Section
\ref{Sec6.5} which contains a variety of examples
illustrating our considerations.
   \section{\label{Sec5.5}Background}
Given a measure space $(X, \ascr,\nu)$ such that
$\{x\} \in \ascr$ for every $x \in X$, we put
   \begin{align}  \label{atom}
\at{\nu}=\{x \in X\colon \nu(x) > 0\}.
   \end{align}
Elements of the set $\at{\nu}$ are called {\em atoms}
of $\nu$. We say that a measure space $(X, \ascr,\nu)$
is {\em discrete} (or that $\nu\colon \ascr \to \rbop$
is a {\em discrete} measure on $X$) if $\ascr = 2^X$,
$\card{\at{\nu}} \Le \aleph_0$, $\nu(X\setminus
\at{\nu})=0$ and $\nu(x)<\infty$ for all $x \in X$.
Clearly, such $\nu$ is $\sigma$-finite, $\at{\nu}=\{x
\in X\colon 0<\nu(x)< \infty\}$, $\nu(\varDelta) =
\nu(\varDelta \cap \at{\nu})$ for all $\varDelta \in
2^X$ and
   \begin{align} \label{ciup}
   \begin{minipage}{70ex}
{\em for every $\varDelta \in 2^X$, $\nu
(\varDelta)=0$ if and only if $\varDelta \subseteq
X\setminus \at{\nu}$.}
   \end{minipage}
   \end{align}
As a consequence, we see that
   \begin{align} \label{jesu2}
   \begin{minipage}{70ex}
{\em if $\mathcal P$ is a property which a point $x
\in X$ may or may not have, then $\mathcal P$ holds
a.e.\ $[\nu]$ if and only if $\mathcal P$ is valid for
every $x\in \at{\nu}$.}
   \end{minipage}
   \end{align}

To avoid the repetition, we state the following
assumption which will be used frequently in this
chapter.
   \begin{align} \label{stand4} \tag{$\mathrm{AS4}$}
   \begin{minipage}{70ex}
The triplet $(X, 2^X,\mu)$ is a discrete measure
space, $\phi$ is a transformation of $X$ and $w\colon
X \to \cbb$ is a function.
   \end{minipage}
   \end{align}
If \eqref{stand4} holds, then, as usual, $\cfw$
denotes the weighted composition operator in
$L^2(\mu)$ with the symbol $\phi$ and the weight $w$.
Note that the measure $\mu_w$ is also discrete and
$\at{\mu_w} \subseteq \at{\mu}$. To simplify notation,
we write
   \begin{align*}
\phi_{w}^{-1}(\{x\}):= \phi^{-1}(\{x\}) \cap
\at{\mu_w}, \quad x \in X.
   \end{align*}
It follows from \eqref{ciup} that
$\phi_{w}^{-1}(\{\cdot\})$ possesses the following
property
   \begin{align} \label{jarmulka}
x\in \at{\mu_w\circ \phi^{-1}} \implies
\phi_{w}^{-1}(\{x\}) \neq \emptyset.
   \end{align}
Since $X = \bigsqcup_{x\in X} \phi^{-1}(\{x\})$, it is
easily seen that
   \begin{align}  \label{bomb}
\at{\mu_w} = \bigsqcup_{x\in \at{\mu_w\circ
\phi^{-1}}} \phi_{w}^{-1}(\{x\}).
   \end{align}

A necessary and sufficient condition for $\cfw$ to be
well-defined and an explicit description of $\hfw$ are
given below.
   \begin{pro}\label{erq1}
Suppose \eqref{stand4} holds. Then the following
assertions are~ valid{\em :}
   \begin{enumerate}
   \item[(i)] $\cfw$ is well-defined if and only if
$\at{\mu_w\circ \phi^{-1}} \subseteq \at{\mu}$,
   \item[(ii)] if $\cfw$ is
well-defined, then
   \begin{align} \label{bas1}
\hfw(x) = \frac{\mu_w(\phi^{-1}{(\{x\})})}{\mu(x)},
\quad x\in \at{\mu}.
   \end{align}
   \end{enumerate}
   \end{pro}
   \begin{proof}
(i) The ``only if'' part is obvious due to Proposition
\ref{wco1}. To prove the ``if'' part assume that
$\mu_w(\phi^{-1}(\{x\}))=0$ for every $x\in X$ such
that $\mu(x)=0$. Using \eqref{ciup}, we easily verify
that
   \begin{align*}
\text{if $x\in X\setminus \at{\mu}$, then $\{w\neq
0\}\cap \phi^{-1}(\{x\}) \subseteq X\setminus
\at{\mu}$.}
   \end{align*}
This combined with \eqref{ciup} implies that $\mu_w
\circ \phi^{-1} \ll \mu$, which by Proposition
\ref{wco1} shows that $\cfw$ is well defined.

(ii) Take $\varDelta \in 2^X$. It follows from
Proposition \ref{wco1} that $\mu_w \circ \phi^{-1} \ll
\mu$. As a consequence, since $\mu\big(\varDelta\cap
(X \setminus \at{\mu})\big)=0$, we see that
$\mu_w\circ \phi^{-1}\big(\varDelta\cap (X \setminus
\at{\mu})\big)=0$. Hence, $\mu_w\circ
\phi^{-1}(\varDelta)=\mu_w\circ
\phi^{-1}\big(\varDelta\cap \at{\mu}\big)$ which
implies that
   \begin{align} \label{jtjn}
\int_{\varDelta}
\frac{\mu_w(\phi^{-1}{(\{x\})})}{\mu(x)} \D \mu(x) =
\int_{\varDelta \cap \at{\mu}}
\frac{\mu_w(\phi^{-1}{(\{x\})})}{\mu(x)} \D \mu(x) =
\mu_w\circ \phi^{-1}(\varDelta).
   \end{align}
Note that according to \eqref{conv-j} the integrand in
\eqref{jtjn} equals $0$ for every $x\in X$ such that
$\mu(x)=0$. This completes the proof.
   \end{proof}
Assume that \eqref{stand4} holds and $\cfw$ is
well-defined. It follows from \eqref{jesu2} and
\eqref{bas1} and the fact that the measure $\mu_w$ is
discrete that
   \begin{align} \label{dzisleje}
   \begin{minipage}{60ex}
{\em $\hfw > 0$ a.e.\ $[\mu]$ $($resp., $\hfw > 0$
a.e.\ $[\mu_w]$$)$ if and only if $\at{\mu} \subseteq
\at{\mu_w\circ\phi^{-1}}$ $($resp., $\at{\mu_w}
\subseteq \at{\mu_w\circ\phi^{-1}}$$)$.}
   \end{minipage}
   \end{align}
By Proposition \ref{lemS2} and \eqref{bas1}, $\cfw$ is
densely defined if and only if
$\mu_w(\phi^{-1}{(\{x\})})<\infty$ for every $x\in
\at{\mu}$.

In Proposition \ref{cpd1} below we explicitly describe
the conditional expectation $\efw$. We begin by
observing that
   \begin{align} \label{zjbzt}
   \begin{minipage}{70ex}
{\em a function $f\colon X \to \rbop$ is
$\phi^{-1}(2^X)$-measurable if and only if for every
$x \in \phi(X)$, $f$ is constant on
$\phi^{-1}(\{x\})$.}
   \end{minipage}
   \end{align}
   \begin{pro}\label{cpd1}
Suppose \eqref{stand4} holds and $\cfw$ is densely
defined. Then the following assertions are valid{\em
:}
   \begin{enumerate}
   \item[(i)] the measure $\mu_w \circ \phi^{-1}$ is
discrete,
   \item[(ii)] $\bigsqcup_{x\in \varOmega}
\phi^{-1}(\{x\})$ is the smallest
$\phi^{-1}(2^X)$-measurable set of full
$\mu_w$-measure,
   \item[(iii)] for every function $f\colon X \to
\rbop$, we have\footnote{\;Note that if $x \in
\phi(X)$ and $\mu_w(\phi^{-1}(\{x\}))=0$, then the
constant value of $\efw(f)$ on $\phi^{-1}(\{x\})$ can
be defined arbitrarily; however, according to
\eqref{conv-j}, the right-hand side of the equality in
\eqref{dzis1} makes sense for such $x$ and equals
$0$.}
   \begin{align}  \label{dzis1}
\efw(f) = \frac{\int_{\phi^{-1}(\{x\})} f \D
\mu_w}{\mu_w(\phi^{-1}(\{x\}))} \text{ on
$\phi^{-1}(\{x\})$ }, \quad x \in \varOmega,
   \end{align}
   \end{enumerate}
where $\varOmega:=\at{\mu_w \circ \phi^{-1}}$.
   \end{pro}
   \begin{proof}
(i) Since the measure $\mu_w$ is discrete, we see that
for every $x \in X$, $x\in \at{\mu_w \circ \phi^{-1}}$
if and only if $\at{\mu_w} \cap \phi^{-1}(\{x\}) \neq
\emptyset$. Since the fibers $\phi^{-1}(\{z\})$, $z\in
X$, are disjoint and the set $\at{\mu_w}$ is at most
countable, we deduce that the set $\varOmega$ is at
most countable and, by the discreteness of $\mu_w$,
$\mu_w(X\setminus \bigsqcup_{x\in \varOmega}
\phi^{-1}(\{x\}))=0$. In turn, by Proposition
\ref{lemS2}, $\mu_w \circ \phi^{-1}$ is
$\sigma$-finite and thus $\mu_w \circ \phi^{-1}$ is
discrete.

(ii) For this, note that if $\varDelta \in 2^X$, then,
by \eqref{ciup} applied to the discrete measure $\mu_w
\circ \phi^{-1}$, $\phi^{-1}(\varDelta)$ is a set of
full $\mu_w$-measure if and only if $X \setminus
\varDelta \subseteq X \setminus \varOmega$, which
proves our claim.

(iii) Using \eqref{zjbzt} and \eqref{B1}, we easily
verify that the formula \eqref{dzis1} holds.
   \end{proof}
As shown below, if the weight $w$ of a well-defined
weighted composition operator $\cfw$ over a discrete
measure space does not vanish on a set of positive
$\mu$-measure, then $\cfw$ is unitarily equivalent to
a weighted composition operator over a ``purely
atomic'' measure space.
   \begin{pro}\label{unitrown}
Suppose \eqref{stand4} holds and $w(x) \neq 0$ for all
$x\in \at{\mu}$. Then
   \begin{enumerate}
   \item[(i)] $\cfw$ is well-defined if and only if
$\phi(\at{\mu}) \subseteq \at{\mu}$,
   \item[(ii)] if $\cfw$ is well-defined, $X_0:=\at{\mu}$,
$\mu_0:=\mu|_{2^{X_0}}$, $\phi_0:=\phi|_{X_0}$ and
$w_0:=w|_{X_0}$, then $\at{\mu_0}=X_0$, $w_0(x)\neq 0$
for all $x\in X_0$, $C_{\phi_0,w_0}$ is well-defined
in $L^2(\mu_0)$ and the mapping $U\colon L^2(\mu) \to
L^2(\mu_0)$ defined by $Uf = f|_{X_0}$ for $f\in
L^2(\mu)$ is a unitary isomorphism such that $U\cfw =
C_{\phi_0,w_0}U$.
   \end{enumerate}
   \end{pro}
   \begin{proof} Note that the measures $\mu$
and $\mu_w$ are mutually absolutely continuous.

(i) If $\cfw$ is well-defined, $x\in X$ and
$\mu(\phi(x))=0$, then, by Proposition \ref{wco1},
   \begin{align*}
0=\mu(\phi^{-1}(\{\phi(x)\}))\Ge \mu(x),
   \end{align*}
and so $x \notin \at{\mu}$. In turn, if
$\phi(\at{\mu}) \subseteq \at{\mu}$, $x\in X$ and
$\mu(\phi^{-1}(\{x\}))>0$, then there exists $y \in
\phi^{-1}(\{x\})\cap\at{\mu}$, which implies that
$x=\phi(y)\in \at{\mu}$. Hence, by Proposition
\ref{wco1}, $\cfw$ is well-defined.

(ii) This is a consequence of (i).
   \end{proof}
   \begin{rem} \label{figurar}
Under the assumptions of Proposition \ref{unitrown},
$\cfw$ is well-defined if and only if $C_{\phi}$ is
well-defined, and if this is the case, then $M_w
C_{\phi} \subseteq \cfw$ (see Section \ref{Sec7.1} for
more information on this matter). What is more, if
$\cfw$ is well-defined, then the underlying measure
space can always be replaced by a complete measure
space such that $\at{\mu}=X$ and $w(x)\neq 0$ for all
$x\in X$. In this particular case, the measure space
$(X,\phi^{-1}(2^X),\mu|_{\phi^{-1}(2^X)})$ is
complete. This is no longer true if the weight $w$ of
$\cfw$ vanishes on a set of positive $\mu$-measure. To
see this, consider the transformation $\phi$ of $X$ as
in Figure \ref{figura20}, a discrete measure $\mu$ on
$X$ such that $X \setminus \at{\mu}=\{A_1,B_1\}$ and a
weight $w\colon X \to \cbb$ such that $\{x\in X\colon
w(x)=0\}= \{A_2,B_2\}$. Then the composition operator
$C_{\phi}$ is not well-defined, the weighted
composition operator $\cfw$ is well-defined and the
measure space
$(X,\phi^{-1}(2^X),\mu|_{\phi^{-1}(2^X)})$ is not
complete (because $\mu(\phi^{-1}(\{0\}))=0$ and
$\{A_1\} \varsubsetneq \phi^{-1}(\{0\})$).
   \end{rem}
   \begin{center}
   \begin{figure}[t]
\subfigure {
\includegraphics[scale=0.20]{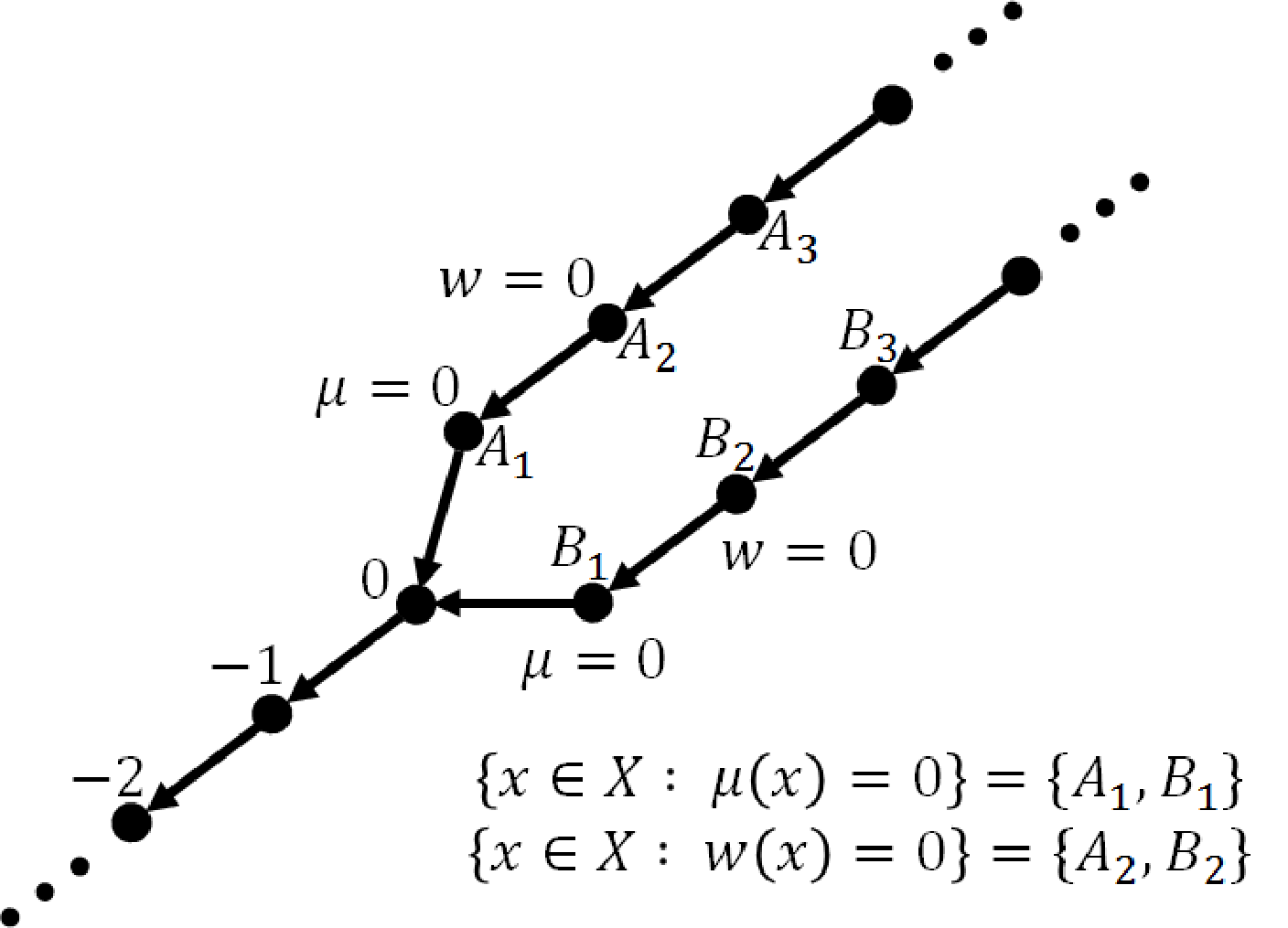}
}
   \caption{\label{6.3}An example illustrating Remark
\ref{figurar}.} \label{figura20}
   \end{figure}
   \end{center}
   \section{\label{Sec5.6}Seminormality}
First, we provide necessary and sufficient conditions
for $\cfw$ to be hyponormal. Let us note that in view
of \eqref{jarmulka} the summation in \eqref{dzis4} is
over a nonempty set.
   \begin{thm} \label{hypdisc}
Suppose \eqref{stand4} holds and $\cfw$ is densely
defined. Then the following statements are
equivalent{\em :}
   \begin{enumerate}
   \item[(i)] $\cfw$ is hyponormal,
   \item[(ii)] the following two conditions
hold{\em :}
   \allowdisplaybreaks
   \begin{gather} \label{dzis3}
\at{\mu_w} \subseteq \at{\mu_w\circ\phi^{-1}},
   \\ \label{dzis4}
\frac{1}{\mu(x)} \sum_{y\in \phi_{w}^{-1}(\{x\})}
\frac{\mu(y)\mu_w(y)}{\mu_w(\phi^{-1}(\{y\}))} \Le 1,
\quad x \in \at{\mu_w\circ \phi^{-1}}.
   \end{gather}
   \end{enumerate}
   \end{thm}
   \begin{proof}
It follows from Corollary \ref{hipinj} and
\eqref{dzisleje} that there is no loss of generality
in assuming that \eqref{dzis3} holds. By Lemma
\ref{lemS6}, Proposition \ref{lemS2} and
\eqref{dzisleje}, there exists a function
$\vartheta\colon X \to (0,\infty)$ such that
$\vartheta(x) = \sqrt{\frac{\hfw \circ
\phi(x)}{\hfw(x)}}$ for every $x \in \at{\mu_w}$.
Hence, by \eqref{bas1}, we have
   \begin{align*}
\vartheta^2(y) =
\frac{\mu_w(\phi^{-1}(\{\phi(y)\}))}{\mu(\phi(y))}
\cdot \frac{\mu(y)}{\mu_w(\phi^{-1}(\{y\}))}, \quad
y\in \at{\mu_w}.
   \end{align*}
(Note that according to our assumptions the numerators
and the denominators of fractions appearing above are
positive and finite.) Now fix $x \in
\at{\mu_w\circ\phi^{-1}}$. Then, by Proposition
\ref{erq1}(i), $x \in \at{\mu}$. As a consequence, we
have
   \allowdisplaybreaks
   \begin{align*}
\efw(\vartheta^2)(z) &\overset{\eqref{dzis1}}=
\frac{\int_{\phi^{-1}(\{x\})} \vartheta^2 \D
\mu_w}{\mu_w(\phi^{-1}(\{x\}))}
   \\
&\hspace{.75ex}=\frac{1}{\mu_w(\phi^{-1}(\{x\}))}
\int_{\phi_{w}^{-1}(\{x\})}
\frac{\mu_w(\phi^{-1}(\{x\}))}{\mu(x)} \cdot
\frac{\mu(y)}{\mu_w(\phi^{-1}(\{y\}))} \D \mu_w(y)
   \\
&\hspace{.75ex}=\frac{1}{\mu(x)}
\int_{\phi_{w}^{-1}(\{x\})}
\frac{\mu(y)}{\mu_w(\phi^{-1}(\{y\}))} \D \mu_w(y)
   \\
&\hspace{.75ex}=\frac{1}{\mu(x)} \sum_{y\in
\phi_{w}^{-1}(\{x\})}
\frac{\mu(y)\mu_w(y)}{\mu_w(\phi^{-1}(\{y\}))}, \quad
z \in \phi^{-1}(\{x\}).
   \end{align*}
Applying Theorem \ref{hypon} and Proposition
\ref{cpd1} completes the proof.
   \end{proof}
The following is a discrete counterpart of Theorem
\ref{cohypon-m}.
   \begin{thm} \label{cohypon-dis}
   Suppose \eqref{stand4} holds and $\cfw$ is densely
defined. Then the following statements are
equivalent{\em :}
   \begin{enumerate}
   \item[(i)] $\cfw$ is cohyponormal,
   \item[(ii)] the following three conditions are satisfied{\em :}
   \begin{enumerate}
   \item[(ii-a$^\dag$)] $\at{\mu_w\circ\phi^{-1}}
\subseteq \{w\neq 0\}$,
   \item[(ii-b$^\dag$)] for every $f\in L^2(\mu_w)$, if
$\hat f= 0$ on $\at{\mu_w\circ \phi^{-1}}$, then $f=0$
on $\at{\mu_w} \cap \at{\mu_w\circ \phi^{-1}}$, where
$\hat f(x)=\int_{\phi^{-1}(\{x\})} f \D \mu_w$ for
$x\in X$,
   \item[(ii-c$^\dag$)]  for every
$x\in \at{\mu_w} \cap \phi(X)$,
   \begin{align*}
\frac{\mu(\phi(x))}{\mu(x)} \sum_{y
\in\phi^{-1}(\{x\})} |w(y)|^2 \mu(y) \Le \sum_{y
\in\phi^{-1}(\{\phi(x)\})} |w(y)|^2 \mu(y).
   \end{align*}
   \end{enumerate}
   \end{enumerate}
Moreover, the condition {\em \mbox{(ii-b$^\dag$)}} is
equivalent to
   \begin{enumerate}
\item[]
   \begin{enumerate}
   \item[(ii-b$^\prime$)] for every $\varDelta \in 2^X$
such that $\varDelta \cap \at{\mu_w} \subseteq
\at{\mu_w \circ \phi^{-1}}$, there exists
$\tilde\varDelta \in 2^X$ such that $\varDelta
\vartriangle \phi^{-1} (\tilde \varDelta) \subseteq X
\setminus \at{\mu_w}$.
   \end{enumerate}
   \end{enumerate}
   \end{thm}
   \begin{proof}
To prove that (i) is equivalent to (ii), it is enough
to show that under the present circumstances the
condition \mbox{(ii-a)} (resp., \mbox{(ii-b)},
\mbox{(ii-c)}) of Theorem \ref{cohypon-m} is
equivalent to the condition \mbox{(ii-a$^\dag$)}
(resp., \mbox{(ii-b$^\dag$)}, \mbox{(ii-c$^\dag$)}).

\mbox{(ii-a)}$\Leftrightarrow$\mbox{(ii-a$^\dag$)}
Note that $\hfw=0$ on $\{w=0\}$ a.e.\ $[\mu]$ if and
only if $\mu(\{w=0\} \cap \{\hfw
> 0\})=0$, or equivalently by \eqref{ciup}, if and
only if
   \begin{align} \label{jesu1}
\{w=0\} \cap \{\hfw > 0\} \cap \at{\mu} = \emptyset.
   \end{align}
Since, by \eqref{bas1} and Proposition \ref{erq1}(i),
$\{\hfw > 0\} \cap \at{\mu} = \at{\mu_w \circ
\phi^{-1}}$, we see that \eqref{jesu1} is equivalent
to
   \begin{align*}
\{w=0\} \cap \at{\mu_w \circ \phi^{-1}} = \emptyset,
   \end{align*}
which is easily seen to be equivalent to
\mbox{(ii-a$\dag$)}.

\mbox{(ii-b)}$\Leftrightarrow$\mbox{(ii-b$^\dag$)} Use
Lemma \ref{cohypon1}(ii) with $\varOmega=\{\hfw
> 0\}$, Proposition \ref{cpd1} and \eqref{jesu2}.

\mbox{(ii-c)}$\Leftrightarrow$\mbox{(ii-c$^\dag$)}
Apply \eqref{jesu2} and \eqref{bas1}.

Summarizing, we have proved that the conditions (i)
and (ii) are equivalent.

\mbox{(ii-b$^\dag$)}$\Leftrightarrow$\mbox{(ii-b$^\prime$)}
In view of Lemma \ref{cohypon1} and the equivalence
\mbox{(ii-b)}$\Leftrightarrow$\mbox{(ii-b$^\dag$)}, it
suffices to prove that the condition (iv) of this
lemma is equivalent to \mbox{(ii-b$^\prime$)}. It is
easily seen that if $\varDelta \in 2^X$, then by
\eqref{ciup}, $\varDelta \subseteq \{\hfw > 0\}$ a.e.\
$[\mu_w]$ if and only if $(\varDelta \setminus \{\hfw
> 0\}) \cap \at{\mu_w} = \emptyset$, or equivalently,
if and only if
   \begin{align*}
\big(\varDelta \cap \at{\mu_w}\big) \setminus
\big(\{\hfw > 0\} \cap \at{\mu_w}\big) = \emptyset.
   \end{align*}
Since $\{\hfw > 0\} \cap \at{\mu_w} = \at{\mu_w\circ
\phi^{-1}} \cap \at{\mu_w}$, we see that $\varDelta
\subseteq \{\hfw > 0\}$ a.e.\ $[\mu_w]$ if and only if
$\varDelta \cap \at{\mu_w} \subseteq \at{\mu_w \circ
\phi^{-1}}$. This together with \eqref{ciup} proves
the claimed equivalence.
   \end{proof}
   \begin{cor} \label{cohypon-c1}
Under the assumptions of Theorem {\em
\ref{cohypon-dis}}, if additionally $\at{\mu}=X$ and
$w(x) \neq 0$ for every $x\in X$, then $\cfw$ is
cohyponormal if and only if the following two
conditions hold{\em :}
   \begin{enumerate}
   \item[(i)] $\{x\} = \phi^{-1}(\{\phi(x)\})$ for
every $x\in \phi(X)$,
   \item[(ii)]  $\mu(\phi(x)) \sum_{y
\in\phi^{-1}(\{x\})} |w(y)|^2 \mu(y) \Le |w(x)|^2
\mu(x)^2$ for every $x\in \phi(X)$.
   \end{enumerate}
   \end{cor}
   \begin{proof}
In view of Theorem \ref{cohypon-dis}, it is enough to
show that the condition \mbox{(ii-b$^\prime$)} is
equivalent to (i). Indeed, if \mbox{(ii-b$^\prime$)}
holds, then for every set $\varDelta \subseteq
\phi(X)$ there exists a set $\tilde \varDelta
\subseteq X$ such that $\varDelta = \phi^{-1}(\tilde
\varDelta)$. In particular, if $\varDelta=\{x\}$,
where $x\in \phi(X)$, then $\phi(x)\in \tilde
\varDelta$ and
   \begin{align*}
\{x\} \subseteq \phi^{-1}(\{\phi(x)\}) \subseteq
\phi^{-1}(\tilde \varDelta) = \{x\},
   \end{align*}
which implies (i). Now we prove the reverse
implication. For this, note that for every set
\mbox{$\varDelta \subseteq \phi(X)$},
   \begin{align*}
\phi^{-1}(\phi(\varDelta)) =
\phi^{-1}\bigg(\bigcup_{x\in \varDelta}
\phi({\{x\}})\bigg) = \bigcup_{x\in \varDelta}
\phi^{-1}(\phi{\{x\}}) = \varDelta,
   \end{align*}
which completes the proof.
   \end{proof}
   The characterization of cohyponormality of weighted
shifts on directed trees stated below appeared in
\cite[Remark 5.2.4]{j-j-s}. The present proof is based
on Theorem \ref{cohypon-dis}. We refer the reader to
the parts (d) and (e) of Section \ref{pco} for
necessary information on weighted shifts on directed
trees.
   \begin{thm} \label{cohypws}
Let $\slam$ be a densely defined weighted shift on a
directed tree $\tcal= (V,E)$ with weights $\lambdab =
\{\lambda_v\}_{v \in V^\circ}$. Then the following
statements hold{\em :}
    \begin{enumerate}
    \item[(i)] if $\tcal$ has a root, then $\slam$ is cohyponormal
if and only if $\slam = 0$,
    \item[(ii)] if $\tcal$ is rootless, then $\slam$ is cohyponormal
if and only if for every $u \in V$ the following two
conditions are satisfied{\em :}
   \begin{enumerate}
   \item[(a)] $\card {\dziplus u} \Le  1$,
   \item[(b)] if $\card{\dziplus u} = 1$, then
$0 < \|\slam e_v\| \Le |\lambda_v|$ for $v \in
\dziplus u$ and $\lambda_x = 0$ for every $x \in
\dzi{u} \setminus \dziplus{u}$,
   \end{enumerate}
   \end{enumerate}
where $e_v=\chi_{\{v\}}$ and $\dziplus u = \{x\in
\dzi{u}\colon \|\slam e_x\| > 0\}$ for $v, u\in V$.
   \end{thm}
   \begin{proof}
Set $\phi(x) = \pa{x}$ and $w(x)=\lambda_x$ for $x \in
V^{\circ}$. If $\tcal$ has a root, then we put
$\phi(\koo) = \koo$ and $w(\koo) = 0$. Note that
$\phi^{-1}(\{x\})=\dzi{x}$ for all $x\in V$. It
follows from \cite[Proposition 3.1.3]{j-j-s} that
$\{e_u\colon u\in V\} \subseteq \dz{\slam}$ and
   \begin{gather} \label{dzi-1}
\mu_w\circ \phi^{-1}(x) = \|\slam e_x\|^2, \quad x\in
V,
   \\ \label{dzi0}
\at{\mu_w\circ \phi^{-1}} = \{x\in V\colon \|\slam
e_x\|>0\} = \{x\in V\colon \sum_{y\in \dzi{x}}
|\lambda_y|^2>0\}.
   \end{gather}

(i) Suppose $\slam$ is cohyponormal. By \eqref{dzi0}
and Theorem \ref{cohypon-dis}\mbox{(ii-a$^\dag$)}, we
see that $\{w=0\} \subseteq V \setminus \at{\mu_w\circ
\phi^{-1}}$. This together with \eqref{dzi0} implies
that $\lambda_u=0$ for all $u \in \dzi{\koo}$, and
$\lambda_v=0$ for every $v\in \dzi{u}$ and every $u
\in V^\circ$ such that $\lambda_u=0$. Applying
\cite[Eqs.\ (2.1.3) and (6.1.3)]{j-j-s}, we deduce
that $\slam =0$. The reverse implication is trivial.

(ii) Assume that $\tcal$ is rootless. Suppose $\slam$
is cohyponormal. Arguing as above, we infer from
Theorem \ref{cohypws}\mbox{(ii-a$^\dag$)} that
   \begin{align} \label{dzi1}
\text{$\lambda_v=0$ for every $v\in \dzi{u}$ and every
$u \in V$ such that $\lambda_u=0$. }
   \end{align}
Now we prove that
   \begin{align} \label{dzi2}
\text{if $u \in V$, $v_1\in \dziplus{u}$, $v_2 \in
\dzi{u}$ and $v_1 \neq v_2$, then $\lambda_{v_2}=0$. }
   \end{align}
Indeed, otherwise $\lambda_{v_2}\neq 0$. By
\eqref{dzi1}, $\lambda_{v_1} \neq 0$. Define $f\in
L^2(\mu_w)$ by
   \begin{align*}
f=\frac{1}{|\lambda_{v_1}|^2} \chi_{\{v_1\}} -
\frac{1}{|\lambda_{v_2}|^2} \chi_{\{v_2\}}.
   \end{align*}
Then, by \eqref{dzi0}, $\hat f= 0$ on $\at{\mu_w\circ
\phi^{-1}}$. Since, by \eqref{dzi0} and \eqref{dzi1},
$v_1 \in \at{\mu_w} \cap \at{\mu_w\circ\phi^{-1}}$, we
infer from Theorem
\ref{cohypon-dis}\mbox{(ii-b$^\dag$)} that $f(v_1)=0$,
which is a contradiction. This justifies our claim.
Clearly, \eqref{dzi1} and \eqref{dzi2} imply the
conditions (a) and (b) except for the requirement that
``$0 < \|\slam e_v\| \Le |\lambda_v|$ for $v \in
\dziplus u$''. To prove the latter, take $u\in V$ such
that $\card{\dziplus u} = 1$. Let $v\in \dziplus{u}$.
By \eqref{dzi1}, $\lambda_v\neq 0$. This combined with
\eqref{dzi2} and the fact that $\mu$ is the counting
measure, enables us to deduce from Theorem
\ref{cohypon-dis}\mbox{(ii-c$^\dag$)} that
   \begin{align*}
0<\|\slam e_v\|^2 \Le \sum_{y \in\dzi{u}}
|\lambda_y|^2 = |\lambda_v|^2,
   \end{align*}
which yields the ``only if'' part of (ii).

To prove the converse implication assume that the
conditions (a) and (b) hold for every $u \in V$. It is
easily seen that (a) and (b) implies \eqref{dzi2}.
Combining \eqref{dzi0} and \eqref{dzi2}, we see that
the conditions \mbox{(ii-a$\dag$)} and
\mbox{(ii-c$\dag$)} of Theorem \ref{cohypon-dis} are
satisfied. To prove that the condition
\mbox{(ii-b$\dag$)} of this theorem is satisfied, take
$f\in L^2(\mu_w)$ such that $\hat f= 0$ on
$\at{\mu_w\circ \phi^{-1}}$. Fix $v\in \at{\mu_w} \cap
\at{\mu_w\circ \phi^{-1}}$. Set $u=\pa{v}$. Then
$\lambda_v\neq 0$ and, by \eqref{dzi0}, $u \in
\at{\mu_w\circ\phi^{-1}}$. Since, again by
\eqref{dzi0}, $v\in \dziplus{u}$, we infer from
\eqref{dzi2} that
   \begin{align*}
0 = \hat f(u) = \int_{\dzi{u}} f \D \mu_w =
|\lambda_v|^2f(v).
   \end{align*}
This means that $f=0$ on $\at{\mu_w} \cap
\at{\mu_w\circ \phi^{-1}}$. As a consequence, we
conclude that \mbox{(ii-b$\dag$)} is satisfied. This
completes the proof.
   \end{proof}
   The normality of weighted composition operators
over discrete measure spaces is characterized in
Theorem \ref{normal-dis} below. We omit its proof
because it is similar to that of Theorem
\ref{cohypon-dis} (apply Lemma \ref{normal1} and
Theorem \ref{normal-m} in place of Lemma
\ref{cohypon1} and Theorem \ref{cohypon-m}).
   \begin{thm} \label{normal-dis}
Suppose \eqref{stand4} holds and $\cfw$ is densely
defined. Then the following statements are
equivalent{\em :}
   \begin{enumerate}
   \item[(i)] $\cfw$ is normal,
   \item[(ii)] the following three conditions are satisfied{\em :}
   \begin{enumerate}
   \item[(ii-a)] $\at{\mu_w\circ\phi^{-1}}
\subseteq \{w\neq 0\}$,
   \item[(ii-b)] for every $f\in L^2(\mu_w)$, if
$\hat f= 0$ on $\at{\mu_w\circ \phi^{-1}}$, then $f=0$
on $\at{\mu_w}$, where $\hat
f(x)=\int_{\phi^{-1}(\{x\})} f \D \mu_w$ for $x\in X$,
   \item[(ii-c)]  for every
$x\in \at{\mu_w}$, $\phi^{-1}(\{x\}) \neq \emptyset$
and
   \begin{align*}
\frac{\mu(\phi(x))}{\mu(x)} \sum_{y
\in\phi^{-1}(\{x\})} |w(y)|^2 \mu(y) = \sum_{y
\in\phi^{-1}(\{\phi(x)\})} |w(y)|^2 \mu(y).
   \end{align*}
   \end{enumerate}
   \end{enumerate}
Moreover, the condition {\em \mbox{(ii-b)}} is
equivalent to
   \begin{enumerate}
\item[]
   \begin{enumerate}
   \item[(ii-b$^\prime$)] for every $\varDelta \in 2^X$,
there exists $\tilde\varDelta \in 2^X$ such that
$\varDelta \vartriangle \phi^{-1} (\tilde \varDelta)
\subseteq X \setminus \at{\mu_w}$.
   \end{enumerate}
   \end{enumerate}
   \end{thm}
   \begin{cor} \label{normal-c1}
Under the assumptions of Theorem {\em
\ref{normal-dis}}, if additionally $\at{\mu}=X$ and
$w(x) \neq 0$ for every $x\in X$, then $\cfw$ is
normal if and only if the following two conditions
hold{\em :}
   \begin{enumerate}
   \item[(i)] $\phi$ is a bijection,
   \item[(ii)]  $|w(\phi^{-1}(x))|^2 \mu(\phi(x))
\mu(\phi^{-1}(x)) = |w(x)|^2 \mu(x)^2$ for every $x\in
X$.
   \end{enumerate}
   \end{cor}
   \begin{proof}
Apply Theorem \ref{normal-dis}(ii) and observe that
the bijectivity of $\phi$ can be inferred from
Corollary \ref{cohypon-c1}(i) and Theorem
\ref{normal-dis}\mbox{(ii-c)}.
   \end{proof}
Regarding Corollary \ref{normal-c1}, we refer the
reader to \cite[Remark 37]{b-j-j-sS} for more
information on composition operators coming from
injections of types I, II and III (cf.\ \cite{StB}).
   \section{\label{Sec6.3}Subnormality}
We begin by stating a criterion for subnormality of
weighted composition operators over discrete measure
spaces.
   \begin{thm} \label{MAIN1-disc}
Suppose \eqref{stand4} holds, $\cfw$ is densely
defined and $\at{\mu_w} \subseteq
\at{\mu_w\circ\phi^{-1}}$. Assume, moreover, that
there exists a family of probability measures $P\colon
X \times \borel{\rbb_+} \to [0,1]$ which satisfies the
following condition
   \begin{align*}
\sum_{y \in \phi^{-1}(\{x\})} \frac{\mu_w(y)}{\mu(x)}
P(y,\sigma) = \int_{\sigma} t P(x,\D t), \quad \sigma
\in \borel{\rbb_+}, \, x \in \at{\mu_w\circ
\phi^{-1}}.
   \end{align*}
Then $\cfw$ is subnormal.
   \end{thm}
   \begin{proof}
Apply \eqref{dzisleje}, Propositions \ref{erq1} and
\ref{cpd1} and Theorem \ref{MAIN1}.
   \end{proof}
We refer the reader to \cite[Theorem 3]{b-d-j-s} for a
criterion for subnormality of weighted shifts on
directed trees. Let us point out that this criterion
has been deduced from Theorem \ref{MAIN1}. It can be
also deduced from Theorem \ref{MAIN1-disc}. This
criterion has found applications in producing some
surprising examples of unbounded subnormal operators
(cf.\ \cite{b-d-j-s} and \cite{b-j-j-sq}).

The following is a generalization of \cite[Theorem
41]{b-j-j-sS} to the case of weighted composition
operators over discrete measure spaces.
   \begin{thm} \label{determcr}
Suppose \eqref{stand4} holds and $\cfw$ is densely
defined. Assume, moreover, that for every $x\in
\at{\mu_w\circ \phi^{-1}}$,
   \begin{align}  \label{sa5}
   \begin{minipage}{71ex}
$\{\hfwn{n}(x)\}_{n=0}^\infty$ is a Stieltjes moment
sequence and $\{\hfwn{n+1}(x)\}_{n=0}^\infty$ is a
determinate Stieltjes moment sequence.
   \end{minipage}
   \end{align}
Then $\cfw$ is subnormal if and only if $\at{\mu_w}
\subseteq \at{\mu_w\circ\phi^{-1}}$.
   \end{thm}
To prove Theorem \ref{determcr}, we need the following
lemma which generalizes \cite[Lemma 38]{b-j-j-sS} to
the case of weighted composition operators. Since its
proof is essentially the same as that of \cite[Lemma
38]{b-j-j-sS}, we leave it to the reader (use Lemma
\ref{lemS11}(ii) and the equality \eqref{dzis1} in
place of \cite[Lemma 15]{b-j-j-sS} and \cite[Eq.\
(57)]{b-j-j-sS}, respectively).
   \begin{lem} \label{determcr-l}
Suppose \eqref{stand4} is satisfied and $\cfw$ is
densely defined. Let $x \in \at{\mu_w\circ \phi^{-1}}$
be such that for every $y \in \phi_{w}^{-1}(\{x\})$,
$\{\hfwn{n}(y)\}_{n=0}^\infty$ is a Stieltjes moment
sequence with a representing measure $\vartheta_y$.
Then the following assertions~ hold.
   \begin{enumerate}
   \item[(i)] If
   \begin{align} \label{ajajZenon}
\sum_{y \in \phi_{w}^{-1}(\{x\})}
\frac{\mu_w(y)}{\mu(x)} \int_0^\infty \frac {1}{t} \,
\vartheta_y(\D t) \Le 1,
   \end{align}
then $\{\hfwn{n}(x)\}_{n=0}^\infty$ is a Stieltjes
moment sequence with a representing measure
$\widetilde\vartheta_x$ given by
   \begin{align} \label{war1}
\widetilde\vartheta_x(\sigma) = \sum_{y \in
\phi_{w}^{-1}(\{x\})} \frac{\mu_w(y)}{\mu(x)}
\int_{\sigma} \frac {1}{t} \, \vartheta_y(\D t) +
\varepsilon_x \cdot \delta_0(\sigma), \quad \sigma \in
\borel{\rbb_+},
   \end{align}
where
   \begin{align} \label{war2}
\varepsilon_x = 1 - \sum_{y \in \phi_{w}^{-1}(\{x\})}
\frac{\mu_w(y)}{\mu(x)} \int_0^\infty \frac {1}{t} \,
\vartheta_y(\D t).
   \end{align}
   \item[(ii)] If $\{\hfwn{n}(x)\}_{n=0}^\infty$
is a Stieltjes moment sequence, and
$\{\hfwn{n+1}(x)\}_{n=0}^\infty$ is a determinate
Stieltjes moment sequence, then the inequality
\eqref{ajajZenon} holds, the Stieltjes moment sequence
$\{\hfwn{n}(x)\}_{n=0}^\infty$ is determinate and its
unique representing measure $\widetilde\vartheta_x$ is
given by \eqref{war1} and \eqref{war2}.
   \end{enumerate}
   \end{lem}
   \begin{proof}[Proof of Theorem  \ref{determcr}]
The ``only if'' part follows from Corollary
\ref{hipinj} and \eqref{dzisleje}. To prove the ``if''
part, assume that $\at{\mu_w} \subseteq
\at{\mu_w\circ\phi^{-1}}$. By Lemma
\ref{determcr-l}(ii), for every $x\in \at{\mu_w\circ
\phi^{-1}}$, the Stieltjes moment sequence
$\{\hfwn{n}(x)\}_{n=0}^\infty$ is determinate; let us
denote its unique representing measure by $P(x,
\cdot)$. Set $P(x,\cdot)=\delta_0$ for $x \in X
\setminus \at{\mu_w \circ \phi^{-1}}$. Since
$\at{\mu_w\circ \phi^{-1}} \subseteq \at{\mu}$ and
$\hfwn{0}(x)= 1$ for all $x\in \at{\mu}$, we see that
$P\colon X\times \borel{\rbb_+} \to [0,1]$ is a family
of probability measures. It follows from Lemma
\ref{determcr-l}(ii) that for every $x \in
\at{\mu_w\circ \phi^{-1}}$,
   \begin{align} \label{krak1}
P(x,\sigma) = \sum_{y \in \phi_{w}^{-1}(\{x\})}
\frac{\mu_w(y)}{\mu(x)} \int_{\sigma} \frac {1}{t}
P(y,\D t) + \varepsilon_x \cdot \delta_0(\sigma),
\quad \sigma\in \borel{\rbb_+},
   \end{align}
where $\varepsilon_x$ is given by \eqref{war2} for
$x\in \at{\mu_w\circ \phi^{-1}}$. It is easily seen
that \eqref{krak1} implies that $P(y, \{0\})=0$ for
all $y \in \phi_{w}^{-1}(\{x\})$ and $x\in
\at{\mu_w\circ \phi^{-1}}$. Hence, integrating the
function $\rbb_+ \ni t \mapsto t \cdot
\chi_{\sigma}(t) \in \rbb_+$ with respect to the
measures appearing on both sides of the equality in
\eqref{krak1} yields
   \begin{align} \label{krak2}
\int_{\sigma} t P(x,\D t) = \sum_{y \in
\phi_{w}^{-1}(\{x\})} \frac{\mu_w(y)}{\mu(x)}
P(y,\sigma), \quad \sigma \in \borel{\rbb_+}, \, x \in
\at{\mu_w\circ \phi^{-1}}.
   \end{align}
Applying Theorem \ref{MAIN1-disc} completes the proof.
   \end{proof}
The following corollary generalizes \cite[Theorem
41]{b-j-j-sS}.
   \begin{cor} \label{determcr-c}
Suppose \eqref{stand4} holds, $\cfw$ is densely
defined and $w(x) \neq 0$ for every $x\in \at{\mu}$.
Then the following assertions are valid{\em :}
   \begin{enumerate}
   \item[(i)] if \eqref{sa5} holds for every
$x\in \at{\mu\circ \phi^{-1}}$, then $\cfw$ is
subnormal if and only if $\at{\mu} \subseteq
\at{\mu\circ\phi^{-1}}$ or, equivalently, if and only
if $\at{\mu} = \at{\mu\circ\phi^{-1}}$,
   \item[(ii)] if \eqref{sa5} holds for every
$x\in \at{\mu}$ and $\at{\mu}=X$, then $\phi(X)=X$,
$\at{\mu\circ \phi^{-1}}=X$ and $\cfw$ is subnormal.
   \end{enumerate}
Moreover, if \eqref{sa5} holds for every $x\in
\at{\mu}$, then $\overline{\dzn{\cfw}}=L^2(\mu)$.
   \end{cor}
   \begin{proof}
Applying Theorem \ref{determcr} and noting that the
measures $\mu$ and $\mu_w$ are mutually absolutely
continuous and $\at{\mu\circ\phi^{-1}} \subseteq
\at{\mu}$ (see Proposition \ref{erq1}(i)), we get (i).
The proof of (ii) is essentially the same as that of
the ``in particular'' part of \cite[Theorem
41]{b-j-j-sS}, and so we leave it to the reader. The
``moreover'' part is a direct consequence of
Proposition \ref{potegi-p} and Theorem \ref{Mittag}.
   \end{proof}
Finally, regarding Theorem \ref{determcr}, we note
that it may happen that $\phi(X)=X$~ and
   \begin{align}  \label{jasno}
\at{\mu_w} \varsubsetneq \at{\mu_w \circ \phi^{-1}}
\varsubsetneq \at{\mu}.
   \end{align}
This phenomenon is illustrated by Figure \ref{figura1}
in which $\mu$ is an arbitrary discrete measure on $X$
such that $\at{\mu}=X$ and
$w=\chi_{X\setminus\{-1,0\}}$.
   \begin{center}
   \begin{figure}[t]
\subfigure {
\includegraphics[scale=0.20]{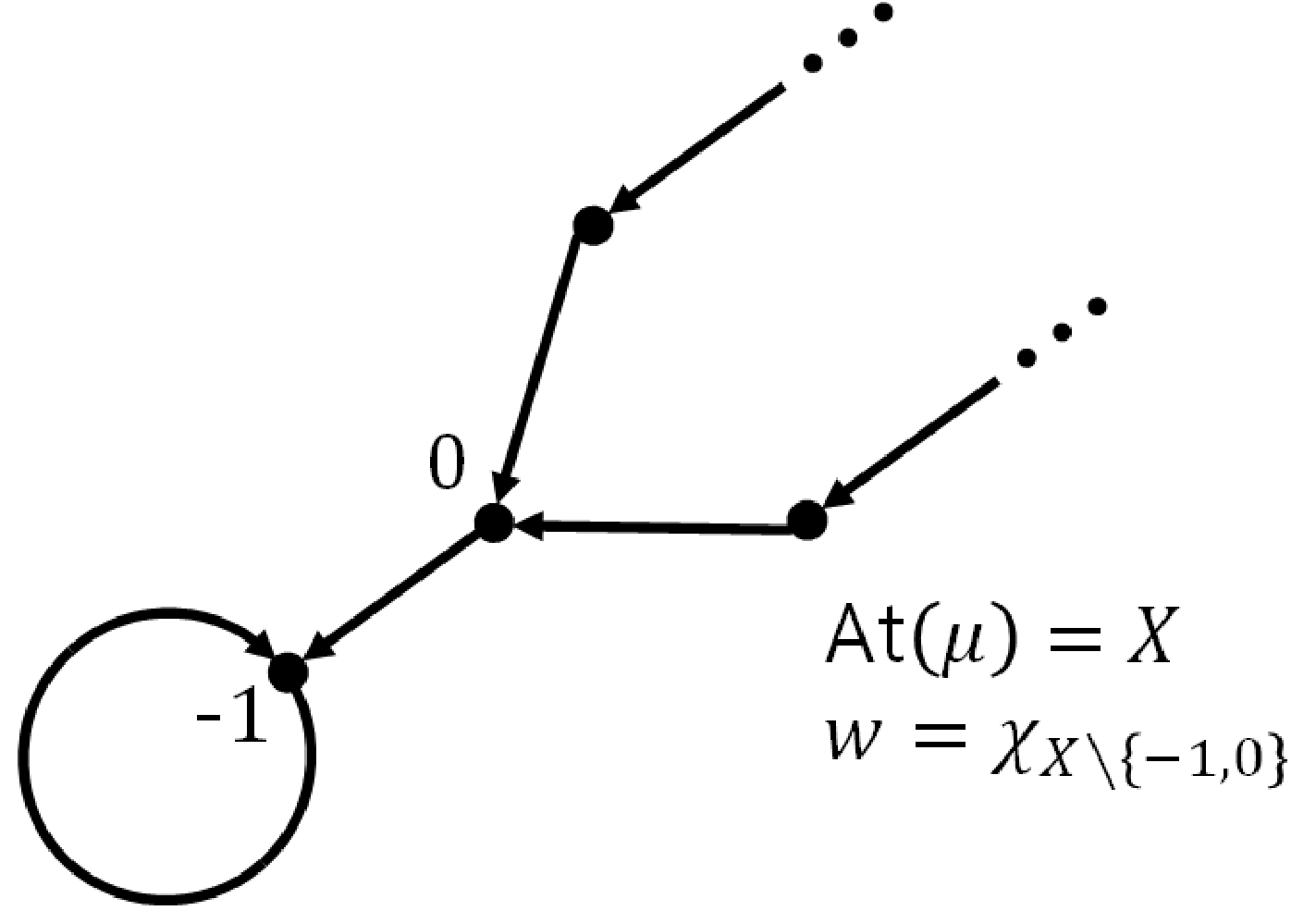}
}
   \caption{\label{6.3}An example illustrating
   \eqref{jasno}.} \label{figura1}
   \end{figure}
   \end{center}
   \section{\label{Sec6.4}Moments and injectivity}
We begin this section by shedding more light on
relationships between the generation of Stieltjes
moment sequences and the geometry of graphs induced by
symbols of weighted composition operators in the
discrete case. We also refer the reader to Remarks
\ref{b-s} and \ref{mgla2} for more comments on this
matter.
   \begin{thm} \label{determcr-c3}
Suppose \eqref{stand4} holds, $\at{\mu}=X$ and $w(x)
\neq 0$ for every $x\in X$. Then $\cfw$ is
well-defined, $\at{\mu\circ \phi^{-1}}=\phi(X)$ and
the following assertions are valid{\em :}
   \begin{enumerate}
   \item[(i)] if $x \in \phi(X)$ and
$\{\hfwn{n}(x)\}_{n=0}^\infty$ is a Stieltjes moment
sequence, then $x \in \phi^{\infty}(X)$, where
$\phi^{\infty}(X):=\bigcap_{n=1}^{\infty} \phi^n(X)$,
   \item[(ii)] if $\cfw$ is
densely defined, $x\in \phi(X \setminus \phi(X))$ and
$\{\hfwn{n}(y)\}_{n=0}^\infty$ is a Stieltjes moment
sequence for all $y\in \{x\} \cup Y_x$, where
\mbox{$Y_x=\phi^{-1}(\{x\}) \cap \phi(X)$}, then $x
\in \phi^{\infty}(X)$, $Y_x=\phi^{-1}(\{x\}) \cap
\phi^{\infty}(X)$ and $\{\hfwn{n+1}(x)\}_{n=0}^\infty$
is an indeterminate Stieltjes moment sequence,
   \item[(iii)] if for every $x\in \phi(X)$,
$\{\hfwn{n}(x)\}_{n=0}^\infty$ is a Stieltjes moment
sequence, then $\cfw$ generates Stieltjes moment
sequences and $\phi(X) = \phi^{\infty}(X)$,
   \item[(iv)] if \eqref{sa5} holds for every
$x\in \phi(X)$, then $\phi(X)=X$ and $\cfw$ is
subnormal,
   \item[(v)] if $\cfw$ is subnormal and
$\overline{\dzn{\cfw}} = L^2(\mu)$, then $\phi(X) =
\phi^{\infty}(X)$.
   \end{enumerate}
   \end{thm}
   \begin{proof}
Obviously, the measures $\mu$ and $\mu_w$ are mutually
absolutely continuous, $\at{\mu\circ
\phi^{-1}}=\phi(X)$ and $\cfw$ is well-defined. It is
also clear that $\phi^{-n}(\{x\}) = \emptyset$ for all
$n\in \nbb$ and $x\in X\setminus \phi(X)$. This,
Proposition \ref{erq1}(ii) and Lemma \ref{lemS3} imply
that
   \begin{align} \label{ajaks}
   \begin{minipage}{70ex}
{\em for every $x\in X \setminus \phi(X)$,
$\{\hfwn{n}(x)\}_{n=0}^{\infty}$ is a determinate
Stieltjes moment sequence with the representing
measure $\delta_0$.}
   \end{minipage}
   \end{align}

(i) Suppose, contrary to our claim, that there exists
$k\in \nbb$ such that $x \in X \setminus \phi^k(X)$.
Then $\phi^{-k}(\{x\})=\emptyset$ and thus, by
Proposition \ref{erq1}(ii), $\hfwn{k}(x)=0$. It
follows from our assumptions and Lemma \ref{lemS3}
that $\hfw(x)=0$. This combined with Proposition
\ref{erq1}(ii) yields $x\in X \setminus \phi(X)$,
which is a contradiction.

(ii) Since, by (i), $x \in \phi^{\infty}(X)$ and
$Y_x=\phi^{-1}(\{x\}) \cap \phi^{\infty}(X)$, it
remains to show that $\{\hfwn{n+1}(x)\}_{n=0}^\infty$
is an indeterminate Stieltjes moment sequence.
Suppose, contrary to our claim, that
$\{\hfwn{n+1}(x)\}_{n=0}^\infty$ is a determinate
Stieltjes moment sequence. By our assumption and
\eqref{ajaks}, for every $y \in \{x\} \cup
\phi^{-1}(\{x\})$, $\{\hfwn{n}(y)\}_{n=0}^{\infty}$ is
a Stieltjes moment sequence with a representing
measure, say $\vartheta_y$. Noting that $x \in
\at{\mu_w \circ \phi^{-1}}$, we infer from Lemma
\ref{determcr-l}(ii) that the inequality
\eqref{ajajZenon} holds. Hence, since $x=\phi(y_0)$
for some $y_0 \in X \setminus \phi(X)$, and
consequently $y_0 \in \phi^{-1}(\{x\}) \cap
(X\setminus \phi(X))$, we have
$\int_0^{\infty}\frac{1}{t} \D \vartheta_{y_0}(t) <
\infty$, which contradicts the equality
$\vartheta_{y_0}=\delta_0$ (see \eqref{ajaks}).

(iii) In view of Theorem \ref{gsms} and \eqref{ajaks},
$\cfw$ generates Stieltjes moment sequences. The
equality $\phi(X) = \phi^{\infty}(X)$ is a direct
consequence of (i).

(iv) According to \eqref{ajaks}, the condition
\eqref{sa5} holds for every $x\in X$. Therefore, the
assertion (iv) follows from Corollary
\ref{determcr-c}(ii).

(v) Apply \cite[Proposition 3.2.1]{b-j-j-sA}, Theorem
\ref{gsms} and (iii).
   \end{proof}
Regarding the assertions (iii) and (v) of Theorem
\ref{determcr-c3}, the following well-known and easy
to prove set-theoretical result is worth recalling.
   \begin{pro} \label{rfin-1}
If $\phi$ is a transformation of a nonempty set $X$,
then the following conditions are equivalent:
   \begin{enumerate}
   \item[(i)] $\phi(X) = \phi^{\infty}(X)$,
   \item[(ii)] $\phi(X)=\phi^2(X)$,
   \item[(iii)] $\phi(X)=\phi^n(X)$ for some integer
$n\Ge 2$,
   \item[(iv)] $\phi(X)=\phi^n(X)$ for every integer
$n\Ge 2$.
   \end{enumerate}
   \end{pro}
Below, we show that the assertion (iii) of Theorem
\ref{determcr-c3} is related to the injectivity
problem (see Problem \ref{IP}).
   \begin{rem} \label{rfin}
Suppose that \eqref{stand4} holds, $\at{\mu}=X$ and
$w(x) \neq 0$ for every $x\in X$. By Theorem
\ref{gsms} and \eqref{ajaks}, the assertion (iii) of
Theorem \ref{determcr-c3} is equivalent to the
statement that, if $\cfw$ generates Stieltjes moment
sequences, then $\phi(X) = \phi^{\infty}(X)$. Hence,
the question arises as to whether this equality
implies the surjectivity of $\phi$ when $\cfw$
generates Stieltjes moment sequences. If the answer to
this question is in the negative, then it can (and
does) happen that such $\cfw$ is not injective because
$\chi_{\{x\}} \in \jd{\cfw}$ for every $x \in X
\setminus \phi(X)$. In turn, if the answer is in the
affirmative, then such $\cfw$ is always injective (see
Lemma \ref{jadro}). This question is a particular case
of a more general problem (see Problem \ref{IP}
below), called the injectivity problem, which was
originally stated for composition operators in
\cite[Problem 3.3.6]{b-j-j-sg}.
   \end{rem}
   \begin{opq}[Injectivity problem] \label{IP}
Suppose that \eqref{stand2} holds, $\cfw$ generates
Stieltjes moment sequences and $w \neq 0$ a.e.\
$[\mu]$. Is it true that $\cfw$ is injective{\em ?}
   \end{opq}
It is worth mentioning that Problem \ref{IP} has a
negative answer if the hypothesis that $w \neq 0$
a.e.\ $[\mu]$ is dropped (cf.\ Section \ref{pco}(g)).
What is more, it may happen that $\cfw$ is an isometry
when $w$ vanishes on a set of positive $\mu$-measure.
   \begin{center}
   \begin{figure}[t]
   \subfigure {
\includegraphics[scale=0.19]{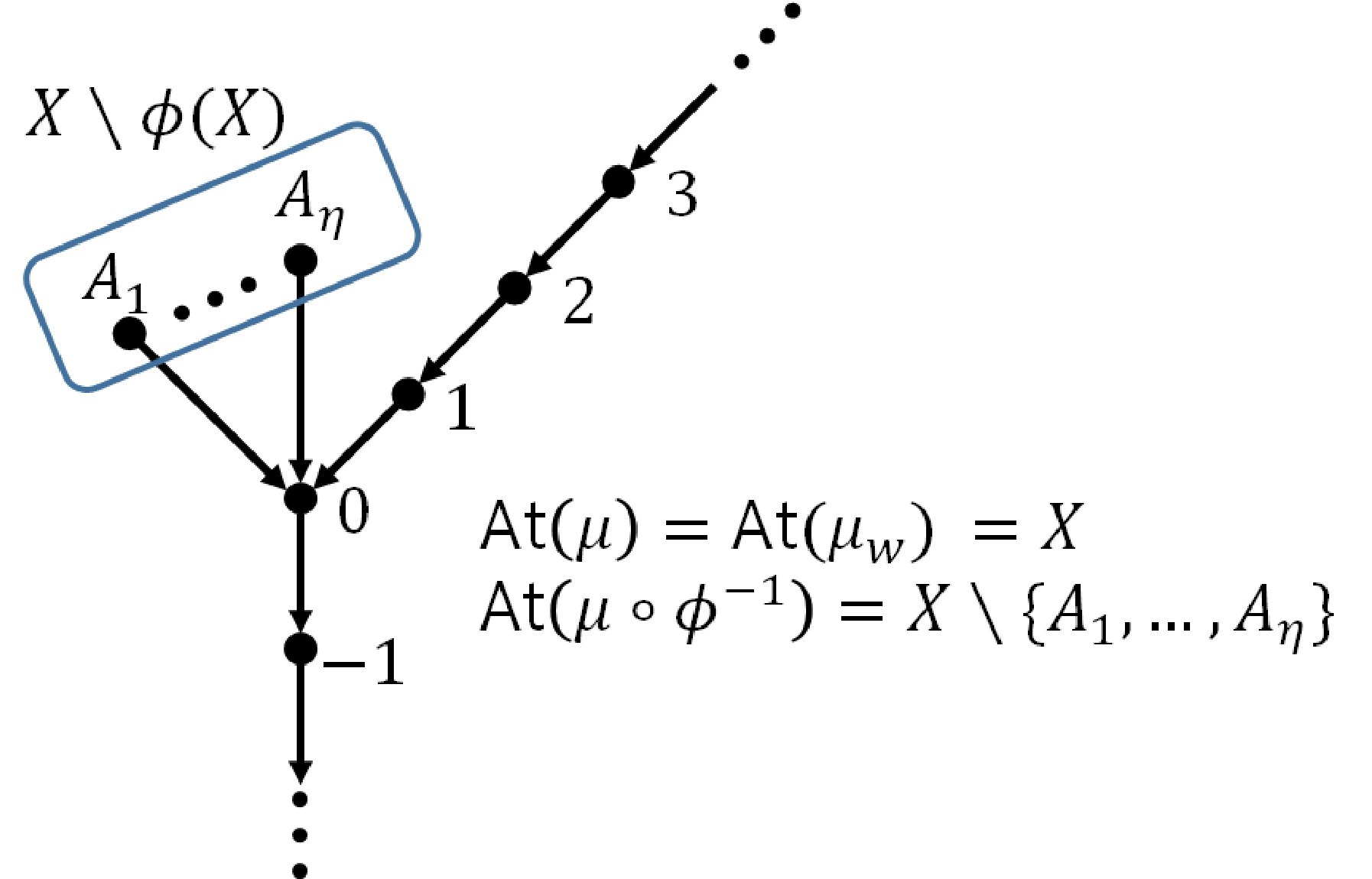}
} \subfigure {
\includegraphics[scale=0.19]{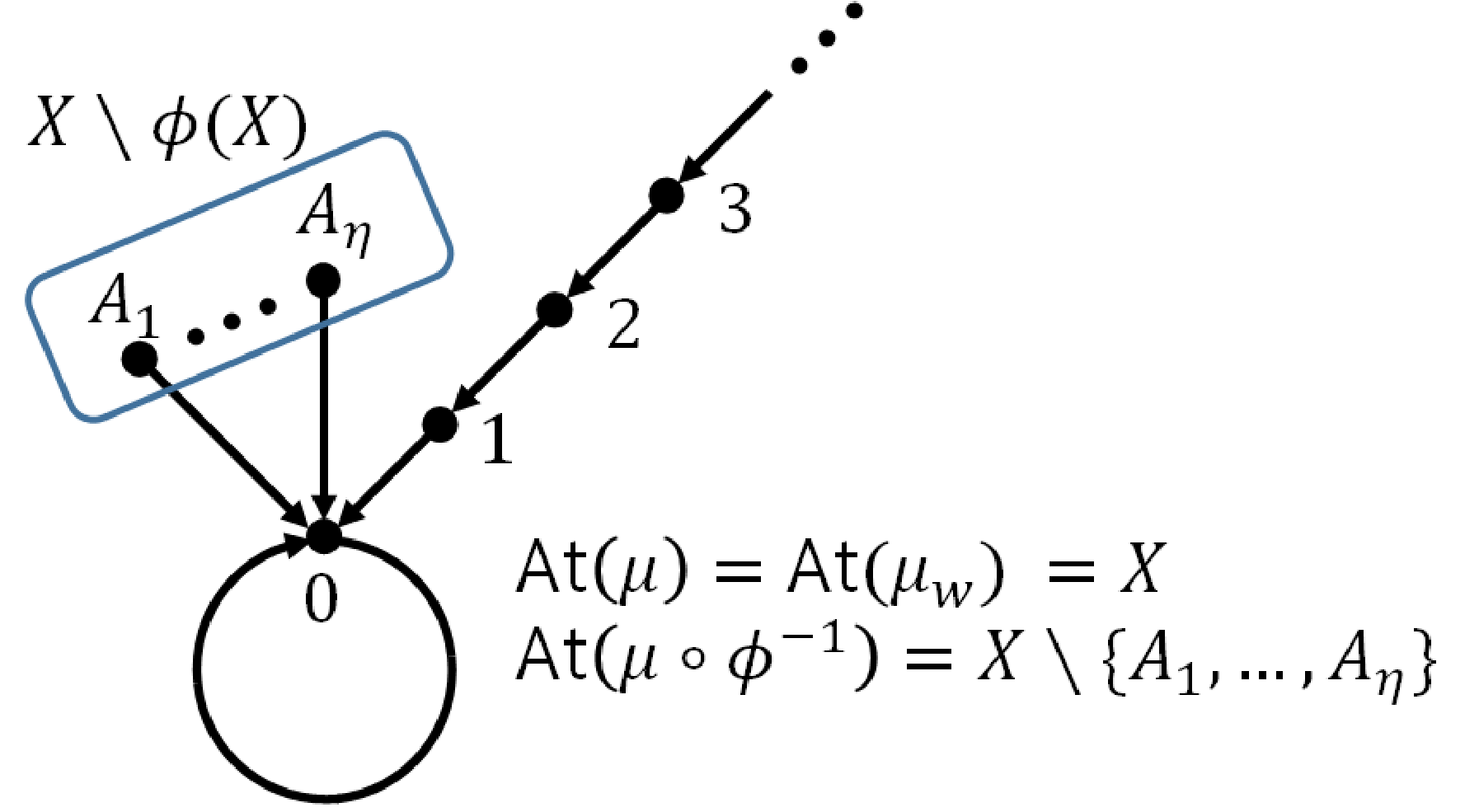}
} \caption{Illustrations of two transformations $\phi$
appearing in Example \ref{cciemno1}.} \label{figura3}
   \end{figure}
   \end{center}

Now, we discuss an example (or, in fact, two examples)
which shows that the injectivity problem is hard to
solve.
   \begin{exa} \label{cciemno1} Fix $\eta
\in \nbb$. Let $\phi$ be a transformation of $X$ as in
one of the two subfigures of Figure \ref{figura3}.
Take any discrete measure $\mu$ on $X$ such that
$\at{\mu}=X$. Let $w\colon X \to \cbb$ be any function
such that $w(x) \neq 0$ for every $x\in X$. Then the
weighted composition operator $\cfw$ is easily seen to
be densely defined (in fact, $\overline{\dzn{\cfw}} =
L^2(\mu)$ because $\chi_{\{x\}} \in \dzn{\cfw}$ for
every $x\in X$). Clearly, $X\setminus \phi(X) = \{A_1,
\ldots, A_{\eta}\}$ and thus $\cfw$ is not injective.
What is more, the symbol $\phi$ satisfies the equality
$\phi(X) = \phi^{\infty}(X)$. It follows from Theorem
\ref{determcr-c3}(ii) that if
$\{\hfwn{n}(x)\}_{n=0}^\infty$ is a Stieltjes moment
sequence for $x\in\{0,1\}$, then
$\{\hfwn{n+1}(0)\}_{n=0}^\infty$ is an indeterminate
Stieltjes moment sequence. We refer the reader to
\cite[Propositions 3.3.3 and 3.3.4]{b-j-j-sg} for more
examples of this kind.
   \end{exa}
Theorem \ref{determcr-c3} turns out to be useful when
localizing points $x$ in $X$ with the property that
$\{\hfwn{n}(x)\}_{n=0}^{\infty}$ is a Stieltjes moment
sequence.
   \begin{rem} \label{b-s}
Assume \eqref{stand4} holds, $\at{\mu}=X$ and $w(x)
\neq 0$ for every $x\in X$. Take any $x\in X$. Then we
have three disjunctive possibilities:
   \begin{enumerate}
   \item[(a)] $x\in X \setminus \phi(X)$,
   \item[(b)] $x \in \phi(X) \setminus \phi^{\infty}(X)$,
   \item[(c)] $x \in \phi^{\infty}(X)$.
   \end{enumerate}
If (a) holds, then by \eqref{ajaks},
$\{\hfwn{n}(x)\}_{n=0}^{\infty}$ is a determinate
Stieltjes moment sequence. If (b) holds, then by the
assertion (i) of Theorem \ref{determcr-c3},
$\{\hfwn{n}(x)\}_{n=0}^{\infty}$ is never a Stieltjes
moment sequence. Finally, if (c) holds, then
$\{\hfwn{n}(x)\}_{n=0}^{\infty}$ may or may not happen
to be a Stieltjes moment sequence. The last
possibility will be discussed in Example \ref{jesien1}
below.
   \end{rem}
   \begin{exa} \label{jesien1}
Let $\phi$ be a transformation of $X$ as in one of the
two subfigures of Figure \ref{figura33} and let $\mu$
be any discrete measure on $X$ such that $\at{\mu}=X$.
Note that $X\setminus \phi(X) = A_2$, $\phi(X)
\setminus \phi^{\infty}(X) = A_1$ and
$\phi^{\infty}(X)=X \setminus \{A_1,A_2\}$. Note also
that $\overline{\dzn{C_{\phi}}} = L^2(\mu)$ and
$C_{\phi}$ is not injective (cf.\ Example
\ref{cciemno1}).

It is easily seen that if
$\{\mathsf{h}_{\phi^n}(x)\}_{n=0}^{\infty}$ is a
Stieltjes moment sequence for a fixed $x\in \nbb$,
then so is
$\{\mathsf{h}_{\phi^n}(x+1)\}_{n=0}^{\infty}$ (because
$\mu(x+1)\mathsf{h}_{\phi^n} (x+1) =
\mu(x)\mathsf{h}_{\phi^{n+1}}(x)$). Suppose now that
$\{\mathsf{h}_{\phi^n}(0)\}_{n=0}^{\infty}$ is a
Stieltjes moment sequence. By analogy with $x=1$, we
can ask whether
$\{\mathsf{h}_{\phi^n}(x)\}_{n=0}^{\infty}$ is a
Stieltjes moment sequence for every $x\in
\phi^{-1}(\{0\})$. In view of Remark \ref{b-s}, this
is not the case for $x=A_1$. We will show that this is
also {\em not} the case for $x=1$ if the Stieltjes
moment sequence
$\{\mathsf{h}_{\phi^{n+2}}(0)\}_{n=0}^{\infty}$ is
determinate. Suppose, contrary to our claim, that
$\{\mathsf{h}_{\phi^n}(x)\}_{n=0}^{\infty}$ is a
Stieltjes moment sequence for $x\in\{0,1\}$ and
$\{\mathsf{h}_{\phi^{n+2}}(0)\}_{n=0}^{\infty}$ is a
determinate Stieltjes moment sequence. Note that Lemma
\ref{determcr-l} is not applicable to $x=0$. Since
$\{\mu(1)\mathsf{h}_{\phi^{n+1}}(1)\}_{n=0}^{\infty} =
\{\mu(n+2)\}_{n=0}^{\infty}$ is a Stieltjes moment
sequence, we deduce that
   \begin{align} \label{jesien3}
   \begin{minipage}{65ex}
{\em $\{\mu(n+2) + \mu(A_2)
\delta_{n,0}\}_{n=0}^{\infty}$ is a Stieltjes moment
sequence with the representing measure
$\nu+\mu(A_2)\delta_0$, where $\nu$ is a representing
measure of
$\{\mu(1)\mathsf{h}_{\phi^{n+1}}(1)\}_{n=0}^{\infty}$.}
   \end{minipage}
   \end{align}

Assume first that $\phi$ is as in the left subfigure
of Figure \ref{figura33}. Then
   \begin{align*}
\{\mu(0)\mathsf{h}_{\phi^{n+2}}(0)\}_{n=0}^{\infty} =
\{\mu(n+2) + \mu(A_2) \delta_{n,0}\}_{n=0}^{\infty}.
   \end{align*}
Since the Stieltjes moment sequence
$\{\mu(0)\mathsf{h}_{\phi^{n+2}}(0)\}_{n=0}^{\infty}$
is determinate, we deduce from \eqref{jesien3} and
\cite[Lemma 2.4.1]{b-j-j-sA} that
$\{\mu(0)\mathsf{h}_{\phi^{n+1}}(0)\}_{n=0}^{\infty}$
is not a Stieltjes moment sequence, which contradicts
the fact that
$\{\mu(0)\mathsf{h}_{\phi^n}(0)\}_{n=0}^{\infty}$ is a
Stieltjes moment sequence.

Suppose now that $\phi$ is as in the right subfigure
of Figure \ref{figura33}. Note that the sequence
$\{\gamma_n\}_{n=0}^{\infty}$ defined by $\gamma_n =
\mu(0)\mathsf{h}_{\phi^{n+1}}(0)$ for $n\in \zbb_+$ is
a Stieltjes moment sequence. Let $\rho$ be a
representing measure of $\{\gamma_n\}_{n=0}^{\infty}$.
Then
   \begin{align*}
\int_0^{\infty} t^n (t-1) \D
\rho(t)=\gamma_{n+1}-\gamma_{n}=\mu(n+2) + \mu(A_2)
\delta_{n,0}, \quad n\in \zbb_+.
   \end{align*}
This and \eqref{jesien3} imply that
   \allowdisplaybreaks
   \begin{multline} \label{dzik1}
\int_0^{\infty} t^n |t-1| \chi_{(1,\infty)}(t) \D
\rho(t) = \int_0^{\infty} t^n |t-1| \chi_{[0,1)}(t) \D
\rho(t)
   \\
+ \int_0^{\infty} t^n \D (\nu + \mu(A_2) \delta_0)(t),
\quad n \in \zbb_+.
   \end{multline}
Observe that
   \begin{align*}
\int_{\sigma} |t-1| \chi_{(1,\infty)}(t) \D \rho(t)
\Le \int_{\sigma} t \D \rho(t), \quad \sigma \in
\borel{\rbb_+}.
   \end{align*}
Since $t\D\rho(t)$ is a representing measure of the
determinate Stieltjes moment sequence
$\{\gamma_{n+1}\}_{n=0}^{\infty}=\{\mu(0)
\mathsf{h}_{\phi^{n+2}}(0)\}_{n=0}^{\infty}$, we
deduce from \cite[Proposition 2.1.3]{b-j-j-sg} that
the Stieltjes moment sequence $\{\int_0^{\infty} t^n
|t-1| \chi_{(1,\infty)}(t) \D
\rho(t)\}_{n=0}^{\infty}$ is determinate. This
together with \eqref{dzik1} implies that
   \begin{align} \label{jesien34}
\int_{\sigma} |t-1| \chi_{(1,\infty)}(t) \D \rho(t) =
\int_{\sigma} |t-1| \chi_{[0,1)}(t) \D \rho(t) + (\nu
+ \mu(A_2) \delta_0)(\sigma)
   \end{align}
for every $\sigma \in \borel{\rbb_+}$. Substituting
$\sigma = \{0\}$ into \eqref{jesien34}, we are led to
a~ contradiction.

Summarizing, in both cases,
$\{\mathsf{h}_{\phi^n}(1)\}_{n=0}^{\infty}$ is not a
Stieltjes moment sequence whenever
$\{\mathsf{h}_{\phi^n}(0)\}_{n=0}^{\infty}$ is a
Stieltjes moment sequence and
$\{\mathsf{h}_{\phi^{n+2}}(0)\}_{n=0}^{\infty}$ is a
determinate Stieltjes moment sequence. It would be
interesting to know whether the determinacy assumption
could be dropped. This leaves us with the following
open question (cf.\ Remark \ref{b-s}).
   \end{exa}
   \begin{center}
   \begin{figure}[t] \subfigure {
\includegraphics[scale=0.19]{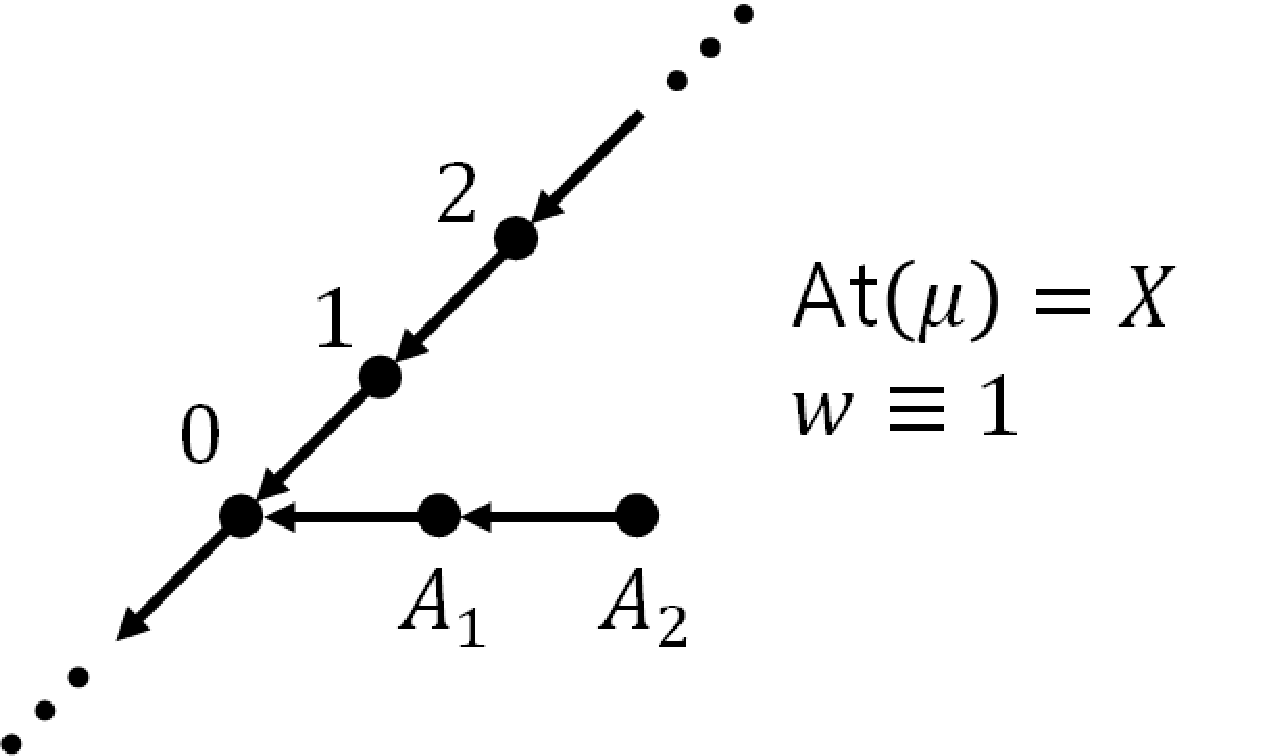}
} \quad \subfigure {
\includegraphics[scale=0.19]{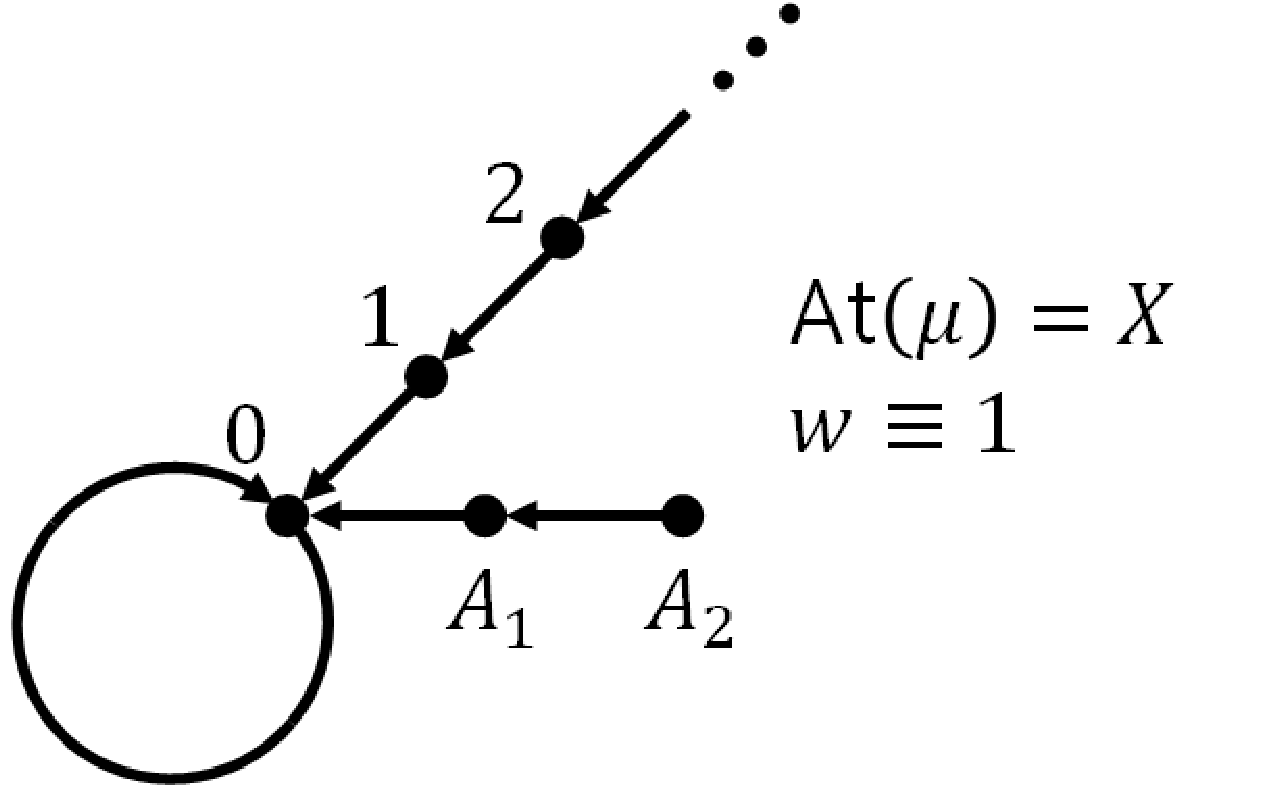}
} \caption{Illustrations of two transformations $\phi$
appearing in Example \ref{jesien1}.} \label{figura33}
   \end{figure}
   \end{center}
   \begin{opq} \label{IP2}
Suppose \eqref{stand4} holds, $\at{\mu}=X$ and $w$ has
no zeros in $X$. Assume that $x \in \phi^{\infty}(X)$
and $Z_{x} \setminus \{x\} \neq \emptyset$, where
   \begin{align} \label{zx}
Z_{x}:= \phi^{-1}(\{x\}) \cap \phi^{\infty}(X).
   \end{align}
Is it true that if $\{\hfwn{n}(x)\}_{n=0}^{\infty}$ is
a Stieltjes moment sequence, then
$\{\hfwn{n}(z)\}_{n=0}^{\infty}$ is a Stieltjes moment
sequence for every $z \in Z_{x}\setminus \{x\}$.
   \end{opq}

   We conclude this section with the following remark.
   \begin{rem} \label{mgla2}
The reader should be aware of the fact that there is a
transformation $\phi$ such that $Z_{x}=\emptyset$ for
some $x \in \phi^{\infty}(X)$, and that there is a
transformation $\phi$ such that $Z_{x}\neq \emptyset$
and $Z_{x}\setminus \{x\} = \emptyset$ (equivalently,
$Z_{x}=\{x\}$) for some $x \in \phi^{\infty}(X)$,
where $Z_{x}$ is as in \eqref{zx}. Indeed, if $\phi$
is the transformation of $X$ as in the left subfigure
of Figure \ref{figura43}, then
   \begin{gather*}
X\setminus \phi(X) = \{A_1, A_2, A_3, \ldots\}, \quad
\phi^{\infty}(X) = \{0,1,2, \ldots\},
   \end{gather*}
and consequently
   \begin{align} \label{jasno55}
   \left. \begin{gathered} \phi^{-1}(\{0\}) \cap
(X\setminus \phi(X)) = \{A_1\},
   \\
\phi^{-1}(\{0\}) \cap (\phi(X) \setminus
\phi^{\infty}(X)) = \{B_1, B_2, B_3, \ldots\},
   \\
Z_0=\phi^{-1}(\{0\}) \cap \phi^{\infty}(X) =
\emptyset.
   \end{gathered}
   \quad \right\}
   \end{align}
Moreover, by \eqref{jasno55}, the point $x=0$ has the
following property (cf.\ Theorem
\ref{determcr-c3}(ii))
   \begin{align} \label{prop1}
\text{$x \in \phi(X\setminus \phi(X))$,
$\phi^{-1}(\{x\}) \cap (\phi(X) \setminus
\phi^{\infty}(X)) \neq \varnothing$ and $Z_x
\varsubsetneq Y_x$.}
   \end{align}
In turn, if $\phi$ is the transformation of $X$ as in
the right subfigure of Figure \ref{figura43}, then
$0\in \phi^{\infty}(X)$, $Z_0 \neq \emptyset$ and
$Z_0\setminus \{0\}=\emptyset$; moreover, as in the
previous case, the point $x=0$ satisfies the condition
\eqref{prop1}.
   \end{rem}
   \begin{figure}[t]
   \begin{center}
   \subfigure {
\includegraphics[scale=0.19]{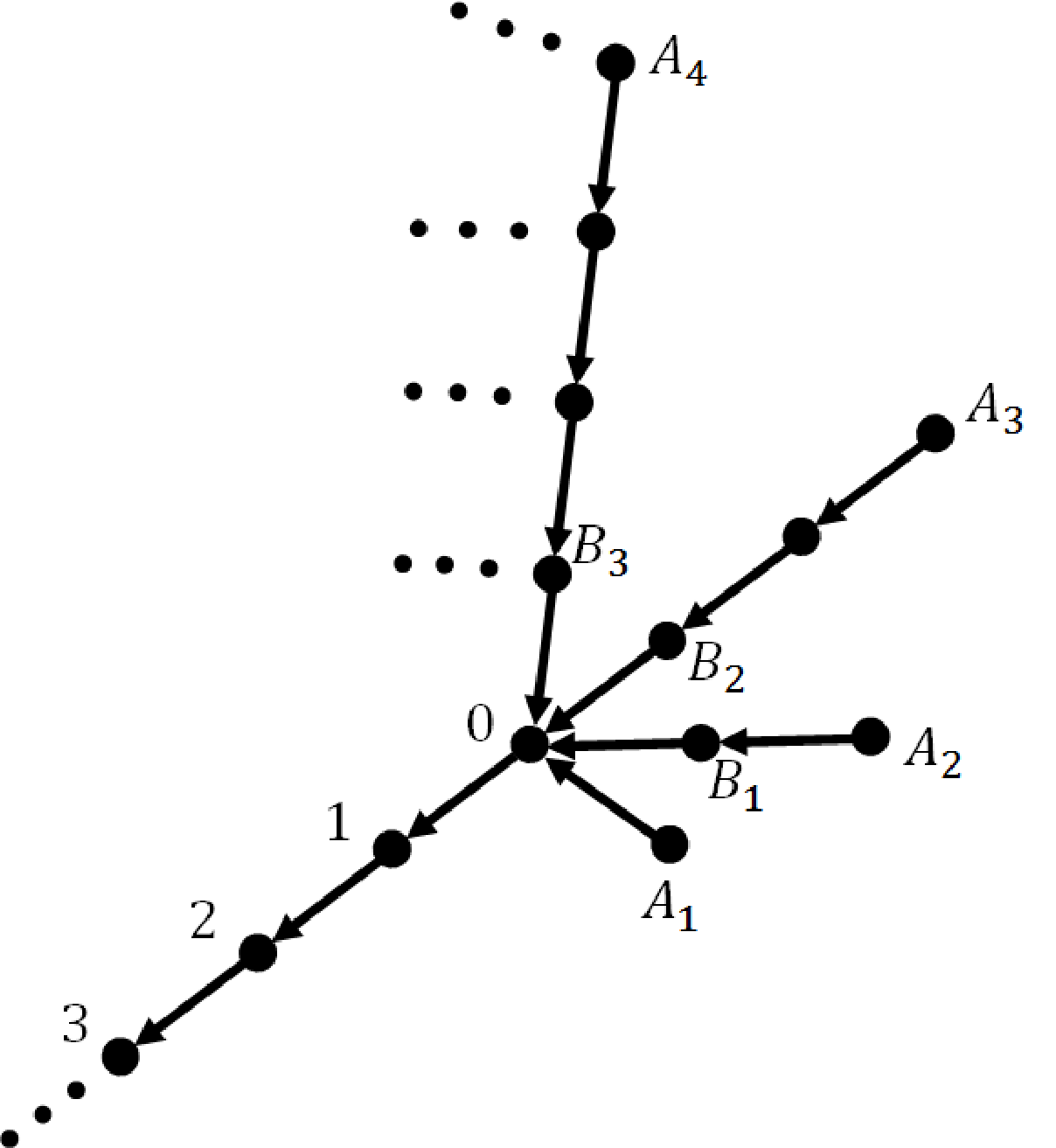}
} \quad \subfigure {
\includegraphics[scale=0.19]{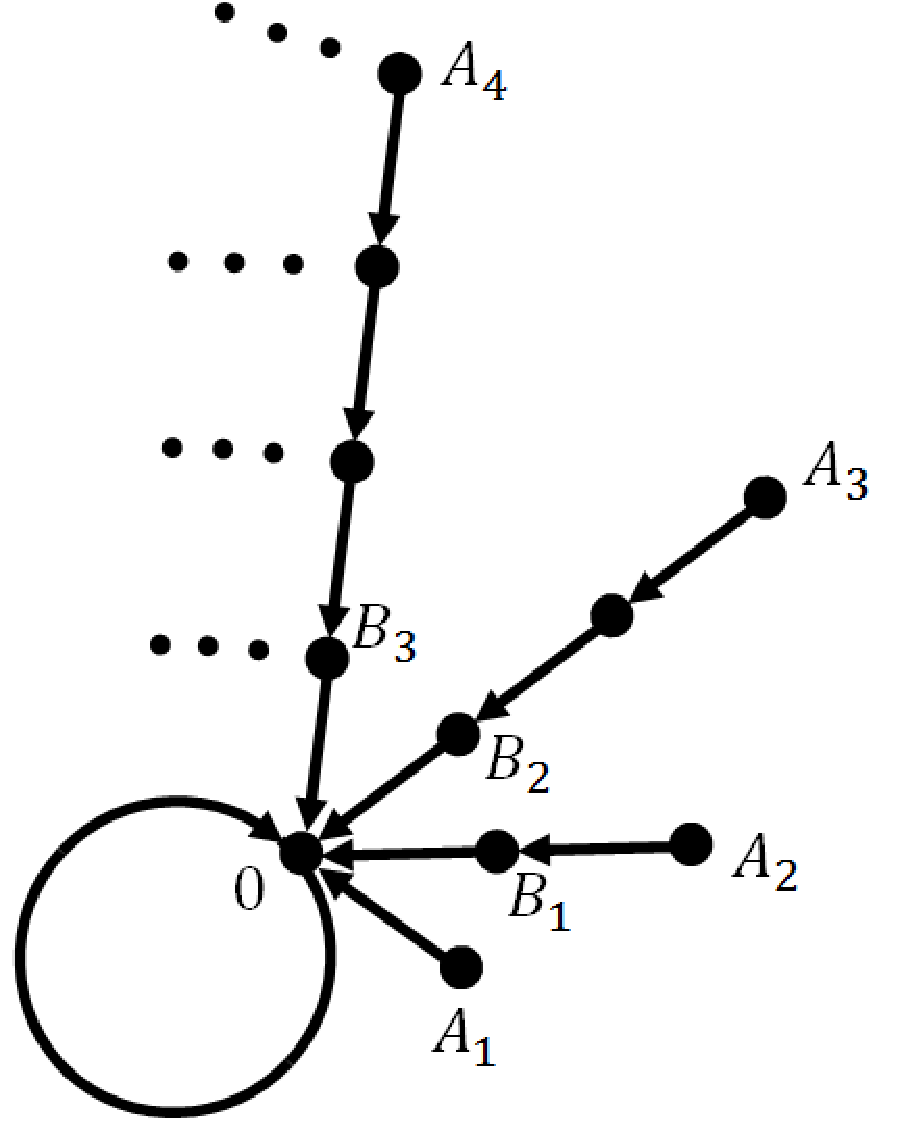}
} \caption{Illustrations of two transformations $\phi$
appearing in Remark \ref{mgla2}.} \label{figura43}
   \end{center}
   \end{figure}
   \section{\label{Sec6.5}Examples} We begin by providing an example
of an isometric weighted composition operator $\cfw$
for which the corresponding composition operator
$C_{\phi}$ is not well-defined (see Example
\ref{cohyp-count2}). This example also shows that the
assumption \mbox{(ii-a)} of Theorem \ref{cohypon-m}
cannot be omitted. That none of the remaining
assumptions \mbox{(ii-b)} and \mbox{(ii-c)} of Theorem
\ref{cohypon-m} can be omitted is demonstrated in
Examples \ref{cohyp-count1} and \ref{cohyp-count3}
respectively.
   \begin{exa}[Condition \mbox{(ii-a)} is missing]
\label{cohyp-count2} Set $X=\zbb_+$. Let $\mu\colon
2^X \to \rbop$ be the $\sigma$-finite measure such
that $\mu(n) = 1$ for every $n\Ge 1$ and $\mu(0)=0$.
Define the transformation $\phi$ of $X$ by
   \begin{align*}
\phi(n) =
   \begin{cases}
n-1 & \text{if } n\Ge 1,
   \\
0 & \text{if } n=0,
   \end{cases}
\quad n\in X.
   \end{align*}
Let $w\colon X \to \cbb$ be the function given by
   \begin{align*}
w(n) =
   \begin{cases}
0 & \text{if } n\in \{0,1\},
   \\
1 & \text{if } n\Ge 2,
   \end{cases}
\quad n\in X.
   \end{align*}
Using Proposition \ref{erq1}, we verify that $\cfw$ is
well-defined and
   \begin{align*}
\hfw(n) = 1, \quad n\Ge 1,
   \end{align*}
Hence, the condition \mbox{(ii-c)} of Theorem
\ref{cohypon-m} holds and, by \eqref{chisom}, $\cfw$
is an isometry on $L^2(\mu)$. Setting
$e_n=\chi_{\{n\}}$ for all $n\in \nbb$ and noting that
$\{e_n\}_{n=1}^{\infty}$ is an orthonormal basis of
$L^2(\mu)$ such that $\cfw e_n = e_{n+1}$ for every
$n\in \nbb$, we deduce that the weighted composition
operator $\cfw$ is unitarily equivalent to the
unilateral shift of multiplicity $1$. This means that
$\cfw$ is not cohyponormal. It is easily seen that the
condition (v) of Lemma \ref{cohypon1} is satisfied, so
by this lemma the condition \mbox{(ii-b)} of Theorem
\ref{cohypon-m} holds. As a consequence, the condition
\mbox{(ii-a)} of this theorem does not hold. The
latter can be also verified directly. Since $\mu(0)=0$
and $\mu\circ \phi^{-1}(0)=1$, we infer from
Proposition \ref{wco1} that the corresponding
composition operator $C_{\phi}$ is not well-defined.
   \end{exa}
In view of Example \ref{cohyp-count2}, one may ask
whether there exists a unitary weighted composition
operator $\cfw$ for which the corresponding
composition operator $C_{\phi}$ is not well-defined.
However, this possibility is disproved by Proposition
\ref{mwc1}(iii).

As shown below, the case when the condition
\mbox{(ii-b)} of Theorem \ref{cohypon-m} is missing
can be realized even by a composition operator in a
two-dimensional $L^2$-space.
   \begin{exa}[Condition \mbox{(ii-b)} is missing]
\label{cohyp-count1} Let $N$, $X$, $\ascr$, $\mu$ and
$\phi$ be as in Example \ref{pico} and let $C_{\phi}$
be the corresponding composition operator in
$L^2(\mu)$. Then $C_{\phi}\in \ogr{L^2(\mu)}$ and
$C_{\phi}$ is a rank-one operator which is not
paranormal. Plainly, the condition \mbox{(ii-a)} of
Theorem \ref{cohypon-m} holds. Since $\mathsf h_{\phi}
= \frac{\mu(X)}{\mu(0)} \, \chi_{\{0\}}$, we see that
$\mathsf{h}_{\phi} \Le \mathsf{h}_{\phi} \circ \phi$,
which means that the condition \mbox{(ii-c)} of
Theorem \ref{cohypon-m} is satisfied as well. Since
$C_{\phi}$ is compact, it is not cohyponormal. Indeed,
otherwise, by \cite[Problem 206]{hal4}, $C_{\phi}$ is
normal and, as such, it is paranormal, which is a
contradiction. (That $C_{\phi}$ is not cohyponormal
can be justified in a more elementary manner by
exploiting the fact that $C_{\phi}$ is a rank-one
operator.) By Theorem \ref{cohypon-m}, the condition
\mbox{(ii-b)} of Theorem \ref{cohypon-m} does not
hold. The same can be deduced from Lemma
\ref{cohypon1}(v) by noticing that the function
$f=\chi_{\{0\}}$ satisfies the equality
$f=\chi_{\{\hfw > 0\}} \cdot f$, though it is not of
the form $f=g\circ \phi$ for any $g\colon X \to
\rbb_+$. Since $\|C_{\phi}\|^2 =
\frac{\mu(X)}{\mu(0)}$, we can immediately obtain an
example of an unbounded composition operator
possessing all the required properties simply by
considering a countably infinite orthogonal sum of
bounded composition operators discussed above (cf.\
\cite[Corollary C.2]{b-j-j-sS}).
   \end{exa}
   \begin{exa}[Condition \mbox{(ii-c)} is missing]
\label{cohyp-count3}
   Set $X=\zbb$ and $\ascr=2^X$. Let $\mu\colon \ascr
\to \rbop$ be the counting measure, $\phi$ be the
transformation of $X$ given by $\phi(n)=n+1$ for $n\in
X$ and $w\colon X \to \cbb\setminus \{0\}$ be any
function. Clearly, the weighted composition operator
$\cfw$ is well-defined and
   \begin{align} \label{slon1}
\hfw(n)=|w(n-1)|^2, \quad n\in \zbb,
   \end{align}
so $\cfw$ is densely defined. The condition
\mbox{(ii-a)} of Theorem \ref{cohypon-m} is obviously
satisfied. As $\phi^{-1}(\ascr) = \ascr$, we see that
$\efw = I_{L^2(\mu_w)}$, and consequently the
condition \mbox{(ii-b)} of Theorem \ref{cohypon-m} is
satisfied as well. It follows from \eqref{slon1} that
the condition \mbox{(ii-c)} of Theorem \ref{cohypon-m}
holds if and only if $|w(n-1)| \Le |w(n)|$ for all
$n\in \zbb$. Choosing an appropriate $w$, we easily
obtain the example with the required properties.
   \end{exa}
Now we demonstrate how restrictive the condition
\mbox{(ii-c)} of Theorem \ref{cohypon-m} really~is.
   \begin{rem} \label{howres}
Let $\slam$ be a weighted shift on a rootless and
leafless directed tree $\tcal=(V,E)$ with nonzero
weights $\lambdab=\{\lambda_v\}_{v \in V} \subseteq
\cbb$. Assume that $\card{V} = \aleph_0$. As we know
(see Sect.\ \ref{pco}(e)), $\slam=\cfw$, where $X=V$,
$\ascr=2^V$, $\mu$ is the counting measure,
$\phi(u)=\pa{u}$ and $w(u)=\lambda_u$ for all $u\in
V$. Clearly, the condition \mbox{(ii-a)} of Theorem
\ref{cohypon-m} is satisfied. Since $\at{\mu_w} = X$,
we deduce that
   \begin{align} \label{okno}
\phi^{-1}(\ascr) = \sigma(\{\dzi{u}\colon u \in V\}).
   \end{align}
Since $\phi$ is surjective, we see that $\hfw(x)>0$
for every $x\in X$. Hence, by Lemma
\ref{cohypon1}(iv), the condition \mbox{(ii-b)} of
Theorem \ref{cohypon-m} holds if and only if $\ascr =
\phi^{-1}(\ascr)$. In view of \eqref{okno} and
\cite[Proposition 2.1.2]{j-j-s}, the latter is valid
if and only if $\card{\dzi{u}} = 1$ for every $u \in
V$. This in turn is equivalent to the fact that
$\tcal$ is graph-isomorphic to the directed tree
$(\zbb, \{(n,n+1)\colon n\in \zbb\})$ (see the second
paragraph of the proof of \cite[Theorem
5.2.2]{j-j-s}).
   \end{rem}
Our next goal is to discuss issues related to the
selfadjointness of weighted composition operators. To
begin with, note that if a composition operators
$C_{\phi}$ is symmetric, then, by \eqref{chisom} and
\cite[Proposition B.1]{b-j-j-sS}, $\mathsf{h}_{\phi} =
1$ a.e.\ $[\mu]$. This is no longer true for symmetric
weighted composition operators. Indeed, it is enough
to consider the multiplication operator $M_{w}$ by an
$\ascr$-measurable function $w\colon X \to \rbb$. This
operator is selfadjoint and $\hfw = |w|^2$ a.e.\
$[\mu]$ (see Lemma \ref{psa-add}).

The next two examples show that neither of the
conditions \mbox{(ii-a)} and \mbox{(ii-b)} of Theorem
\ref{self} is sufficient for $\cfw$ to be selfadjoint
even if $\cfw$ has some additional properties.
   \begin{exa} \label{self-72}
Set $X=\zbb$ and $\ascr=2^X$. Let $\mu\colon \ascr \to
\rbop$ be a $\sigma$-finite measure such that
$\at{\mu} = X$, $\phi$ be the transformation of $X$
given by $\phi(n)=n+1$ for $n\in X$ and $w\colon X \to
(0,\infty)$ be any function. Clearly, $\cfw$ is
densely defined. Note that the condition \mbox{(ii-b)}
of Theorem \ref{self} is not satisfied because
otherwise by Corollary \ref{sasq}(iii) and the fact
that $\at{\mu_w}=X$, we see that $\id_X = \id_X \circ
\phi^2$, which is a contradiction. The condition
\mbox{(ii-a)} of Theorem \ref{self} takes the
following form
   \begin{align} \label{slon2}
\frac{w(n) \, w(n+1)}{w(n-1)^2} = \frac
{\mu(n-1)}{\mu(n)}, \quad n \in \zbb.
   \end{align}
Given any function $w\colon X \to (0,\infty)$, we can
solve the functional equation \eqref{slon2} with
respect to $\mu$ (by fixing the value of $\mu$ at any
point of $\zbb$). For such $\mu$, the condition
\mbox{(ii-a)} of Theorem \ref{self} is satisfied and,
by this theorem, $\cfw$ is not selfadjoint. Since, by
the condition \mbox{(ii-a)} of Theorem \ref{self},
$\hfw(n)=w(n) \, w(n+1)$ for all $n\in \nbb$, we
deduce that $\{\hfw(n)\}_{n\in \zbb}$ can be an
arbitrary two-sided sequence of positive real numbers.
This means that $\cfw$ can be either a bounded or an
unbounded operator. What is more, if
$\hfw=\boldsymbol{1}$, then by \eqref{chisom}, $\cfw$
is an isometry. Since $\Big\{\frac{1}{\sqrt{\mu(n)}}
\, \chi_{\{n\}}\colon n\in \zbb\Big\}$ is an
orthonormal basis of $L^2(\mu)$ and
   \begin{align*}
\cfw(\chi_{\{n\}}) = w(n-1) \, \chi_{\{n-1\}}, \quad
n\in \zbb,
   \end{align*}
we conclude that $\cfw$ is in fact a unitary operator.
   \end{exa}
   \begin{exa} \label{self-3}
We show here that the following phenomena related to
the selfadjointness of $\cfw$ can happen (cf.\
Propositions \ref{psa-pre} and \ref{wkwlp}):
   \begin{enumerate}
   \item[(A)] $\cfw$ is not selfadjoint,
the condition \mbox{(ii-b)} of Theorem \ref{self}
holds and $w = \sqrt{\hfw}$ a.e.\ $[\mu]$,
   \item[(B)] $\cfw$ is  bounded selfadjoint and
not positive, the equality $w = \sqrt{\hfw}$ a.e.\
$[\mu]$ does not hold, $C_{\phi}$ is a densely defined
nonsymmetric operator in $L^2(\mu_w)$ whose square
$C_{\phi}^2$ is not closed and $C_{\phi}^2
\varsubsetneq C_{\phi^2}$,
   \item[(C)] $\cfw$ is selfadjoint and
not positive, and $w=\sqrt{\hfw}$ a.e.\ $[\mu]$ (this
can be done even when $w=\boldsymbol{1}$).
   \end{enumerate}
For this purpose, set $X=\zbb_+$ and $\ascr=2^X$. Let
$w\colon X \to (0,\infty)$ be any function, $\mu\colon
\ascr \to \rbop$ be a $\sigma$-finite measure such
that $\at{\mu}=X$ and $\phi$ be a transformation of
$X$ given by $\phi(2k)=2k+1$ and $\phi(2k+1)=2k$ for
all $k\in \zbb_+$. Then $\phi$ is a bijection such
that $\phi^2=\id_X$. As a consequence, the condition
\mbox{(ii-b)} of Theorem \ref{self} is satisfied. Note
also that the condition \mbox{(ii-b)} of Theorem
\ref{psa} is not satisfied. Indeed, otherwise by
Corollary \ref{sasq}, we have $\id_X \circ \phi =
\id_X$, which is a contradiction. Clearly, $\cfw$ is
densely defined and
   \begin{align} \label{Vienna2}
\hfw(x)=
   \begin{cases}
\frac{w(2k+1)^2\mu(2k+1)}{\mu(2k)} & \text{if } x=2k
\text{ with } k\in\zbb_+,
   \\[1ex]
\frac{w(2k)^2\mu(2k)}{\mu(2k+1)} & \text{if } x=2k+1
\text{ with } k\in\zbb_+,
   \end{cases}
\quad x\in X.
   \end{align}
Hence, the condition \mbox{(ii-a)} of Theorem
\ref{self} takes the form (cf.\ \eqref{slon2})
   \begin{align} \label{Vienna1}
\frac{w(2k+1)}{w(2k)} = \frac {\mu(2k)}{\mu(2k+1)},
\quad k \in \zbb_+.
   \end{align}
In turn, the equality $w = \sqrt{\hfw}$ a.e.\ $[\mu]$
holds if and only if
   \begin{align}  \label{Vienna1+}
\bigg(\frac{w(2k+1)}{w(2k)}\bigg)^2 = \frac
{\mu(2k)}{\mu(2k+1)}, \quad k \in \zbb_+.
   \end{align}
Plainly, for any function $w\colon X \to (0,\infty)$
such that $w(0) \neq w(1)$, there exists a measure
$\mu$ that satisfies \eqref{Vienna1+} and does not
satisfy \eqref{Vienna1}. By Theorem \ref{self}, for
such $w$ and $\mu$, the condition (A) holds. In turn,
for any function $w\colon X \to (0,\infty)$ such that
$w(2k+1) = w(2k)$ for all $k \in \zbb_+$, there exists
a measure $\mu$ that satisfies \eqref{Vienna1} and
\eqref{Vienna1+}. Since the condition \mbox{(ii-b)} of
Theorem \ref{psa} is not satisfied, we infer from
Theorems \ref{self} and \ref{psa} that for such $w$
and $\mu$, the condition (C) holds. In particular,
choosing $w=\boldsymbol{1}$, we get the composition
operator that satisfies (C).

Let us now observe that each of the items (A) and (C)
can be realized by bounded and also by unbounded
operators $\cfw$. Indeed, in each of these items we
can assume that the equalities in \eqref{Vienna1} and
\eqref{Vienna1+} hold for every $k\in \nbb$. Then, by
\eqref{Vienna2}, we have
   \begin{align*}
\hfw(2k) = \hfw(2k+1) = w(2k)^2, \quad k\in\nbb,
   \end{align*}
which together with Proposition \ref{lemS1} enables us
to obtain weighted composition operators with the
required properties.

Finally, we discuss the condition (B). First, we
specify the weight $w$ and the measure $\mu$. Set
$w(2k)=1$ and $w(2k+1)=\frac{1}{k+1}$ for $k\in
\zbb_+$, and choose the measure $\mu$ such that
$\frac{\mu(2k)}{\mu(2k+1)} = \frac{1}{k+1}$ for all
$k\in \zbb_+$. It follows from \eqref{Vienna2} that
$\cfw \in \ogr{L^2(\mu)}$. As \eqref{Vienna1} is
valid, we infer from Theorem \ref{self} that $\cfw$ is
selfadjoint. Since \eqref{Vienna1+} does not hold, we
infer from Proposition \ref{psa-pre} that $\cfw$ is
not positive. Clearly, $C_{\phi}$ is densely defined
as an operator in $L^2(\mu_w)$ and
   \begin{align*}
\mathsf{h}_{\phi} (2k+1) = k+1, \quad k\in \zbb_+,
   \end{align*}
where $\mathsf{h}_{\phi} := \frac{\D \mu_w \circ
\phi^{-1}}{\D \mu_w}$, and so by Proposition
\ref{lemS1}, $C_{\phi}$ is not bounded. Hence, since
$C_{\phi}$ is closed and $C_{\phi^2} =
I_{L^2(\mu_w)}$, we deduce that $C_{\phi}^2
\varsubsetneq C_{\phi^2}$. By Lemma
\ref{13-12-13}(iii), $C_{\phi}^2$ is not closed. In
turn, by Corollary \ref{jfacor}(i), $C_{\phi}$ is not
symmetric.
   \end{exa}
In closing this section, we make a few comments on
selfadjoint weighted composition operators.
   \begin{rem}
First, note that the situation described in the
condition (A) of Example \ref{self-3} never happens in
the case of composition operators $C_{\phi}$. Indeed,
otherwise the condition \mbox{(ii-a)} of Theorem
\ref{self} is automatically satisfied and thus, by
this theorem, $C_{\phi}$ is selfadjoint, which is a
contradiction. Second, in view of Corollary
\ref{jfacor}(i) and \eqref{chisom}, the situation
described in the condition (B) of Example \ref{self-3}
also never happens in the case of composition
operators. Third, the condition \mbox{(ii-b)} of
Theorem \ref{psa} is not sufficient for $\cfw$ to be
positive and selfadjoint. Indeed, if $\mu_w$ is
nonzero and the condition \mbox{(ii-b)} of Theorem
\ref{psa} holds, then by Corollary \ref{psa-add-c}, we
can find an $\ascr$-measurable function $u\colon X \to
\cbb$ such that $\mu(u^{-1}(\cbb\setminus \rbb)) > 0$,
$|w|=|u|$ and $C_{\phi,u}=M_u$. Since $\mu_w=\mu_u$,
the condition \mbox{(ii-b)} of Theorem \ref{psa} holds
with $u$ in place of $w$ and $C_{\phi,u}$ is neither
selfadjoint nor positive. Fourth, also the condition
\mbox{(ii-a)} of Theorem \ref{psa} is not sufficient
for $\cfw$ to be selfadjoint and positive (cf.\
Proposition \ref{fn2nC}). This can be seen even in the
case of composition operators (cf.\ \cite[Example
B.2]{b-j-j-sS}). Fifth, one can show, by modifying
\cite[Example 3.2]{Bu-St1}, that the equality
$C_{\phi} = I_{L^2(\mu)}$ does not imply that $\phi =
\id_X$ a.e.\ $[\mu]$; what is worse, it may happen
that the set $\{x \in X\colon \phi(x) = x\}$ is not
even $\ascr$-measurable.
   \end{rem}
   \chapter{Relationships Between $\cfw$
and $C_\phi$} In this chapter, we investigate the
interplay between selected properties of a weighted
composition operator $\cfw$ and the corresponding
composition operator $C_\phi$. In Section
\ref{Sec7.1}, we discuss the questions of when the
product $M_w C_{\phi}$ is closed and when it coincides
with $\cfw$ (see Theorems \ref{mwc2} and \ref{mwc4}).
The relationships between the Radon-Nikodym
derivatives $\mathsf h_{\phi}$ and $\hfw$ are
described in Section \ref{13IX} (see Propositions
\ref{caroll} and \ref{caroll2a}). In Section
\ref{Sec7.3}, using a result due to Berg and
Dur\'{a}n, we give conditions enabling us to deduce
the subnormality of $\cfw$ from that of $C_{\phi}$
(see Theorem \ref{be-du}). The converse possibility is
discussed in Theorem \ref{be-du2}. In Section
\ref{Sec7.4}, we provide a criterion for a bounded
weighted composition operator with matrix symbol to be
subnormal (see Theorem \ref{wcomc}). We conclude the
chapter with Section \ref{Sec7.5} containing numerous
examples illustrating our considerations.

The reader should be aware that in this chapter we
apply our previous results, not only to $\cfw$, but
also to $C_\phi$. Recall that the Radon-Nikodym
derivative $\mathsf h_{\phi}$ and the conditional
expectation $\mathsf{E}_\phi$ correspond to
$C_{\phi}$.
   \section{\label{Sec7.1}$M_w C_\phi$ versus $\cfw$}
We begin by stating the most basic relationship
between them. Recall that the composition operator
$C_{\phi}$ may not be well-defined even if the
weighted composition operator $\cfw$ is an isometry
(see Example \ref{cohyp-count2}; see also Examples
\ref{mwc6} and \ref{ciach}).
   \begin{pro} \label{mwc1}
Suppose \eqref{stand1} holds. Then the following
assertions hold{\em :}
   \begin{enumerate}
   \item[(i)] if the composition operator $C_{\phi}$ is
well-defined, then $\cfw$ is well-defined and $M_w
C_{\phi} \subseteq \cfw$,
   \item[(ii)] if $w\neq 0$ a.e.\ $[\mu]$ and $\cfw$ is
well-defined, then $C_{\phi}$ is well-defined,
   \item[(iii)] if $\cfw$ is well-defined and has
dense range, then $C_{\phi}$ is well-defined.
   \end{enumerate}
   \end{pro}
   \begin{proof}
(i)\&(ii) This essentially follows from Proposition
\ref{wco1}.

(iii) Suppose, contrary to our claim, that $C_{\phi}$
is not well-defined. Hence, there exists $\varDelta
\in \ascr$ such that $\mu(\varDelta)=0$ and
$\mu(\phi^{-1}(\varDelta)) > 0$. By the
$\sigma$-finiteness of $\mu$, there exists a sequence
$\{X_n\}_{n=1}^{\infty} \subseteq \ascr$ such that
$X_n \nearrow X$ as $n\to\infty$ and $\mu(X_n) <
\infty$ for all $n\in \nbb$. By continuity of measure,
there exists $N \in \nbb$ such that
   \begin{align}  \label{Palg1}
0 < \mu(\phi^{-1}(\varDelta) \cap X_{N})< \infty.
   \end{align}
Since $\cfw$ is well-defined, we have
   \begin{align*}
0 = \mu_w(\phi^{-1}(\varDelta)) =
\int_{\phi^{-1}(\varDelta)} |w|^2 \D\mu,
   \end{align*}
which shows that $w=0$ on $\phi^{-1}(\varDelta)$ a.e.\
$[\mu]$. Therefore, we have
   \begin{align} \label{Palg2}
\text{$w=0$ on $\phi^{-1}(\varDelta) \cap X_{N}$ a.e.\
$[\mu]$.}
   \end{align}
Set $g = \chi_{\phi^{-1}(\varDelta) \cap X_{N}}$.
Then, by \eqref{Palg1} and \eqref{Palg2}, $g \in
L^2(\mu)\setminus \{0\}$ and
   \begin{align*}
\is{\cfw f}{g} = \int_{\phi^{-1}(\varDelta) \cap
X_{N}} w \cdot f\circ \phi \D\mu = 0, \quad f \in
\dz{\cfw}.
   \end{align*}
This means that $g \perp \ob{\cfw}$, which is a
contradiction.
   \end{proof}
It is to be emphasized that in general the operators
$\cfw$ and $M_w C_{\phi}$ may not coincide even if the
composition operator $C_{\phi}$ is densely defined
(see Example \ref{mwc8}). Theorem \ref{mwc2} below
provides a necessary and sufficient condition for the
operators $\cfw$ and $M_w C_{\phi}$ to be equal. Note
that in the proofs of Theorems \ref{mwc2} and
\ref{mwc4} and Lemma \ref{l2g1} we use the fact that
intersections and inclusions of $L^2$-spaces make
sense whenever the underlying measures are mutually
absolutely continuous.
   \begin{thm}\label{mwc2}
Suppose \eqref{stand1} holds and $C_{\phi}$ is
well-defined. Then $\cfw$ is well-defined and the
following two conditions are equivalent{\em: }
   \begin{enumerate}
   \item[(i)] $\cfw = M_w C_{\phi}$,
   \item[(ii)] there exists $c\in \rbb_+$ such that
$\mathsf h_{\phi} \Le c(1+\hfw)$ a.e.\ $[\mu]$.
   \end{enumerate}
   \end{thm}
   \begin{proof}
It follows from Proposition \ref{mwc1} that $C_{\phi}$
is well-defined and $M_w C_{\phi} \subseteq \cfw$.
Since, by Proposition \ref{lemS1}(i), $\dz{\cfw} =
L^2((1+\hfw)\D \mu)$ and $\dz{M_w C_{\phi}} =
L^2((1+\mathsf h_{\phi})\D \mu) \cap L^2((1+\hfw)\D
\mu)$, we deduce that $\cfw = M_w C_{\phi}$ if and
only if $L^2((1+\hfw)\D \mu) \subseteq L^2((1+\mathsf
h_{\phi})\D \mu)$. Now applying \cite[Corollary
12.4]{b-j-j-sC} completes the proof of the equivalence
(i)$\Leftrightarrow$(ii).
   \end{proof}
   As shown below, the operators $\cfw$ and $M_w
C_{\phi}$ coincide if either $C_{\phi}$ is a bounded
operator on $L^2(\mu)$ or if $M_w$ is bounded from
below.
   \begin{pro}\label{mwc3}
Suppose \eqref{stand1} holds. Then the following
assertions hold{\em :}
   \begin{enumerate}
   \item[(i)] if $C_{\phi} \in \ogr{L^2(\mu)}$,
then $\cfw$ is well-defined and $\cfw = M_w C_{\phi}$,
   \item[(ii)] if $\alpha:=\mu\text{-}\mathrm{ess\,inf} |w| > 0$
and $\cfw$ is well-defined, then $C_{\phi}$ is
well-defined, $\cfw = M_w C_{\phi}$ and
   \begin{align} \label{tymczas}
h_{\phi} \Le \frac{1}{\alpha^2} \, \hfw \text{ a.e. }
[\mu];
   \end{align}
if, moreover, $\cfw$ is densely defined $($resp.,
$\cfw$ is in $\ogr{L^2(\mu)}$$)$, then so is
$C_{\phi}$,
   \item[(iii)] if $M_w \in \ogr{L^2(\mu)}$ and
$C_{\phi}$ is well-defined, then $\cfw$ is
well-defined and
   \begin{align*}
\hfw \Le \|M_w\|^2 \mathsf \, h_{\phi} \text{ a.e. }
[\mu];
   \end{align*}
if, moreover, $C_{\phi}$ is densely defined
$($$C_{\phi}$ is in $\ogr{L^2(\mu)}$$)$, then so is
$\cfw$.
   \end{enumerate}
   \end{pro}
   \begin{proof}
The assertion (i) is a consequence of Proposition
\ref{lemS1}(v) and Theorem \ref{mwc2}. In turn,
applying \eqref{l1} twice, we get
   \begin{align*}
\int_{\varDelta} \mathsf{h}_{\phi} \D\mu = \mu\circ
\phi^{-1}(\varDelta) \Le \frac{1}{\alpha^2}
\int_{\varDelta} \chi_{\varDelta}\circ \phi \D \mu_w =
\frac{1}{\alpha^2} \int_{\varDelta} \hfw \D \mu, \quad
\varDelta\in \ascr,
   \end{align*}
which together with \cite[Theorem 1.6.11]{Ash} yields
\eqref{tymczas}. The ``moreover'' part of (ii) is a
direct consequence of \eqref{tymczas} and Propositions
\ref{lemS1} and \ref{lemS2}. The assertion (iii) can
be proved in a similar way.
   \end{proof}
Our next goal is to answer the question of when the
product $M_wC_{\phi}$ is closed. The solution given
below resembles the characterization of when the
operators $\cfw$ and $M_wC_{\phi}$ coincide (see
Theorem \ref{mwc2}).
   \begin{thm} \label{mwc4}
Suppose \eqref{stand1} holds and $C_{\phi}$ is
well-defined. Then $\cfw$ is well-defined and the
following two conditions are equivalent{\em :}
   \begin{enumerate}
   \item[(i)] $M_w C_{\phi}$  is a closed operator,
   \item[(ii)] there exists $c\in \rbb_+$ such that
$\mathsf h_{\phi} \Le c(1+\hfw)$ a.e.\ $[\mu]$ on
$\{\mathsf h_{\phi} < \infty\}$.
   \end{enumerate}
   \end{thm}
Before proving Theorem \ref{mwc4}, we need the
following auxiliary lemma which seems to be of
independent interest. Note that this lemma is no
longer true if we drop the requirement that the
measure $\nu$ be $\sigma$-finite (see \cite[Appendix
A]{b-j-j-sC} for a counterexample as well as for
related results).
   \begin{lem} \label{l2g1}
Let $(Y, \bscr, \nu)$ be a $\sigma$-finite measure
space and let $g_1, g_2$ be $\bscr$-measurable scalar
functions on $Y$ such that $0 < g_1 \Le g_2 \Le
\infty$ a.e.\ $[\nu]$. Then $L^2(g_2\D\mu)$ is a
closed subspace of $L^2(g_1 \D\mu)$ if and only if
there exists $c\in \rbb_+$ such that $g_2 \Le c \,
g_1$ a.e.\ $[\nu]$ on $\{g_2 < \infty\}$.
   \end{lem}
   \begin{proof}
Set $E=\{g_2 < \infty\}$. First, we note that $L^2(g_2
\D \nu)$ is a vector subspace of $L^2(g_1 \D \nu)$ and
   \begin{align} \label{fonta}
   \begin{aligned}
L^2(g_2 \D \nu) & = \chi_{E} \cdot L^2(g_2 \D \nu),
   \\
L^2(g_1 \D \nu) & = \chi_{E} \cdot L^2(g_1 \D \nu)
\oplus \chi_{Y \setminus E} \cdot L^2(g_1 \D \nu),
   \end{aligned}
   \end{align}
where the orthogonality refers to the inner product of
$L^2(g_1 \D \nu)$. Since the mapping
   \begin{align*}
U\colon \chi_{E} \cdot L^2(g_1 \D \nu) \ni f
\longmapsto f|_{E} \in L^2(E, g_1 \D \nu)
   \end{align*}
is a well-defined unitary isomorphism such that
   \begin{align*}
U\big(\chi_{E} \cdot L^2(g_2 \D \nu)\big) = L^2(E, g_2
\D \nu),
   \end{align*}
and, by \cite[Lemma 12.1]{b-j-j-sC}, $L^2(E, g_2 \D
\nu)$ is dense in $L^2(E, g_1 \D \nu)$, we deduce from
\eqref{fonta} that $L^2(g_2 \D \nu)$ is a closed
subspace of $L^2(g_1 \D \nu)$ if and only if $L^2(E,
g_2 \D \nu) = L^2(E, g_1 \D \nu)$. An application of
\cite[Lemma 12.4]{b-j-j-sC} completes the proof.
   \end{proof}
   \begin{proof}[Proof of Theorem \ref{mwc4}]
By Proposition \ref{mwc1}, $\cfw$ is well-defined. Set
$g_1 = 1 + \hfw$ and $g_2 = 1 + \mathsf h_{\phi} +
\hfw$. It follows from Proposition \ref{lemS1} that
   \begin{align*}
\dz{M_w C_{\phi}} = L^2((1+\mathsf h_{\phi})\D \mu)
\cap L^2((1+\hfw)\D \mu) = L^2(g_2 \D \mu),
   \end{align*}
and that $\dz{\cfw}$ equipped with the graph norm
\mbox{$\|\cdot\|_{\cfw}$} coincides with the Hilbert
space $L^2((1+\hfw)\D \mu)$. Since $M_w C_{\phi}
\subseteq \cfw$, we deduce that the product $M_w
C_{\phi}$ is closed if and only if $L^2(g_2 \D \mu)$
is a closed subspace of $L^2(g_1 \D \mu)$. An
application of Lemma \ref{l2g1} completes the proof.
   \end{proof}
The relationship between the closedness of the product
$M_wC_{\phi}$ and the equality $\cfw = M_w C_{\phi}$
is explained in Proposition \ref{fonta2} below. As
shown in Example \ref{mwc5}, the implication
(ii)$\Rightarrow$(i) in Proposition \ref{fonta2} is
false if the assumption that $C_{\phi}$ is densely
defined is dropped (however the reverse implication is
true if $C_{\phi}$ is well-defined).
   \begin{pro} \label{fonta2}
Suppose \eqref{stand1} holds and $C_{\phi}$ is densely
defined. Then $\cfw$ is well-defined and the following
two conditions are equivalent{\em :}
   \begin{enumerate}
   \item[(i)] $\cfw = M_w C_{\phi}$,
   \item[(ii)] $M_w C_{\phi}$  is a closed operator.
   \end{enumerate}
   \end{pro}
   \begin{proof}
Propositions \ref{wco1} and \ref{mwc1} guarantee that
$\cfw$ is well-defined and \eqref{stand2} holds. The
implication (i)$\Rightarrow$(ii) is a direct
consequence of Proposition \ref{lemS1}(iv), while the
implication (ii)$\Rightarrow$(i) follows from Theorem
\ref{mwc4}, Propositions \ref{wco1} and \ref{lemS2}
applied to $w=\boldsymbol{1}$ and Theorem \ref{mwc2}.
   \end{proof}
Concluding this section, we provide a factorization of
a weighted composition operator into the product of a
unitary multiplication operator and a weighted
composition operator with nonnegative weight.
   \begin{pro}
Suppose \eqref{stand1} holds. Then the following
assertions hold{\em :}
   \begin{enumerate}
   \item[(i)] $\mu_{w} = \mu_{|w|}$, $\mu_{w} \circ
\phi^{-1} = \mu_{|w|} \circ \phi^{-1}$ and
$\hfw=\hfwm$ a.e.\ $[\mu]$,
   \item[(ii)] $\cfw$ is well-defined if and only if $\cfwm$ is
well-defined,
   \item[(iii)] if $\cfw$ is well-defined, then $($see Section
{\em \ref{pco}(a)} for notation$)$
   \begin{align}  \label{dysz}
\cfw = M_{\mathrm{sign}(w)} \cfwm,
   \end{align}
where
   \begin{align*}
(\mathrm{sign}(w))(x) =
   \begin{cases}
1 & \text{if } w(x)=0,
   \\[.3ex]
\frac{w(x)}{|w(x)|} & \text{if } w(x)\neq 0;
   \end{cases}
   \end{align*}
in particular, $\dz{\cfw} = \dz{\cfwm}$ and
   \begin{align}  \label{font1}
\dim \overline{\ob{\cfw}} & = \dim
\overline{\ob{\cfwm}},
   \\ \label{font2}
\dim \overline{\ob{\cfw}}^{\perp} & = \dim
\overline{\ob{\cfwm}}^{\perp},
   \end{align}
   \item[(iv)]
if $\cfw$ is well-defined, then
   \begin{align*}
\jd{\cfw} = \jd{\cfwm},
   \end{align*}
   \item[(v)] $\cfw$ is densely defined if and only if $\cfwm$ is
densely defined; moreover, if $\cfw$ is densely
defined, then
   \begin{align*}
\dim \jd{\cfw^*} = \dim \jd{\cfwm^*}.
   \end{align*}
   \end{enumerate}
   \end{pro}
   \begin{proof}
The assertion (i) is obviously true. The assertions
(ii) and (iv) follow from (i), Proposition \ref{wco1}
and Lemma \ref{jadro}.

(iii) It is a simple matter to verify that $\dz{\cfw}
= \dz{\cfwm}$ and that \eqref{dysz} holds. Since
$M_{\mathrm{sign}(w)}$ is a unitary operator, we
obtain \eqref{font1} and \eqref{font2}.

The ``moreover'' part of (v) is a direct consequence
of \eqref{font2}. The remaining part of (v) is a
direct consequence of the equality $\dz{\cfw} =
\dz{\cfwm}$.
   \end{proof}
It is worth pointing out that, despite of the equality
\eqref{dysz}, the weighted composition operators
$\cfw$ and $\cfwm$ may not be unitarily equivalent (or
even similar). To see this it is enough to consider
multiplication operators.
   \section{\label{13IX}Radon-Nikodym derivative
and conditional expectation}
   We begin by expressing the Radon-Nikodym derivative
$\hfw$, that corresponds to $\cfw$, in terms of the
Radon-Nikodym derivative $\mathsf{h}_{\phi}$ and the
conditional expectation $\mathsf{E}_\phi(\cdot)$, that
correspond to $C_{\phi}$. This can be done under the
weakest possible assumption that the conditional
expectation $\mathsf{E}_\phi(\cdot)$ exists, or
equivalently, by Proposition \ref{lemS2}, that
$C_{\phi}$ is densely defined. Surprisingly, it can
happen that $\dz{\cfw}=\{0\}$ even if
$C_{\phi}\in\ogr{L^2(\mu)}$ (see Example
\ref{ciach4}). Fortunately, if the weight function $w$
satisfies some natural constraints, then $\cfw$ is
densely defined (see the assertion (iii) of
Proposition \ref{mwc3}).

Since in this section we investigate powers of
weighted composition operators, it is worth making a
comment on the existence of the conditional
expectation $\mathsf{E}_{\phi^n,\widehat w_n}$ for a
fixed $n\in \nbb$. For this purpose, assume that
$\cfw$ is well-defined. Then, by Lemma \ref{lemS11},
the operator $C_{\phi^n,\widehat w_n}$ is well-defined
and $\cfw^n \subseteq C_{\phi^n,\widehat w_n}$. As we
know, the conditional expectation
$\mathsf{E}_{\phi^n,\widehat w_n}$ exists if and only
if the measure $\mu_{\widehat
w_n}|_{\phi^{-n}(\ascr)}$ is $\sigma$-finite, or
equivalently, by Proposition \ref{lemS2} applied to
$(\phi^{n},\widehat w_n)$, if and only if the operator
$C_{\phi^n,\widehat w_n}$ is densely defined. Thus the
following holds.
   \begin{align} \label{mgla}
   \begin{minipage}{70ex}
{\em If $\cfw$ is well-defined and $\cfw^n$ is densely
defined for a fixed $n\in \nbb$, then
$C_{\phi^k,\widehat w_k}$ is densely defined and the
conditional expectation $\mathsf{E}_{\phi^k,\widehat
w_k}$ exists for all $k\in \{0,\ldots,n\}$.}
   \end{minipage}
   \end{align}
We also refer the reader to Lemma \ref{lemS11p}(vi),
Proposition \ref{potegi-p} and Example \ref{powers}
for more information on the question of density of
domains of the operators $\cfw^n$ and
$C_{\phi^n,\widehat w_n}$.

In this section, we will frequently use the fact,
without explicitly mentioning it, that a weighted
composition operator $\cfw$ is well-defined if and
only if $\mu_{w} \circ \phi^{-1} \ll \mu$ (see
Proposition \ref{wco1}). It is worth mentioning that
the right-hand side of the formula \eqref{poz1} below
is denoted in \cite[Sect.\ 6]{ca-hor} by $J$ and that
the measure $\D \nu := J \D \mu$ considered therein
coincides with our $\mu_w\circ \phi^{-1}$.
   \begin{pro} \label{caroll}
Suppose \eqref{stand1} holds and $C_{\phi}$ is densely
defined. Then $C_{\phi,w}$ is well-defined and
   \begin{align} \label{poz1}
\hfw = \mathsf{h}_{\phi} \cdot \mathsf{E}_\phi
(|w|^2)\circ \phi^{-1} \text{ a.e.\ $[\mu]$},
   \end{align}
where $\mathsf{E}_{\phi}(\cdot) :=
\mathsf{E}(\,\cdot\,;\phi^{-1}(\ascr),\mu)$.
   \end{pro}
   \begin{proof} In view of Proposition \ref{lemS2},
the measure $\mu|_{\phi^{-1}(\ascr)}$ is
$\sigma$-finite and consequently the conditional
expectation $\mathsf{h}_{\phi}(\cdot)$ exists. Hence,
by \eqref{B3}, \eqref{fifi} and \eqref{l2}, we have
   \allowdisplaybreaks
   \begin{align*}
\mu_{w} \circ \phi^{-1} (\varDelta) & = \int_X
\chi_{\varDelta} \circ \phi \cdot |w|^2 \D \mu
   \\
& = \int_X \chi_{\varDelta} \circ \phi \cdot
\mathsf{E}_{\phi}(|w|^2) \D \mu
   \\
& = \int_X \chi_{\varDelta} \circ \phi \cdot
(\mathsf{E}_{\phi}(|w|^2) \circ \phi^{-1}) \circ \phi
\D \mu
   \\
& = \int_{\varDelta}\mathsf{h}_{\phi} \cdot
\mathsf{E}_{\phi}(|w|^2)\circ \phi^{-1} \D \mu, \quad
\varDelta \in \ascr.
   \end{align*}
This implies that $\mu_{w} \circ \phi^{-1} \ll \mu$,
the equality \eqref{poz1} holds and the weighted
composition operator $C_{\phi,w}$ is well-defined
(cf.\ Proposition \ref{mwc1}).
   \end{proof}
   \begin{cor} \label{caroll2}
If \eqref{stand1} holds and $C_{\phi}^n$ is densely
defined for every $n\in \nbb$, then $C_{\phi^n,
\widehat w_n}$ is well-defined for every $n\in \nbb$
and
   \begin{align} \label{poz1-1}
\hfwn{n} = \mathsf{h}_{\phi^n} \cdot \cen{n}
(|\widehat w_n|^2)\circ \phi^{-n} \text{ a.e.\
$[\mu]$}, \quad n \in \zbb_+,
   \end{align}
where $\cen{n}(\cdot) =
\mathsf{E}(\,\cdot\,;\phi^{-n}(\ascr),\mu)$ and
$\widehat w_n$ is as in \eqref{Zak}.
   \end{cor}
   \begin{proof}
Apply Proposition \ref{caroll} to $(\phi^n,\widehat
w_n)$ in place of $(\phi,w)$ and use \eqref{mgla}.
   \end{proof}
The sequence $\{\cen{n} (|\widehat w_n|^2)\circ
\phi^{-n}\}_{n=0}^{\infty}$ which appears in Corollary
\ref{caroll2} can be described by the following
recurrence relation.
   \begin{lem} \label{recfor}
Suppose $(X,\ascr,\mu)$ is a $\sigma$-finite measure
space and $\phi$ is an $\ascr$-measurable
transformation of $X$ such that $C_{\phi}$ is
well-defined and $C_{\phi}^n$ is densely defined for
some fixed $n\in \nbb$. Then
   \begin{align*}
\gamma_{k+1} = \cen{k+1}\big(\gamma_k \circ \phi^k
\cdot |w \circ \phi^k|^2 \big) \circ \phi^{-(k+1)}
\text{ a.e.\ $[\mu]$}, \quad k \in \{0,\ldots, n-1\},
   \end{align*}
where $\gamma_k := \cen{k} (|\widehat w_k|^2)\circ
\phi^{-k}$ for $k\in \{0,\ldots,n\}$.
   \end{lem}
   \begin{proof} It follows from \eqref{mgla} that
the conditional expectation $\cen{k}$ exists for every
$k\in \{0,\ldots,n\}$. In turn, the inclusion
$\phi^{-(k+1)}(\ascr) \subseteq \phi^{-k}(\ascr)$,
which is valid for every $k\in \zbb_+$, yields
   \begin{align*}
\gamma_{k+1} &= \cen{k+1}\big(|\widehat w_k|^2 \cdot
|w\circ \phi^k|^2\big) \circ \phi^{-(k+1)}
   \\
&\hspace{-2ex}\overset{\eqref{Dag9}}=
\cen{k+1}\Big(\cen{k}\big(|\widehat w_k|^2 \cdot
|w\circ \phi^k|^2\big)\Big) \circ \phi^{-(k+1)}
   \\
&\hspace{-2ex}\overset{\eqref{B6}}=
\cen{k+1}\Big(\cen{k}(|\widehat w_k|^2\big) \cdot
|w\circ \phi^k|^2\Big) \circ \phi^{-(k+1)}
      \\
&\hspace{-.8ex}\overset{\eqref{fifi}}=
\cen{k+1}\big(\gamma_k \circ \phi^k \cdot |w \circ
\phi^k|^2\big) \circ \phi^{-(k+1)} \textit{ a.e.\
$[\mu]$}, \quad k \in \{0,\ldots, n-1\}.
   \end{align*}
This completes the proof.
   \end{proof}
Below, we express the Radon-Nikodym derivative
$\mathsf{h}_{\phi}$, that corresponds to $C_{\phi}$,
in terms of the Radon-Nikodym derivative $\hfw$ and
the conditional expectation $\efw(\cdot)$, that
correspond to $\cfw$. In contrast to the previous
case, now the weakest possible assumption that the
conditional expectation $\efw(\cdot)$ exists is not
sufficient for this purpose. In fact, it is not even
sufficient for $C_{\phi}$ to be well-defined, and,
what is worse, this can happen even if $\cfw$ is an
isometry (see Example \ref{cohyp-count2}; see also
Examples \ref{mwc6} and \ref{ciach}). However, even if
$C_{\phi}$ is well-defined, $C_{\phi}$ may have
trivial domain (see Example \ref{ciach3}).
Nevertheless, if the weight function $w$ satisfies
some constraints, then $C_{\phi}$ may be densely
defined (see the assertion (ii) of Proposition
\ref{mwc3}).
   \begin{pro} \label{caroll2a}
Suppose \eqref{stand1} holds, $w \neq 0$ a.e.\ $[\mu]$
and $\cfw$ is densely defined. Then $C_{\phi}$ is
well-defined and
   \begin{align} \label{poz2}
\mathsf{h}_{\phi} = \hfw \cdot \efw
\bigg(\frac{1}{|w|^2}\bigg)\circ \phi^{-1} \text{
a.e.\ $[\mu]$.}
   \end{align}
   \end{pro}
   \begin{proof} Arguing as in the proof of Proposition
\ref{caroll}, we see that
   \begin{align*}
\mu \circ \phi^{-1} (\varDelta) & = \int_X
\chi_{\varDelta} \circ \phi \cdot \frac{1}{|w|^2} \D
\mu_{w}
   \\
& = \int_X \chi_{\varDelta} \circ \phi \cdot \efw
\bigg(\frac{1}{|w|^2}\bigg) \circ \phi^{-1}\bigg)
\circ \phi \D \mu_{w}
   \\
& = \int_{\varDelta} \hfw \cdot
\efw\bigg(\frac{1}{|w|^2}\bigg) \circ \phi^{-1}\D \mu,
\quad \varDelta \in \ascr.
   \end{align*}
This completes the proof.
   \end{proof}
Applying Proposition \ref{caroll2a} to
$(\phi^n,\widehat w_n)$ in place of $(\phi,w)$ and
using Lemma \ref{potegi-p}, we get the following.
   \begin{cor} \label{caroll3}
Suppose \eqref{stand1} holds, $w \neq 0$ a.e.\ $[\mu]$
and $\cfw^n$ is densely defined for every $n\in \nbb$.
Then $\widehat w_n \neq 0$ a.e.\ $[\mu]$ and
$C_{\phi^n}$ is well-defined for every $n \in \zbb_+$,
and
   \begin{align} \label{poz2-2}
\mathsf{h}_{\phi^n} = \hfwn{n} \cdot
\mathsf{E}_{\phi^n, \widehat w_n}
\bigg(\frac{1}{|\widehat w_n|^2}\bigg)\circ \phi^{-n}
\text{ a.e.\ $[\mu]$,} \quad n\in \zbb_+.
   \end{align}
   \end{cor}
   Now, we provide formulas that connect the
conditional expectations $\mathsf{E}_{\phi}(\cdot)$
and $\efw(\cdot)$ calculated at $|w|^2$ and
$\frac{1}{|w|^2}$, respectively. This is done under
the weakest possible assumption that the conditional
expectations $\mathsf{E}_{\phi}(\cdot)$ and
$\efw(\cdot)$ exist and the measures $\mu$ and $\mu_w$
are mutually absolutely continuous.
   \begin{pro} \label{caroll4}
Suppose \eqref{stand1} holds, $w \neq 0$ a.e.\ $[\mu]$
and the operators $C_{\phi}$ and $\cfw$ are densely
defined. Then the following conditions hold{\em :}
   \begin{enumerate}
   \item[(i)] $\{\hfw > 0\} = \{\mathsf{h}_{\phi} > 0\}$
a.e.\ $[\mu]$,
   \item[(ii)] $\mathsf{E}_{\phi} (|w|^2)\circ \phi^{-1} \cdot
\efw \big(\frac{1}{|w|^2}\big)\circ \phi^{-1} =
\chi_{\{\mathsf{h}_{\phi} > 0\}}$ a.e.\ $[\mu]$,
   \item[(iii)] $\mathsf{E}_{\phi} (|w|^2) \cdot
\efw \big(\frac{1}{|w|^2}\big) = 1$ a.e.\ $[\mu]$,
   \item[(iv)] $\jd{C_{\phi}}=\jd{\cfw}$.
   \end{enumerate}
   \end{pro}
   \begin{proof}
It follows from \eqref{poz1} and \eqref{poz2} that
$\{\hfw > 0\} \subseteq \{\mathsf{h}_{\phi}
> 0\}$ a.e.\ $[\mu]$ and $\{\mathsf{h}_{\phi}
> 0\} \subseteq \{\hfw > 0\}$ a.e.\ $[\mu]$, which
yields (i). The assertion (ii) is a direct consequence
of (i), the definitions of $\mathsf{E}_{\phi}
(\cdot)\circ \phi^{-1}$ and $\efw(\cdot)\circ
\phi^{-1}$ and the equalities \eqref{poz1} and
\eqref{poz2} (use also Proposition \ref{lemS2}). The
assertion (iii) can be deduced from (ii) by applying
Lemma \ref{nuklear} and \eqref{fifi}. Finally, the
assertion (iv) follows from (i) and Lemma \ref{jadro}.
   \end{proof}
   \begin{cor} 
Suppose \eqref{stand1} holds, $w \neq 0$ a.e.\ $[\mu]$
and the operators $C_{\phi}^n$ and $\cfw^n$ are
densely defined for every $n\in \nbb$. Then the
following conditions hold{\em :}
   \begin{enumerate}
   \item[(i)] $\widehat w_n \neq 0$ a.e.\ $[\mu]$ for all $n\in \zbb_+$,
   \item[(ii)] $\{\hfwn{n} > 0\} = \{\mathsf{h}_{\phi^n} > 0\}$
a.e.\ $[\mu]$ for all $n\in \zbb_+$,
   \item[(iii)] $\cen{n} (|\widehat w_n|^2)\circ \phi^{-n} \cdot
\mathsf{E}_{\phi^n, \widehat w_n}
\big(\frac{1}{|\widehat w_n|^2}\big)\circ \phi^{-n} =
\chi_{\{\mathsf{h}_{\phi^n} > 0\}}$ a.e.\ $[\mu]$ for
all $n\in \zbb_+$,
   \item[(iv)] $\cen{n} (|\widehat w_n|^2) \cdot
\mathsf{E}_{\phi^n, \widehat w_n}
\big(\frac{1}{|\widehat w_n|^2}\big) = 1$ a.e.\
$[\mu]$ for all $n\in \zbb_+$.
   \end{enumerate}
Moreover, if the operators $C_{\phi}^n$ and $\cfw^n$
are closed for all integers $n\Ge 2$, then
   \begin{enumerate}
   \item[(v)] $\jd{C_{\phi}^n}=\jd{C_{\phi, w}^n}$
for all $n\in \zbb_+$.
   \end{enumerate}
   \end{cor}
   \begin{proof}
Applying Proposition \ref{caroll4} to $(\phi^n,
\widehat w_n)$ in place of $(\phi, w)$ and using
Proposition \ref{potegi-p} we obtain the conditions
(i)-(iv). Employing Proposition \ref{caroll4}(iv) and
using Lemma \ref{13-12-13}(iii) yields (v).
   \end{proof}
Let us make some comments regarding Proposition
\ref{caroll4} and its proof.
   \begin{rem}
First, note that the assertion (iii) can be proved
more directly. Namely, applying \eqref{B3} twice, we
get
   \begin{align*}
\int_{\phi^{-1}(\varDelta)} 1 \D\mu &=
\int_{\phi^{-1}(\varDelta)} \frac{1}{|w|^2}\D\mu_w =
\int_{\phi^{-1}(\varDelta)} |w|^2
\efw\Big(\frac{1}{|w|^2}\Big) \D\mu
   \\
&= \int_{\phi^{-1}(\varDelta)}
\mathsf{E}_{\phi}(|w|^2) \efw\Big(\frac{1}{|w|^2}\Big)
\D\mu, \quad \varDelta \in \ascr,
   \end{align*}
which together with \cite[Theorem 1.6.11]{Ash} and the
fact that $\mu|_{\phi^{-1}(\ascr)}$ is $\sigma$-finite
proves our claim. Second, the assertions (ii) and
(iii) are logically equivalent. Indeed, in view of the
proof of Proposition \ref{caroll4}, it suffices to
show that (iii)$\Rightarrow$(ii). However, this can be
deduced from (i), \eqref{fifi}, Lemma \ref{nuklear}
and the definitions of $\mathsf{E}_{\phi} (\cdot)\circ
\phi^{-1}$ and $\efw(\cdot)\circ \phi^{-1}$.
   \end{rem}
   \section{\label{Sec7.3}Application to subnormality}
Using results of Section \ref{13IX}, we give criteria
for subnormality of some classes of weighted
composition operators. This is the right place to
refer the reader to Lemma \ref{recfor} for the
recurrence relation for the sequence $\{\cen{n}
(|\widehat w_n|^2)\circ \phi^{-n}\}_{n=0}^{\infty}$
which is frequently used in this section.
   \begin{lem} \label{subnci}
Suppose \eqref{stand1} holds, $\cfw\in \ogr{L^2(\mu)}$
and $C_{\phi}^n$ is densely defined for every $n\in
\nbb$. Assume also that $\{\cen{n} (|\widehat
w_n|^2)\circ \phi^{-n}(x)\}_{n=0}^{\infty}$ is a
Stieltjes moment sequence for $\mu$-a.e.\ $x\in X$ and
$C_{\phi}$ is subnormal. Then $\cfw$ is~ subnormal.
   \end{lem}
   \begin{proof}
Note that, by \eqref{mgla}, $\cen{n}$ exists for all
$n\in \nbb$. It follows from \cite[Corollary
10.3]{b-j-j-sC} that
$\{\mathsf{h}_{\phi^n}(x)\}_{n=0}^{\infty}$ is a
Stieltjes moment sequence for $\mu$-a.e.\ $x\in X$.
Since the pointwise product of two Stieltjes moment
sequences is a Stieltjes moment sequence (cf.\
\cite[Lemma 2.1]{B-D}), we deduce from \eqref{poz1-1}
that $\{\hfwn{n}(x)\}_{n=0}^{\infty}$ is a Stieltjes
moment sequence for $\mu$-a.e.\ $x\in X$. Applying
Theorem \ref{gsms+} completes the proof.
   \end{proof}
For the reader's convenience, we state the
Berg-Dur\'{a}n theorem which is an important
ingredient of the proof of Theorem \ref{be-du} (see
\cite[p.\ 252]{B-D}; see also \cite[Theorem 1.1]{Berg}
for the version presented below).
   \begin{thm}
\label{Ber-Dur} If $\{a_n\}_{n=0}^{\infty}$ is a
non-degenerate Hausdorff moment sequence $($i.e., $a_n
\neq 0$ for all $n\in \zbb_+$$)$, then the sequence
$\{s_n\}_{n=0}^{\infty}$ defined by
   \begin{align*}
s_n =
   \begin{cases}
1 & \text{ if } n=0,
   \\
\displaystyle{\frac{1}{a_1 \cdot \ldots \cdot a_n}} &
\text{ if } n\Ge 1,
   \end{cases}
\quad n \in \zbb_+,
   \end{align*}
is a Stieltjes moment sequence.
   \end{thm}
   Now we are in a position to establish a criterion
for deriving subnormality of a bounded weighted
composition operator from subnormality of a
composition operator with the same symbol.
   \begin{thm} \label{be-du}
Suppose \eqref{stand1} holds, $\phi$ is a bijection
whose inverse $\phi^{-1}$ is $\ascr$-measurable and
$\{1/|w(\phi^{-n}(x))|^{2}\}_{n=0}^{\infty}$ is a
Hausdorff moment sequence for $\mu$-a.e.\ $x\in X$.
Assume also that $\cfw\in \ogr{L^2(\mu)}$,
$C_{\phi}^n$ is densely defined for every $n\in \nbb$
and $C_{\phi}$ is subnormal. Then $\cfw$ is subnormal.
   \end{thm}
   \begin{proof}
In view of Lemma \ref{subnci}, it is enough to show
that $\{\gamma_n(x)\}_{n=0}^{\infty}$ is a Stieltjes
moment sequence for $\mu$-a.e.\ $x\in X$, where
$\gamma_n:=\cen{n} (|\widehat w_n|^2)\circ \phi^{-n}$
for $n\in \zbb_+$. Knowing that the transformation
$\phi$ is bijective and $\ascr$-bimeasurable, we see
that $\cen{n}(f)=f$ a.e.\ $[\mu]$ for all
$\ascr$-measurable functions $f\colon X \to \rbop$ and
all $n\in \zbb_+$. Since $C_{\phi}$ is subnormal, then
by \cite[Corollary 10.3]{b-j-j-sC} (see also Theorem
\ref{gsms}) we have $C_{\phi}^n=C_{\phi^n}$ for all
$n\in \zbb_+$. Thus for every $n\in \nbb$, the
operator $C_{\phi^n}$ is subnormal as a densely
defined $n$th power of a subnormal operator. By
\cite[Section 6]{b-j-j-sC}, $\mathsf{h}_{\phi^{n}} >
0$ a.e.\ $[\mu]$ for every $n\in \zbb_+$. As a
consequence, we see that
   \begin{align} \label{dozac}
\text{$\cen{n} (f) \circ \phi^{-n} = f\circ \phi^{-n}$
a.e.\ $[\mu]$}
   \end{align}
for all $\ascr$-measurable functions $f\colon X \to
\rbop$ and all $n\in \zbb_+$, where $f\circ \phi^{-n}$
is the usual composition of functions. Hence, for
every $n \in \zbb_+$ and for $\mu$-a.e.\ $x \in X$,
   \begin{align*}
\gamma_n(x) =
   \begin{cases}
1 & \text{ if } n=0,
   \\
\displaystyle{\frac{1}{a_1(x) \cdot \ldots \cdot
a_n(x)}} & \text{ if } n\Ge 1,
   \end{cases}
   \end{align*}
where for every $n\in \zbb_+$, $a_n\colon X \to
(0,\infty)$ is a function such that
   \begin{align*}
a_n(x) = \frac{1}{|w(\phi^{-n}(x))|^2} \quad
\textit{for $\mu$-a.e.\ $x\in X$};
   \end{align*}
such functions exist because, by our assumption,
$w(\phi^{-n}(x)) \neq 0$ a.e.\ $[\mu]$ for every $n\in
\zbb_+$. Now applying Theorem \ref{Ber-Dur} completes
the proof.
   \end{proof}
   It is worth noting that under the assumptions of
Theorem \ref{be-du} the composition operator
$C_{\phi^{-1}}$ is well-defined. This observation is a
particular case of a more general result stated below
(because subnormal composition operators are
automatically injective; see \cite[Section
6]{b-j-j-sC}). On the other hand, as can be deduced
from Proposition \ref{portor}, the equality
\eqref{dozac} may hold under much weaker assumptions
than those of Theorem \ref{be-du}.
   \begin{pro}\label{portor}
Suppose $(X,\ascr,\mu)$ is a $\sigma$-finite measure
space and $\phi$ is a bijective and
$\ascr$-bimeasurable transformation of $X$ such that
$C_{\phi}$ is well-defined and injective. Then
$C_{\phi^{-1}}$ is well-defined and injective.
Moreover, if $C_{\phi}$ is densely defined, then for
all $\ascr$-measurable functions $f\colon X \to
\rbop$,
   \begin{align*}
\text{$\mathsf{E}_{\phi}(f) \circ \phi^{-1} = f\circ
\phi^{-1}$ a.e.\ $[\mu]$},
   \end{align*}
where $f\circ \phi^{-1}$ is the usual composition of
functions.
   \end{pro}
   \begin{proof}
It follows from \cite[Proposition 6.2]{b-j-j-sC} that
$\mathsf{h}_{\phi} > 0$ a.e.\ $[\mu]$. This
immediately implies the ``moreover'' part. Using
\eqref{l1} with $w=\boldsymbol{1}$, we verify that the
measures $\mu\circ \phi^{-1}$ and $\mu$ are mutually
absolutely continuous. As a consequence, $\mu\circ
(\phi^{-1})^{-1} \ll \mu$, so by Proposition
\ref{wco1} (again with $w=\boldsymbol{1}$), the
composition operator $C_{\phi^{-1}}$ is well-defined.
A direct application of Lemma \ref{nuklear} (as before
with $w=\boldsymbol{1}$) shows that $C_{\phi^{-1}}$ is
injective.
   \end{proof}
Below, we give an example which shows that Proposition
\ref{portor} is no longer true if the assumption on
injectivity of $C_{\phi}$ is removed. What is more,
the equality \eqref{dozac} may not hold even if $\phi$
is bijective and $\ascr$-bimeasurable, and $C_{\phi}
\in \ogr{L^2(\mu)}$.
   \begin{exa} \label{portor-1}
Let $X=\zbb$ and $\ascr=2^X$. Then there exists a
unique (necessarily $\sigma$-finite) measure
$\mu\colon \ascr \to \rbop$ such that
   \begin{align*}
\mu(n) =
   \begin{cases}
1 & \text{ if } n \Ge 0,
   \\
0 & \text{ if } n < 0,
   \end{cases}
\quad n\in X.
   \end{align*}
Let $\phi$ be the transformation of $X$ given by
$\phi(n)=n+1$ for $n\in X$. Clearly $\phi$ is
bijective (and $\ascr$-bimeasurable), $C_{\phi}$ is
well-defined and $C_{\phi^{-1}}$ is not well-defined.
Thus by Proposition \ref{portor} (or by a direct
verification) $C_{\phi}$ is not injective. Observe
that $\mathsf{h}_{\phi^n} = \chi_{n + \zbb_+}$ a.e.\
$[\mu]$ for all $n\in \nbb$. Hence, by Proposition
\ref{lemS1}, we see that $C_{\phi} \in
\ogr{L^2(\mu)}$. Plainly $\cen{n}(f)=f$ a.e.\ $[\mu]$
for all functions $f\colon X \to \rbop$ and all $n\in
\zbb_+$. However, for any $n\in \nbb$, there exists a
function $f\colon X \to \rbop$ for which the condition
\eqref{dozac} fails to hold (e.g., the function
$f=\boldsymbol{1}$ does the job perfectly, cf.\
\eqref{jedynka}).
   \end{exa}
The following result is a ``dual'' version of Lemma
\ref{subnci} in which the roles of operators
$C_{\phi}$ and $\cfw$ are interchanged.
   \begin{lem} \label{subnci2}
Suppose \eqref{stand1} holds, $w \neq 0$ a.e.\
$[\mu]$, $C_{\phi}\in \ogr{L^2(\mu)}$ and
$C_{\phi,w}^n$ is densely defined for every $n\in
\nbb$. Assume also that $\cfw$ is subnormal and
$\{\mathsf{E}_{\phi^n, \widehat w_n}
\big(\frac{1}{|\widehat w_n|^2}\big)\circ \phi^{-n}(x)
\}_{n=0}^{\infty}$ is a Stieltjes moment sequence for
$\mu$-a.e.\ $x\in X$. Then $C_{\phi}$ is subnormal.
   \end{lem}
   \begin{proof}
Note that $\mathsf{E}_{\phi^n, \widehat w_n}$ exists
for all $n\in \nbb$ (see \eqref{mgla}). Since, by
Theorem \ref{Mittag} and \cite[Proposition
3.2.1]{b-j-j-sA}, $\cfw$ generates Stieltjes moment
sequences, we infer from Theorem \ref{gsms} that
$\{\hfwn{n}(x)\}_{n=0}^\infty$ is a Stieltjes moment
sequence for $\mu$-a.e.\ $x \in X$. Then following the
proof of Lemma \ref{subnci} with \eqref{poz2-2} in
place of \eqref{poz1-1} completes the proof.
   \end{proof}
Now arguing as in the proof of Theorem \ref{be-du}
one can show that the following ``dual'' version of
this theorem holds (use: Lemma \ref{subnci2} in
place of Lemma \ref{subnci}; Theorems \ref{Mittag}
and \ref{gsms} and \cite[Proposition
3.2.1]{b-j-j-sA} in place of \cite[Corollary
10.3]{b-j-j-sC}; Corollary \ref{hipinj} in place of
   \cite[Section 6]{b-j-j-sC}; and Corollary
   \ref{caroll3}).
   \begin{thm} \label{be-du2}
Suppose \eqref{stand1} holds, $\phi$ is a bijection
whose inverse $\phi^{-1}$ is $\ascr$-measurable and
$\{|w(\phi^{-n}(x))|^{2}\}_{n=0}^{\infty}$ is a
non-degenerate Hausdorff moment sequence for
$\mu$-a.e.\ $x\in X$. Assume also that $C_{\phi}\in
\ogr{L^2(\mu)}$, $C_{\phi,w}^n$ is densely defined for
every $n\in \nbb$ and $\cfw$ is subnormal. Then
$C_{\phi}$ is subnormal.
   \end{thm}
   \section{\label{Sec7.4}Subnormality in the matrix symbol case}
This section deals with
the question of subnormality of weighted composition
operators with matrix symbols. First, we discuss the
case of composition operators. We refer the reader to
\cite{sto} for more information on this class of
operators (see also \cite{ml}).

Denote by $\hscr$ the set of all entire functions
$\eta$ on $\cbb$ of the form
   \begin{align*}
\eta(z) = \sum_{n=0}^\infty a_n z^n, \quad z \in \cbb,
   \end{align*}
where $\{a_n\}_{n=0}^{\infty}$ is a sequence of
nonnegative real numbers such that $a_k > 0$ for some
$k \Ge 1$. Clearly, if $\eta \in \hscr$, then the
function $\eta|_{\rbb_+}$ is nonnegative and strictly
increasing, and $\lim_{\rbb_+ \ni t \to \infty}
\eta(t) = \infty$. Fix a positive integer $\kappa$.
Let $\eta\in \hscr$ and $\|\cdot\|$ be a norm on
$\rbb^\kappa$. Define the $\sigma$-finite Borel
measure $\mu$ on $\rbb^\kappa$ by
   \begin{align} \label{xidef}
\mu(\varDelta) = \int_{\varDelta} \eta(\|x\|^2) \D x,
\quad \varDelta \in \borel{\rbb^\kappa},
   \end{align}
where $\D x$ indicates integration with respect to the
$\kappa$-dimensional Lebesgue measure. Plainly the
measure $\mu$ and the $\kappa$-dimensional Lebesgue
measure are mutually absolutely continuous. Given a
linear transformation $A$ of $\rbb^\kappa$, one can
prove that the composition operator $C_{A}$ in
$L^2(\mu)$ with the symbol $A$ is well-defined if and
only if $A$ is invertible (see, e.g., the proof of
\cite[Theorem 13.1]{2xSt3}). If this is the case, then
the composition operator $C_{A^n}$ is well-defined for
every $n\in \nbb$ (cf.\ Lemma \ref{lemS11}) and
   \begin{align} \label{pochodna}
\mathsf{h}_{A^n} (x) = \frac{1}{|\det A|^n}
\frac{\eta(\|A^{-n}x\|^2)}{\eta(\|x\|^2)} \quad
\text{for $\mu$-a.e.\ $x\in \rbb^\kappa$}.
   \end{align}
Observe that the rational function appearing on the
right-hand side of the equality in \eqref{pochodna} is
continuous on $\rbb^\kappa \setminus \{0\}$. Hence,
there is only one continuous representative of
$\mathsf{h}_{A^n}$ on $\rbb^\kappa \setminus \{0\}$.
Combining \eqref{pochodna} with Lemmas
\ref{lemS11p}(v) and \ref{jadro} (both applied to
$w=\boldsymbol{1}$), we conclude that if $C_A$ is
well-defined, then $C_A^n$ is densely defined and
injective for every $n\in \nbb$. It follows from
\cite[Proposition 2.2]{sto} that $C_A \in
\ogr{L^2(\mu)}$ if and only if either $\eta$ is a
polynomial, or $\eta$ is not a polynomial and
$\|A^{-1}\|\Le 1$.

Now let $w\colon \rbb^\kappa \to \cbb$ be a Borel
function and $A$ be an invertible linear
transformation of $\rbb^\kappa$. Since $C_A$ is
well-defined, so is $C_{A,w}$, the weighted
composition operator in $L^2(\mu)$ with the symbol $A$
and the weight $w$ (cf.\ Proposition \ref{mwc1}). The
reader should be aware of the fact that the notation
``$C_{A,r}$'' that appears in \cite{sto} has nothing
to do with weighted composition operators $C_{A,w}$;
it simply denotes the composition operator in
$L^2(r(x)\D x)$ with the symbol $A$. By
\eqref{pochodna} and Propositions \ref{caroll} and
\ref{portor}, we have
   \begin{align} \label{haa1}
\mathsf{h}_{A,w}(x) = \frac{1}{|\det A|}
\frac{\eta(\|A^{-1}x\|^2)}{\eta(\|x\|^2)} \cdot
|w(A^{-1}x)|^2 \quad \text{for $\mu$-a.e.\ $x\in
\rbb^\kappa$}.
   \end{align}

In the next result, which is the main application of
Theorem \ref{be-du}, we provide a wide class of
subnormal weighted composition operators with matrix
symbols.
   \begin{thm} \label{wcomc}
Suppose $\eta, \xi \in \hscr$, $\|\cdot\|$ is a norm
on $\rbb^\kappa$ induced by an inner product
\mbox{$\langle \cdot, \mbox{-}\rangle$}, $\mu$ is a
Borel measure on $\rbb^\kappa$ defined by
\eqref{xidef}, $A$ is an invertible linear
transformation of $\rbb^\kappa$ and $w\colon
\rbb^\kappa \to \cbb$ is a Borel function such that
$\|A^{-1}\| \Le 1$ and
   \begin{align} \label{klima1}
|w(x)|^2 = \frac{1}{\xi(\|x\|^2)}, \quad x\in
\rbb^\kappa \setminus \{0\}.
   \end{align}
Then the following two statements are valid{\em :}
   \begin{enumerate}
   \item[(i)] $C_{A,w} \in \ogr{L^2(\mu)}$ if and only if
$\xi(0) > 0$,
   \item[(ii)] if $C_{A,w} \in \ogr{L^2(\mu)}$ and $A$
is a normal operator on \mbox{$(\rbb^\kappa, \langle
\cdot, \mbox{-}\rangle)$}, then $C_{A,w}$ is
subnormal.
   \end{enumerate}
   \end{thm}
   \begin{proof}
(i) It follows from \eqref{haa1} and \eqref{klima1}
that
   \begin{align} \label{haa2}
\mathsf{h}_{A,w}(x) = \frac{1}{|\det A|}
\frac{\eta(\|A^{-1}x\|^2)}{\eta(\|x\|^2)} \cdot
\frac{1}{\xi(\|A^{-1} x\|^2)} \quad \text{for
$\mu$-a.e.\ $x\in \rbb^\kappa$}.
   \end{align}
Suppose $\xi(0) > 0$. Knowing that $\eta|_{\rbb_+}$ is
increasing and $\|A^{-1}\| \Le 1$, we have
   \begin{align} \label{klima3}
\sup_{x\in \rbb^{\kappa} \setminus \{0\}}
\frac{\eta(\|A^{-1}x\|^2)}{\eta(\|x\|^2)} \Le 1.
   \end{align}
Since $\inf_{t\in \rbb_+} \xi(t) > 0$, we infer from
\eqref{haa2} and Proposition \ref{lemS1}(v) that
$C_{A,w} \in \ogr{L^2(\mu)}$. To prove the converse,
assume that $C_{A,w} \in \ogr{L^2(\mu)}$. Suppose
contrary to our claim that $\xi(0)=0$. Denote by $m$
and $n$ the multiplicities of zero of $\eta$ and $\xi$
at $0$, respectively. By our assumption $n\Ge 1$ and
there exist $\tilde \eta, \tilde \xi \in \hscr \cup
\{\boldsymbol{1}\}$ such that $\eta(z) = z^m \tilde
\eta(z)$ and $\xi(z) = z^n \tilde \xi(z)$ for all
$z\in \cbb$, $\tilde \eta(0) > 0$ and $\tilde \xi(0)\
> 0$.  Then  we get
   \begin{align*}
\sup_{x\neq 0} &
\frac{\eta(\|A^{-1}x\|^2)}{\eta(\|x\|^2)} \cdot
\frac{1}{\xi(\|A^{-1} x\|^2)}
   \\
& = \sup_{x\neq 0} \sup_{t>0}
\frac{\eta(\|A^{-1}(tx)\|^2)}{\eta(\|tx\|^2)} \cdot
\frac{1}{\xi(\|A^{-1} (tx)\|^2)}
   \\
& = \sup_{x\neq 0} \sup_{t>0} \frac{1}{t^{2n}} \cdot
\frac{\|A^{-1}(x)\|^{2m}}{\|x\|^{2m}} \cdot
\frac{\tilde\eta(t^2\|A^{-1}x\|^2)}{\tilde\eta(t^2\|x\|^2)}
\cdot \frac{1}{\|A^{-1} x\|^{2n} \tilde
\xi(t^2\|A^{-1} x\|^2)}
   \\
& \hspace{-.3ex}\overset{(\dag)}= \infty.
   \end{align*}
(To obtain $(\dag)$ consider $t\to 0+$.) This combined
with \eqref{haa2} implies that $\mathsf{h}_{A,w}
\notin L^{\infty}(\mu)$, which by Proposition
\ref{lemS1}(v) yields $C_{A,w} \notin \ogr{L^2(\mu)}$,
a contradiction.

   (ii) It follows from \cite[Proposition 2.2]{sto}
(or directly from \eqref{pochodna}, \eqref{klima3} and
Proposition \ref{lemS1}(v) applied to
$w=\boldsymbol{1}$) that $C_A \in \ogr{L^2(\mu)}$.
Hence, by \cite[Theorem 2.5]{sto}, the composition
operator $C_A$ is subnormal. According to
\eqref{klima1}, we have
   \begin{align} \label{klima2}
\frac{1}{|w(A^{-n}x)|^2} = \xi(\|A^{-n}x\|^2), \quad x
\in \rbb^\kappa \setminus \{0\}.
   \end{align}
In view of \cite[Theorem 6.3]{sto-a}, the sequence
$\{\xi(\|A^{-n}x\|^2)\}_{n=0}^{\infty}$ is positive
definite for every $x\in \rbb^{\kappa}$. Substituting
$A^{-1}x$ in place of $x$, we see that the sequence
$\{\xi(\|A^{-(n+1)}x\|^2)\}_{n=0}^{\infty}$ is
positive definite for every $x\in \rbb^{\kappa}$.
Since $\|A^{-1}\| \Le 1$ and $\xi|_{\rbb_+}$ is
increasing, we deduce that the sequence
$\{\xi(\|A^{-n}x\|^2)\}_{n=0}^{\infty}$ is bounded for
every $x\in \rbb^{\kappa}$. Finally, by \eqref{Stiogr}
and \eqref{klima2}, the sequence
$\{1/|w(A^{-n}x)|^{2}\}_{n=0}^{\infty}$ is a Hausdorff
moment sequence for every $x\in \rbb^{\kappa}
\setminus \{0\}$. Applying Theorem \ref{be-du}
completes the proof.
   \end{proof}
Now we make a comment on the proof of Theorem
\ref{wcomc}.
   \begin{rem} \label{klima4}
Assume that $A$ is an invertible normal operator on
\mbox{$(\rbb^\kappa, \langle \cdot, \mbox{-}\rangle)$}
and \eqref{klima1} holds for some $\xi \in \hscr$.
Then $\|A^{-1}\|\Le 1$ if and only if the sequence
$\{1/|w(A^{-n}x)|^{2}\}_{n=0}^{\infty}$ is a Hausdorff
moment sequence for $\mu$-a.e.\ $x \in \rbb^{\kappa}$,
where $\mu$ is as in \eqref{xidef} (note that, in view
of the proof of \cite[Theorem 5.1]{sto-a} and
\eqref{klima2},
$\{1/|w(A^{-n}x)|^{2}\}_{n=0}^{\infty}$ is a Stieltjes
moment sequence for every $x \in
\rbb^{\kappa}\setminus \{0\}$). For the justification
of the ``only if'' part see the proof of the statement
(ii) of Theorem \ref{wcomc}. The ``if'' part can be
proved as follows. Since each Hausdorff moment
sequence is bounded, we infer from \eqref{klima2} that
the sequence $\{\xi(\|A^{-n}x\|^2)\}_{n=0}^{\infty}$
is bounded for $\mu$-a.e.\ $x \in \rbb^{\kappa}$.
Knowing that $\lim_{\rbb_+ \ni t \to \infty} \xi(t) =
\infty$, we deduce that the sequence
$\{\|A^{-n}x\|\}_{n=0}^{\infty}$ is bounded for
$\mu$-a.e.\ $x \in \rbb^{\kappa}$. Combined with
\cite[Lemma 1.4(v)]{sto-05}, this implies that the set
$\mathscr M:= \{x\in \rbb^{\kappa}\colon
\lim_{n\to\infty} \|A^{-n}\|^{1/n} \Le 1\}$ is a
vector subspace of $\rbb^{\kappa}$ which is a set of
full $\mu$-measure. Since $\mu$ and the
$\kappa$-dimensional Lebesgue measure are mutually
absolutely continuous, we deduce that $\mathscr M =
\rbb^{\kappa}$. Hence, by the normality of $A^{-1}$
and \cite[Lemma 1.4(i)]{sto-05}, we conclude that~
\mbox{$\|A^{-1}\|\Le 1$}.

It is worth observing that if $A$ is not normal, then
the assumption that the sequence
$\{\xi(\|A^{-n}x\|^2)\}_{n=0}^{\infty}$ is bounded for
$\mu$-a.e.\ $x \in \rbb^{\kappa}$ may not imply that
$\|A^{-1}\|\Le 1$. Indeed, fix $t \in (0,1)$ and take
a nilpotent operator $N \in \ogr{\rbb^{\kappa}}$ such
that $\|N\| > 1 + t$. Then the operator $t + N$ is
invertible and the spectral radius $r(t+N)$ of $t+N$
equals $t$ (because the spectrum of $t + N$ equals
$\{t\}$). Set $A=(t + N)^{-1}$. Then $\|A^{-1}\| \Ge
\|N\|-t > 1$ and $r(A^{-1})=t<1$, so the sequence
$\{\|A^{-n}x\|\}_{n=0}^{\infty}$ is bounded for every
$x \in \rbb^{\kappa}$, or equivalently the sequence
$\{\xi(\|A^{-n}x\|^2)\}_{n=0}^{\infty}$ is bounded for
every $x \in \rbb^{\kappa}$.
   \end{rem}
   \section{\label{Sec7.5}Examples}
As shown in Example \ref{cohyp-count2}, it may happen
that a weighted composition operator $\cfw$ is an
isometry while the corresponding composition operator
$C_{\phi}$ is not even well-defined. In this example
the measure $\mu$ is infinite and all the measures
$\mu|_{\phi^{-n}(\ascr)}$, $n\in \nbb$, are
$\sigma$-finite. Here we provide two more examples of
this kind. The first one, Example \ref{mwc6}, is built
over a finite complete measure space (in this
situation, all the measures $\mu|_{\phi^{-n}(\ascr)}$,
$n\in \nbb$, are automatically $\sigma$-finite). The
second one, Example \ref{ciach}, which is built over
an infinite $\sigma$-finite complete measure space, is
such that the measure $\mu|_{\phi^{-n}(\ascr)}$ is not
$\sigma$-finite for every $n\in \nbb$.

The next three examples are based on \cite[Example
3.1]{b-j-j-sC}.
   \begin{exa} \label{mwc6}
Set $X=\{0\}\cup\{1\}\cup [2,3]$ and
   \begin{align*}
\ascr=\{\varDelta \in 2^{X}\colon \varDelta \text{ is
a Lebesgue measurable set in $\rbb$}\}.
   \end{align*}
Clearly, $\ascr$ is a $\sigma$-algebra in $X$. Define
the finite complete measure $\mu$ on $\ascr$ by
   \begin{align*}
\mu(\varDelta)= \delta_{0}(\varDelta) +
\delta_{1}(\varDelta) + m(\varDelta\cap [2,3]), \quad
\varDelta \in \ascr,
   \end{align*}
where $m$ stands for the Lebesgue measure on $\rbb$.
Let $\phi$ be the $\ascr$-measurable transformation of
$X$ given by $\phi(0)=2$, $\phi(1)=1$ and $\phi(x)=0$
for $x \in [2,3]$. Since $\phi^{2k-1}=\phi$ and
$\phi^{2k}=\phi^2$ for any $k\in \nbb$, we easily see
that for every $n\in \nbb$, the sigma algebra
$\phi^{-n}(\ascr)$ is relatively $\mu$-complete, or
equivalently, the measure $\mu|_{\phi^{-n}(\ascr)}$ is
complete (because $\mu$ is complete). Noting that
$\mu(2)=0$ and $(\mu\circ \phi^{-1})(\{2\})=1$, we
infer from Proposition \ref{wco1} that $C_{\phi}$ is
not well-defined. Set $w=\chi_{X\setminus \{0\}}$.
Since
   \begin{align*}
\mu_w \circ \phi^{-1} (\varDelta) =
   \begin{cases}
1 & \text{if } \varDelta = \{0\},
   \\
1 & \text{if } \varDelta = \{1\},
   \\
0 & \text{if } \varDelta \in \ascr \cap [2,3],
   \end{cases}
   \end{align*}
we deduce that $\mu_w \circ \phi^{-1} \ll \mu$ and
$\hfw = \chi_{\{0,1\}}$ a.e.\ $[\mu]$. Hence it
follows from Propositions \ref{wco1} and
\ref{lemS1}(v) that the operator $\cfw$ is
well-defined and $\cfw \in \ogr{L^2(\mu)}$.
   \end{exa}
Below we adapt the above example to the case of
infinite $\sigma$-finite measure spaces loosing the
$\sigma$-finiteness of the measures $\mu\circ
\phi^{-n}$, $n\in \nbb$. As before, the weighted
composition operator is well-defined and bounded.
   \begin{exa} \label{ciach}
Set $X=(-\infty,0] \cup \{1\} \cup [2,3]$ and
   \begin{align*}
\ascr=\{\varDelta \in 2^{X}\colon \varDelta \text{ is
a Lebesgue measurable set in $\rbb$}\}.
   \end{align*}
Define the $\sigma$-finite complete measure $\mu\colon
\ascr\to \rbop$~ by
   \begin{align*}
\mu(\varDelta)= m(\varDelta\cap (-\infty,0)) +
\delta_{0}(\varDelta) + \delta_{1}(\varDelta) +
m(\varDelta\cap [2,3]), \quad \varDelta \in \ascr,
   \end{align*}
where $m$ is the Lebesgue measure on $\rbb$. Let
$\phi$ be the $\ascr$-measurable transformation of $X$
defined by
   \begin{align*}
\phi(x) =
   \begin{cases}
2 & \text{if } x\in(-\infty,0],
   \\
1 & \text{if } x=1,
   \\
0 & \text{if } x \in [2,3].
   \end{cases}
   \end{align*}
Since $\mu(2)=0$ and $(\mu\circ
\phi^{-1})(\{2\})=\infty$, Proposition \ref{wco1}
implies that $C_{\phi}$ is not well-defined. Observe
that for every $n\in \nbb$, the measure
$\mu|_{\phi^{-n}(\ascr)}$ is not $\sigma$-finite and
thus the conditional expectation
$\mathsf{E}(\,\cdot\,;\phi^{-n}(\ascr),\mu)$ does not
exist. Indeed, since $\phi^{2k-1}=\phi$ and
$\phi^{2k}=\phi^2$ for all $k\in \nbb$, we see that
$\mu(\phi^{-(2k-1)}(\{2\}))=\infty$ and
$\mu(\phi^{-2k}(\{0\}))=\infty$ for all $k\in \nbb$,
which easily implies that for every $n\in \nbb$, the
measure $\mu|_{\phi^{-n}(\ascr)}$ is not
$\sigma$-finite. As in Example \ref{mwc6}, we verify
that for every $n\in \nbb$, the sigma algebra
$\phi^{-n}(\ascr)$ is relatively $\mu$-complete. Set
$w=\chi_{\{1\} \cup [2,3]}$. Noting that
   \begin{align*}
\mu_w \circ \phi^{-1} (\varDelta) =
   \begin{cases}
0 & \text{if } \varDelta \in \ascr \cap (-\infty,0),
   \\
1 & \text{if } \varDelta = \{0\},
   \\
1 & \text{if } \varDelta = \{1\},
   \\
0 & \text{if } \varDelta \in \ascr \cap [2,3],
   \end{cases}
   \end{align*}
we deduce that $\mu_w \circ \phi^{-1} \ll \mu$ and
$\hfw = \chi_{\{0,1\}}$ a.e.\ $[\mu]$. By Proposition
\ref{lemS1}(v), $\cfw \in \ogr{L^2(\mu)}$.
   \end{exa}
In the subsequent example we construct a bounded
weighted composition operator $\cfw$ such that the
corresponding composition operator $C_{\phi}$ is
well-defined, but the conditional expectation
$\mathsf{E}(\,\cdot\,;\phi^{-n}(\ascr),\mu)$ does not
exist for any $n\in \nbb$. Though Example \ref{ciach2}
is, in a sense, weaker than Example \ref{ciach3}, the
former one is simpler and the underlying measure space
appearing in it is not discrete.
   \begin{exa} \label{ciach2}
Set $X=(-\infty,0] \cup \{1\} \cup [2,3]$ and
   \begin{align*}
\ascr=\{\varDelta \in 2^{X}\colon \varDelta \text{ is
a Lebesgue measurable set in $\rbb$}\}.
   \end{align*}
Define the $\sigma$-finite complete measure $\mu\colon
\ascr\to \rbop$ by
   \begin{align*}
\mu(\varDelta)= m(\varDelta\cap (-\infty,0)) +
\delta_{0}(\varDelta) + \delta_{1}(\varDelta) +
\delta_{2}(\varDelta) + m(\varDelta\cap (2,3]), \quad
\varDelta \in \ascr,
   \end{align*}
where $m$ is the Lebesgue measure on $\rbb$. Let
$\phi$ be the $\ascr$-measurable transformation of $X$
defined by
   \begin{align*}
\phi(x) =
   \begin{cases}
2 & \text{if } x\in(-\infty,0],
   \\
1 & \text{if } x=1,
   \\
0 & \text{if } x \in [2,3].
   \end{cases}
   \end{align*}
Set $w=\chi_{\{1\} \cup [2,3]}$. Since $\mu(2) > 0$,
it is easily seen that $\mu\circ \phi^{-1} \ll \mu$
and so, by Proposition \ref{wco1}, $C_{\phi}$ is
well-defined. Observe that for every $n\in \nbb$, the
measure $\mu|_{\phi^{-n}(\ascr)}$ is not
$\sigma$-finite and thus the conditional expectation
$\mathsf{E}(\,\cdot\,;\phi^{-n}(\ascr),\mu)$ does not
exist. Indeed, since $\phi^{2k-1}=\phi$ and
$\phi^{2k}=\phi^2$ for all $k\in \nbb$, we see that
$\mu(\phi^{-(2k-1)}(\{2\}))=\infty$ and
$\mu(\phi^{-2k}(\{0\}))=\infty$ for all $k\in \nbb$,
which proves our claim. Clearly, the sigma algebras
$\phi^{-n}(\ascr)$, $n\in \nbb$, are relatively
$\mu$-complete. Since
   \begin{align*}
\mu_w \circ \phi^{-1} (\varDelta) =
   \begin{cases}
0 & \text{if } \varDelta \in \ascr \cap (-\infty,0),
   \\
2 & \text{if } \varDelta = \{0\},
   \\
1 & \text{if } \varDelta = \{1\},
   \\
0 & \text{if } \varDelta \in \ascr \cap [2,3],
   \end{cases}
   \end{align*}
we conclude that $\mu_w \circ \phi^{-1} \ll \mu$ and
$\hfw = 2 \chi_{\{0\}} + \chi_{\{1\}}$ a.e.\ $[\mu]$.
Hence, in view of Proposition \ref{lemS1}(v), $\cfw
\in \ogr{L^2(\mu)}$.
   \end{exa}
Below, we construct a unitary weighted composition
operator $\cfw$ such that $C_{\phi}^n$ is densely
defined for every $n\in \nbb$, $C_{\phi} \notin
\ogr{L^2(\mu)}$ and $C_{\phi}$ is not hyponormal. This
example sheds more light on part (iii) of Proposition
\ref{mwc1} and Example \ref{cohyp-count2}.
   \begin{exa} \label{palgong1}
Set $X=\zbb$. Let $\mu\colon 2^X \to \rbop$ be a
$\sigma$-finite measure such that $\at{\mu}=X$. Define
the transformation $\phi$ of $X$ by $\phi(n)=n+1$ for
$n \in X$. Plainly, $C_{\phi}$ is well-defined. Since
$\phi$ is a bijection, we see that
$\phi^{-n}(2^X)=2^X$ for every $n\in \nbb$ (so
$\mathsf{E}(\,\cdot\,;\phi^{-n}(2^X),\mu)$ acts as the
identity mapping). Hence, by Propositions \ref{lemS2}
and \ref{potegi-p}, $C_{\phi}^n$ is densely defined
for every $n\in \nbb$. Let $w\colon X \to \cbb$ be any
function such that $w(n)\neq 0$ for all $n\in X$.
Clearly, $\cfw$ is well-defined and
   \begin{align} \label{palgong2}
\hfw(n) & = \frac{|w(n-1)|^2 \mu(n-1)}{\mu(n)}, \quad
n\in X,
   \\  \label{palgong2b}
\mathsf{h}_{\phi}(n) & = \frac{\mu(n-1)}{\mu(n)},
\quad n\in X.
   \end{align}
Note that, by Theorem \ref{hypdisc}, $C_{\phi}$ is
hyponormal if and only
   \begin{align} \label{palgong3}
\mu(n-1)^2 \Le \mu(n)\mu(n-2), \quad n\in X.
   \end{align}
Fix $k\in \zbb$ and choose $\mu$ so that
$\mu(k-2)=\mu(k-1)=1$, $0 < \mu(k)< 1$, and
   \begin{align} \label{palgong2c}
\sup_{n\in X} \frac{\mu(n)}{\mu(n+1)} = \infty.
   \end{align}
Since \eqref{palgong3} does not hold for $n=k$,
$C_{\phi}$ is not hyponormal. In turn,
\eqref{palgong2b} and \eqref{palgong2c} imply that
$C_{\phi}$ is not bounded (see Proposition
\ref{lemS1}(v)). Now we specify $w$ by requiring it to
satisfy the following condition
   \begin{align*}
|w(n)|^2 & = \frac{\mu(n+1)}{\mu(n)}, \quad n\in X.
   \end{align*}
It follows from \eqref{palgong2} that $\hfw(n)=1$ for
every $n\in X$. This together with \eqref{chisom}
implies that $\cfw$ is an isometry on $L^2(\mu)$. Set
$e_n=\chi_{\{n\}}$ for $n\in \zbb$. Since $\cfw e_n =
e_{n+1}$ for every $n\in\zbb$ and the vectors
$\{e_n\}_{n=0}^{\infty}$ are linearly dense in
$L^2(\mu)$, we conclude that $\cfw$ is a unitary
operator. Since $\{e_n\}_{n=0}^{\infty}$, after
rescaling, is an orthonormal basis, we see that $\cfw$
is unitarily equivalent to the bilateral shift of
multiplicity $1$.
   \end{exa}
Our next goal is to give two examples of weighted
composition operators which have, in a sense, opposite
properties. In Example \ref{ciach3} we construct an
injective weighted composition operator $\cfw\in
\ogr{L^2(\mu)}$ such that the corresponding
composition operator $C_{\phi}$ is well-defined,
$\dz{C_{\phi^n}}=\{0\}$ and the measure
$\mu|_{\phi^{-n}(\ascr)}$ is not $\sigma$-finite for
every $n\in \nbb$, and (cf.\ Corollary \ref{caroll3})
   \begin{align} \label{etyk1}
\bigg(\mathsf{E}_{\phi^n, \widehat w_n}
\bigg(\frac{1}{|\widehat w_n|^2}\bigg)\circ
\phi^{-n}\bigg)(x) = \infty, \quad x \in X, \, n \in
\nbb.
   \end{align}
In turn, in Example \ref{ciach4} we construct a
well-defined weighted composition operator $\cfw$ in
$L^2(\mu)$ such that $\dz{C_{\phi^n,\widehat
w_n}}=\{0\}$ and the measure $\mu_{\widehat
w_n}|_{\phi^{-n}(\ascr)}$ is not $\sigma$-finite for
every $n\in \nbb$, $C_{\phi} \in \ogr{L^2(\mu)}$,
$C_{\phi}$ is injective, and (cf.\ Corollary
\ref{caroll2})
   \begin{align} \label{etyk2}
(\cen{n} (|\widehat w_n|^2)\circ \phi^{-n})(x) =
\infty, \quad x \in X, \, n \in \nbb.
   \end{align}
(See \eqref{Zak} for the definition of $\widehat
w_n$.) In both these examples the weight functions $w$
are positive.

We begin with the following general procedure.
   \begin{preexa} \label{D-c} Set $X=\nbb$ and
$\ascr=2^X$. Suppose $N\in \nbb \cup \{\infty\}$ and
$\{\varOmega_n\}_{n=1}^N$ is a partition of $X$. Let
$\{\omega_n\}_{n=1}^N$ be a fixed injective sequence
in $X$. Define the ($\ascr$-measurable) transformation
$\phi$ of $X$ by
   \begin{align} \label{fidef}
\text{$\phi(x) = \omega_n$ for all $x \in \varOmega_n$
and $n \in J_N$,}
   \end{align}
where $J_N:=\nbb \cap [1,N]$. Let $\mu: \ascr \to
\rbop$ be a $\sigma$-finite measure such that
$\at{\mu}=X$ and let $w \colon X \to (0,\infty)$ be
any function. Then $\at{\mu_w}=X$, $\mu_w\circ
\phi^{-1}\ll\mu$ and, by the injectivity of
$\{\omega_n\}_{n=1}^N$ (see \eqref{bas1}),
   \begin{align} \label{Ter6}
\hfw(x) =
   \begin{cases}
\frac{\mu_w(\varOmega_n)}{\mu(\omega_n)} & \text{if }
x = \omega_n \text{ for some } n\in J_N,
   \\
0 & \text{otherwise.}
   \end{cases}
   \end{align}
Similarly, $\mu\circ \phi^{-1} \ll \mu$ and
   \begin{align} \label{Ter7}
\mathsf{h_{\phi}}(x) =
   \begin{cases}
\frac{\mu(\varOmega_n)}{\mu(\omega_n)} & \text{if } x
= \omega_n \text{ for some } n\in J_N,
   \\
0 & \text{otherwise.}
   \end{cases}
   \end{align}
By \eqref{Ter6} and Proposition \ref{lemS1}(v), we
have
   \begin{align} \label{Ter4}
\cfw\in \ogr{L^2(\mu)} \iff \sup_{n\in J_N}
\frac{\mu_w(\varOmega_n)}{\mu(\omega_n)} < \infty.
   \end{align}
Similarly,
   \begin{align} \label{Ter5}
C_{\phi} \in \ogr{L^2(\mu)} \iff \sup_{n\in J_N}
\frac{\mu(\varOmega_n)}{\mu(\omega_n)} < \infty.
   \end{align}

Now, we calculate the conditional expectations
$\mathsf{E}_{\phi}(\cdot)$ and $\efw(\cdot)$. If the
conditional expectation $\efw(\cdot)$ exists
(equivalently:\ $\hfw(x) < \infty$ for all $x\in X$),
then by Proposition \ref{cpd1} we get
   \begin{align*}
\efw(f)(x) = \frac{\int_{\varOmega_n} f \D
\mu_w}{\mu_w(\varOmega_n)}, \quad x \in \varOmega_n,
\, n \in J_N.
   \end{align*}
In particular, we have
   \begin{align*}
\efw\bigg(\frac{1}{w^2}\bigg)(x) =
\frac{\mu(\varOmega_n)}{\mu_w(\varOmega_n)}, \quad x
\in \varOmega_n, \, n \in J_N,
   \end{align*}
which implies that
   \begin{align*}
\bigg(\efw\bigg(\frac{1}{w^2}\bigg) \circ
\phi^{-1}\bigg) (x) =
   \begin{cases}
\frac{\mu(\varOmega_n)}{\mu_w(\varOmega_n)} & \text{if
} x = \omega_n \text{ for some } n \in J_N,
   \\
0 & \text{otherwise.}
   \end{cases}
   \end{align*}
Similarly, if the conditional expectation
$\mathsf{E}_{\phi}(\cdot)$ exists (equivalently:
$\mathsf{h}_{\phi}(x) < \infty$ for all $x\in X$),
then
   \begin{align*}
\big(\mathsf{E}_{\phi}\big(w^2\big) \circ
\phi^{-1}\big) (x) =
   \begin{cases}
\frac{\mu_w(\varOmega_n)}{\mu(\varOmega_n)} & \text{if
} x = \omega_n \text{ for some } n \in J_N,
   \\
0 & \text{otherwise.}
   \end{cases}
   \end{align*}
   \end{preexa}
A careful look at the calculations of the conditional
expectations $\mathsf{E}_{\phi}(\cdot)$ and
$\efw(\cdot)$ reveals that it is hard to establish
explicit formulas for the conditional expectations
$\cen{n}(\cdot)$ and $\mathsf{E}_{\phi^n, \widehat
w_n}(\cdot)$ of higher orders. The same refers to the
Radon-Nikodym derivatives $\mathsf{h}_{\phi^n}$ and
$\hfwn{n}$ of higher orders.

To get two aforementioned examples, we specify
parameters in Procedure \ref{D-c}.
   \begin{exa} \label{ciach3}
Assume that $N=\infty$,
$\{\varOmega_n\}_{n=1}^{\infty}$ is a partition of $X$
such that each set $\varOmega_n$ is infinite and the
mapping $\nbb \ni n \longmapsto \omega_n \in X$ is a
bijection. Let $\phi$ be as in \eqref{fidef}. Clearly,
$\phi$ is a surjection. Let $\mu$ be the counting
measure on $X$. Then, it is easily seen that there
exists $w \colon X \to (0,\infty)$ such that
   \begin{align*}
\sup_{n\in \nbb}
\frac{\mu_w(\varOmega_n)}{\mu(\omega_n)} < \infty.
   \end{align*}
Applying \eqref{Ter4}, we deduce that $\cfw \in
\ogr{L^2(\mu)}$. Since $\phi$ is a surjection and
$\widehat w_n(x) >0$ for all $x\in X$ and $n\in
\zbb_+$, we infer from Proposition \ref{lemS1}(v) that
   \begin{align} \label{fidef2}
0 < \frac{\mu_{\widehat
w_n}(\phi^{-n}(\{x\}))}{\mu(x)} = \hfwn{n}(x) \Le
\|\cfw^n\|^2, \quad x \in X, \, n\in \zbb_+.
   \end{align}
Hence, by Lemma \ref{jadro}, $\cfw$ is injective. It
follows from \eqref{Ter7} that $C_{\phi}$ is
well-defined and $\mathsf{h}_{\phi}(x)=\infty$ for all
$x\in X$. In view of Proposition \ref{lemS1}(iii),
$\dz{C_{\phi}} = \{0\}$. We show that
$\mathsf{h}_{\phi^n}(x)=\infty$ for all $x\in X$ and
$n\in\nbb$. Indeed, using induction and the
surjectivity of the mapping $\nbb \ni n \longmapsto
\omega_n \in X$, we easily verify that
   \begin{align} \label{Ter8}
\text{for any $(n,x)\in \nbb\times X$ there exists
$k\in \nbb$ such that $\varOmega_k \subseteq
\phi^{-n}(\{x\})$.}
   \end{align}
This implies that
   \begin{align} \label{dr1}
\mathsf{h_{\phi^n}}(x) = \mu(\phi^{-n}(\{x\})) =
\infty, \quad x\in X, \, n \in \nbb.
   \end{align}
Hence, by Propositions \ref{lemS1}(iii) and
\ref{lemS2}, $\dz{C_{\phi^n}}=\{0\}$ and the measure
$\mu|_{\phi^{-n}(\ascr)}$ is not $\sigma$-finite for
every $n\in \nbb$. It follows from \eqref{poz2-2},
\eqref{fidef2} and \eqref{dr1} that \eqref{etyk1}
holds.
   \end{exa}
   \begin{center}
   \begin{figure}[t]
   \subfigure{\includegraphics[scale=0.80]{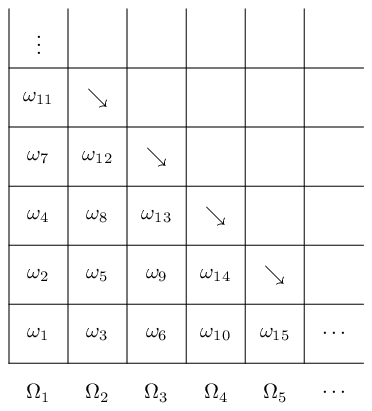}}
\caption{A sample construction of
$\{\varOmega_n\}_{n=1}^{\infty}$ and
$\{\omega_n\}_{n=1}^{\infty}$ satisfying the
requirements of Example \ref{ciach4}.}
   \label{figura6}
   \end{figure}
   \end{center}
   \begin{exa} \label{ciach4}
As in Example \ref{ciach3}, we assume that $N=\infty$.
Let $\{\varOmega_n\}_{n=1}^{\infty}$ be a partition of
$X$ such that each set $\varOmega_n$ is infinite and
let $\nbb \ni n \longmapsto \omega_n \in X$ be a
bijection such that
   \begin{align} \label{inda}
\text{$\omega_1 \in \varOmega_1$ and $\omega_n \in
\bigsqcup_{j=1}^{n-1} \varOmega_j$ for all $n\in
\nbb\setminus \{1\}$.}
   \end{align}
We refer the reader to Figure \ref{figura6} for an
idea of how to construct such objects. Let $\phi$ be
as in \eqref{fidef}. Obviously, $\phi$ is a
surjection. Our next goal is to show that for any
fixed $c\in (1,\infty)$, there exists a finite measure
$\mu$ on $\ascr$ such that
   \begin{align}  \label{nocik}
\sup_{n\in \nbb}
\frac{\mu(\varOmega_n)}{\mu(\omega_n)} \Le c.
   \end{align}
(By rescaling, the above measure $\mu$ can always be
made probabilistic.) We will construct it by
induction. First, we choose
$\{\mu(\omega)\}_{\omega\in \varOmega_1} \subseteq
(0,\infty)$ such that
$\frac{\mu(\varOmega_1)}{\mu(\omega_1)} \Le c$.
Suppose we have constructed $\{\mu(\omega)\}_{\omega
\in \bigsqcup_{j=1}^{n}\varOmega_j} \subseteq
(0,\infty)$ such that
   \begin{align} \label{inda2}
\text{$\frac{\mu(\varOmega_j)}{\mu(\omega_j)} \Le c$
and $\mu(\varOmega_j) \Le 2^{-(j-1)} \mu(\varOmega_1)$
for any $j=1, \ldots, n$,}
   \end{align}
where $n\in \nbb$ is arbitrarily fixed. Since, by
\eqref{inda}, $\omega_{n+1} \in
\bigsqcup_{j=1}^{n}\varOmega_j$ and $\varOmega_{n+1}$
has empty intersection with
$\bigsqcup_{j=1}^{n}\varOmega_j$, we can choose
$\{\mu(\omega)\}_{\omega\in \varOmega_{n+1}} \subseteq
(0,\infty)$ such that
$\frac{\mu(\varOmega_{n+1})}{\mu(\omega_{n+1})} \Le c$
and $\mu(\varOmega_{n+1}) \Le 2^{-n}
\mu(\varOmega_1)$. This completes the induction
argument. As a consequence of \eqref{inda2}, the
measure $\mu$ is finite. Applying \eqref{Ter5} and
\eqref{nocik}, we deduce that $C_{\phi} \in
\ogr{L^2(\mu)}$. Since $\phi$ is a surjection, we
deduce from Proposition \ref{lemS1}(v) that
   \begin{align} \label{fidef2+}
0 < \frac{\mu(\phi^{-n}(\{x\}))}{\mu(x)} =
\mathsf{h}_{\phi^n}(x) \Le \|C_{\phi}^n\|^2, \quad x
\in X, \, n\in \zbb_+.
   \end{align}
Thus, by Lemma \ref{jadro}, $C_{\phi}$ is injective.
It is easily seen that there exists $w \colon X \to
[1,\infty)$ such that $\mu_w(\varOmega_n) = \infty$
for all $n\in \nbb$. This implies that
   \begin{align} \label{Ter9}
\mu_{\widehat w_m}(\varOmega_n) = \infty, \quad m,n\in
\nbb.
   \end{align}
Using \eqref{Ter9}, we see that
   \begin{align}  \label{Wro}
\hfwn{n}(x) = \frac{\mu_{\widehat
w_n}(\phi^{-n}(\{x\}))}{\mu(x)}
\overset{\eqref{Ter8}}\Ge \frac{\mu_{\widehat
w_n}(\varOmega_k)}{\mu(x)} = \infty, \quad x\in X, \,
n\in \nbb,
   \end{align}
where $k$ depends on $n$ and $x$. Hence, by
Propositions \ref{lemS1}(iii) and \ref{lemS2},
$\dz{C_{\phi^n,\widehat w_n}}=\{0\}$ and the measure
$\mu_{\widehat w_n}|_{\phi^{-n}(\ascr)}$ is not
$\sigma$-finite for every $n\in \nbb$. In particular,
$\dz{\cfw}=\{0\}$. Finally, by \eqref{poz1-1},
\eqref{fidef2+} and \eqref{Wro}, the equality
\eqref{etyk2} is satisfied.
   \end{exa}
   \begin{rem}
Note that Example \ref{ciach4} can also be
accomplished for infinite $\sigma$-finite measures
$\mu$. It suffices to consider a countably infinite
orthogonal sum of (weighted) composition operators
constructed in Example \ref{ciach4} with probability
measures (cf.\ \cite[Corollary C.2]{b-j-j-sS}). What
we get is a well-defined weighted composition operator
$\cfw$ in $L^2(\mu)$ with an infinite $\sigma$-finite
measure $\mu$ and a positive weight function $w$ such
that $\dz{C_{\phi^n,\widehat w_n}}=\{0\}$ and the
measure $\mu_{\widehat w_n}|_{\phi^{-n}(\ascr)}$ is
not $\sigma$-finite for every $n\in \nbb$, $C_{\phi}
\in \ogr{L^2(\mu)}$, $C_{\phi}$ is injective and the
equality \eqref{etyk2} holds.
   \end{rem}

Now we show that using Procedure \ref{D-c} we can
construct a bounded finite rank weighted composition
operator with predetermined rank.
   \begin{exa} \label{byby}
Let $(X, \ascr, \mu)$, $\phi$ and $w$ be as in
Procedure \ref{D-c}. Assume that $N\in \nbb$ and
$\mu_w$ is finite. In view of \eqref{Ter6}, $\hfw \in
L^{\infty}(\mu)$, which by Proposition \ref{lemS1}(v)
implies that $\cfw\in \ogr{L^2(\mu)}$. It follows from
\eqref{fidef} that
   \begin{align*}
\cfw f =
{\sum\nolimits^{\smalloplus}}_{\hspace{-1ex}n=1}^N \,
f(\omega_n) \big(w \cdot \chi_{\varOmega_n}\big),
\quad f\in L^2(\mu),
   \end{align*}
where the symbol $\sum\nolimits^{\smalloplus}$ denotes
the orthogonal sum. This implies that $\cfw$ is a
finite rank operator with $\dim \ob{\cfw} = N$. Since
$L^2(\mu)$ is infinite dimensional, the range of
$\cfw$ is not dense in $L^2(\mu)$. In particular, if
the measure $\mu$ is finite, then the composition
operator $C_{\phi}=C_{\phi,\boldsymbol{1}}$ is a
bounded finite rank operator with $\dim \ob{\cfw} =
N$.
   \end{exa}
The next example shows that the operators $\cfw$ and
$M_w C_{\phi}$ may not coincide even if the operators
$C_{\phi}$ and $\cfw$ are subnormal.
   \begin{exa} \label{mwc8}
Set $X=\zbb$ and $\ascr=2^X$. Let
$\{\lambda_n\}_{n\in\zbb}$ be a two-sided sequence of
positive real numbers. Then there exists a unique
(necessarily $\sigma$-finite) measure $\mu\colon \ascr
\to \rbop$ such that
   \begin{align*}
\mu(n)=
   \begin{cases}
\lambda_0^2 \cdot \ldots \cdot \lambda_{n-1}^2 &
\text{for } n\Ge 1,
   \\
1 & \text{for } n=0,
   \\
(\lambda_{-1}^2 \cdot \ldots \cdot \lambda_{n}^2)^{-1}
& \text{for } n \Le -1,
   \end{cases}
\quad n \in \zbb.
   \end{align*}
Define the ($\ascr$-measurable) transformation $\phi$
of $X$ by $\phi(n)=n-1$ for $n \in \zbb$. Let $w\colon
X \to \cbb$ be any weight function. Clearly, the
operators $C_{\phi}$ and $\cfw$ are well-defined. Note
that
   \begin{align*}
\mu(n+1) = \lambda_n^2 \, \mu(n), \quad n \in \zbb,
   \end{align*}
which implies that
   \begin{align} \label{sob1}
   \left.
\begin{gathered} \mathsf{h}_{\phi}(n) =
\frac{\mu(n+1)}{\mu(n)} = \lambda_n^2, \quad n \in
\zbb,
   \\
\hfw(n) = \frac{\mu_w(n+1)}{\mu(n)} = \lambda_n^2
|w(n+1)|^2, \quad n \in \zbb.
   \end{gathered}
   \quad \right\}
   \end{align}
In view of Proposition \ref{lemS2}, the operators
$C_{\phi}$ and $\cfw$ are densely defined. In turn, by
Lemma \ref{jadro} and \eqref{sob1}, $C_{\phi}$ is
always injective, and $\cfw$ is injective whenever
$w(n)\neq 0$ for all $n\in \zbb$. It follows from
\eqref{sob1} and Theorem \ref{mwc2} that
   \begin{align} \label{bialyz2}
\cfw = M_w C_{\phi} \iff \inf_{n \in \zbb} \frac{1+
\lambda_n^2|w(n+1)|^2}{\lambda_n^2}
> 0.
   \end{align}

Now we show that (see Section \ref{pco}(h))
   \begin{align} \label{bialyz}
   \begin{minipage}{65ex}
{\em the composition operator $C_{\phi}$ is unitarily
equivalent to the bilateral weighted shift in
$\ell^2(\zbb)$ with weights $\{\lambda_n\}_{n\in
\zbb}$.}
   \end{minipage}
   \end{align}
Indeed, it is a matter of routine to verify that the
sequence $\{\hat e_n\}_{n\in \zbb}$ defined by
   \begin{align*}
\hat e_n=\frac{1}{\sqrt{\mu(n)}} \chi_{\{n\}}, \quad
n\in \zbb,
   \end{align*}
is an orthonormal basis of $L^2(\mu)$ such that
$\escr:=\lin\{\hat e_n\colon n\in \zbb\} \subseteq
\dz{C_{\phi}}$ and
   \begin{align} \label{Amanda3}
C_{\phi} \hat e_n = \lambda_n \hat e_{n+1}, \quad n
\in \zbb.
   \end{align}
Since $C_{\phi}$ is closed (see Proposition
\ref{lemS1}(iv)), it remains to prove that $\escr$ is
a core for $C_{\phi}$. Take a function $f\in
\dz{C_{\phi}}$ which is orthogonal to $\escr$ with
respect to the graph inner product
$\is{\cdot}{\mbox{-}}_{C_{\phi}}$. Then
   \begin{align*}
0 = \is{f}{\hat e_n}_{C_{\phi}} =
\frac{\mu(n)+\mu(n+1)}{\sqrt{\mu(n)}} f(n), \quad n
\in \zbb,
   \end{align*}
which yields $f=0$. This completes the proof of
\eqref{bialyz}.

The assertion \eqref{bialyz} enables us to construct a
sequence $\{\lambda_n\}_{n=0}^{\infty}$ of positive
real numbers and a weight $w\colon X \to (0,\infty)$
such that (cf.\ Proposition \ref{mwc3})
   \begin{itemize}
   \item[(a)] $C_{\phi}$ and $\cfw$ are subnormal
operators,
   \item[(b)] $C_{\phi} \notin
\ogr{L^2(\mu)}$, $\cfw \notin \ogr{L^2(\mu)}$ and
$\inf_{n\in \zbb} |w(n)|=0$,
   \item[(c)] $M_w C_{\phi} \varsubsetneq \cfw$ and
$\overline{M_w C_{\phi}}=\cfw$,
   \item[(d)] there is no constant $c\in\rbb_+$
such that $\hfw \Le c \, \mathsf{h}_{\phi}$ (a.e.\
$[\mu]$),
   \item[(e)]  there is no constant $c\in\rbb_+$
such that $\mathsf{h}_{\phi} \Le c \, \hfw$ (a.e.\
$[\mu]$).
   \end{itemize}
For this purpose, we fix $q\in (0,1)$ and set
   \begin{align*}
\text{$\lambda_n=q^{-\frac{1}{4} (2n+1)}$ and
$w(n)=q^{\frac{1}{8}(2n-1)}$ for $n\in \zbb$.}
   \end{align*}
Using \eqref{Amanda3}, we see that
$C_{\phi}(\escr)=\escr$. Since $C_{\phi}$ is
injective, we infer from \eqref{Amanda3} that
   \begin{align}  \label{Amanda4}
\|C_{\phi}^n \hat e_0\|^2=q^{-\frac{1}{2}n^2}, \quad
n\in \zbb.
   \end{align}
Knowing that $\{q^{-\frac{1}{2}n^2}\}_{n\in \zbb}$ is
a two-sided Stieltjes moment sequence (see \cite[p.\
J.106]{sti}; see also \cite{chi,lei,ber2}), we deduce
from \eqref{bialyz}, \eqref{Amanda4} and \cite[Theorem
5]{StSz2} that
   \begin{align} \label{kaczka1}
\text{the composition operator $C_{\phi}$ is
subnormal.}
   \end{align}
Clearly, $\escr \subseteq \dz{M_w C_{\phi}}$ and $M_w
C_{\phi} \subseteq \cfw$, so by \eqref{Amanda3}, we
have
   \begin{align} \label{Amanda6}
\cfw \hat e_n = w(n+1)\lambda_n \hat e_{n+1}, \quad n
\in \zbb.
   \end{align}
Arguing as in the previous paragraph (now $\is{f}{\hat
e_n}_{\cfw} =
\frac{\mu(n)+\mu(n+1)w(n+1)^2}{\sqrt{\mu(n)}} f(n)$),
we verify that $\escr$ is a core for $\cfw$. In
particular, we have
   \begin{align} \label{kaczka2}
\overline{M_w C_{\phi}}=\cfw.
   \end{align}
Since $\escr$ is a core for $\cfw$, we infer from
\eqref{Amanda6} that $\cfw$ is unitarily equivalent to
the bilateral weighted shift in $\ell^2(\zbb)$ with
weights $\{w(n+1)\lambda_n\}_{n\in \zbb}$. Using
\eqref{Amanda6}, we see that $\cfw(\escr)=\escr$.
Since $\cfw$ is injective, we deduce from
\eqref{Amanda6} that
   \begin{align*}
\|\cfw^n \hat e_0\|^2 = q^{-\frac{1}{4}n^2}, \quad
n\in \zbb.
   \end{align*}
Applying \cite[Theorem 5]{StSz2} once more, we see
that the weighted composition operator $\cfw$ is
subnormal. This, together with \eqref{kaczka1}, proves
(a). By \eqref{sob1} and Proposition~ \ref{lemS1}, the
operators $C_{\phi}$ and $\cfw$ are unbounded, and so
(b) holds. In view of \eqref{bialyz2}, $M_w C_{\phi}
\varsubsetneq \cfw$. This, combined with
\eqref{kaczka2}, justifies (c). Finally, (d) and (e)
follow from \eqref{sob1}.
   \end{exa}
As shown below, it may happen that a nonzero weighted
composition operator $\cfw$ is densely defined, the
composition operator $C_{\phi}$ is well-defined and
not densely defined (cf.\ Example \ref{ciach3}) and,
what is more interesting, the product $M_{w}C_{\phi}$
can be made closed or not, according to our needs.
   \begin{exa} \label{mwc5}
Set $X=\zbb$ and $\ascr=2^X$. Let $\mu\colon \ascr \to
\rbop$ be any (necessarily $\sigma$-finite) measure
such that $0 < \mu(n) < \infty$ for all $n \in X$.
Define the ($\ascr$-measurable) transformation $\phi$
of $X$ by
   \begin{align*}
\phi(n) =
   \begin{cases}
n-1 & \text{if } n \Le 0,
   \\
0 & \text{if } n\Ge 1,
   \end{cases}
\quad n \in X.
   \end{align*}
Let $w\colon X \to \cbb$ be a weight function such
that $w(n)\neq 0$ for all $n\in X$. Clearly, the
operators $C_{\phi}$ and $\cfw$ are well-defined and
the following equalities hold:
   \begin{align}  \label{was1}
\mathsf{h}_{\phi}(n) & =
   \begin{cases}
0 & \text{if } n \Ge 1,
   \\[1ex]
\frac{\sum_{k=1}^{\infty} \mu(k)}{\mu(0)} & \text{if }
n=0,
   \\[1ex]
\frac{\mu(n+1)}{\mu(n)} & \text{if } n \Le -1,
   \end{cases}
\quad n \in X,
   \\[1ex]  \label{was2}
\hfw(n) & =
   \begin{cases}
0 & \text{if } n \Ge 1,
   \\[1ex]
\frac{\sum_{k=1}^{\infty} |w(k)|^2 \mu(k)}{\mu(0)} &
\text{if } n=0,
   \\[1ex]
\frac{|w(n+1)|^2 \mu(n+1)}{\mu(n)} & \text{if } n \Le
-1,
   \end{cases}
\quad n \in X.
   \end{align}
Assume now that $\sum_{k=1}^{\infty} \mu(k)=\infty$
and $\sum_{k=1}^{\infty} |w(k)|^2\mu(k) < \infty$. It
follows from \eqref{was1}, \eqref{was2} and
Proposition \ref{lemS2} that $\cfw$ is densely
defined, while $C_{\phi}$ is not. In view of
Proposition \ref{mwc1} and Theorem \ref{mwc2}, $M_w
C_{\phi} \varsubsetneq \cfw$. By Theorem \ref{mwc4},
we have
   \begin{align*}
M_{w} C_{\phi} \text{ is closed } \iff \inf_{n\Le -1}
\bigg(\frac{\mu(n)}{\mu(n+1)} + |w(n+1)|^2\bigg)
> 0.
   \end{align*}
Now it is easily seen that, by appropriate choices of
the sequences $\{\mu(n)\}_{n=-\infty}^{-1}$ and
$\{w(n)\}_{n=-\infty}^{-1}$, the product
$M_{w}C_{\phi}$ can be made closed or not, according
to our needs. This also shows that the implication
(ii)$\Rightarrow$(i) in Proposition \ref{fonta2} is
false if we do not assume that $C_{\phi}$ is densely
defined.
   \end{exa}
   \chapter{Miscellanea}
This chapter consists of three sections. In Section
\ref{Sec8.1}, we discuss the problem of when the
tensor product of (finitely many) weighted composition
operators can be regarded as a weighted composition
operator. A partial answer is given in Theorem
\ref{amierz} (see also Corollary \ref{amierz3+}). In
Theorem \ref{cltenp} we show that the closure of a
tensor product of densely defined weighted composition
operators can be regarded as a weighted composition
operator. Two open question related to the above
topics are stated as well (see Problems \ref{probl1}
and \ref{Pro2}). Section \ref{Sect8.2} proposes a
method of modifying the symbol of a weighted
composition operator $\cfw$ which preserves many
properties of objects attached to $\cfw$ and does not
change the operator $\cfw$ itself. As shown in Section
\ref{Sec8.3}, this method enables to modify the symbol
$\phi$ of a quasinormal weighted composition operator
$\cfw$ so as to get a $\phi^{-1}(\ascr)$-measurable
family $P\colon X \times \borel{\rbb_+} \to [0,1]$ of
probability measures that satisfies \eqref{cc-1mu}
(see Proposition \ref{munich3}). We conclude Section
\ref{Sec8.3} with an example of a quasinormal weighted
composition operator $\cfw$ which has no
$\phi^{-1}(\ascr)$-measurable family of probability
measures satisfying \eqref{cc-1mu} (see Example
\ref{quasif}).
   \section{\label{Sec8.1}Tensor products}
In this section, we show that, under certain
circumstances, the tensor product of weighted
composition operators can be regarded as a weighted
composition operator.

The following notations and assumptions are fixed
throughout this section. Let $N$ be an integer such
that $N \Ge 2$ and let $J_N=\{1, \ldots, N\}$. Suppose
that for each $i \in J_N$, $(X_i,\ascr_i,\mu_i)$ is a
$\sigma$-finite measure space. Set
   \begin{align*}
X&=X_1 \times \ldots \times X_N,
   \\
\ascr & = \ascr_1 \otimes \ldots \otimes \ascr_N =
\sigma_X\big(\ascr_1 \boxtimes \ldots \boxtimes
\ascr_N\big),
   \\ & \hspace{3ex}\text{where } \ascr_1 \boxtimes
\ldots \boxtimes \ascr_N = \big\{\varDelta_1\times
\ldots \times \varDelta_N\colon \varDelta_i \in
\ascr_i \text{ for } i\in J_N\big\}.
   \end{align*}
Then there exists a unique (necessarily
$\sigma$-finite) measure $\mu=\mu_1 \otimes \ldots
\otimes \mu_N \colon \ascr \to \rbop$ such that (see
\cite[Sect.\ 2.6]{Ash})
   \begin{align*}
\mu_1 \otimes \ldots \otimes \mu_N(\varDelta_1\times
\ldots \times \varDelta_N) = \mu_1(\varDelta_1) \cdot
\ldots \cdot \mu_N(\varDelta_N), \quad \varDelta_i \in
\ascr_i \text{ for } i\in J_N.
   \end{align*}
The measure $\mu$ is called the {\em product of
measures} $\mu_1, \ldots, \mu_N$. It is well-known
that the Hilbert space $L^2(\mu)$ can be identified
with the complete tensor product of Hilbert spaces
$L^2(\mu_1) \hat\otimes \ldots \hat\otimes L^2(\mu_N)$
(see \cite[Sect.\ II.4]{R-S} and \cite[Sect.\
3.4]{Weid}), where the tensor product mapping
   \begin{align*}
L^2(\mu_1) \times \ldots \times L^2(\mu_N) \ni (f_1,
\ldots,f_N) \longmapsto f_1 \otimes \ldots \otimes f_N
\in L^2(\mu)
   \end{align*}
is given by
   \begin{align} \label{tenpr}
(f_1 \otimes \ldots \otimes f_N)(x_1, \ldots, x_N) =
f_1(x_1) \cdot \ldots \cdot f_N(x_N), \quad (x_1,
\ldots, x_N) \in X.
   \end{align}
Next suppose that for each $i \in \{1, \ldots, N\}$,
$w_i$ is an $\ascr_i$-measurable complex function on
$X_i$ and $\phi_i$ is an $\ascr_i$-measurable
transformation of $X_i$. Set $w = w_1 \otimes \ldots
\otimes w_N$ and $\phi = \phi_1 \times \ldots \times
\phi_N \colon X \to X$, where
   \begin{align*}
(\phi_1 \times \ldots \times \phi_N)(x_1, \ldots, x_N)
= (\phi_1(x_1), \ldots, \phi_N(x_N)) \text{ for }
(x_1, \ldots, x_N) \in X.
   \end{align*}
Note that the function $w$ is $\ascr$-measurable.
Since $\phi^{-1}(\sigma_Z(\cscr)) =
\sigma_Y(\phi^{-1}(\cscr))$ for any function
$\phi\colon Y \to Z$ and any family $\cscr \subseteq
2^Z$, we see that the transformation $\phi$ is
$\ascr$-measurable as well. The above notations and
assumptions will be used throughout this section.

We begin by discussing the question of when the
weighted composition operator $\cfw$ is the zero
operator on the entire $L^2(\mu)$.
   \begin{lem} \label{poprokr0}
The composition operator $\cfw$ is the zero operator
on $L^2(\mu)$ if and only if there exists $i\in J_N$
such that the operator $C_{\phi_i,w_i}$ is the zero
operator on $L^2(\mu_i)$ $($the other operators
$C_{\phi_j,w_j}$, $j \neq i$, may not even be
well-defined\/$)$.
   \end{lem}
   \begin{proof} By Fubini theorem we have
   \begin{align} \label{oro}
\mu_w (X_1 \times \ldots \times X_N) = \prod_{i=1}^N
\int_{X_i} |w_i|^2 \D \mu_i.
   \end{align}
Applying Proposition \ref{lemS1}(vi) and \eqref{oro}
completes the proof.
   \end{proof}
The second issue we want to discuss is the question of
whether the assumption that the operator $\cfw$ is
well-defined implies that the operators
$C_{\phi_i,w_i}$, $i\in J_N$, are well-defined.
   \begin{lem} \label{poprokr}
If the composition operator $\cfw$ is well-defined and
if it is not the zero operator on $L^2(\mu)$, then
each operator $C_{\phi_i,w_i}$, $i\in J_N$, is
well-defined.
   \end{lem}
   \begin{proof}
It follows from Lemma \ref{poprokr0} that
   \begin{align} \label{ajaj}
\int_{X_i} |w_i|^2 \D \mu_i \in (0,\infty], \quad i\in
J_N.
   \end{align}
By Proposition \ref{wco1}, $\mu_w \circ \phi^{-1} \ll
\mu$. Let us fix $i\in J_N$. If $\varDelta \in
\ascr_i$ is such that $\mu_i(\varDelta)=0$, then
   \begin{align*}
\mu(X_1 \times \ldots \times X_{i-1} \times \varDelta
\times X_{i+1} \times \ldots \times X_N) = 0,
   \end{align*}
and consequently, by the Fubini theorem, we have
   \begin{align*}
(\mu_i)_{w_i} \circ \phi_i^{-1} (\varDelta) \;\; \cdot
& \prod_{j\in J_N \setminus \{i\}}\int_{X_i} |w_i|^2
\D \mu_i
   \\
& = \mu\circ \phi^{-1} (X_1 \times \ldots \times
X_{i-1} \times \varDelta \times X_{i+1} \times \ldots
\times X_N)=0.
   \end{align*}
This and \eqref{ajaj} imply that $(\mu_i)_{w_i} \circ
\phi_i^{-1} (\varDelta)=0$. Hence, by Proposition
\ref{wco1}, the operator $C_{\phi_i,w_i}$ is
well-defined. This completes the proof.
   \end{proof}
Now we discuss the converse question.
   \begin{opq} \label{probl1}
Does the assumption that the operators
$C_{\phi_i,w_i}$, $i\in J_N$, are well-defined imply
that the operator $\cfw$ is well-defined\/{\em ?}
   \end{opq}
This problem remains unsolved in full generality.
Below we provide some partial solutions. First, we fix
some notation. Given $\cscr\subseteq 2^Z$, we denote
by $\cscr_{\sigma}$ the class of countable unions of
sets in $\cscr$. Note that $\cscr_{\sigma}$ is the
smallest subclass of $2^Z$ containing $\cscr$ and
closed under the formation of countable unions.
Moreover, if $\cscr$ is an algebra, then
$\cscr_{\sigma}$ coincides with the class of countable
increasing unions of sets in $\cscr$. If $Z$ is a
disjoint union of sets $Z_1$ and $Z_2$,
$\cscr_1\subseteq 2^{Z_1}$ and $\cscr_2 \subseteq
2^{Z_2}$, then we write
   \begin{align*}
\cscr_1 \uplus \cscr_2 = \{C_1 \cup C_2 \colon C_1 \in
\cscr_1, \, C_2 \in \cscr_2\}.
   \end{align*}
It is easy to see that
   \begin{align} \label{c1c2}
(\cscr_1 \uplus \cscr_2)_{\sigma} = (\cscr_1)_{\sigma}
\uplus (\cscr_2)_{\sigma}.
   \end{align}
   \begin{lem} \label{amierz+}
Suppose $(\mu_i)_{w_i} \circ \phi_i^{-1} \ll \mu_i$
for each $i\in J_N$. Set
   \begin{align*}
\hfwi{i} = \frac{\D (\mu_i)_{w_i} \circ
\phi_i^{-1}}{\D\mu_i}, \quad i\in J_N,
   \end{align*}
and $\varOmega_{\infty} = \{\hfwi{1} \otimes \ldots
\otimes \hfwi{N} < \infty\}$ $($see \eqref{tenpr} for
notation$)$. Then the following assertions are
valid{\em :}
   \begin{enumerate}
   \item[(i)] $\varOmega_{\infty}
\in \mathrm{alg}(\ascr_1 \boxtimes \ldots \boxtimes
\ascr_N):=\text{the algebra generated by $\ascr_1
\boxtimes \ldots \boxtimes \ascr_N$}$,
   \item[(ii)]  $\qscr := \Big\{\varDelta \in \ascr \colon \varDelta
\subseteq \varOmega_{\infty}\Big\} \uplus
\Big\{\varDelta \in \mathrm{alg}(\ascr_1 \boxtimes
\ldots \boxtimes \ascr_N)\colon \varDelta \subseteq
X\setminus \varOmega_{\infty}\Big\}$ is an algebra of
sets which contains $\ascr_1 \boxtimes \ldots
\boxtimes \ascr_N$, and
   \begin{align*}
\qscr_{\sigma} = \Big[\ascr \cap
\varOmega_{\infty}\Big] \uplus
\Big[(\mathrm{alg}(\ascr_1 \boxtimes \ldots \boxtimes
\ascr_N))_{\sigma} \cap (X \setminus
\varOmega_{\infty})\Big],
   \end{align*}
   \item[(iii)] $\mu_w \circ \phi^{-1}(\varDelta) = \int_{\varDelta}
\hfwi{1} \otimes \ldots \otimes \hfwi{N} \D \mu$ for
every $\varDelta \in \mathcal \qscr_{\sigma}$,
   \item[(iv)] the measure $\mu_w \circ \phi^{-1}$ is
$\sigma$-finite on $\varOmega_{\infty}$.
   \end{enumerate}
   \end{lem}
   \begin{proof}
If $(\varDelta_1, \ldots, \varDelta_N) \in \ascr_1
\times \ldots \times \ascr_N$, then by the Fubini
theorem we have
   \allowdisplaybreaks
   \begin{align} \notag
\mu_w \circ \phi^{-1}(\varDelta_1\times \ldots \times
\varDelta_N) & = \int_X
\chi_{\phi_1^{-1}(\varDelta_1)} \otimes \ldots \otimes
\chi_{\phi_N^{-1}(\varDelta_N)} |w_1 \otimes \ldots
\otimes w_N|^2 \D \mu
   \\ \notag
& = \prod_{i=1}^N (\mu_i)_{w_i} \circ
\phi_i^{-1}(\varDelta_i)
   \\ \label{alpr}
& = \int_{\varDelta_1\times \ldots \times \varDelta_N}
\hfwi{1} \otimes \ldots \otimes \hfwi{N} \D \mu.
   \end{align}
Fix $i\in J_N$. Let $\varXi_i=\{\hfwi{i} < \infty\}$
and $\cscr_i = \ascr_i \cap \varXi_i$. Set $\varXi=
\varXi_1 \times \ldots \times \varXi_N$ and
$\cscr=\{\varDelta \in \ascr \colon \varDelta
\subseteq \varXi\}$. Clearly $\varXi \in \ascr$ and
the class of sets $\cscr_i$ is a $\sigma$-algebra of
subsets of $\varXi_i$. Notice that
   \begin{align} \notag
\cscr & = \ascr \cap \varXi
   \\ \notag
&= \sigma_X \big(\ascr_1\boxtimes \ldots \boxtimes
\ascr_N\big) \cap \varXi
   \\ \notag
& \hspace{-.25ex}\overset{(\dag)} = \sigma_{\varXi}
\big((\ascr_1\boxtimes \ldots \boxtimes \ascr_N) \cap
\varXi\big)
   \\ \notag
& = \sigma_{\varXi} \big(\cscr_1\boxtimes \ldots
\boxtimes \cscr_N\big)
   \\ \label{dyctu0}
& = \cscr_1 \otimes \ldots \otimes \cscr_N,
   \end{align}
where $(\dag)$ follows from \cite[Sect.\ 1.2.2]{Ash}.
Since $\mu_i$ is $\sigma$-finite, there exists a
sequence $\{\varXi_{i,n}\}_{n=1}^\infty \subseteq
\ascr_i$ such that
   \begin{align} \label{wkonc}
\text{$\mu_i(\varXi_{i,n}) < \infty$ and $\hfwi{i} \Le
n$ on $\varXi_{i,n}$ for every $n\in \nbb$,}
   \end{align}
and $\varXi_{i,n} \nearrow \varXi_i$ as $n\to \infty$.
Define the measure $\nu\colon \ascr \to \rbop$ by
   \begin{align} \label{defnu}
\nu(\varDelta) = \int_{\varDelta} \hfwi{1} \otimes
\ldots \otimes \hfwi{N} \D \mu, \quad \varDelta \in
\ascr.
   \end{align}
It follows from \eqref{alpr} that the measures $\mu_w
\circ \phi^{-1}|_{\cscr}$ and $\nu|_{\cscr}$ coincide
on the semi-algebra $\cscr_1\boxtimes \ldots \boxtimes
\cscr_N$. Moreover, the sequence $\{\varXi_{1,n}
\times \ldots \times \varXi_{N,n}\}_{n=1}^{\infty}
\subseteq \cscr_1\boxtimes \ldots \boxtimes \cscr_N$
is such that $\varXi_{1,n} \times \ldots \times
\varXi_{N,n} \nearrow \varXi$ as $n\to \infty$ and, by
\eqref{wkonc}, $\nu(\varXi_{1,n} \times \ldots \times
\varXi_{N,n}) < \infty$ for all $n\in \nbb$. Hence, by
\eqref{dyctu0} and Lemma \ref{2miary}, we have
   \begin{align} \label{Daeg0}
\mu_w \circ \phi^{-1}(\varDelta)= \nu(\varDelta),
\quad \varDelta \in \ascr, \, \varDelta \subseteq
\varXi.
   \end{align}

(i) For this, note that the set $\varOmega_{\infty}$
can be decomposed into a finite disjoint union of sets
in the semi-algebra $\ascr_1 \boxtimes \ldots
\boxtimes \ascr_N$ as follows:
   \begin{align} \label{Daeg1}
\varOmega_{\infty} & = \varXi \sqcup
\bigsqcup_{\substack{\alpha \in \{0,1,\infty\}^N: \\
|\alpha|=\infty, \, \sqcap(\alpha) = 0}}
\varGamma_{1,\alpha_1} \times \ldots \times
\varGamma_{N,\alpha_N},
   \end{align}
where $|\alpha|=\sum_{i=1}^N\alpha_i$ and
$\sqcap(\alpha) = \alpha_1 \cdot \ldots \cdot
\alpha_N$ for $\alpha=(\alpha_1, \ldots, \alpha_N)\in
\{0,1,\infty\}^N$, and
   \begin{align}  \label{notation}
\varGamma_{i,\beta} =
   \begin{cases}
\{\hfwi{i} = 0\} & \text{if } \beta=0,
   \\
\{0 < \hfwi{i} < \infty\} & \text{if } \beta=1,
   \\
\{\hfwi{i} = \infty\} & \text{if } \beta=\infty,
   \end{cases}
\qquad \beta \in \{0,1,\infty\}, \, i\in J_N.
   \end{align}
This yields (i).

(ii) Use (i) and \eqref{c1c2}.

(iii) If $\alpha \in \{0,1,\infty\}^N$ is such that
$|\alpha|=\infty$ and $\sqcap(\alpha) = 0$, then by
\eqref{alpr} we have
   \begin{align} \label{sun1}
\mu_w \circ \phi^{-1}(\varGamma_{1,\alpha_1} \times
\ldots \times \varGamma_{N,\alpha_N})=
\nu(\varGamma_{1,\alpha_1} \times \ldots \times
\varGamma_{N,\alpha_N})=0,
   \end{align}
and consequently $\mu_w \circ \phi^{-1}(\varDelta)=
\nu(\varDelta)=0$ for every $\varDelta \in \ascr$ such
that $\varDelta \subseteq \varGamma_{1,\alpha_1}
\times \ldots \times \varGamma_{N,\alpha_N}$. This
together with \eqref{Daeg0} and \eqref{Daeg1} implies
that
   \begin{align} \label{14.IV.2016}
\mu_w \circ \phi^{-1}(\varDelta)= \nu(\varDelta),
\quad \varDelta \in \ascr, \, \varDelta \subseteq
\varOmega_{\infty}.
   \end{align}
Applying \cite[Proposition I-6-1]{Nev} and
\eqref{alpr}, we see that
   \begin{align*}
\mu_w \circ \phi^{-1}(\varDelta)= \nu(\varDelta),
\quad \varDelta \in \mathrm{alg}(\ascr_1 \boxtimes
\ldots \boxtimes \ascr_N), \, \varDelta \subseteq
X\setminus \varOmega_{\infty}.
   \end{align*}
This combined with \eqref{14.IV.2016} shows that
$\mu_w \circ \phi^{-1}(\varDelta) = \nu(\varDelta)$
for all $\varDelta \in \mathcal \qscr$. An application
of \cite[Theorem 1.19(d)]{Rud} yields (iii).

(iv) For this, observe that $\varXi_{1,n} \times
\ldots \times \varXi_{N,n} \cup (\varOmega_{\infty}
\setminus \varXi) \nearrow \varOmega_{\infty}$ as
$n\to \infty$ and, by \eqref{Daeg0}, \eqref{Daeg1} and
\eqref{sun1}, $\mu_w \circ \phi^{-1} (\varXi_{1,n}
\times \ldots \times \varXi_{N,n} \sqcup
(\varOmega_{\infty} \setminus \varXi)) < \infty$ for
all $n\in \nbb$. This completes the proof.
   \end{proof}
Let us make some comments regarding Lemma
\ref{amierz+}.
   \begin{rem} \label{Zentolub}
First, observe that, as in the case of
$\varOmega_{\infty}$, the set $X \setminus
\varOmega_{\infty}$ can be decomposed into a finite
disjoint union of sets in $\ascr_1 \boxtimes \ldots
\boxtimes \ascr_N$ as follows (see~ \eqref{notation}):
   \begin{align*}
   X \setminus \varOmega_{\infty} & =
\bigsqcup_{\substack{\alpha \in \{0,1,\infty\}^N: \\
|\alpha|=\infty, \, \sqcap(\alpha) > 0}}
\varGamma_{1,\alpha_1} \times \ldots \times
\varGamma_{N,\alpha_N}.
   \end{align*}
Hence, applying either the above equality or assertion
(i) of Lemma \ref{amierz+}, we obtain
   \begin{align*}
\{\varDelta \in \mathrm{alg}(\ascr_1 \boxtimes \ldots
\boxtimes \ascr_N)\colon \varDelta \subseteq
X\setminus \varOmega_{\infty}\} =
\big(\mathrm{alg}(\ascr_1 \boxtimes \ldots \boxtimes
\ascr_N)\big) \cap (X\setminus \varOmega_{\infty}).
   \end{align*}
Second, note that the measure $\nu$ defined by
\eqref{defnu} is {\em completely non-$\sigma$-finite}
on the set $X \setminus \varOmega_{\infty}$ (i.e.,
$\nu$ is not $\sigma$-finite on any $\ascr$-measurable
subset of $X \setminus \varOmega_{\infty}$ of positive
$\nu$-measure), or equivalently, $\nu$ has the
property that $\nu(\varDelta) \in \{0,\infty\}$ for
every $\ascr$-measurable subset of $X \setminus
\varOmega_{\infty}$; for such $\varDelta$'s,
$\nu(\varDelta) = 0$ if and only if $\mu(\varDelta) =
0$. It follows from Lemma \ref{amierz+} that the
measure $\nu$ is $\sigma$-finite on the set
$\varOmega_{\infty}$. This provides a decomposition of
the measure $\nu$ into a $\sigma$-finite and a
completely non-$\sigma$-finite parts (see
\cite[Exercise 30.11]{hal3}). As a consequence, by
assertion (iii) of Lemma \ref{amierz+}, $\mu_w \circ
\phi^{-1}(\varDelta) \in \{0,\infty\}$ for every
$\varDelta \in (\mathrm{alg}(\ascr_1 \boxtimes \ldots
\boxtimes \ascr_N))_{\sigma} \cap (X \setminus
\varOmega_{\infty})$.
   \end{rem}
The next two theorems give partial solutions to
Problem \ref{probl1}.
   \begin{thm} \label{amierz}
Suppose $C_{\phi_i,w_i}$ is well-defined for every
$i\in J_N$ and
   \begin{align} \label{noc-1}
\ascr=(\mathrm{alg}(\ascr_1 \boxtimes \ldots \boxtimes
\ascr_N))_{\sigma}.
   \end{align}
Then the following assertions are valid{\em :}
   \begin{enumerate}
   \item[(i)] the operator $\cfw$ is well-defined,
   \item[(ii)]  $\hfw := \frac{\D \mu_w \circ
\phi^{-1}}{\D\mu} = \hfwi{1} \otimes \ldots \otimes
\hfwi{N}$ a.e.\ $[\mu]$.
   \end{enumerate}
   \end{thm}
   \begin{proof}
Apply Proposition \ref{wco1} and Lemma \ref{amierz+}.
   \end{proof}
   Let $Z$ be a nonempty set and $\cscr \subseteq 2^Z$
be a $\sigma$-algebra. We say that a set $\varDelta
\in \cscr$ is an {\em atom} of $\cscr$ if the only
proper $\cscr$-measurable subset of $\varDelta$ is the
empty set. Denote by $\ats(\cscr)$ the set of all
atoms of $\cscr$. It is obvious that the atoms of
$\cscr$ are nonempty and pairwise disjoint. If
$\ats(\cscr)\neq \emptyset$ and $\sigma(\ats(\cscr)) =
\cscr$, then the $\sigma$-algebra $\cscr$ is called
{\em atomic}. As shown in Example \ref{predki} below,
there may exist atomic $\sigma$-algebras for which
   \begin{align*}
\bigcup \ats(\cscr) \varsubsetneq Z.
   \end{align*}
   \begin{exa} \label{predki}
Let $Z=[0,1]$ and $\cscr =
\sigma_Z\big(\big\{\{x\}\colon x \in
\big[0,\frac{1}{2}\big]\big\}\big)$. Then a subset
$\varDelta$ of $Z$ is in $\cscr$ if and only if
$\varDelta$ is a countable subset of
$\big[0,\frac{1}{2}\big]$ or it is a disjoint union of
the interval $\big(\frac{1}{2},1\big]$ and a subset of
$\big[0,\frac{1}{2}\big]$ with countable relative
complement in $\big[0,\frac{1}{2}\big]$. This implies
that $\ats(\cscr)=\big\{\{x\}\colon x \in
\big[0,\frac{1}{2}\big]\big\}$, which means that
$\ats(\cscr)$ is not a partition of $Z$.
   \end{exa}
It is a matter of routine to verify that if $Z$ is a
nonempty set, $\cscr \subset 2^Z$ is an atomic
$\sigma$-algebra and $\ats(\cscr)$ is a partition of
$Z$, then $\cscr$ consists of all sets of the form
$\bigsqcup_{\varDelta \in \fscr} \varDelta$, where
$\fscr$ is a countable subset of $\ats(\cscr)$ or it
is a subset of $\ats(\cscr)$ whose relative complement
in $\ats(\cscr)$ is countable. As a consequence, we
have:
   \begin{align} \label{bye}
   \begin{minipage}{70ex}
{\em if $\cscr \subset 2^Z$ is an atomic
$\sigma$-algebra and $\ats(\cscr)$ is a countable
partition of $Z$, then $\cscr$ consists of all sets of
the form $\bigsqcup_{\varDelta \in \fscr} \varDelta$,
where $\fscr \subseteq \ats(\cscr)$.}
   \end{minipage}
   \end{align}

We say that a measure $\tau\colon \cscr \to \rbop$ is
{\em atomic} if $\cscr$ is an atomic $\sigma$-algebra
of subsets of a nonempty set $Z$, the measure $\tau$
is $\sigma$-finite and $\tau(\varDelta)
> 0$ for every $\varDelta \in \ats(\cscr)$. If this
is the case, then the set $\ats(\cscr)$ is countable
(see \cite[Theorem 10.2(iv)]{bill}), and thus $\cscr$
can be described as in \eqref{bye}.

As shown below, the atomic measures fit well into the
scope of Theorem \ref{amierz}.
   \begin{cor}\label{amierz3+}
Suppose that for every $i\in J_N$, $C_{\phi_i,w_i}$ is
well-defined, $\mu_i$ is an atomic measure and
$\ats(\ascr_i)$ is a partition of $X_i$. Then the
following assertions are valid{\em :}
   \begin{enumerate}
   \item[(i)] the operator $\cfw$ is well-defined,
   \item[(ii)]  $\hfw = \hfwi{1} \otimes \ldots \otimes
\hfwi{N}$ a.e.\ $[\mu]$.
   \end{enumerate}
   \end{cor}
   \begin{proof}
Since each $\mu_i$ is $\sigma$-finite, $\ats(\ascr_i)$
is countable for every $i\in J_N$. Let $\cscr$ be the
class of all subsets of $X$ of the form
   \begin{align} \label{noc}
\bigsqcup_{(\varDelta_1, \ldots, \varDelta_N) \in
\fscr} \varDelta_1 \times \ldots \times \varDelta_N,
   \end{align}
where $\fscr$ is a subset of $\ats(\ascr_1) \times
\ldots \times \ats(\ascr_N)$. It is a routine matter
to check that $\cscr$ is a $\sigma$-algebra which is
contained in $\ascr$. On the other hand, if $(\tilde
\varDelta_1, \ldots, \tilde \varDelta_N)\in \ascr_1
\times \ldots \times \ascr_N$, then by \eqref{bye} for
every $i\in J_N$, there exists $\fscr_i \subseteq
\ats(\ascr_i)$ such that $\tilde \varDelta_i =
\bigsqcup_{\varDelta_i \in \fscr_i} \varDelta_i$. This
implies that $\tilde \varDelta_1 \times \ldots \times
\tilde \varDelta_N$ is of the form \eqref{noc} with
$\fscr=\fscr_1 \times \ldots \times \fscr_N$. As a
consequence of the above, $\cscr = \ascr$. This means
that the condition \eqref{noc-1} is satisfied.
Applying Theorem \ref{amierz} completes the proof.
   \end{proof}
Suppose now that the operators $C_{\phi_i,w_i}$, $i\in
J_N$, are well-defined. Then there exists a unique
operator in $L^2(\mu)$, denoted by
$C_{\phi_1,w_1}\otimes \ldots \otimes C_{\phi_N,w_N}$,
such that
   \begin{gather*}
\dz{C_{\phi_1,w_1}\otimes \ldots \otimes
C_{\phi_N,w_N}} = \dz{C_{\phi_1,w_1}}\otimes \ldots
\otimes \dz{C_{\phi_N,w_N}},
   \\
(C_{\phi_1,w_1}\otimes \ldots \otimes C_{\phi_N,w_N})
(f_1 \otimes \ldots \otimes f_N) = (C_{\phi_1,w_1}f_1)
\otimes \ldots \otimes (C_{\phi_N,w_N} f_N),
   \end{gather*}
whenever $f_i \in \dz{C_{\phi_i,w_i}}$ for $i\in J_N$
(see \cite[p.\ 262]{Weid}); here
   \begin{align*}
\dz{C_{\phi_1,w_1}}\otimes \ldots \otimes
\dz{C_{\phi_N,w_N}} = \lin \big\{f_1 \otimes \ldots
\otimes f_N\colon f_i \in \dz{C_{\phi_i,w_i}} \text{
for }i\in J_N\big\}.
   \end{align*}
Since the operators $C_{\phi_i,w_i}$, $i\in J_N$, are
closed (see Proposition \ref{lemS1}) and tensor
products of closable operators are
closable\footnote{\;Adapting the proof of
\cite[Proposition, p.\ 298]{R-S} to the context of
operators acting between different Hilbert spaces and
restricting the tensor product of operators in
question to the closure of its domain, we may drop the
assumption that the tensor factors are densely
defined.}, we get the following.
   \begin{pro}\label{closab}
If the operators $C_{\phi_i,w_i}$, $i\in J_N$, are
well-defined, then their tensor product
$C_{\phi_1,w_1}\otimes \ldots \otimes C_{\phi_N,w_N}$
is closable.
   \end{pro}
   The next result sheds more light on the
relationship between the tensor product
$C_{\phi_1,w_1}\otimes \ldots \otimes C_{\phi_N,w_N}$
whose factors are densely defined and the weighted
composition operator $\cfw$.
   \begin{thm} \label{cltenp}
Suppose $C_{\phi_i,w_i}$ is densely defined for each
$i\in J_N$. Then the following assertions are
valid{\em :}
   \begin{enumerate}
   \item[(i)] the operator $\cfw$ is densely defined,
   \item[(ii)]  $\hfw = \hfwi{1} \otimes \ldots \otimes
\hfwi{N}$ a.e.\ $[\mu]$,
   \item[(iii)] $\mu_w \circ \phi^{-1} = [(\mu_1)_{w_1} \circ
\phi_1^{-1}] \otimes \ldots \otimes [(\mu_N)_{w_N}
\circ \phi_N^{-1}]$,
   \item[(iv)] the operator $C_{\phi_1,w_1}\otimes \ldots \otimes
C_{\phi_N,w_N}$ is densely defined, closable and
   \begin{align*}
\overline{C_{\phi_1,w_1}\otimes \ldots \otimes
C_{\phi_N,w_N}} = \cfw.
   \end{align*}
   \end{enumerate}
   \end{thm}
   \begin{proof}
By Proposition \ref{lemS2}, we can modify $\hfwi{i}$
so as to obtain the equality $\{\hfwi{i} < \infty\} =
X_i$ for each $i\in J_N$. Then we have $\{\hfwi{1}
\otimes \ldots \otimes \hfwi{N} < \infty\}=X$.
Applying Lemma \ref{amierz+}, we deduce that $\mu_w
\circ \phi^{-1}$ is absolutely continuous with respect
to $\mu$ and (ii) is satisfied. By Propositions
\ref{wco1} and \ref{lemS2}, (i) holds as well. Since,
by Proposition \ref{lemS2}, the measures
$\nu_i:=(\mu_i)_{w_i} \circ \phi_i^{-1}$, $i\in J_N$,
are $\sigma$-finite, there exists their product
measure $\nu_1 \otimes \ldots \otimes \nu_N$. Hence,
by \eqref{alpr}, (iii) holds.

It remains to prove (iv). By Proposition \ref{closab},
$C_{\phi_1,w_1}\otimes \ldots \otimes C_{\phi_N,w_N}$
is closable. Observe that $C_{\phi_1,w_1}\otimes
\ldots \otimes C_{\phi_N,w_N} \subseteq \cfw$. Since,
by Proposition \ref{lemS1}, $\cfw$ is closed, it is
enough to show that the orthogonal complement of
$\dz{C_{\phi_1,w_1}\otimes \ldots \otimes
C_{\phi_N,w_N}}$ in $\dz{\cfw}$ with respect to the
graph inner product $\is{\cdot}{\mbox{-}}_{\cfw}$ is
the zero space. For this, take a function $u \in
\dz{\cfw}$ which is orthogonal to
$\dz{C_{\phi_1,w_1}}\otimes \ldots \otimes
\dz{C_{\phi_N,w_N}}$ with respect to the graph inner
product $\is{\cdot}{\mbox{-}}_{\cfw}$. Let $f_i \in
\dz{C_{\phi_i,w_i}}$ for $i\in J_N$. Applying
\eqref{l2}, we deduce that
   \begin{align} \label{igss}
u \bar f_1 \otimes \ldots \otimes \bar f_N (1+\hfw)
\in L^1(\mu)
   \end{align}
   and
   \begin{align} \notag
0= \is{u}{f_1 \otimes \ldots \otimes f_N}_{\cfw} & =
\int_X u \bar f_1 \otimes \ldots \otimes \bar f_N \D
\mu
   \\ \notag
& \hspace{4ex}+ \int_X (u \circ \phi) \big((\bar f_1
\otimes \ldots \otimes \bar f_N) \circ \phi\big)
\D\mu_w
   \\  \label{cieplo}
& = \int_X u \bar f_1 \otimes \ldots \otimes \bar f_N
(1+\hfw) \D \mu.
   \end{align}
It follows from Lemma \ref{aproks} and Proposition
\ref{lemS2} that for every $i \in J_N$, there exists a
sequence $\{\varOmega_{i,n}\}_{n=1}^{\infty} \subseteq
\ascr_i$ such that
   \begin{align} \label{ukryt}
\text{$\mu_i(\varOmega_{i,n}) < \infty$ and $\hfwi{i}
\Le n$ a.e.\ $[\mu_i]$ on $\varOmega_{i,n}$ for every
$n\in \nbb$,}
   \end{align}
and $\varOmega_{i,n} \nearrow X_i$ as $n\to \infty$.
Fix $n\in \nbb$. Let $\cscr_n=\{\varDelta \in \ascr
\colon \varDelta \subseteq \varOmega_n\}$, where
$\varOmega_n:=\varOmega_{1,n} \times \ldots \times
\varOmega_{N,n}$. Clearly, for $i\in J_N$, the family
$\ascr_{i,n} := \ascr_i \cap \varOmega_{i,n}$ is a
$\sigma$-algebra of subsets of $\varOmega_{i,n}$.
Arguing as in \eqref{alpr}, we get
   \begin{align} \label{dyctu}
\cscr_n & = \sigma_{\varOmega_n}
\big(\ascr_{1,n}\boxtimes \ldots \boxtimes
\ascr_{N,n}\big).
   \end{align}
Since, by Proposition \ref{lemS1}(i) and
\eqref{ukryt}, $f_i:=\chi_{\varOmega_{i,n}} \in
\dz{C_{\phi_i,w_i}}$ for $i\in J_N$, we infer from
\eqref{igss} that $u(1+\hfw) \in L^1(\varOmega_n,
\cscr_n, \mu|_{\cscr_n})$. Hence, by Lebesgue's
dominated convergence theorem, the set function
$\nu_n\colon \cscr_n \to \cbb$ defined by
   \begin{align} \label{lypie}
\nu_n(\varDelta) = \int_{\varDelta} u(1+\hfw) \D \mu,
\quad \varDelta \in \cscr_n,
   \end{align}
is a complex measure. Since $\chi_{\varDelta_i} \in
\dz{C_{\phi_i,w_i}}$ for $\varDelta_i \in \ascr_{i,n}$
and $i\in J_N$, we deduce from \eqref{cieplo} that
$\nu_n$ vanishes on the semi-algebra
$\ascr_{1,n}\boxtimes \ldots \boxtimes \ascr_{N,n}$.
By \cite[Proposition I-6-1]{Nev}, $\nu_n$ vanishes on
the algebra $\dscr$ generated by $\ascr_{1,n}\boxtimes
\ldots \boxtimes \ascr_{N,n}$. It follows from
\cite[Proposition 1.19]{Rud} that the class of sets
$\{\varDelta \in \cscr_n\colon \nu_n(\varDelta)=0\}$
is a monotone class which contains $\dscr$. Hence, by
\eqref{dyctu} and the monotone class theorem (see
\cite[Theorem 1.3.9]{Ash}), we see that $\nu_n$
vanishes on the entire $\cscr_n$. This fact combined
with $u(1+\hfw) \in L^1(\varOmega_n, \cscr_n,
\mu|_{\cscr_n})$ and \eqref{lypie} implies that $u=0$
a.e.\ $[\mu]$ on $\varOmega_n$. Since $\varOmega_n
\nearrow X$ as $n\to \infty$, we conclude that $u=0$
a.e.\ $[\mu]$. This completes the proof.
   \end{proof}
It is worth mentioning that the following problem,
which is closely related to Theorem \ref{cltenp},
remains unsolved.
   \begin{opq} \label{Pro2}
Does the assumption that the operator $\cfw$ is
nonzero and densely defined imply that the operators
$C_{\phi_i,w_i}$, $i\in J_N$, are densely
defined\/{\em ?}
   \end{opq}
It may also be useful to provide yet another comment
on Theorem \ref{cltenp}.
   \begin{rem} \label{probl1-com}
We show that under the weakest possible assumption
that the operators $C_{\phi_i,w_i}$, $i\in J_N$, are
well-defined, the tensor product
$C_{\phi_1,w_1}\otimes \ldots \otimes C_{\phi_N,w_N}$
acts in a way that still resembles the formal action
of $\cfw$. Given a measure space $(Z,\cscr,\tau)$ and
a $\cscr$-measurable complex function $g$ on $Z$, we
write $[g]_{\tau}$ for the equivalence class of $g$
with respect to the equivalence relation ``a.e.\
$[\tau]$''. Denote by $\escr$ the linear span of the
set of all functions of the form $f_1 \otimes \ldots
\otimes f_N$, where, for $i\in J_N$, $f_i$ is an
$\ascr_i$-measurable function on $X_i$ such that
$[f_i]_{\mu_i} \in \dz{C_{\phi_i,w_i}}$. It is a
simple matter to verify that
   \begin{gather*}
\dz{C_{\phi_1,w_1}\otimes \ldots \otimes
C_{\phi_N,w_N}} = \big\{[f]_{\mu}\colon f \in
\escr\big\},
   \\
(C_{\phi_1,w_1}\otimes \ldots \otimes C_{\phi_N,w_N})
[f]_{\mu} = [w \cdot f \circ \phi]_{\mu}, \quad f \in
\escr.
   \end{gather*}
This means that $w \cdot f \circ \phi = w \cdot g
\circ \phi$ a.e.\ $[\mu]$ whenever $f$ and $g$ are
functions in $\escr$ such that $f=g$ a.e.\ $[\mu]$ (in
fact, this property is equivalent to knowing that the
tensor product $C_{\phi_1,w_1}\otimes \ldots \otimes
C_{\phi_N,w_N}$ is well-defined). However, we do not
know whether $w \cdot f \circ \phi = w \cdot g \circ
\phi$ a.e.\ $[\mu]$ if $f$ is an $\ascr$-measurable
complex function on $X$, $g$ is a function in $\escr$
and $f=g$ a.e.\ $[\mu]$ (cf.\ Problem \ref{probl1}).
   \end{rem}
   \section{\label{Sect8.2}Modifying the symbol $\phi$}
As will be seen in this and the next section,
modifying the symbols of weighted composition
operators may be fruitful. Here we provide a certain
modification which will be used in two ways: first, to
show that the conclusion of Theorem \ref{main2} is
optimal in a sense that in general non of the
implications in \eqref{futro} can be reversed (see
Example \ref{rem0+}) and, second, to improve the
measurability properties of families $P$ satisfying
\eqref{cc} in the case of quasinormal operators (see
Proposition \ref{munich3}).
   \begin{lem}\label{modphi}
Suppose \eqref{stand2} holds. Then the transformation
$\tilde \phi\colon X \to X$ given~ by
   \begin{align} \label{numerek}
\tilde \phi(x) =
   \begin{cases}
x & \text{if} \quad w(x) = 0,
   \\
\phi(x) & \text{if} \quad w(x) \neq 0,
   \end{cases}
\quad x \in X,
   \end{align}
is $\ascr$-measurable and has the following
properties{\em :}
   \begin{enumerate}
   \item[(i)] $\tilde\phi=\phi$ a.e.\ $[\mu_w]$,
$\mu_w\circ \tilde \phi^{-1} \ll \mu$, $\htfw = \hfw$
a.e.\ $[\mu]$ and $\cfw=\ctfw$,
   \item[(ii)] if $\hfw < \infty$ a.e.\ $[\mu]$, then
$\etfw(f) = \efw(f)$ a.e.\ $[\mu_w]$ for every
$\ascr$-measurable function $f \colon X\to \rbop$,
   \item[(iii)] if $\hfw \circ \phi = \hfw$ a.e.\
$[\mu_w]$, then $\htfw \circ \tilde \phi = \htfw$
a.e.\ $[\mu]$.
   \end{enumerate}
   \end{lem}
   \begin{proof}
(i)\&(ii) Apply Proposition \ref{wco1} and Proposition
\ref{appcoin}.

(iii) By (i) and Lemma \ref{nuklear}, we have
   \begin{align} \label{wcoZ}
\htfw \circ \tilde \phi = \hfw \circ \tilde \phi
\overset{\eqref{numerek}}= \hfw \circ \phi = \hfw =
\htfw \text{ a.e.\ $[\mu_w]$}.
   \end{align}
Since clearly $\htfw \circ \tilde \phi = \htfw$ on
$\{w=0\}$, we get $\htfw \circ \tilde \phi = \htfw$
a.e.\ $[\mu]$.
   \end{proof}
Regarding Lemma \ref{modphi}, we note that if $\hfw <
\infty$ a.e.\ $[\mu]$, then the assertion (iii) of
this lemma can also be deduced from Proposition
\ref{wco1} and Theorem \ref{quain}.
   \begin{pro} \label{modPhi}
If \eqref{stand3} holds and $\tilde \phi$ is as in
\eqref{numerek}, then the following assertions are
valid{\em :}
   \begin{enumerate}
   \item[(i)] if $\mu_w\circ \phi^{-1} \ll
\mu$, $\hfw < \infty$ a.e.\ $[\mu]$ and $P$ satisfies
\eqref{cc}, then $P$ satisfies \eqref{cc} with $\tilde
\phi$ in place of $\phi$,
   \item[(ii)] if $\rho_W \circ \varPhi^{-1} \ll \rho$ and
$\HFW \circ \varPhi = \HFW$ a.e.\ $[\rho_W]$, then
$\rho_W \circ \tilde \varPhi^{-1} \ll \rho$ and $\HtFW
\circ \tilde \varPhi = \HtFW$ a.e.\ $[\rho]$, where
$\tilde \varPhi\colon X \times \rbb_+ \to X \times
\rbb_+$ is given by $\tilde \varPhi(x,t) = (\tilde
\phi(x),t)$ for $x \in X$ and $t \in \rbb_+$.
   \end{enumerate}
   \end{pro}
   \begin{proof}
(i) Apply Lemma \ref{modphi}(ii) and the equality
$\htfw \circ \tilde \phi = \hfw \circ \phi$ a.e.\
$[\mu_w]$ (see \eqref{wcoZ}).

(ii) First note that
   \begin{align*}
\tilde \varPhi(x,t) =
   \begin{cases}
(x,t) & \text{if} \quad W(x,t) = 0,
   \\
\varPhi(x,t) & \text{if} \quad W(x,t) \neq 0,
   \end{cases}
\quad (x,t) \in X \times \rbb_+.
   \end{align*}
This and Lemma \ref{modphi}(iii) applied to $\varPhi$,
$W$ and $\rho$ in place of $\phi$, $w$ and $\mu$,
respectively, complete the proof.
   \end{proof}
Now we are in a position to provide an example which
was announced in the last paragraph of Section
\ref{sectcc-1}. Recall that in view of Theorem
\ref{main2}, the conditions
\mbox{(i$^\star$)-(v$^\star$)} are equivalent and the
implications
\mbox{(iii$^\star$)}$\Rightarrow$\mbox{(vi$^\star$)},
\mbox{(vii$^\star$)}$\Rightarrow$\mbox{(iii$^\star$)}
and
\mbox{(vii$^\star$)}$\Rightarrow$\mbox{(vi$^\star$)}
hold. The example below demonstrates that in general
none of these implications can be reversed.
   \begin{exa} \label{rem0+}
Consider a subnormal weighted composition operator
$\cfw$ that admits an $\ascr$-measurable family
$P\colon X \times \borel{\rbb_+} \to [0,1]$ of
probability measures satisfying \eqref{cc} and has the
property that
   \begin{align*}
\mu(\varTheta_{00}) > 0, \quad \text{where} \quad
\varTheta_{00} := \{\hfw = 0\} \cap \{w = 0\}.
   \end{align*}
In view of Remark \ref{22.X.2013Daegu} or Example
\ref{quasif}, such an operator exists. By Corollary
\ref{hipinj}, $\hfw > 0$ a.e.\ $[\mu_w]$. Clearly,
(vii$^\star$) does not hold. According to Corollary
\ref{lemS15}(i), $P$ can be modified so as to satisfy
the equality $\int_0^\infty t P(\cdot, \D t) = 0$
a.e.\ $[\mu]$ on $\{\hfw = 0\}$. It follows from
Theorem \ref{main2} that the so-modified $P$ satisfies
\eqref{cc-1mu}. This shows that the implication
\mbox{(iii$^\star$)}$\Rightarrow$\mbox{(vii$^\star$)}
does not hold in general. It remains to prove that
also the implications
\mbox{(vi$^\star$)}$\Rightarrow$\mbox{(iii$^\star$)}
and
\mbox{(vi$^\star$)}$\Rightarrow$\mbox{(vii$^\star$)}
do not hold in general. By Corollary \ref{lemS15}(ii),
$P$ can be modified so as to satisfy the equality
$\int_0^\infty t P(\cdot, \D t) = \infty$ a.e.\
$[\mu]$ on $\varTheta_{00}$. The so-modified $P$
satisfies \eqref{cc} and, by Theorem \ref{main2}, does
not satisfy \eqref{momomu} and \eqref{cc-1mu}. To
guarantee that \mbox{(vi$^\star$)} holds for the above
$P$, we have to modify $\phi$ and $\varPhi$ in
accordance with Lemma \ref{modphi} and Proposition
\ref{modPhi} (this is possible due to the implication
(vii)$\Rightarrow$(vi) of Theorem \ref{main1});
summarizing, our $P$ does not satisfy
\mbox{(iii$^\star$)} and \mbox{(vii$^\star$)}.
   \end{exa}
   \section{\label{Sec8.3}Quasinormality revisited}
It follows from \cite[Sect.\ 6]{b-j-j-sC} and
\cite[Theorem 7 and Proposition 10]{b-j-j-sS} that
each quasinormal composition operator $C_{\phi}$
admits a $\phi^{-1}(\ascr)$-measurable family of
probability measures that satisfies \eqref{cc-1mu}. As
shown in Example \ref{quasif} below, this is no longer
true for weighted composition operators. Nevertheless,
this is the case when $\phi^{-1}(\ascr)$-measurability
is replaced by $\ascr$-measurability (see Proposition
\ref{munich2}). What is more,
$\phi^{-1}(\ascr)$-measurability can always be
restored by modifying the symbol $\phi$ of $\cfw$ (see
Proposition \ref{munich3}).

We begin by showing that each quasinormal weighted
composition operator $\cfw$ has an $\ascr$-measurable
family of probability measures $P$ that satisfies
\eqref{cc-1mu} (and consequently \eqref{cc}; see
\eqref{cc-1to-cc}). If $w \neq 0$ a.e.\ $[\mu]$, then
$P$ can always be chosen to be
$\phi^{-1}(\ascr)$-measurable (this covers
\cite[Proposition 10]{b-j-j-sC}). We also discuss the
question of uniqueness of $P$.
   \begin{pro}\label{munich2}
Suppose \eqref{stand2} holds and $\cfw$ is
quasinormal. Then the following assertions are
valid{\em :}
   \begin{enumerate}
   \item[(i)] a mapping $P\colon X \times \borel{\rbb_+}\to[0,1]$ defined by
   \begin{align} \label{def1}
P(x,\sigma)=\chi_{\sigma}(\hfw(x)), \quad x\in X,
\sigma\in\borel{\rbb_+},
   \end{align}
is a $\bscr$-measurable family of probability measures
that satisfies \eqref{cc-1mu}, where
$\bscr=\phi^{-1}(\ascr)^{\mu_w}$ $($see
\eqref{complascr} for notation$)$,
      \item[(ii)] if $w\neq 0$ a.e.\ $[\mu]$, then a mapping $P\colon X \times
\borel{\rbb_+}\to[0,1]$ defined by
   \begin{align} \label{def2}
P(x,\sigma)=\chi_{\sigma}(\hfw(\phi(x))), \quad x\in
X, \sigma\in\borel{\rbb_+},
   \end{align}
is a $\phi^{-1}(\ascr)$-measurable family of
probability measures satisfying \eqref{cc-1mu},
   \item[(iii)] if $P_1,P_2\colon X \times \borel{\rbb_+} \to [0,1]$ are
$\ascr$-measurable families of probability measures
satisfying \eqref{cc}, then $P_1(x, \cdot) = P_2(x,
\cdot)$ for $\mu_w$-a.e.\ $x\in X$,
   \item[(iv)] if $P_1,P_2\colon X \times \borel{\rbb_+} \to [0,1]$ are
$\ascr$-measurable families of probability measures
satisfying \eqref{cc-1mu}, then $P_1(x, \cdot) =
P_2(x, \cdot)$ for $\mu$-a.e.\ $x\in X$.
   \end{enumerate}
   \end{pro}
   \begin{proof}
(i) By \eqref{def1} and Theorem \ref{quain},
$P(\phi(x), \sigma) = P(x, \sigma)$ for $\mu_w$-a.e.\
$x\in X$ and for all $\sigma \in \borel{\rbb_+}$.
Hence, $P$ is $\bscr$-measurable (see \cite[Lemma 1,
p. 169]{Rud}) and $\efw(P(\cdot, \sigma))(x) =
P(\phi(x), \sigma)$ for $\mu_w$-a.e.\ $x \in X$ and
for all $\sigma \in \borel{\rbb_+}$. This implies that
$P$ satisfies \eqref{cc}. Applying the implication
\mbox{(iv$^\star$)}$\Rightarrow$\mbox{(i$^\star$)} of
Theorem \ref{main2}, we see that $P$ satisfies
\eqref{cc-1mu} as well.

(ii) Since $\mu \ll \mu_w$, we infer from Theorem
\ref{quain} that $\hfw \circ \phi = \hfw$ a.e.\
$[\mu]$. Hence, by Lemma \ref{nuklear}, $\hfw \circ
\phi^2 = \hfw\circ \phi$ a.e.\ $[\mu_w]$. Since, by
\eqref{def2}, $P(\cdot,\sigma)$ is
$\phi^{-1}(\ascr)$-measurable, we easily verify that
$P$ satisfies \eqref{cc}. Applying the equality $\hfw
\circ \phi = \hfw$ a.e.\ $[\mu]$ and the implication
\mbox{(iv$^\star$)}$\Rightarrow$\mbox{(i$^\star$)} of
Theorem \ref{main2}, we conclude that $P$ satisfies
\eqref{cc-1mu}.

(iii) Let $P$ be as in (i). By Corollary \ref{hipinj},
$\hfw > 0$ a.e.\ $[\mu_w]$. Using the implication
(vii)$\Rightarrow$(ii) of Theorem \ref{main1}, we
deduce that the measures $P(x,\cdot)$, $P_1(x,\cdot)$
and $P_2(x,\cdot)$ are representing measures of a
Stieltjes moment sequence
$\{(\hfw(x))^n\}_{n=0}^\infty$ for $\mu_w$-a.e.\ $x\in
X$. This and \eqref{Stiogr} prove (iii).

(iv) Argue as in (iii), using Theorem \ref{main2} in
place of Theorem \ref{main1}.
   \end{proof}
Now we show that the modification $\phi
\rightsquigarrow \tilde \phi$ of the symbol of a
quasinormal operator $\cfw$ enables us to obtain a
$\phi^{-1}(\ascr)$-measurable family of probability
measures that satisfies \eqref{cc-1mu}.
   \begin{pro} \label{munich3}
Suppose \eqref{stand2} holds and $\cfw$ is
quasinormal. Let $\tilde \phi$ be as in
\eqref{numerek}. Then $\tilde\phi$ is
$\ascr$-measurable, $C_{\tilde\phi,w}$ is
well-defined, $C_{\tilde\phi,w} = \cfw$ and a mapping
$P\colon X \times \borel{\rbb_+} \to [0,1]$ defined by
   \begin{align*}
P(x,\sigma)=\chi_{\sigma}(\htfw(\tilde\phi(x))),\quad
x\in X, \sigma\in\borel{\rbb_+},
   \end{align*}
is a $\tilde\phi^{-1}(\ascr)$-measurable family of
probability measures that satisfies \eqref{cc-1mu}
with $\tilde\phi$ in place of $\phi$. Moreover,
$\mathsf{h}_{\tilde\phi,w} \circ \tilde\phi =
\mathsf{h}_{\tilde\phi,w}$ a.e.\ $[\mu]$ and $P$ is a
$(\phi^{-1}(\ascr))^{\mu_w}$-measurable.
   \end{pro}
   \begin{proof}
Clearly, $P$ is a $\tilde\phi^{-1}(\ascr)$-measurable
family of probability measures. Hence, by Lemma
\ref{modphi} and Proposition \ref{wco1},
$\tilde\phi=\phi$ a.e.\ $[\mu_w]$, $P$ is
$(\phi^{-1}(\ascr))^{\mu_w}$-measurable,
$C_{\tilde\phi,w}$ is well-defined and
$C_{\tilde\phi,w} = \cfw$. In view of Theorem
\ref{quain}, $\mathsf{h}_{\tilde\phi,w}
(\tilde\phi(x)) = \mathsf{h}_{\tilde\phi,w}(x)$ for
$\mu$-a.e.\ $x \in \{w \neq 0\}$. By the definition of
$\tilde\phi$, $\mathsf{h}_{\tilde\phi,w}
(\tilde\phi(x)) = \mathsf{h}_{\tilde\phi,w}(x)$ for
every $x \in \{w = 0\}$. Hence
$\mathsf{h}_{\tilde\phi,w} \circ \tilde\phi =
\mathsf{h}_{\tilde\phi,w}$ a.e.\ $[\mu]$. Arguing as
in the proof of assertion (ii) of Proposition
\ref{munich2} completes the proof.
   \end{proof}
Below we give an example of a quasinormal weighted
composition operator $\cfw$ which has no
$\phi^{-1}(\ascr)$-measurable family of probability
measures satisfying \eqref{cc}. In this particular
case, $\mu(\varTheta_{++})
>0$, $\mu(\varTheta_{+0}) >0$ and $\mu(\varTheta_{00})
>0$ (see Remark \ref{4sets}). However,
$\mu(\varTheta_{0+}) = 0$ due to Corollary
\ref{hipinj}.
   \begin{exa} \label{quasif}
Fix $M,N \in \nbb \cup \{\infty\}$. Set $X =
\varOmega_0 \sqcup \varOmega_1$ and $\ascr=2^X$, where
$\varOmega_0 = \bigsqcup_{m=1}^M \zbb_+ \times \{m\}
\times \{0\}$ and $\varOmega_1 = \bigsqcup_{n=1}^N
\zbb_+ \times \{n\} \times \{1\}$. Let
$\{\lambda_n\}_{n=1}^N$ be a sequence of positive real
numbers. Define the $\sigma$-finite measure $\mu\colon
\ascr \to \rbop$ by
   \begin{align} \label{defmu}
\mu(\{(j,m,s)\}) =
   \begin{cases}
\delta_0(j) & \text{ if } s=0,
   \\
\lambda_m^j & \text{ if } s=1,
   \end{cases}
\quad (j,m,s) \in X.
   \end{align}
Set $w=\chi_{\varOmega}$ with $\varOmega =
\big(\bigsqcup_{m=1}^M \nbb \times \{m\} \times
\{0\}\big) \sqcup \big(\bigsqcup_{n=1}^N \nbb \times
\{n\} \times \{1\}\big)$. The transformation $\phi$ of
$X$ is defined by
   \begin{align}  \label{defp}
\phi((j,m,s)) =
   \begin{cases}
(j-1,m,s) & \text{ if } j\Ge 1,
   \\
(0,1,s) & \text{ if } j=0,
   \end{cases}
\quad (j,m,s) \in X.
   \end{align}
The reader is referred to Figure \ref{figura5} which
illustrates the measure space $(X,\ascr,\mu)$ and the
transformation $\phi$ defined above.
   \begin{center}
   \begin{figure}[t]
\subfigure {
\includegraphics[scale=0.20]{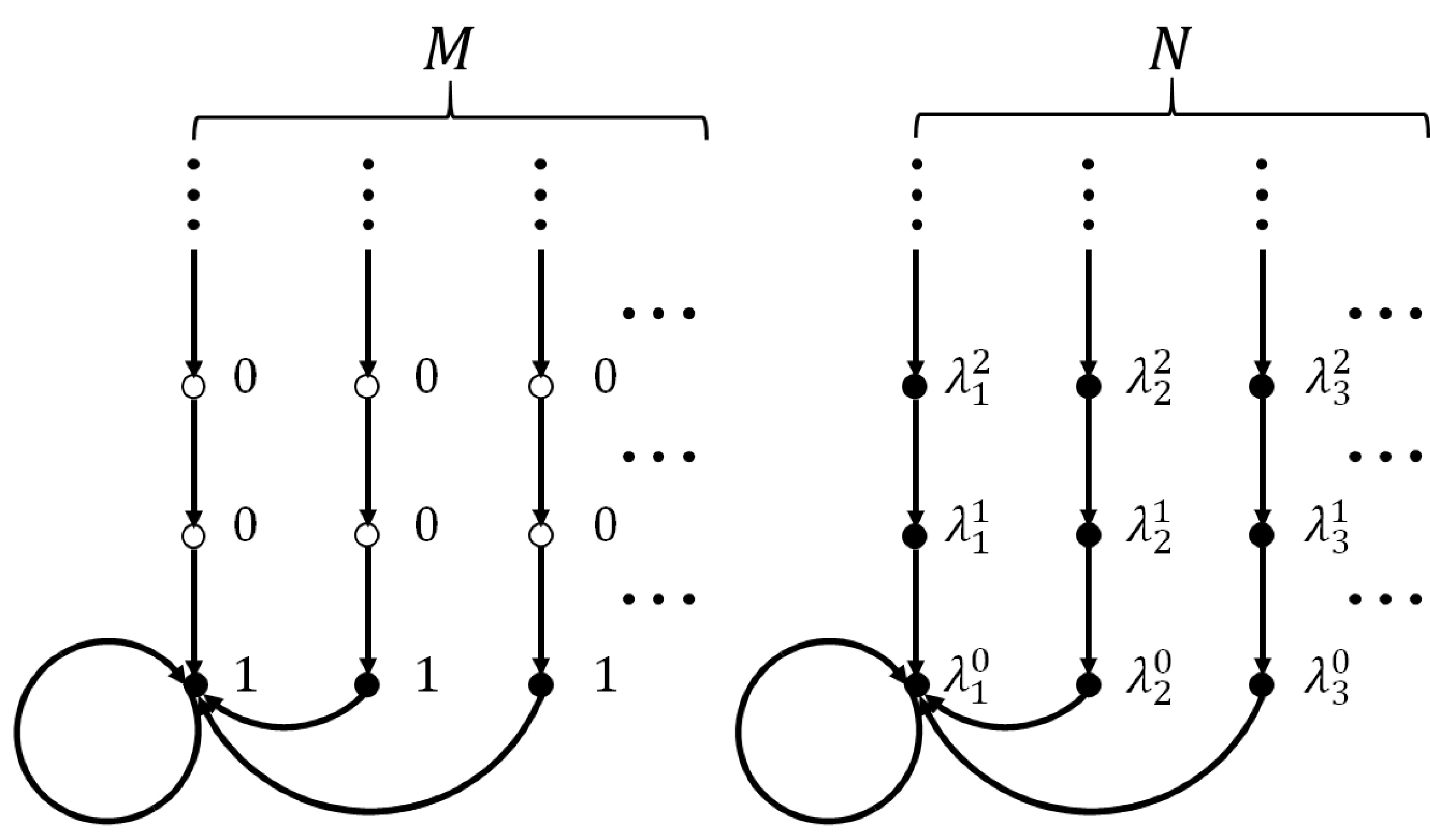}
}
   \caption{The measure $\mu$ and the transformation
$\phi$ that appear in Example \ref{quasif}.}
   \label{figura5}
   \end{figure}
   \end{center}

First, note that $C_{\phi,w}$ is well-defined. Indeed,
if $\mu(\varDelta)=0$, then $\varDelta \subseteq
\varOmega \cap \varOmega_0$ and thus
$\phi^{-1}(\varDelta) \subseteq \varOmega \cap
\varOmega_0$, which implies that
$\mu_w(\phi^{-1}(\varDelta))=0$. It is clear that
   \begin{align} \label{hsf}
\hfw((j,m,s)) =
   \begin{cases}
0 & \text{ if } s=0,
   \\
\lambda_m & \text{ if } s = 1,
   \end{cases}
\quad (j,m,s) \in X.
   \end{align}
(The function $\hfw$ is uniquely determined on the set
${X \setminus (\varOmega \cap \varOmega_0)}$ which
coincides with $\at{\mu}$ (see \eqref{atom}); so we
may put $\hfw|_{\varOmega \cap \varOmega_0}=0$.) By
\eqref{hsf}, $\cfw$ is densely defined, $\hfw > 0$
a.e.\ $[\mu_w]$, $\mu(\varTheta_{+0}) = N >0$ and
$\mu(\varTheta_{00}) = M >0$, which means that the
condition ``$\hfw > 0$ a.e.\ $[\mu]$'' does not hold.
Hence, by Lemma \ref{jadro}, $\cfw$ is not injective.
It follows from \eqref{defp} and \eqref{hsf} that
$\hfw(\phi(x)) = \hfw(x)$ for all $x \in \varOmega
\cap \varOmega_1$. Since $\varOmega \cap
\varOmega_1=\at{\mu_w}$, we conclude that $\hfw \circ
\phi = \hfw$ a.e.\ $[\mu_w]$. By Theorem \ref{quain},
$\cfw$ is quasinormal and, by Proposition
\ref{munich2}, $\cfw$ has an $\ascr$-measurable family
of probability measures that satisfies \eqref{cc-1mu}.
Note that if $P\colon X \times \borel{\rbb_+} \to
[0,1]$ is any $\ascr$-measurable family of probability
measures that satisfies \eqref{cc-1mu}, then
$P(x,\cdot) = \delta_0(\cdot)$ for every $x\in
\varTheta_{00}$. Indeed, this is a direct consequence
of the implication
\mbox{(i$^\star$)}$\Rightarrow$\mbox{(iv$^\star$)} of
Theorem \ref{main2}.

Suppose $P\colon X \times \borel{\rbb_+} \to [0,1]$ is
an $\ascr$-measurable family of probability measures
that satisfies \eqref{cc}. We will show that
$P(x,\cdot)$ is uniquely determined for all $x \in
\varTheta_{+0} = \{(0,n,1) \colon n \in J_N\}$, where
$J_N = \nbb \cap [1,N]$. In fact, we will prove~ that
   \begin{align} \label{pon}
P((0,n,1),\cdot) = \delta_{\lambda_n}(\cdot), \quad n
\in J_N.
   \end{align}
Indeed, since $\varOmega \cap \varOmega_1=\at{\mu_w}$,
we easily verify that for every function $f\colon X
\to \rbop$,
   \begin{align*}
\big(\efw(f)\big)(y) = \frac{\int_{\phi^{-1}(\{x\})} f
\D \mu_w}{\mu_w(\phi^{-1}(\{x\}))} = f(y), \quad y \in
\phi^{-1}(\{x\}), \quad x\in \varOmega \cap
\varOmega_1.
   \end{align*}
(The denominator is nonzero because
$\phi^{-1}(\varOmega \cap \varOmega_1) \subseteq
\varOmega \cap \varOmega_1$.) Then clearly \eqref{cc}
takes the following form
   \begin{align*}
P(x,\sigma) \cdot \hfw(\phi(x)) = \int_{\sigma} t
P(\phi(x), \D t), \quad x\in \varOmega \cap
\varOmega_1.
   \end{align*}
Since $(j+1, n, 1) \in \varOmega \cap \varOmega_1$ for
$j\in \zbb_+$ and $n \in J_N$, we get
   \begin{align*}
P((j+1, n, 1), \sigma) \cdot \lambda_n = \int_{\sigma}
t P((j, n, 1), \D t), \quad \sigma \in \borel{\rbb_+},
\, j \in \zbb_+, \, n \in J_N.
   \end{align*}
By induction on $j$, we have
   \begin{align*}
P((j,n,1), \sigma) = \frac{1}{\lambda_n^j}
\int_{\sigma} t^j P((0,n,1), \D t), \quad \sigma \in
\borel{\rbb_+}, \, j \in \zbb_+, \, n \in J_N.
   \end{align*}
Substituting $\sigma=\rbb_+$, we obtain
   \begin{align*}
\lambda_n^j = \int_0^\infty t^j P((0,n,1), \D t),
\quad j \in \zbb_+, \, n \in J_N,
   \end{align*}
which together with \eqref{Stiogr} implies
\eqref{pon}.

Assume additionally that the family $P$ is
$\phi^{-1}(\ascr)$-measurable. Then for every $\sigma
\in \borel{\rbb_+}$, $P(\cdot, \sigma)$ is a constant
function on the set
   \begin{align*}
\phi^{-1}(\{(0,1,1)\}) = \{(0,n,1)\colon n \in J_N\}
\sqcup \{(1,1,1)\}.
   \end{align*}
This yields
   \begin{align*}
\delta_{\lambda_n}(\cdot) \overset{\eqref{pon}}=
P((0,n,1), \cdot) = P((0,1,1), \cdot)
\overset{\eqref{pon}} = \delta_{\lambda_1}(\cdot),
\quad n \in J_N,
   \end{align*}
which implies that $\lambda_n = \lambda_1$ for every
$n \in J_N$. Hence, if the sequence
$\{\lambda_n\}_{n=1}^N$ is non-constant, then $\cfw$
has no $\phi^{-1}(\ascr)$-measurable family of
probability measures satisfying \eqref{cc}.
   \end{exa}
   \begin{rem} \label{Zenchce}
It is worth pointing out that the operator $\cfw$
constructed in Example \ref{quasif} is of the form
   \begin{align*}
\cfw = C_{\psi,u} \oplus C_{\zeta,v},
   \end{align*}
where $C_{\psi,u}$ is a weighted composition operator
in $L^2(\varOmega_0,
2^{\varOmega_0},\mu|_{2^{\varOmega_0}})$ with
$u=w|_{\varOmega_0}$ and
$\psi=\phi|_{\varOmega_0}\colon \varOmega_0 \to
\varOmega_0$, and $C_{\zeta,v}$ is a weighted
composition operator in $L^2(\varOmega_1,
2^{\varOmega_1},\mu|_{2^{\varOmega_1}})$ with
$v=w|_{\varOmega_1}$ and
$\zeta=\phi|_{\varOmega_1}\colon \varOmega_1 \to
\varOmega_1$. Indeed, it is a matter of routine to
show that $C_{\psi,u}$ and $C_{\zeta,v}$ are
well-defined, $\mathsf{h}_{\psi,u} =
\hfw|_{\varOmega_0}=0$ a.e.\
$[\mu|_{2^{\varOmega_0}}]$ and $\mathsf{h}_{\zeta,v} =
\hfw|_{\varOmega_1}$ a.e.\ $[\mu|_{2^{\varOmega_1}}]$.
This implies that $\cfw = C_{\psi,u} \oplus
C_{\zeta,v}$, $C_{\psi,u}$ is the zero operator on
$L^2(\varOmega_0,
2^{\varOmega_0},\mu|_{2^{\varOmega_0}})$ and
$C_{\zeta,v}$ is a quasinormal operator (use Theorem
\ref{quain}). By \eqref{defmu} and \eqref{hsf}, we
have
   \begin{gather*} \mu(\{\mathsf{h}_{\psi,u} = 0\}
\cap \{u = 0\}) = M > 0,
   \\
\mu(\{\mathsf{h}_{\zeta,v} = 0\} \cap \{v = 0\}) =
0,
   \\
\mu(\{\mathsf{h}_{\zeta,v} > 0\} \cap \{v = 0\}) = N
> 0.
   \end{gather*}
   \end{rem}
   \begin{appendix}
   \numberwithin{equation}{section}
   \numberwithin{thm}{section}

\chapter{Nonprobabilistic Expectation and Operators}

   In this chapter, we discuss some properties of
conditional expectation in the nonprobabilistic
context and Hilbert space operators that have been
used in this paper.
   \section{Conditional expectation} \label{AppB}
In this section, we discuss some basic properties of
conditional expectation in non-probabilistic setting.

Let $(X, \ascr, \nu)$ be a measure space and $\bscr
\subseteq \ascr$ be a $\sigma$-algebra such that the
measure $\nu|_{\bscr}$ is $\sigma$-finite. It follows
from the Radon-Nikodym theorem that for every
$\ascr$-measurable function $f \colon X \to \rbop$
there exists a unique (up to a set of $\nu$-measure
zero) $\bscr$-measurable function $\ceb{f} \colon X
\to \rbop$ such that
   \begin{align} \label{B1}
\int_{\varDelta} f \D \nu = \int_{\varDelta} \ceb{f}
\D \nu, \quad \varDelta \in \bscr.
   \end{align}
We call $\ceb{f}$ the {\em conditional expectation} of
$f$ with respect to the $\sigma$-algebra $\bscr$ and
the measure $\nu$ (see \cite{Rao} for the theory of
conditional expectation in the probabilistic setting).
Clearly, if $g \colon X \to \rbop$ is an
$\ascr$-measurable function such that $f=g$ a.e.\
$[\nu]$, then $\ceb{f}=\ceb{g}$ a.e.\ $[\nu]$.
Moreover, by \eqref{B1}, if $\tilde\bscr\subseteq
\ascr$ is another $\sigma$-algebra such that the
measure $\nu|_{\tilde\bscr}$ is $\sigma$-finite and
$\tilde\bscr \subseteq \bscr$, then for every
$\ascr$-measurable function $f \colon X \to \rbop$,
   \begin{align} \label{Dag9}
\text{$\mathsf{E}(f;\tilde\bscr,\nu) =
\mathsf{E}\big(\ceb{f};\tilde\bscr,\nu\big)$ a.e.\
$[\nu]$.}
   \end{align}
It follows from \eqref{B1} that
   \begin{align}  \label{B2}
   \begin{minipage}{70ex}
{\em for all $\alpha, \beta \in \rbb_+$ and for all
$\ascr$-measurable functions $f,g \colon X \to \rbop$,}
   $$ \text{$\ceb{\alpha f + \beta g} = \alpha \ceb{f} +
\beta \ceb{g}$ a.e.\ $[\nu]$.}
   $$
   \end{minipage}
   \end{align}
Applying the standard approximation procedure to \eqref{B1}, we see that
if $f \colon X \to \rbop$ is $\ascr$-measurable and $g \colon X \to \rbop$
is $\bscr$-measurable, then
   \begin{align} \label{B3}
\int_X g f \D \nu = \int_X g \ceb{f} \D \nu.
   \end{align}
In view of \eqref{B1} and the $\sigma$-finiteness of $\nu|_{\bscr}$, we
have
   \begin{align} \label{B4}
   \begin{minipage}{70ex}
{\em if $f,g \colon X \to \rbop$ are
$\ascr$-measurable functions and $f \Le g$ a.e.\
$[\nu]$, then $\ceb{f} \Le \ceb{g}$ a.e.\ $[\nu]$.}
   \end{minipage}
   \end{align}
Using \eqref{B1}, \eqref{B4} and Lebesque's monotone convergence theorem,
we verify that
   \begin{align} \label{B5}
   \begin{minipage}{72ex}
{\em if $f, f_n\colon X \to \rbop$, $n \in \nbb$, are
$\ascr$-measurable functions such that $f_n \nearrow
f$ a.e.\ $[\nu]$ as $n\to\infty$, then $\ceb{f_n}
\nearrow \ceb{f}$ a.e.\ $[\nu]$ as $n\to\infty$,}
   \end{minipage}
   \end{align}
where the expression ``$f_n \nearrow f$ a.e.\ $[\nu]$
as $n\to \infty$'' means that the sequence
$\{f_n(x)\}_{n=1}^\infty$ is monotonically increasing
to $f(x)$ for $\nu$-a.e.\ $x \in X$. Substituting $g
\cdot \chi_{\varDelta}$ with $\varDelta \in \bscr$ in
place of $g$ in \eqref{B3}, we infer from \eqref{B1},
\eqref{B3} and \eqref{B5} that for every
$\ascr$-measurable function $f \colon X \to \rbop$ and
for every $\bscr$-measurable function $g \colon X \to
\rbop$,
   \begin{align} \label{B6}
\ceb{g f} = g \ceb{f} \text{ a.e.\ $[\nu]$.}
   \end{align}
The following fact is a direct consequence of \eqref{B1}:
   \begin{align} \label{B7}
   \begin{minipage}{70ex}
{\em if $\varDelta \in \bscr$ and $f\colon X \to
\rbop$ is an $\ascr$-measurable function, then $f=0$
a.e.\ $[\nu]$ on $\varDelta$ if and only if $\ceb{f} =
0$ a.e.\ $[\nu]$ on $\varDelta$.}
   \end{minipage}
   \end{align}
The conditional expectation has the property of
``separating points'' like $L^p$-norms.
   \begin{align} \label{B7.5}
   \begin{minipage}{70ex}
{\em If $p\in (0,\infty)$ and $f\colon X \to \rbb_+$
or $f\colon X \to \cbb$ is an $\ascr$-measurable
function such that $\ceb{|f|^p}=0$ a.e.\ $[\nu]$, then
$f=0$ a.e.\ $[\nu]$.}
   \end{minipage}
   \end{align}
This is a direct consequence of the equalities
   \begin{align*}
0=\int_X \ceb{|f|^p} \D \nu \overset{\eqref{B1}}=
\int_X |f|^p \D \nu.
   \end{align*}

The next result resembles H\"{o}lder's inequality.
   \begin{lem} \label{BL1}
Suppose $f, g \colon X \to \rbop$ are $\ascr$-measurable functions. If
$p,q \in (1,\infty)$ are such that $\frac{1}{p} + \frac{1}{q}=1$, then
   \begin{align}  \label{B8}
\ceb{f g} \Le \ceb{f^p}^{1/p} \, \ceb{g^q}^{1/q} \text{ a.e.\ $[\nu]$.}
   \end{align}
Moreover, if $g \in L^{\infty}(\nu)$, then $\ceb{f g} \Le \ceb{f}
\|g\|_{L^{\infty}(\nu)}$ a.e.\ $[\nu]$.
   \end{lem}
   \begin{proof}
Set $u = \ceb{f^p}^{1/p}$, $v=\ceb{g^q}^{1/q}$,
$\varOmega_u = \{0 < u < \infty\}$ and $\varOmega_v =
\{0 < v < \infty\}$. Define the functions $U, V \colon
X \to \rbop$ by
   \begin{align*}
U(x) =
   \begin{cases}
\frac{1}{u(x)} & \text{if } x \in \varOmega_u,
   \\[1ex]
0 & \text{otherwise,}
   \end{cases}
\quad V(x) =
   \begin{cases}
\frac{1}{v(x)} & \text{if } x \in \varOmega_v,
   \\[1ex]
0 & \text{otherwise.}
   \end{cases}
   \end{align*}
Clearly, the functions $U$ and $V$ are
$\bscr$-measurable. Set $F=U \cdot f$ and $G=V \cdot
g$. Then, by the weighted arithmetic-geometric-mean
inequality, we have
   \begin{align*}
F G = (F^p)^{1/p} (G^q)^{1/q} \Le \frac{1}{p} F^p + \frac{1}{q} G^q.
   \end{align*}
This, combined with \eqref{B2}, \eqref{B4} and \eqref{B6}, yields
   \begin{align*}
U V \ceb{f g} = \ceb{F G} \Le \frac{1}{p} U^p \ceb{f^p} + \frac{1}{q} V^q
\ceb{g^q} \text{ a.e.\ $[\nu]$.}
   \end{align*}
Since $U^p\ceb{f^p} = \chi_{\varOmega_u}$ a.e.\ $[\nu]$ and $V^q\ceb{g^q}
= \chi_{\varOmega_v}$ a.e.\ $[\nu]$, we get
   \begin{align*}
U V \ceb{f g} \Le 1 \text{ a.e.\ $[\nu]$.}
   \end{align*}
This implies that the inequality in \eqref{B8} holds
a.e.\ $[\nu]$ on $\varOmega_u \cap \varOmega_v$.
Applying \eqref{B7} to the set $\widehat\varOmega_u :=
\{x \in X \colon u(x) = 0\}$ (resp.,
$\widehat\varOmega_v := \{x \in X \colon v(x) = 0\}$)
and the functions $f^p$, $fg$ (resp., $g^q$ and $fg$),
we deduce that the inequality in \eqref{B8} holds
a.e.\ $[\nu]$ on $\widehat \varOmega_u \cup
\widehat\varOmega_v$. Since the right-hand side of the
inequality in \eqref{B8} is equal to $\infty$ outside
of the set $\widehat\varOmega_u \cup
\widehat\varOmega_v \cup (\varOmega_u \cap
\varOmega_v)$, the proof of \eqref{B8} is complete.
The ``moreover'' part is a direct consequence of
\eqref{B2} and \eqref{B4}.
   \end{proof}
   \begin{cor} \label{BL2}
Suppose $p \in [1,\infty]$ and $f \colon X \to \rbop$
is an $\ascr$-measurable function such that $\int_X
|f|^p \D \nu < \infty$ if $p<\infty$ or
$\mathrm{ess\,sup}_{\nu} |f|<\infty$ if $p=\infty$.
Then $\int_X |\ceb{f}|^p \D \nu < \infty$ if
$p<\infty$ or $\mathrm{ess\,sup}_{\nu}
|\ceb{f}|<\infty$ if $p=\infty$.
   \end{cor}
   \begin{proof}
If $p \in [1,\infty)$, then by applying Lemma
\ref{BL1} to $g = \boldsymbol{1}$, we see that
$\ceb{f}^p \Le \ceb{f^p}$ a.e.\ $[\nu]$, which implies
that
   \begin{align*}
\int_X \ceb{f}^p \D \nu \Le \int_X \ceb{f^p} \D \nu =
\int_X f^p \D \nu.
   \end{align*}
If $p=\infty$, then by \eqref{B4}, $\ceb{f} \Le
\|f\|_{L^{\infty}(\nu)}$.
   \end{proof}
It follows from \eqref{B2} and Corollary \ref{BL2}
that for every $p \in [1,\infty]$, the mapping
$\ceb{\,\cdot\,}$ can be extended linearly from the
convex cone $L_+^p(\nu)$ to all of $L^p(\nu)$ via
standard extension procedure. We use the same symbol
$\ceb{\,\cdot\,}$ for the extended mapping, and call
it the {\em operator of conditional expectation} (we
will suppress the dependence on $p$ in notation). By
Corollary \ref{BL2} and the extension procedure, the
following holds.
   \begin{align} \label{apme}
   \begin{minipage}{70ex}
{\em Suppose $p \in [1,\infty]$ and $\efw$ is regarded
as an operator in $L^p(\nu)$. If $f\in L^p(\nu)$, then
$f\in \ob{\efw}$ if and only if there exists
$\bscr$-measurable function $\tilde f \in L^p(\nu)$
such that $f=\tilde f$ a.e.\ $[\nu]$.}
   \end{minipage}
   \end{align}

By Corollary \ref{BL2}, we can always choose $\ceb{f}$
to be a $\bscr$-measurable complex-valued function
whenever $f \in L^p(\nu)$ and $p \in [1,\infty]$.
Hence it is clear that
   \begin{align} \label{B9}
   \begin{minipage}{70ex}
{\em $\overline{\ceb{f}} = \ceb{\bar f}$ a.e.\ $[\nu]$
whenever $f\in L^p(\nu)$ and $p\in [1,\infty]$,}
   \end{minipage}
   \end{align}
and, by \eqref{B3},
   \begin{align} \label{B9.5}
   \begin{minipage}{70ex}
{\em $\int_X g f \D \nu = \int_X g \ceb{f} \D \nu$
whenever $g\colon X \to \cbb$ is $\bscr$-measur\-able,
$g \in L^q(\nu)$ and $f\in L^p(\nu)$, where $p,q \in
[1,\infty]$ are such that $\frac{1}{p} +
\frac{1}{q}=1$.}
   \end{minipage}
   \end{align}
Moreover, by applying \eqref{B6}, we obtain
   \begin{align} \label{B10}
   \begin{minipage}{70ex}
{\em $\ceb{g f} = g \ceb{f}$ a.e.\ $[\nu]$ whenever
$g\colon X \to \cbb$ is $\bscr$-measur\-able, $f\in
L^p(\nu)$, $gf \in L^r(\nu)$ and $p, r \in
[1,\infty]$.}
   \end{minipage}
   \end{align}
In the case of complex-valued functions, Lemma
\ref{BL1} takes the following form.
   \begin{pro} \label{BL3}
{\em (i)} If $p\in [1,\infty]$, then for every $f\in L^p(\nu)$,
   \begin{align*}
\text{$|\ceb{f}| \Le \ceb{|f|}$ a.e.\ $[\nu]$.}
   \end{align*}

\noindent {\em (ii)} If $p,q \in (1,\infty)$ and $\frac{1}{p} +
\frac{1}{q}=1$, then for all $f\in L^p(\nu)$ and $g \in L^q(\nu)$,
   \begin{align*}
|\ceb{f g}| \Le \ceb{|f|^p}^{1/p} \, \ceb{|g|^q}^{1/q} \text{ a.e.\
$[\nu]$.}
   \end{align*}

\noindent {\em (iii)} If $f \in L^1(\nu)$ and $g \in L^{\infty}(\nu)$,
then
   \begin{align*}
\text{$|\ceb{f g}| \Le \ceb{|f|} \|g\|_{L^{\infty}(\nu)}$ a.e.\ $[\nu]$.}
   \end{align*}
   \end{pro}
   \begin{proof}
(i) Let $g\colon X\to \cbb$ be a $\bscr$-measurable function such that
$|g| = 1$ and $|\ceb{f}| = g\ceb{f}$. Since $|f| - \mathrm{Re}(gf) \Ge 0$
and $\mathrm{Re}(gf), f \in L^p(\nu)$, we infer from \eqref{B4} that
$\ceb{\mathrm{Re}(gf)} \Le \ceb{|f|}$ a.e.\ $[\nu]$, and thus
   \begin{align*}
|\ceb{f}| & = g\ceb{f} \overset{\eqref{B10}} = \ceb{gf}
   \\
& = \mathrm{Re} (\ceb{gf}) \overset{\eqref{B9}} = \ceb{\mathrm{Re}(gf)}
\Le \ceb{|f|} \text{ a.e.\ $[\nu]$.}
   \end{align*}

The conditions (ii) and (iii) follow from (i) and
Lemma \ref{BL1}.
   \end{proof}
Now we prove that the operator of conditional
expectation is contractive with respect to each
$L^p$-norm, where $p\in [1,\infty]$.
   \begin{thm}  \label{Ath}
The following assertions hold:
   \begin{enumerate}
   \item[(i)] if $p\in [1,\infty)$, then $|\ceb{f}|^p
\Le \ceb{|f|^p}$ a.e.\ $[\nu]$ for each $f \in L^p(\nu)$,
   \item[(ii)] if $p\in [1,\infty]$, then
the mapping $L^p(\nu) \ni f \longmapsto \ceb{f} \in
L^p(\nu)$ is a linear contraction which leaves the
convex cone $L_+^p(\nu)$ invariant,
   \item[(iii)] the mapping $L^2(\nu) \ni f \longmapsto
\ceb{f} \in L^2(\nu)$ is an orthogonal projection.
   \end{enumerate}
   \end{thm}
   \begin{proof}
(i) The case of $p=1$ follows from Proposition
\ref{BL3}(i). Suppose $p\in (1,\infty)$ and $f\in
L^p(\nu)$. Since the measure $\nu|_{\bscr}$ is
$\sigma$-finite, there exists a sequence
$\{X_k\}_{k=1}^{\infty} \subseteq \bscr$ such that
$\nu(X_n) < \infty$ for every $n\in \nbb$, and $X_n
\nearrow X$ as $n \to \infty$. Applying Proposition
\ref{BL3}(ii) to the function $g=\chi_{X_n}$, which is
in $L^q(\nu)$, we get
   \begin{align*}
|\ceb{f}|^p \cdot \chi_{X_n} & \overset{\eqref{B10}}=
|\ceb{f \cdot \chi_{X_n}}|^p
   \\
& \hspace{2.3ex}\Le \ceb{|f|^p} \cdot \chi_{X_n}
\text{ a.e.\ $[\nu]$}, \quad n\in \nbb.
   \end{align*}
Passing to the limit with $n$ completes the proof of
(i).

(ii) The invariance is clear. If $p\in[1,\infty)$, then (i) yields
   \begin{align*}
\int_X |\ceb{f}|^p \D \nu \Le \int_X \ceb{|f|^p} \D \nu
\overset{\eqref{B1}}= \int_X |f|^p \D \nu, \quad f \in L^p(\nu).
   \end{align*}
If $p=\infty$, then by Proposition \ref{BL3}(i), we have
   \begin{align*}
\text{$|\ceb{f}| \Le \ceb{|f|} \overset{\eqref{B4}} \Le
\|f\|_{L^{\infty}(\nu)}$ a.e.\ $[\nu]$,} \quad f \in L^{\infty}(\nu),
   \end{align*}
which completes the proof of (ii).

(iii) Since, by (ii), the mapping $L^2(\nu) \ni f
\longmapsto \ceb{f} \in L^2(\nu)$ is a linear
contraction and, by \eqref{Dag9}, $\ceb{\ceb{f}} =
\ceb{f}$ for all $f\in L^2(\nu)$, we deduce that
$L^2(\nu) \ni f \longmapsto \ceb{f} \in L^2(\nu)$ is
an orthogonal projection. This completes the proof.
   \end{proof}
Below, we show that the operators of conditional
expectation with respect to two $\sigma$-algebras
coincide if and only if the relative $\nu$-complements
of these $\sigma$-algebras are equal to each other
(see Sect.\ \ref{sec-2.1} for definition).
   \begin{pro}\label{appcoin}
Let $\bscr, \tilde\bscr\subseteq \ascr$ be
$\sigma$-algebras such that the measures
$\nu|_{\bscr}$ and $\nu|_{\tilde\bscr}$ are
$\sigma$-finite. Then the following conditions are
equivalent{\em :}
   \begin{enumerate}
   \item[(i)] $\mathsf{E}(f;\bscr,\nu) =
\mathsf{E}(f;\tilde\bscr,\nu)$ a.e.\ $[\nu]$ for every
$\ascr$-measurable functions $f \colon X\to \rbop$,
   \item[(ii)]  $\bscr^\nu=\tilde\bscr^\nu$.
   \end{enumerate}
   \end{pro}
   \begin{proof}
(i)$\Rightarrow$(ii) It is enough to show that
$\bscr^\nu \subseteq \tilde\bscr^\nu$. Take $\varDelta
\in \bscr$. Since
   \begin{align*}
\chi_{\varDelta} =
\mathsf{E}(\chi_{\varDelta};\bscr,\nu) =
\mathsf{E}(\chi_{\varDelta};\tilde \bscr,\nu) \text{
a.e.\ $[\nu]$}
   \end{align*}
and the function $\mathsf{E}(\chi_{\varDelta};\tilde
\bscr,\nu)$ is $\tilde\bscr$-measurable, we infer from
\cite[Lemma 13.1]{b-j-j-sC} that $\chi_{\varDelta}$ is
$\tilde \bscr^\nu$-measurable. This implies that
$\varDelta \in \tilde\bscr^\nu$. Hence $\bscr^\nu
\subseteq \tilde\bscr^\nu$.

(ii)$\Rightarrow$(i) Let $\mathscr C$ denote any of
the $\sigma$-algebras $\bscr$ or $\tilde \bscr$. Take
an $\ascr$-measurable function $f \colon X\to \rbop$.
First, we show that
   \begin{align} \label{nabla}
\int_{\varDelta} f \D \nu = \int_{\varDelta}
\mathsf{E}(f;\mathscr C,\nu) \D \nu, \quad \varDelta
\in \mathscr C^\nu.
   \end{align}
Indeed, if $\varDelta \in \mathscr C^\nu$, then by
\eqref{complascr} there exists $\varDelta^{\prime} \in
\mathscr C$ such that $\nu(\varDelta \vartriangle
\varDelta^{\prime})=0$. Then
   \begin{align*}
\int_{\varDelta} f \D \nu = \int_{\varDelta^{\prime}}
f \D \nu \overset{\eqref{B1}}=
\int_{\varDelta^{\prime}} \mathsf{E}(f;\mathscr C,\nu)
\D \nu = \int_{\varDelta} \mathsf{E}(f;\mathscr C,\nu)
\D \nu,
   \end{align*}
which proves our claim. Applying \eqref{nabla} to the
$\sigma$-algebras $\bscr$ and $\tilde\bscr$ and using
the equality $\bscr^\nu=\tilde\bscr^\nu$, we obtain
   \begin{align*}
\int_{\varDelta} \mathsf{E}(f;\bscr,\nu) \D \nu =
\int_{\varDelta} f \D \nu = \int_{\varDelta}
\mathsf{E}(f;\tilde \bscr,\nu) \D \nu, \quad \varDelta
\in \bscr^\nu.
   \end{align*}
Since the functions $\mathsf{E}(f;\bscr,\nu)$ and
$\mathsf{E}(f;\tilde\bscr,\nu)$ are
$\bscr^\nu$-measurable and the measure
$\nu|_{\bscr^\nu}$ is $\sigma$-finite, the proof is
complete.
   \end{proof}
We conclude Appendix \ref{AppB} by making the
following important caution regarding the conditional
expectation.
   \begin{cau}
We say that the conditional expectation
$\mathsf{E}(\,\cdot\,;\cscr,\nu)$ exists if $\cscr
\subseteq \ascr$ is a $\sigma$-algebra and the measure
$\nu|_{\cscr}$ is $\sigma$-finite.
   \end{cau}
   \section{Powers of operators} \label{App}
   This section deals with operators whose powers are
all densely defined. We begin by stating the
following fact, which is a direct consequence of
the Mittag-Leffler theorem (see \cite[Lemma
1.1.2]{Sch}).
   \begin{lem}\label{closedpow}
Let $\{k_n\}_{n=0}^\infty$ be a strictly increasing
sequence of nonnegative integers. Suppose $A$ is an
operator in a complex Hilbert space $\hh$ such that
for every $n \in \zbb_+$, the normed space
\mbox{$(\dz{A^{k_n}}, \|\cdot\|_{A;k_n})$} is complete
and $\dz{A^{k_{n+1}}}$ is dense in $\dz{A^{k_n}}$ with
respect to the norm \mbox{$\|\cdot\|_{A;k_n}$}. Then
for every $n\in \zbb_+$, $\dzn{A}$ is dense in
$\dz{A^{k_n}}$ with respect to the norm
\mbox{$\|\cdot\|_{A;k_n}$}.
   \end{lem}
Theorem \ref{closedpow2} below is an abstract version
of Theorem 4.7 in \cite{b-j-j-sC}. The reader should
be aware of the fact that the former theorem does not
imply the latter because there exists a composition
operator $C_{\phi}$ which has a dense set of
$C^\infty$-vectors but $C_{\phi}^n$ is not closed for
every integer $n \Ge 2$ (see \cite[Example
5.4]{b-j-j-sC}).
   \begin{thm}\label{closedpow2}
Let $\{k_n\}_{n=0}^\infty$ be a strictly increasing
sequence of nonnegative integers. Suppose $A$ is a
closed operator in $\hh$ such that
   \begin{enumerate}
   \item[(i)] $A^{k_n}$ is closed
for every $n\in \zbb_+$,
   \item[(ii)] $\dz{A^{k_{n+1}}}$ is a
core for $A^{k_n}$ for every $n\in \zbb_+$.
   \end{enumerate}
Then for every $n\in \zbb_+$, the norms
$\|\cdot\|_{A;k_n}$ and $\|\cdot\|_{A^{k_n}}$ are
equivalent, and $\dzn{A}$ is a core for $A^{k_n}$.
   \end{thm}
\begin{proof}
Fix $n\in \zbb_+$. Since $A$ is closed, we deduce that
the normed space \mbox{$(\dz{A^{k_n}},
\|\cdot\|_{A;k_n})$} is complete (see
\cite[Proposition 1]{StSz4}). By (i), the normed space
\mbox{$(\dz{A^{k_n}}, \|\cdot\|_{A^{k_n}})$} is
complete. Since the identity mapping from
\mbox{$(\dz{A^{k_n}}, \|\cdot\|_{A;k_n})$} to
\mbox{$(\dz{A^{k_n}}, \|\cdot\|_{A^{k_n}})$} is
continuous, the inverse mapping theorem implies that
the norms \mbox{$\|\cdot\|_{A;k_n}$} and
\mbox{$\|\cdot\|_{A^{k_n}}$} are equivalent. Hence, by
(ii), $\dz{A^{k_{n+1}}}$ is dense in $\dz{A^{k_n}}$
with respect to the norm \mbox{$\|\cdot\|_{A;k_n}$}.
An application of Lemma \ref{closedpow} completes the
proof.
   \end{proof}
It is worth pointing out that closed paranormal
operators satisfy the assumption (i) of Theorem
\ref{closedpow2} (see \cite[Proposition
6(iv)]{StSz4}). In particular, this is the case for
closed subnormal, and consequently closed symmetric,
operators (see \cite[Proposition 1]{StSz3.5} for the
proof that symmetric operators are subnormal). Theorem
\ref{closedpow2} is no longer true if the assumption
(ii) is dropped, even in the class of closed symmetric
operators (see \cite[Theorem 4.8]{b-j-j-sC} which is
essentially due to Schm\"udgen \cite{Schm}).
   \end{appendix}
   \vspace{3ex}

\textbf{Acknowledgement}. A substantial part of this
paper was written while the first, the second and the
fourth author visited Kyungpook National University
during the spring and the autumns of 2013, 2016. They
wish to thank the faculty and the administration of
this unit for their warm hospitality.

   \bibliographystyle{amsalpha}
   
   \renewcommand\baselinestretch{1.3}
   \printindex
   \end{document}